%% file: DPPCA_Arxiv.tex
\documentclass[11pt]{article}

\input{commands.tex}

\title{High-Dimensional Asymptotics of Differentially Private PCA}
\author{Youngjoo Yun\thanks{Department of Statistics, University of Wisconsin–Madison. Email: \tt{yyun25@wisc.edu}} \and Rishabh Dudeja\thanks{Department of Statistics, University of Wisconsin–Madison. Email: \tt{rdudeja@wisc.edu}}}

\begin{document}
\maketitle
\input{main_abstract}
%\tableofcontents

%--------
% Codes from https://tex.stackexchange.com/questions/419249/table-of-contents-only-for-the-appendix
\doparttoc % Tell to minitoc to generate a toc for the parts
\faketableofcontents % Run a fake tableofcontents command for the partocs
\part{} % Start the document part
\parttoc % Insert the document TOC
%--------

%%%%%%%%%%%%%%%%%%%%%%%%
\input{main_introduction}
\input{main_preliminaries}
\input{main_main_results}

\input{main_related}
\input{main_proof_overview}
\section*{Acknowledgements}
The authors would like to thank Songbin Liu and Junjie Ma for helpful notes on the 1000 Genomes dataset and Gavin Brown for helpful discussions on differential privacy. The authors acknowledge support from NSF award DMS-2610474.
\bibliographystyle{abbrvnat}
\bibliography{refs}
\newpage
%%%%%%%%%%%%% Appendices %%%%%%%%

\appendix
%--------
% Codes from https://tex.stackexchange.com/questions/419249/table-of-contents-only-for-the-appendix
\addcontentsline{toc}{section}{Appendix} % Add the appendix text to the document TOC
\part{Appendix} % Start the appendix part
\parttoc % Insert the appendix TOC
%--------

\crefalias{section}{appendix}
\crefalias{subsection}{appendix}
\input{appendix_overlap}
\input{appendix_sampling}
\input{appendix_contiguity}
\input{appendix_privacy}
\input{appendix_rank}
\input{appendix_misc}
\input{appendix_notes_on_figures}
\input{appendix_survey}
\end{document}

%% file: commands.tex
\makeatletter
\renewcommand\paragraph{\@startsection{paragraph}{4}{\z@}%
                                    {1ex \@plus1ex \@minus.2ex}%
                                    {-1em}%
                                    {\normalfont\normalsize\bfseries}}
\makeatother
{\begin{list}               %    with flush left bullets.
    {$\bullet$ \hfill}{
        \setlength{\leftmargin}{\parindent}
        \setlength{\parsep}{0.04\baselineskip}
        \setlength{\itemsep}{0.5\parsep}
        \setlength{\labelwidth}{\leftmargin}
        \setlength{\labelsep}{0em}}
    }
{\end{list}}

\usepackage{amsmath,lipsum}
\usepackage{amssymb}
\usepackage{amsthm}
\usepackage{bbm} % added this for the indicator function
\usepackage{fancybox}
\usepackage[margin=1in]{geometry}
\usepackage[sort&compress,numbers]{natbib} % for citet, citep etc
\usepackage{color}
\usepackage{enumitem}
\usepackage{hyperref}
 \hypersetup{
     colorlinks=true,
     linkcolor=blue,
     filecolor=blue,
     citecolor = blue,      
     urlcolor=black,
     }
\usepackage{bm}
\usepackage{mathrsfs}
\usepackage[scr=boondox]{mathalpha}

\usepackage{extarrows}
\usepackage{graphicx}
\usepackage{subfig}
\usepackage{algorithmic,algorithm}
\usepackage{mathrsfs}
\usepackage{cmll}
\usepackage{mathtools}
\usepackage[cmyk]{xcolor} 
\usepackage[title]{appendix}
\usepackage{mleftright}

\usepackage{minitoc}

\usepackage{wrapfig}

\newtheorem{theorem}{Theorem}
\newtheorem{proposition}{Proposition}
\newtheorem{lemma}{Lemma}

\newtheorem{fact}{Fact} %for results from other papers
\theoremstyle{definition}
\newtheorem{claim}{Claim}
\newtheorem{definition}{Definition}
\newtheorem{assumption}{Assumption}

\theoremstyle{remark}
\newtheorem{remark}{Remark}

\usepackage{cleveref}
\usepackage{crossreftools}

\crefname{claim}{claim}{claims}
\crefname{assumption}{assumption}{assumptions}
\crefname{fact}{fact}{facts}

\providecommand{\R}{\ensuremath{\mathbb{R}}}
\providecommand{\N}{\ensuremath{\mathbb{N}}}
\renewcommand{\O}{\mathbb{O}} %% use for orthogonal group
\providecommand{\calM}{\mathcal{M}}
\providecommand{\calN}{\mathcal{N}}

\providecommand{\calQ}{\mathcal{Q}}
\providecommand{\calZ}{\mathcal{Z}}

\renewcommand{\vec}[1]{\ensuremath{{#1}}}
\providecommand{\mat}[1]{\ensuremath{{#1}}}
\providecommand{\mA}{\mat{A}} 
\providecommand{\mB}{\mat{B}}
 
\providecommand{\mD}{\mat{D}}
\providecommand{\mE}{\mat{E}}

\providecommand{\mI}{\mat{I}}

\providecommand{\mL}{\mat{L}} 
\providecommand{\mM}{\mat{M}}

\providecommand{\mQ}{\mat{Q}} 
\providecommand{\mR}{\mat{R}}
\providecommand{\mS}{\mat{S}} 
 
\providecommand{\mU}{\mat{U}} 
\providecommand{\mV}{\mat{V}}

\providecommand{\mX}{\mat{X}}

\providecommand{\mZ}{\mat{Z}}
\providecommand{\mGamma}{\mat{\Gamma}}

\providecommand{\va}{\vec{a}}

\providecommand{\ve}{\vec{e}}

\providecommand{\vm}{\vec{m}}

\providecommand{\vr}{\vec{r}}

\providecommand{\vu}{\vec{u}} 
\providecommand{\vv}{\vec{v}}

\providecommand{\vx}{\vec{x}} 

\providecommand{\vz}{\vec{z}}

\providecommand{\rvar}[1]{\ensuremath{\bm{#1}}}
\providecommand{\rA}{\rvar{A}} 
\providecommand{\rB}{\rvar{B}}

\providecommand{\rL}{\rvar{L}} 
\providecommand{\rM}{\rvar{M}} 
 
\providecommand{\rO}{\rvar{O}} 
 
\providecommand{\rQ}{\rvar{Q}} 
\providecommand{\rR}{\rvar{R}}
\providecommand{\rS}{\rvar{S}} 
 
\providecommand{\rU}{\rvar{U}} 
\providecommand{\rV}{\rvar{V}}

\providecommand{\rZ}{\rvar{Z}}

\providecommand{\rs}{\rvar{s}}

\providecommand{\rv}{\rvar{v}}
\providecommand{\rw}{\rvar{w}}
 
\providecommand{\ry}{\rvar{y}}
\providecommand{\rz}{\rvar{z}}

\newcommand{\ip}[2]{\left\langle {#1}, {#2} \right\rangle} %uinner product
\DeclareMathOperator{\diag}{Diag}
\DeclareMathOperator{\Tr}{Tr}
\DeclareMathOperator{\fr}{Fr}

\newcommand{\explain}[2]{\overset{\text{\tiny{#1}}}{#2}} %to provide explanations

\providecommand{\bydef}{\explain{def}{=}}

\newcommand*\diff{\mathop{}\!\mathrm{d}}

\newcommand{\unif}[1]{\mathsf{Unif}(#1)} %% use for Haar measure/ Uniform measure
\newcommand{\E}{\mathbb{E}} %% for expectations
\renewcommand{\P}{\mathbb{P}} %% for probabilities
\newcommand{\Q}{\mathbb{Q}} %% for probabilities
\newcommand{\Var}{\mathrm{Var}}
\newcommand{\Cov}{\mathrm{Cov}}

\newcommand{\gauss}[2]{\mathsf{N}\!\left( #1,#2 \right)} %% for Gaussian distribution 

\newcommand{\gausss}[2]{\mathsf{N}( #1,#2 )} %% for Gaussian distribution (it's the version that doesn't make the size of the brackets variable)

\newcommand{\pc}{\overset{\P}{\rightarrow}} %conv. in P
\newcommand{\dc}{\explain{d}{\rightarrow}}

\newcommand{\ls}{\underset{\dim\rightarrow\infty}{\lim\sup\;}} % limsup (added) 

\newcommand{\li}{\underset{\dim\rightarrow\infty}{\lim\inf\;}} % liminf (added)
\newcommand{\lm}{\lim_{\dim\rightarrow\infty}} % lim (added)

\renewcommand{\tilde}{\widetilde}
\renewcommand{\hat}{\widehat}
\DeclareFontFamily{U}{mathx}{\hyphenchar\font45}
\DeclareFontShape{U}{mathx}{m}{n}{
      <5> <6> <7> <8> <9> <10>
      <10.95> <12> <14.4> <17.28> <20.74> <24.88>
      mathx10
      }{}
\DeclareSymbolFont{mathx}{U}{mathx}{m}{n}
\DeclareFontSubstitution{U}{mathx}{m}{n}
\DeclareMathAccent{\widecheck}{0}{mathx}{"71}
\DeclareMathAccent{\wideparen}{0}{mathx}{"75}

\renewcommand{\bar}{\overline}
\providecommand{\tv}{\mathrm{d}_\mathrm{TV}}

\providecommand{\ind}{\mathbbm{1}}

\providecommand{\ssize}{n}
\renewcommand{\dim}{p}
\newcommand{\rnk}{k}

\providecommand{\mSigma}{\mat{\Sigma}}
\providecommand{\rLambda}{\bm{\Lambda}}

\providecommand{\mSigmaR}{\mSigma_{\mathrm{rk}}}

\providecommand{\mSigmaRb}{\bar{\mSigma}_{\mathrm{rk}}}

\DeclarePairedDelimiterX{\tfx}[2]{[}{]}{%
  #1\;\delimsize\|\;#2%
}
\newcommand{\tf}{\mathsf{T}\tfx}
\DeclarePairedDelimiterX{\tpx}[2]{[}{]}{%
  #1\;\delimsize\|\;#2%
}

\DeclarePairedDelimiter{\norm}{\lVert}{\rVert}

\DeclarePairedDelimiterX{\renyix}[3]{[}{]}{%
  #2\;\delimsize\|\;#3%
}
\newcommand{\renyi}{\mathsf{D}}
\newcommand{\rdv}[3]{\renyi_{#1}\renyix{#1}{#2}{#3}}

\newcommand{\indep}{\perp\!\!\!\perp}

\newcommand{\mref}{\ref}
\newcommand{\meqref}{\eqref}
\let\oldparagraph\paragraph
\renewcommand{\paragraph}[1]{\oldparagraph{#1.}}

%% file: main_abstract.tex
\begin{abstract}In differential privacy, random noise is introduced to privatize summary statistics of a sensitive dataset before releasing them. The noise level determines the privacy loss, which quantifies how easily an adversary can detect a target individual's presence in the dataset using the published statistic. Most privacy analyses provide non-asymptotic upper bounds on the privacy loss which hold uniformly across all datasets. Sometimes, these bounds can be pessimistic on a given dataset. In such cases, it can be useful to complement these privacy bounds with sharp privacy characterizations that quantify a mechanism's exact privacy loss on a given dataset. With this goal, we study differentially private principal component analysis (PCA), where the goal is to privatize the leading principal components of a dataset with $n$ samples and $p$ features. We analyze the exponential mechanism and provide sharp asymptotic characterizations of its utility and privacy loss in the high-dimensional limit ($p \rightarrow \infty$). We show that in this limit, detecting a target individual's presence using privatized principal components is asymptotically equivalent to distinguishing between two Gaussians with different means, where the mean difference depends on certain spectral properties of the dataset. Our analysis combines the hypothesis-testing formulation of privacy guarantees proposed by Dong, Roth, and Su (2022) with Le Cam's contiguity arguments.

\paragraph{Keywords} principal component analysis, high-dimensional asymptotics, differential privacy, contiguity, spherical integrals, spherical spin glasses.
% \RD{In differential privacy, random noise is introduced to privatize statistics of a sensitive dataset. The noise level determines the privacy guarantee, which often comes in the form of a worst-case bound on the privacy loss over all datasets. These bounds often suggest adding excessive noise that overwhelms meaningful signals, serving as a conservative estimate of the actual privacy loss on the given dataset.} This paper explores whether we can obtain sharp privacy characterizations that identifies the exact privacy loss on a given dataset. We study this problem in the context of differentially private principal component analysis, where the goal is to privatize the leading principal components (PCs) of a dataset with $n$ samples and $p$ features. We analyze the exponential mechanism for this problem in a model-free setting and provide sharp utility and privacy characterizations in the high-dimensional limit ($p \rightarrow \infty$). Our privacy result shows that, in high dimensions, detecting the presence of a target individual in the dataset using the privatized PCs is exactly as hard as distinguishing two Gaussians with slightly different means, where the mean difference depends on certain spectral properties of the dataset. Our privacy analysis combines the hypothesis-testing formulation of privacy guarantees proposed by Dong, Roth, and Su (2022) with classical contiguity arguments due to Le Cam to obtain sharp high-dimensional privacy characterizations.
\end{abstract}

%% file: main_introduction.tex
\section{Introduction}\label{sec:introduction}
In many applications, we analyze high-dimensional datasets containing sensitive information, such as genetic data, electronic health records, and opinion polls on controversial issues. In such cases, even publishing aggregate statistics of the dataset can create privacy risks for participants. Indeed, the work of \citet{homer2008resolving} shows that one can infer whether a target individual is in a DNA database from commonly published aggregate statistics (see \citep{dwork2017exposed} for a survey on such privacy-violating attacks).

\paragraph{Differential Privacy} To address this concern, Differential Privacy (DP), introduced by \citet{dwork2006calibrating}, provides a mathematical framework for privacy-preserving statistical analysis of sensitive datasets. Instead of publishing the statistic as it is, one releases a privatized version, which is a random, noisy version of the statistic. The noise prevents adversaries from inferring sensitive information about individuals in the dataset, and the privacy guarantee reassures the study participants that their data is safe. DP provides an information-theoretic impossibility result that the noise level used for privatization is sufficiently high to ensure that no adversary can reliably detect even the presence or absence of a target individual in the dataset, let alone infer their personal information. 

\paragraph{Limitations of Existing Privacy Bounds} In differential privacy, most existing privacy analyses are examples of non-asymptotic \emph{upper bounds} on the privacy loss of a privatized statistic, which hold uniformly across all datasets. These bounds quantify how easily an adversary can detect the presence or absence of a target individual in the dataset using the published value of the privatized statistic.  Sometimes, these bounds can be pessimistic, prescribing excessive amounts of noise that overwhelm meaningful signals in the statistic, to achieve reasonable privacy guarantees. This paper studies a prototypical situation where this can happen: differentially private principal component analysis (PCA). Here, we are given a dataset $\mX$ consisting of $n$ samples, each represented by a $p$-dimensional feature vector. The goal is to privatize the top $k$ principal components (PCs) or, equivalently, eigenvectors of the empirical second moment matrix $\mSigma(\mX)$ of the dataset. One method for doing so is via the exponential mechanism \citep{mcsherry2007mechanism,chaudhuri2013near}, which generates the privatized PCs by sampling them from a Gibbs distribution with the following density with respect to the uniform (or Haar) distribution on $\O(\dim,\rnk)$ (the set of $\dim \times \rnk$ column-orthonormal matrices):
\begin{align} \label{eq:gibbs-intro}
    \nu(\diff \mV \mid \mSigma(\mX), \beta, \rnk) \propto e^{ \frac{\dim\beta}{2} \Tr[\mV^\top \mSigma(\mX) \mV]}.
\end{align}
The parameter $\beta \geq 0$ controls the noise level used for privatization. For large values of the noise parameter $\beta$, the privatized PCs contain very little noise and are close to the true PCs, but offer weak privacy protection. For small values of $\beta$, the privatized PCs are extremely noisy, but offer strong privacy protection.
\paragraph{An Illustrative Example} As an illustration, we apply the exponential mechanism on a subset of the 1000 Genomes dataset \citep{10002015global} consisting of $n = 2504$ samples with $p = 200$ genetic features (corresponding to 200 single nucleotide polymorphisms (SNPs)) per sample. As shown in the left-hand plot of \Cref{fig:proj}, the data points, when projected onto the true (non-private) PCs, organize into clusters, reflecting the ancestry of the subjects.  We compare this plot with the corresponding plots for the privatized PCs in \Cref{fig:proj}.
\begin{itemize}
\item At small values of $\beta$, such as  $\beta=0.01$, a well-known privacy bound for the exponential mechanism due to \citet{chaudhuri2013near} shows that the privatized PCs satisfy a strong $0.16$-DP guarantee. Roughly speaking, this guarantee states that detecting the presence or absence of a target individual in the dataset from the privatized PCs is at least as difficult as distinguishing between $\mathrm{Bern}(0.46)$ and $\mathrm{Bern}(0.54)$ (based on a single sample) \citep{wasserman2010statistical,dong2022gaussian,kairouz2015composition}. This is, information-theoretically, a very challenging testing problem. However, at such low values of $\beta$, the noise in the privatized PCs overwhelms the signal, and the clustering pattern in the PCA plot is completely lost. 
\item At higher values of $\beta$,  such as $\beta = 0.2$, the privatized PCs remain noisy and fail to capture the clustering in the dataset. At this point, the privacy bound of \citep{chaudhuri2013near} already becomes quite weak and provides a $3.19$-DP guarantee for the privatized PCs, which translates to distinguishing between $\mathrm{Bern}(0.04)$ and $\mathrm{Bern}(0.96)$, a relatively easy testing problem: the Neyman-Pearson test can reliably distinguish between these distributions with a power of $0.96$ and a false rejection probability of $0.05$.
\item To capture the clustering pattern in the dataset, one needs to use a high value of $\beta$ such as $\beta = 1.2$. However, at such a high value of $\beta$, the privacy bound of \citep{chaudhuri2013near} is essentially vacuous and provides a $19.17$-DP guarantee, which translates to testing between Bern($5\times10^{-9}$) and Bern(1).
\end{itemize}
Therefore, in this example, a well-known non-asymptotic privacy bound for the exponential mechanism provides weak privacy guarantees unless the noise level is so high that it overwhelms the signal in the PCs.
\begin{figure}
    \centering
    \includegraphics[width=\linewidth]{plots/1000Genomes_Figure1.pdf}
    \caption{Projections of the 1000 Genomes dataset onto the span of the first $\rnk=2$ PCs. Left to right: (1) non-private PCs, (2--4)  privatized PCs with $\beta=0.01$, $\beta=0.2$, and $\beta = 1.2,$ respectively. }\label{fig:proj}
\end{figure}
\paragraph{Main Objective} One possible explanation for the pessimism of non-asymptotic privacy bounds depicted in \Cref{fig:proj} is that they might be loose upper bounds on the actual privacy loss of the exponential mechanism on the given dataset, which might be significantly lower.  This motivates studying the mechanism's \emph{exact privacy loss} on the given dataset, which is the central question studied in this paper:
\begin{equation}
  \tag{Q} \label{eq:key-question}
  \parbox{\dimexpr\linewidth-4em}{%
    \strut
    \emph{Can we obtain exact theoretical characterizations of the privacy loss and utility (or estimation error) of a mechanism on a given dataset?}
    \strut
  }
\end{equation}
Such exact, data-dependent characterizations can be useful for two reasons:
\begin{itemize}
    \item For a given estimation problem, often there are many differentially private mechanisms which attain the minimax optimal rate in the sense that their estimation error has an optimal dependence on $\dim,\ssize$ under a suitable measure of the privacy loss, up to unspecified universal constants. These constants can be important in empirical applications. Sharp data-dependent utility and privacy characterizations allow for a fine-grained comparison between the privacy-utility trade-offs of various competing mechanisms. Naturally, characterizations are the first step towards understanding whether the privacy-utility trade-off of a mechanism is Pareto-optimal or whether another mechanism can improve both its utility and privacy loss on a given dataset. 
     \item Second, such exact, data-dependent characterizations help a data analyst understand the exact privacy cost of publishing a privatized statistic for a given dataset, complementing worst-case privacy bounds, which can sometimes be pessimistic. Indeed, a similar motivation underlies several prior works  \citep{papernot2016semi,papernot2018scalable,wang2019per}, which develop data-dependent privacy guarantees that can be more informative than worst-case privacy guarantees.  Understanding the gap between worst-case privacy guarantees and the privacy loss for a mechanism on a given dataset also motivates several works in privacy auditing \citep{jagielski2020auditing,nasr2021adversary,lu2022general,steinke2023privacy} that use carefully designed adversarial attacks to obtain empirical privacy loss lower bounds that complement privacy loss upper bounds from worst-case privacy analysis. 
\end{itemize} 
\paragraph{Our Contributions} In this paper, we use the exponential mechanism as a testbed to study the question \eqref{eq:key-question}. For complex mechanisms, such as the exponential mechanism, obtaining sharp utility and privacy characterizations can be challenging. Indeed, it seems unlikely that the exact utility or privacy loss of the exponential mechanism on a given dataset will have a clean and explicit formula.  To overcome this difficulty, we take inspiration from a growing line of work on high-dimensional or mean-field asymptotics in statistics (see \citep{montanari2018mean} for a recent review), which has shown it is often possible to obtain exact characterizations of the estimation error for popular estimators in the high-dimensional limit as $\dim \rightarrow \infty$. This surprising phenomenon was highlighted as a ``blessing of dimensionality'' in an influential lecture by \citet{donoho2000highdimensional}. In fact, \citet{dwork2024differentially} and \citet{bombari2025better} have recently exploited tools from high-dimensional asymptotics to obtain exact characterizations of the \emph{estimation error} of popular privatized estimators for linear regression. We build on this line of work and show that it is also possible to obtain sharp characterizations of the privacy loss in the high-dimensional limit. Our main contributions are as follows.
% Motivated by these limitations, this work investigates whether it is possible to obtain \emph{sharp privacy characterizations} (as opposed to privacy bounds) that pin down the exact privacy loss of a mechanism on a given dataset. For more complicated mechanisms, such as the exponential mechanism, obtaining sharp privacy characterizations appears to be a challenging task. However, we show that this is feasible in the high-dimensional limit ($\dim \rightarrow \infty$). In a sense, our results can be viewed as privacy counterparts to a line of work in high-dimensional or mean-field asymptotics in statistics (see \citep{montanari2018mean} for a recent review), which obtains exact characterizations of the estimation error for popular estimators. \RD{Recently, \citet{dwork2024differentially} and \citet{bombari2025better} have extended this line of work to the analysis of privatized estimators for linear regression, deriving their estimation errors in the high-dimensional limit. Our estimation error analysis serves as an analogous result for PCA. For the privacy loss, however, the works provide a finite-sample lower bound whereas we characterize it in the same high-dimensional regime as our estimation error analysis.}

% Our main contributions are as follows. 
\begin{itemize}
\item \emph{Utility or Estimation Error Analysis:} We characterize the exact limit of the estimation error of the exponential mechanism as $\dim \rightarrow \infty$ (\Cref{thm:utility}). We find that the exponential mechanism exhibits several interesting phase transitions. 
\item \emph{Privacy Analysis:}  We provide a sharp privacy analysis of the exponential mechanism in the high-dimensional limit (\Cref{thm:privacy}). At a high level, our result says that detecting the presence or absence of a target individual in the dataset based on the output of the exponential mechanism on a dataset $X$ is asymptotically exactly as hard as distinguishing between $\gauss{0}{1}$ and $\gauss{\sigma_\beta}{1}$, where the parameter $\sigma_\beta$ has an explicit formula in terms of the noise parameter $\beta$, as well as the asymptotic spectral properties of the dataset $\mX$. Our characterization shows that the exponential mechanism exhibits an interesting \emph{privacy plateau}, a regime where decreasing $\beta$ increases noise in the privatized PCs and degrades utility, without any asymptotic reduction in privacy loss.
\item \emph{A Sampling Algorithm for the Exponential Mechanism:} As a proof technique to analyze the exponential mechanism, we design a simple sampling algorithm for the exponential mechanism. This algorithm generates an approximate sample from the Gibbs distribution \eqref{eq:gibbs-intro} and achieves vanishing total variation error as $\dim \rightarrow \infty$, whenever the exponential mechanism has a non-trivial utility (\Cref{thm:sampling}). This algorithm sheds light on the noise introduced by the exponential mechanism and shows that, in high dimensions, it can be approximated by a non-isotropic Gaussian distribution. The variance of this noise in each direction is calibrated according to the sensitivity of the leading PCs in that direction. As a result, there is more noise in highly sensitive directions and less noise in the remaining directions.
\item \emph{Proof Techniques:} Our work is inspired by a line of work \citep{wasserman2010statistical,kairouz2015composition,dong2022gaussian}, which explores connections between differential privacy and hypothesis testing. Specifically, we rely on the insights of \citet{dong2022gaussian}, who show that the exact privacy loss of a mechanism can be captured by the trade-off function (or the power of the optimal test as a function of the significance level) for the hypothesis testing task of distinguishing between the output distributions of the mechanism on a given dataset and a neighboring dataset, constructed by adding or removing a single data point from the given dataset. Computing these trade-off functions exactly for complicated mechanisms can be challenging, so we resort to asymptotic arguments. We show that in the high-dimensional regime, the output distributions of the exponential mechanism on two neighboring datasets are mutually contiguous and hence, the corresponding trade-off function can be derived using classical techniques due to \citet{le2012asymptotic}. While we focus on the exponential mechanism for differentially private PCA in this paper, we anticipate that this contiguity-based approach could be used to obtain analogous results for other mechanisms. %, and we discuss some possibilities in the Future Work section (\Cref{sec:future-work}). 
\end{itemize}
While many works in high-dimensional asymptotics make distributional assumptions on the dataset, we emphasize that our results above hold under certain deterministic conditions on the dataset and do not require that the dataset is generated from a statistical model.  We also design experiments that try to empirically measure the finite-$\dim$ utility and privacy loss of the exponential mechanism on two datasets: the 1000 Genomes dataset \citep{10002015global} and the American National Election Studies (ANES) survey dataset \citep{ANES2024Pilot}. We find that, in many cases, the asymptotic predictions derived from our results closely track the empirically measured utility and privacy loss of the exponential mechanism on these datasets. This suggests that although our results are asymptotic, they can provide good approximations to the finite-$\dim$ utility and privacy loss of the exponential mechanism in some situations.
% As a consequence, we find that our asymptotic predictions closely track the empirically measured utility and privacy loss of the exponential mechanism on the 1000 Genomes dataset and the American National Election Studies (ANES) survey dataset \cite{ANES2024Pilot}. This suggests that even though our results are asymptotic, they can provide good approximations to the finite-$\dim$ utility and privacy loss of the exponential mechanism on some real-world datasets.
Lastly, we also address the problem of data normalization, which arises when applying our results to real-world datasets. Like many prior works on differentially private PCA in deterministic settings \citep{chaudhuri2013near,  amin2019differentially, leake2021sampling, mangoubi2022re, dwork2014analyze, gonem2018smooth}, we assume that the dataset is normalized so that every data point has a bounded norm and work with the second moment matrix (as opposed to the sample covariance matrix), which implicitly assumes that the dataset is centered. To apply such procedures to real-world datasets, one often preprocesses them to satisfy these assumptions. However, preprocessing can incur additional privacy loss as it is often data-dependent. In \Cref{app:rank},
%the supplement \citep[Appendix E]{supplementary},
we analyze a natural data normalization procedure based on rank transformation, which is often used in the statistics literature to handle non-Gaussian, heavy-tailed, or contaminated data (see e.g., \citep{xue2012regularized,liu2012high,mitra2014multivariate,cape2024robust,liao2025testing}). We also provide an end-to-end privacy analysis for the resulting algorithm (see \Cref{thm:privacy-rank}).
 %the supplement \citep[Theorem E.1]{supplementary}).
\paragraph{Concurrent Work} We also refer the reader to the interesting parallel work of \citet{lin2025high}, which studies noisy stochastic gradient descent (SGD) for differentially private high-dimensional linear regression with the least squares loss and ridge penalty in the model-based setting. Like this paper, \citet{lin2025high} propose high-dimensional asymptotic privacy characterizations for noisy SGD, assuming as a conjecture that, in the high-dimensional limit ($\dim \to \infty$), the distribution of the noisy SGD iterates is indistinguishable from the law of a certain diffusion process called noisy homogenized SGD.

%\paragraph{Proof Techniques} 
% Finally, our proofs also require certain probabilistic results about the Gibbs distribution in \eqref{eq:gibbs-intro}. This distribution has been studied in a line of work in probability \cite{guionnet2005fourier, guionnet2021asymptotics,husson2025spherical}. In particular, a result by \citet{guionnet2005fourier} on the asymptotics of the log-normalizing constant of the Gibbs distribution, as well as a subsequent refinement of their result by \citet{guionnet2021asymptotics}, play a crucial role in the proof of our utility result. Our privacy theorem requires CLT-type results for privatized PCs drawn from the Gibbs distribution in \eqref{eq:gibbs-intro}, which don’t appear to be readily available in the literature. To address this, we will build on the proofs of \citet{guionnet2005fourier} and design a simple sampling algorithm for the exponential mechanism. This algorithm approximates the distribution of the privatized PCs in total variation distance by a non-linear transformation of independent Gaussian random variables. CLTs for the latter can be easily derived using the second-order Poincaré inequality developed by \citet{chatterjee2009fluctuations}.

\paragraph{Notations} We end the introduction by defining some frequently used notations.

\indent\emph{Asymptotics.} We use the standard order notations $O(\cdot)$ and $o(\cdot)$, and use them interchangeably with $\lesssim$ and $\ll,$ respectively. 
% For limits, we use $\bar{\lim}_{\dim\rightarrow\infty}$ and $\underline{\lim}_{\dim\rightarrow\infty}$ to denote the limit superior and the limit inferior, respectively, as $\dim\rightarrow0.$

\indent\emph{Sets.} We use $\R$ and $\N$ to denote the set of real numbers and the set of natural numbers (zero excluded), respectively. For $\dim, \rnk \in \N$,  $\O(\dim,\rnk)$ denotes the set of $\dim \times \rnk$ matrices with orthonormal columns,  $\O(\rnk)$ denotes the set of $\rnk \times \rnk$ orthogonal matrices, and  $[\dim]$ 
denotes the set $\{1, 2, \dotsc,\dim\}.$

\indent\emph{Linear algebra.} We index the entries of a vector and a matrix using subscripts: $\vv_i$ denotes the $i$th entry of a vector $\vv,$ and $\mV_{ij}$ denotes the $ij$th entry of a matrix $\mV.$ For $\dim,\rnk\in\N,$ we use $\mI_\rnk\in\R^{\rnk\times\rnk}$ to denote the identity matrix, $0_{\dim\times\rnk}\in\R^{\dim\times\rnk}$ to denote the matrix of 0s, and $1_\rnk\in\R^\rnk$ to denote the vector of 1s. For a matrix $\mV$, $\vv_i$ denotes its $i$th row, unless specified otherwise. For matrix and vector norms, $\|\vv\|$ for a vector $\vv$ denotes its $\ell_2$ norm; $\|\mV\|$ and $\|\mV\|_{\fr}$ for a matrix $\mV$ denote its operator norm and Frobenius norm, respectively. We let $\langle \mU,\mV\rangle \explain{def}{=} \sum_{ij} U_{ij} V_{ij} $ denote the Frobenius inner product between the matrices $\mU,\mV$ of compatible dimensions. For a symmetric matrix $S$ with eigendecomposition $S=U\Lambda U^\top,$ $(S)_+$ denotes the matrix $U\max\{\Lambda,0\} U^\top,$ where $\max\{\cdot,0\}$ is applied to each eigenvalue. We let $\lambda_i(\mS)$ and $\vu_i(\mS)$ denote its $i$th largest eigenvalue and the corresponding eigenvector, respectively. If $\mS$ is additionally positive semi-definite, we let $\mS^{1/2}$ denote its symmetric square root, which is defined as $\mS^{1/2} \explain{def}{=} \mU \Lambda^{1/2} \mU^\top.$ 

\indent\emph{Probability.} 
$\gauss{\vm}{\mS}$ denotes the Gaussian distribution with mean vector $\vm$ and covariance matrix $\mS,$ and  $\xi_{\dim,\rnk}$ denotes the uniform (Haar) distribution on $\O(\dim,\rnk)$.  A random matrix $\rO \sim \xi_{\dim, \rnk} \bydef \unif{\O(\dim,\rnk)}$ if and only if $\rO^\top \rO = \mI_{\rnk}$ and $\rO \explain{d}{=} \mA \rO$ for any deterministic orthogonal matrix $\mA \in \O(\dim)$ \citep[Section 1]{meckes2019random}. 
% In $\xi_{\dim,1}$ is the uniform distribution on $\O(\dim,1)$ and $\xi_{\rnk,\rnk}$ is the uniform distribution on $\O(\rnk).$ 
For two probability measures $\mu$ and $\nu$, $\mu \otimes \nu$ denotes their product measure. %For random variables $(\rz_1, \dotsc, \rz_\dim)$, we shorthand the statement $(\rz_1, \dotsc, \rz_\dim) \sim \mu_1 \otimes \mu_2 \otimes \dotsb \otimes \mu_\dim$ as $\rz_i \explain{$\indep$}{\sim} \mu_i$ for $i = 1,2 \dotsc, \dim$, 
For random variables $(\rz_1, \dotsc, \rz_\dim)$, we use the shorthand expression $\rz_i \explain{$\indep$}{\sim} \mu_i$ for $i \in[\dim]$ to denote that $(\rz_1, \dotsc, \rz_\dim) \sim \mu_1 \otimes \mu_2 \otimes \dotsb \otimes \mu_\dim,$ where the symbol $\indep$ denotes independent sampling of random variables.  We also  frequently use the usual stochastic order notations $O_\P(\cdot), o_\P(\cdot).$  

%% file: main_preliminaries.tex
\section{Preliminaries}\label{sec:preliminaries}
\paragraph{Setup} We are interested in the setting where we run principal component analysis on a  dataset\footnote{In this paper, we consider datasets as multi-sets. In particular, $X$ may have repeated elements since two or more data points can have the same feature representations.} $\mX$: 
\begin{align*}
    \mX\bydef \{\vx_1, \vx_2, \dotsc,\vx_\ssize\},
\end{align*}
consisting of $\ssize$ data points,  represented using $\dim$-dimensional feature vectors $\vx_{1:\ssize} \in \R^\dim$. Following standard assumptions in the literature \citep{chaudhuri2013near,  amin2019differentially, leake2021sampling, mangoubi2022re, dwork2014analyze, gonem2018smooth}, we assume that the dataset has been centered and normalized so that the data points satisfy the norm constraint:
\begin{align*}
    \|\vx_i\| \leq \sqrt{\dim} \quad \forall \; i \; \in \; [\ssize].
\end{align*}
Let $\mSigma(\mX)$ denote the empirical second moment matrix of the dataset, with eigendecomposition:
\begin{align*}
    \mSigma(\mX)\bydef\frac{1}{|\mX|}\sum_{\vx\in\mX}\vx\vx^\top  = \sum_{i=1}^\dim \lambda_i \vu_i \vu_i^\top,
\end{align*}
where $\lambda_{1} \geq \lambda_2 \geq \dotsb \geq \lambda_\dim$ denote the eigenvalues of $\mSigma(X)$ and $\vu_1, \vu_2, \dotsc, \vu_\dim$  the corresponding eigenvectors. For a given $\rnk \in \N$, our goal is to publish the top $\rnk$ PCs $\mU_\star \in \O(\dim,\rnk)$ of $\mSigma(\mX)$:
\begin{align*}
    \mU_\star  \bydef\begin{bmatrix}\vu_1& \vu_2& \cdots &\vu_\rnk\end{bmatrix}.
\end{align*}
However, publishing $\mU_\star$ may reveal sensitive information about the individuals in the dataset to a potential adversary. Hence, we construct and analyze privatized PCs $\rV \in \O(\dim,\rnk)$ that can be safely released and serve as a reasonable approximation for $\mU_\star$.

\begin{remark}[Empirical v.s. Population Differentially Private PCA] \label{rem:empirical-vs-pop} In this paper, we study the privacy-preserving PCA problem in the empirical setting, where one is given a deterministic dataset $\mX$, and the goal is to construct a private approximation of the leading eigenvectors of the empirical second moment matrix $\mSigma(X)$ of the given dataset. A different formulation of this problem looks at the population setting, where one assumes that the dataset is sampled from an underlying model (such as the spiked covariance model) and the goal is to privately approximate the leading eigenvectors of the population second moment matrix. Both settings are well-studied in the literature, and there are subtle differences between the sample complexities of the two formulations (see \Cref{sec:related-work} for a detailed discussion). In this paper, we focus exclusively on the empirical setting. 
\end{remark}

\paragraph{The Exponential Mechanism} A privatized statistic is a noisy version of the non-private statistic of interest. The noise prevents an adversary from reliably inferring sensitive information about the individuals in the dataset from the released value of the statistic.  This paper studies a specific method to construct privatized PCs called the exponential mechanism, which was introduced by \citet{mcsherry2007mechanism} and first analyzed in the context of PCA by \citet{chaudhuri2013near}.
\begin{definition}[Exponential Mechanism for Differentially Private PCA {\citep{mcsherry2007mechanism,chaudhuri2013near}}]\label{def:Gibbs}
    For $\beta \geq 0$, $k \in \N$, and $\Sigma \in \R^{\dim \times \dim}$,  let $\nu(\cdot \mid \mSigma, \beta, \rnk)$ denote the Gibbs distribution on $\O(\dim,\rnk)$ with density:
\begin{align} \label{eq:gibbs-def}
    \frac{\diff \nu(\mV \mid \Sigma, \beta, \rnk)}{\diff \xi_{\dim,\rnk}} & = \frac{1}{Z(\mSigma, \beta, \rnk)}  \cdot \exp\left( \frac{p\beta}{2} \Tr[ \mV^\top \mSigma \mV]\right),
\end{align}
where $\xi_{\dim,\rnk} \bydef \unif{\O(\dim,\rnk)}$ and $Z(\mSigma,\beta, \rnk)$ denotes the normalizing constant:
\begin{align*}
Z(\mSigma,\beta,\rnk) \explain{def}{=}  \E_{\rV \sim \xi_{\dim,\rnk}} \left[ \exp\left( \frac{p\beta}{2}  \Tr[ \rV^\top \mSigma\rV]\right) \right].
\end{align*}
% \begin{itemize}
%     {\color{magenta}\item We discussed possibly adding more citation(s) for the Gibbs distribution.}
% \end{itemize}
The exponential mechanism (\Cref{alg:ExpM}) for differentially private PCA takes three inputs: a dataset $\mX$, the noise parameter $\beta \geq 0,$ and the rank (number of PCs) $\rnk \in \N$. It generates the privatized PCs $\rV$ by sampling them from the Gibbs distribution $\nu(\cdot \mid \mSigma(\mX), \beta, \rnk)$, where $\mSigma(\mX)$ is the second moment matrix of $\mX$.
\begin{algorithm}[H]
\caption{\textsc{Exponential Mechanism}($X,\beta,\rnk$)}
\label{alg:ExpM}
\begin{algorithmic}
\STATE {\textit{Input:}} Dataset $\mX \subset \R^{p}$, noise parameter $\beta > 0$, rank (number of PCs) $\rnk\in [\dim].$
\STATE {\textit{Output:}} Privatized PCs $\rV \in \O(\dim,\rnk).$
\begin{itemize}
\item Sample {orthonormal vectors} $\rV$ from $\nu(\cdot \mid \mSigma(\mX), \beta, \rnk ).$
\end{itemize}
\STATE \textit{Return:} Privatized PCs $\rV$.
\end{algorithmic}
\end{algorithm}
\end{definition}
As illustrated in \Cref{fig:proj}, the parameter $\beta$ controls the amount of noise introduced in the privatized PCs and hence manages the trade-off between the utility (measured by the error between the true and privatized PCs) and the privacy loss:
\begin{itemize}
    \item For large values of $\beta$, the Gibbs distribution $\nu(\cdot \mid \mSigma(\mX),  \beta, \rnk)$ concentrates on the directions $\mV \in \O(\dim, \rnk)$ that maximize the captured variance $\Tr[\mV^\top \mSigma(\mX) \mV]$. These variance-maximizing directions correspond to any orthonormal basis for the subspace spanned by the true PCs $\mU_\star$. In this situation, the privatized PCs $\rV$ contain very little noise, and the error between the subspaces spanned by the privatized PCs and the true PCs is small. The small noise level also means that the privacy loss for large $\beta$ is higher. 
    \item As $\beta$ decreases, the privatized PCs become noisier, and the error between the subspaces spanned by the true PCs and the privatized PCs increases. The increase in the amount of noise ensures a decrease in the privacy loss. In the extreme case when $\beta = 0$, the privatized PCs $\rV \sim \unif{\O(\dim,\rnk)}$ consist of pure noise and are perfectly private since they carry no information about the dataset. 
\end{itemize}

% \begin{figure}[H]
%     \centering
%     \includegraphics[width=0.3\linewidth]{plots/preliminaries/preliminaries_nonPriv.pdf}\hspace{-0.5cm}
%     \includegraphics[width=0.3\linewidth]{plots/preliminaries/preliminaries_beta2.81.pdf}\hspace{-0.5cm}
%     \includegraphics[width=0.3\linewidth]{plots/preliminaries/preliminaries_beta0.pdf}
    
%     \caption{Projections of the 1000 Genomes dataset onto the first $\rnk=2$ principal components $\mU_\star=\begin{bmatrix}
%         \vu^\star_1 &\vu^\star_2
%     \end{bmatrix}$ (left) and its privatized version $\rV$ sampled from the exponential mechanism with $\beta=2$ (middle) and $\beta=0$ (right). See \Cref{sec:figures-info} for a description of the dataset. \RD{maybe we can add 3 different values of $\beta$}}
%     \label{fig:genomes_vis}
% \end{figure}

\paragraph{Differential Privacy} To characterize the privacy of a privatized statistic $\rO(\cdot)$ (such as the privatized PCs produced by the exponential mechanism), we adopt the differential privacy framework of \citet{dwork2006calibrating}. In this framework, a randomized statistic $\rO(\cdot)$ is private if it ensures that an adversary cannot reliably infer whether a target individual is present in the dataset $\mX$ or not from $\rO(\mX)$, the value of the privatized statistic on the dataset $\mX$. Detecting the presence or absence of a target individual requires an adversary to distinguish between the distributions of $\rO(\cdot)$ on two datasets: one containing the target individual and one without. Such datasets are known as \emph{neighboring datasets}, formally defined below\footnote{The differential privacy literature considers two definitions of neighboring datasets: the \emph{add/remove} definition \citep[Definition 2.4]{dwork2014algorithmic} and the \emph{swap} definition \citep{dwork2006calibrating}. Under the add/remove definition, a neighboring dataset is obtained by adding or removing a single data point. Under the swap definition, it is obtained by replacing a single data point with a new one. This paper adopts the add/remove definition. For a comparison of the two definitions, see \citet{kuleszamean}.}.

% Hence, this view of privacy naturally leads to the following definition of neighboring datasets:
\begin{definition}[Neighboring datasets] Two datasets $\mX,\tilde{\mX}\subset \{x \in \R^\dim: \|x\| \leq \sqrt{\dim}\}$ are neighboring if they differ by a single individual, or, equivalently, if  $\tilde{\mX}=\mX\cup\{\vx\}$ for some $\vx\in\R^{\dim}$ satisfying $\|\vx\|\leq\sqrt{\dim}$ or $\tilde{\mX}=\mX\backslash\{\vx\}$ for some $\vx\in\mX$. We use  $\calN(\mX)=\calN_+(\mX)\cup\calN_-(\mX)$ to denote the set of all neighboring datasets of $\mX$, where $\calN_+(\mX)$ and $\calN_-(\mX)$ correspond to the datasets obtained by adding and subtracting, respectively, a data point from $\mX$.
\end{definition}
% \begin{remark}
% %Some of the related works that adopt the same definition for neighboring datasets include \cite{liu2022dp}, \cite{leake2021sampling}, \cite{gonem2018smooth}, \cite{dwork2024differentially} and \cite{dungler2025iterative}.
% \end{remark}
%For a given dataset $\mX,$ we will let $\mathcal{N}(\mX)$ denote the collection of all neighboring datasets of $X.$

In differential privacy, privacy guarantees are expressed as indistinguishability results between the distributions of a privatized statistic on neighboring datasets. Hence, they provide upper bounds on an adversary's ability to detect the presence or absence of a target individual. This paper employs two popular notions of indistinguishability to measure privacy loss: one based on trade-off functions introduced by \citet{dong2022gaussian} and the other based on Rényi divergence introduced by \citet{mironov2017renyi}.
\paragraph{Trade-off Functions and Differential Privacy} A line of work \citep{wasserman2010statistical,kairouz2015composition,dong2022gaussian} has explored connections between differential privacy and hypothesis testing, culminating in the work of \citet{dong2022gaussian}, who propose the following notion of trade-off functions for expressing privacy guarantees.
\begin{definition}[Trade-off function {\citep{dong2022gaussian}}]\label{def:tf}
    For any distributions $\nu_0$ and $\nu_1$ on some common space $\mathcal{V},$ the trade-off function $\tf{\nu_0}{\nu_1}:[0,1]\rightarrow[0,1]$ is defined as:
    \begin{align*}
        \tf{\nu_0}{\nu_1}(\alpha)\bydef\underset{\phi:\mathcal{V}\rightarrow[0,1]}{\min}\left\{1-\E_{\rV\sim\nu_1}[\phi(\rV)]\;:\;\E_{\rV\sim\nu_0}[\phi(\rV)]\leq\alpha\right\}.
    \end{align*}
In other words, for $\alpha \in [0,1]$, $\tf{\nu_0}{\nu_1}(\alpha)$ is the Type II error (false acceptance probability) of the optimal (most powerful) level-$\alpha$ test for distinguishing between the null hypothesis $H_0: \rV \sim \nu_0$ and the alternative hypothesis $H_1: \rV \sim \nu_1$. With some abuse of notation, for two random variables $\rV_0$ and $\rV_1$ with laws $\nu_0$ and $\nu_1$ respectively, we use $\tf{\rV_0}{\rV_1}$ interchangeably with $\tf{\nu_0}{\nu_1}$.
\end{definition}
Many popular differential privacy guarantees can be expressed using trade-off functions:
\begin{itemize}
    \item The oldest differential privacy (DP) guarantee is $\epsilon$-DP, introduced by \citet{dwork2006calibrating}. For any $\epsilon > 0$, a privatized statistic $\rO(\cdot)$ is $\epsilon$-DP if for any two neighboring datasets $\mX,\tilde{\mX}$ and any event $E$, $\P(\rO(X) \in E) \leq e^\epsilon \cdot \P(\rO(\tilde{X}) \in E)$. It is well-known \citep{wasserman2010statistical,dong2022gaussian,kairouz2015composition} that $\rO(\cdot)$ is $\epsilon$-DP if and only if for any two neighboring datasets $\mX,\tilde{\mX}$, 
    \begin{align*}
    \tf{\rO(\mX)}{\rO(\tilde{\mX})}(\alpha)\geq \tf*{\mathrm{Bern}\left(\frac{1}{1+e^\epsilon}\right)}{\mathrm{Bern}\left(\frac{e^\epsilon}{1+e^\epsilon}\right)}(\alpha)\quad\forall\:\alpha\in[0,1].
\end{align*}
This reformulation in terms of trade-off functions provides an operational interpretation for an $\epsilon$-DP privacy guarantee: for an adversary, detecting the presence or absence of a target individual in the dataset based on the published value of a $\epsilon$-DP statistic with $\epsilon = 0.1$ is at least as hard as distinguishing between $\mathrm{Bern}(0.48)$ and $\mathrm{Bern}(0.52)$.
\item The $\mu$-Gaussian Differential Privacy ($\mu$-GDP) guarantee of \citet{dong2022gaussian} is also expressed in terms of trade-off functions. For any $\mu\geq0,$ a privatized statistic $\rO(\cdot)$ is $\mu$-GDP if for any two neighboring datasets $\mX,\tilde{\mX}$,  
\begin{align*}
    \tf{\rO(\mX)}{\rO(\tilde{\mX})}(\alpha)\geq \tf{\gauss{0}{1}}{\gauss{\mu}{1}}(\alpha)\quad\forall\:\alpha\in[0,1].
\end{align*}
Operationally, this means that for an adversary detecting the presence or absence of a target individual in the dataset based on the published value of a $\mu$-GDP statistic is at least as hard as distinguishing between $\gauss{0}{1}$ and $\gauss{\mu}{1}$.
\end{itemize}
\begin{wrapfigure}[12]{r}{0.4\linewidth}
  \centering
  \vspace{-7mm}
  \includegraphics[width=\linewidth]{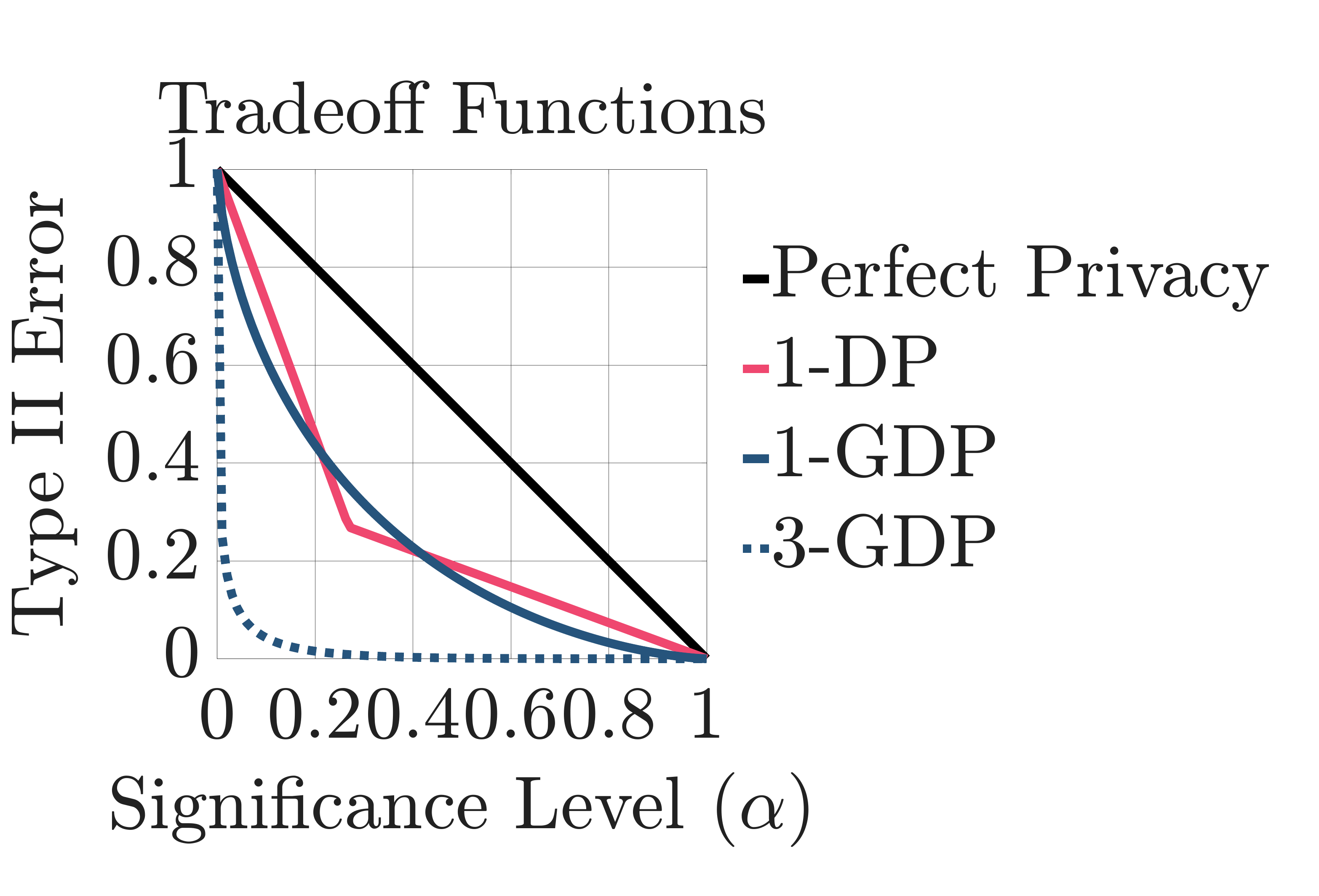}
    \caption{Visualization of trade-off-function lower bounds implied by various privacy guarantees.}
    \label{fig:tf}
\end{wrapfigure}
A useful feature of trade-off-function-based privacy formulations is that they allow for a visualization of often incomparable definitions of privacy together in a single plot, as shown in \Cref{fig:tf}.
\paragraph{Rényi Differential Privacy} The following definition, due to \citet{mironov2017renyi}, introduces a different notion of privacy  which uses the Rényi divergence as the notion of indistinguishability. 
\begin{definition}[Rényi Divergence and Rényi Differential Privacy {\citep{mironov2017renyi}}]\label{def:Rényi}
    For any $\alpha\in(1,\infty)$ and any two distributions $\nu_0$ and $\nu_1$, the Rényi divergence of order $\alpha$ of $\nu_1$ with respect to $\nu_0$ is defined as:
\begin{align*}
    \rdv{\alpha}{\nu_1}{\nu_0} \bydef \begin{cases} \frac{1}{\alpha - 1} \ln \E_{\rV \sim \nu_0} \left[ \left(\frac{\diff \nu_1}{\diff \nu_0}(\rV) \right)^\alpha \right] & \text{ if } \nu_1 \ll \nu_0, \\ \infty & \text{ otherwise}. \end{cases}
\end{align*}
With some abuse of notation, for two random variables $\rV_0$ and $\rV_1$ with laws $\nu_0$ and $\nu_1$ respectively, we use $\rdv{\alpha}{\rV_0}{\rV_1}$ interchangeably with $\rdv{\alpha}{\nu_0}{\nu_1}$. For $\alpha > 1$ and $\epsilon > 0$, a privatized statistic $\rO(\cdot)$ is said to be $(\alpha,\epsilon)$-Rényi differentially private ($(\alpha,\epsilon)$-RDP) if $\rdv{\alpha}{\rO(\tilde{\mX})}{\rO(\mX)}\leq \epsilon$ for any two neighboring datasets $\mX$ and $\tilde{\mX}$. 
\end{definition}
\begin{remark}
    Note that we use $\alpha$ to denote both the order of the Rényi divergence and the significance level for trade-off functions. The meaning of $\alpha$ will be clear from the context.
\end{remark}

%% file: main_main_results.tex
\section{Main Results}\label{sec:main-results}

Since we analyze the exponential mechanism in the high-dimensional limit $\dim \rightarrow \infty$, we assume that we run the exponential mechanism on a sequence of datasets $X^{(p)} \subset \R^p$ with increasing dimension $p$. To ensure that the estimation error and privacy loss of the exponential mechanism converge to well-defined limits, we assume that certain properties of the dataset $X^{(p)}$ stabilize and converge to well-defined limits as $\dim \rightarrow \infty$. We formally state our assumptions below. 

\begin{assumption} \label{assump:data} The sequence of datasets $\{X^{(\dim)}: \dim \in \N\}$ with $X^{(\dim)} \subset \R^\dim$ satisfies:
\begin{enumerate}
    \item \emph{Sample Size Scaling:} $|\mX^{(\dim)}|/\dim^{3/2} \rightarrow \theta$ for some $\theta>0$ as $\dim \rightarrow \infty$, where  $|\mX^{(\dim)}|$  denotes the number of data points in $\mX^{(\dim)}$ (or the sample size).
    \item \emph{Convergence of Spectrum:} The largest $\rnk+1$ eigenvalues of the empirical second moment matrix $\mSigma(\mX^{(p)})$ converge to some finite limits $\gamma_{1:\rnk+1}$ as $\dim \rightarrow \infty$:
    \begin{align*}
        \lambda_i(\mSigma(\mX^{(p)})) & \rightarrow \gamma_{i} \in [0,\infty) \quad \forall \; i \; \in \; [\rnk+1],
    \end{align*}
    where $\lambda_i(\mSigma(\mX^{(p)}))$ for $i\in[\dim]$ denotes the $i$th largest eigenvalue of $\mSigma(\mX^{(p)})$ and $k$ is the number of PCs to be privatized. Moreover, the asymptotic spectral gap $\Delta \bydef \gamma_{\rnk} - \gamma_{\rnk+1}$ is strictly positive. We also require that $\mu_{\mSigma(\mX^{(\dim)})}$, the empirical distribution of the smallest $\dim - \rnk$ eigenvalues of  $\mSigma(\mX^{(p)}),$ converges weakly to some limiting spectral measure $\mu$ as $\dim \rightarrow \infty$:
    \begin{align*}
        \mu_{\mSigma(\mX^{(\dim)})} & \bydef \frac{1}{\dim-\rnk} \sum_{i=k+1}^{\dim} \delta_{\lambda_i(\mSigma(\mX^{(p)}))} \rightarrow \mu.
    \end{align*}
    \item \emph{Normalization:} Each data point $\vx \in \mX^{(p)}$ satisfies $\|x \| \leq \sqrt{p}.$
\end{enumerate}
\end{assumption}
Going forward, for notational convenience, we will simply refer to the observed (sequence of) dataset(s) as $\mX$, rather than $\mX^{(\dim)}$, keeping the dependence on $\dim$ implicit. In the following remarks we provide additional context regarding the assumptions introduced above. 

\begin{remark}[Sample Size Scaling] Assumption 1(1) requires the sample size to scale as $\ssize \propto \dim^{3/2}$. This is the scaling regime in which the exponential mechanism has a nontrivial asymptotic privacy loss. Our results suggest that if $\ssize \gg \dim^{3/2}$, the exponential mechanism is asymptotically perfectly private. That is, its output distribution on the given dataset is indistinguishable from its output distribution on all neighboring datasets. Conversely, if $\ssize \ll \dim^{3/2}$, the optimal test based on the privatized PCs that distinguishes between the given dataset and its worst-case neighboring dataset is consistent. These heuristics are obtained by taking $\theta \rightarrow \infty$ and $\theta \rightarrow 0$ in our privacy characterization. Assumption 1(1) is also consistent with the literature on private PCA in the empirical setting, where existing methods require $\ssize \gtrsim \dim^{3/2}$ samples \citep{dwork2014analyze,gonem2018smooth,tran2025spectral}, which is information-theoretically optimal for certain privacy guarantees \citep{dwork2014analyze}. In contrast, in the population setting (cf. \Cref{rem:empirical-vs-pop}), one can achieve a better sample complexity\footnote{Here $\tilde{O}(p)$ hides poly-logarithmic factors in $\dim$} of $\ssize = \tilde{O}(\dim)$ \citep{cai2024optimal,liu2022dp,dungler2025iterative}. See \Cref{sec:related-work} for further discussion.
\end{remark}

\begin{remark}Assumption 1(2) requires certain spectral properties of the empirical second moment matrix $\mSigma(\mX)$ to stabilize and converge to well-defined limits, ensuring that the utility and privacy loss of the exponential mechanism converge to well-defined limits. Closely related assumptions are common in random matrix theory-based analyses of high-dimensional PCA (see, e.g., \citep{benaych2011eigenvalues,bai2012sample,dey2019asymptotic}). In the random matrix theory literature, the limiting spectral measure $\mu$ is often called the ``bulk’’ spectral measure, and the limiting outlier eigenvalues $\gamma_{1:k}$ are often called ``spikes.’’ We will revisit this assumption and discuss it further in \Cref{rem:non-asymptotic}.
\end{remark}

\begin{remark} Assumption 1(3) imposes an upper bound on the $\ell_2$ norm of any data point in the dataset and is standard for differentially private PCA in the empirical setting (see, e.g., \citep{chaudhuri2013near,dwork2014analyze}). 
\end{remark}

\paragraph{Hilbert transform} Before stating our results, we will need to introduce the Hilbert transform $H_\mu$ of a probability distribution $\mu,$ which is the function:
\begin{align*}
    H_\mu: \R\backslash \mathrm{supp}(\mu)\rightarrow\R,\quad H_\mu(\lambda) \bydef \int_{\R} \frac{\diff \mu(t) }{\lambda - t}.
\end{align*}
In our setting, $\mu$ represents the limiting spectral measure of the empirical second moment matrix $\mSigma(\mX)$ (\Cref{assump:data} (item 2)). Since $\mu$ is unknown, we estimate the Hilbert transform of $\mu$ by its empirical counterpart\footnote{{The choice of the factor $1/p$ (as opposed to $1/(p-k)$) is purely for notational convenience. Note that since $k$ is fixed, both choices lead to the same limit as $p\rightarrow\infty.$}} $H_{\mSigma(\mX)}:$
\begin{align} \label{eq:empirical-hilb}
    H_{\mSigma(\mX)}(\lambda) & \bydef \frac{1}{\dim} \sum_{i=1}^{\dim - \rnk} \frac{1}{\lambda - \lambda_{\rnk+i}} \quad \forall \; \lambda \in (\lambda_{\rnk+1}, \infty),
\end{align}
where $\lambda_{1:\dim}$ denote the eigenvalues of $\mSigma(\mX).$ Note that under \Cref{assump:data}, $H_{\mSigma(\mX)}(\lambda)$ is a consistent estimator for $H_{\mu}(\lambda)$ for any $\lambda>\gamma_{\rnk+1}$ (see \Cref{lem:misc_HK}).
%the supplement \citep[Lemma F.2]{supplementary}). 
It is well-known that the performance of many differentially private methods for PCA depends on the eigengap $\lambda_{\rnk} - \lambda_{\rnk+1}$ between the $\rnk$th and $\rnk+1$th eigenvalues of $\mSigma(\mX)$ \citep{dwork2014analyze, gonem2018smooth, chaudhuri2013near, mangoubi2022re}. In our work, the Hilbert transform values $H_\mu(\gamma_{i})$ and their empirical estimates $H_{\mSigma(\mX)}(\lambda_i)$ for $i \in [\rnk]$ will play an important role in determining the utility and privacy loss of the exponential mechanism. These can be viewed as an average of inverse-eigengaps between the $i$th eigenvalue and the smallest $\dim-\rnk$ eigenvalues, an average-case analog to the worst-case inverse-eigengap $1/(\lambda_\rnk-\lambda_{\rnk+1}).$

\subsection{Utility Analysis} \label{sec:utility-disc}
Our first main result is an asymptotic characterization of the utility or performance of the exponential mechanism (\Cref{alg:ExpM}). We quantify the utility of the privatized PCs $\rV$ returned by the exponential mechanism using the \emph{overlap matrix} $\mU_{\star}^\top \rV \rV^\top \mU_{\star}$, which is a measure of similarity between the privatized PCs and the true PCs $\mU_\star$. 

\begin{theorem}\label{thm:utility} Consider a sequence of datasets $X \subset \R^\dim$ which satisfies \Cref{assump:data}, a sequence of noise parameters $\beta_{\dim} \rightarrow \beta \in (0,\infty)$ as $\dim \rightarrow \infty$, and a fixed rank $\rnk \in \N$ (independent of $\dim$). Let $\mU_{\star} \in \R^{\dim \times \rnk}$ denote the true PCs $(\text{the matrix of the first $\rnk$ eigenvectors of $\mSigma(\mX)$})$ and $\rV$ denote the output of the exponential mechanism run on inputs $(\mX, \beta_{\dim},\rnk)$. Then, 
\begin{align*}
   \mU_{\star}^\top \rV \rV^\top \mU_{\star}  \pc \diag\left(1 - \frac{H_\mu(\gamma_1)}{\beta}, 1 - \frac{H_\mu(\gamma_2)}{\beta}, \dotsc, 1-\frac{H_\mu(\gamma_\rnk)}{\beta} \right)_+ \quad \text{ as $\dim \rightarrow \infty$}.
\end{align*}
\end{theorem}
We now discuss some implications of \Cref{thm:utility}.

\paragraph{Overlap Matrix and Common Error Metrics} Many commonly used error  metrics can be related to the overlap matrix $\mU_{\star}^\top \rV \rV^\top \mU_{\star}$. Examples include the error metrics $\|\mU_\star \mU_\star^\top - \rV \rV^\top\|$ and $\|\mU_\star \mU_\star^\top - \rV \rV^\top\|_{\fr}$, which are commonly used to quantify the error between the subspace spanned by the true PCs $\mU_\star$ and the subspace spanned by the privatized PCs $\rV.$ Indeed, it is well-known that these error metrics are related to the overlap matrix  \citep[Lemma 2.5]{chen2021spectral}: 
\begin{align*}
    \|\mU_\star \mU_\star^\top - \rV \rV^\top\|^2 & = \|\mI_\rnk -  \mU_{\star}^\top \rV \rV^\top \mU_{\star}\|, \quad  \|\mU_\star \mU_\star^\top - \rV \rV^\top\|^2_{\fr} = 2 \rnk - 2\Tr[\mU_{\star}^\top \rV \rV^\top \mU_{\star}].
\end{align*}
Hence \Cref{thm:utility} immediately yields the following asymptotic formulas for the estimation error between the true PCs and the privatized PCs:
\begin{subequations}\label{eq:utility-pred}
\begin{align} 
   \lim_{\dim \rightarrow \infty} \E \|\mU_\star \mU_\star^\top - \rV \rV^\top\|^2 & = \min\left(1, \frac{H_\mu(\gamma_\rnk)}{\beta} \right),\\
   \lim_{\dim \rightarrow \infty} \E \|\mU_\star \mU_\star^\top - \rV \rV^\top\|_{\fr}^2  &= 2 \sum_{i=1}^\rnk \min\left(1, \frac{H_\mu(\gamma_i)}{\beta} \right).
\end{align}
\end{subequations}
% \paragraph{Asymptotic Theory v.s. Empirical Performance.} We also explore if the asymptotic predictions in \eqref{eq:utility-pred} provide good approximations to the finite-$\dim$ performance of the exponential mechanism on the 1000 Genomes dataset. A minor issue is that the asymptotic predictions in \eqref{eq:utility-pred} depend on the limiting spectral properties of the dataset via $H_\mu(\gamma_{1:\rnk})$, which need to be estimated. However, this issue is easily addressed by estimating the Hilbert transform $H_\mu$ by its empirical counterpart $H_{\mSigma(\mX)}$:
% \begin{align} \label{eq:empirical-hilb}
%     H_{\mSigma(\mX)}(\lambda) & \bydef \frac{1}{\dim} \sum_{i=1}^{\dim - \rnk} \frac{1}{\lambda - \lambda_{\rnk+i}} \quad \forall \; \lambda \in [\lambda_\rnk, \infty),
% \end{align}
% where $\lambda_{1:\dim}$ denote the eigenvalues of the covariance matrix $\mSigma(\mX)$. Under \Cref{assump:data}, $H_{\mSigma(\mX)}(\lambda_{1}), \dotsc, H_{\mSigma(\mX)}(\lambda_{\rnk})$ are consistent estimators for $H_{\mu}(\gamma_{1}), \dotsc, H_\mu(\gamma_\rnk)$ (see \citep[Lemma 14]{supplementary}).
Estimating the Hilbert transform $H_\mu$ by its empirical counterpart $H_{\mSigma(\mX)}$ defined in \eqref{eq:empirical-hilb}, we obtain the following data-driven asymptotic approximations from \eqref{eq:utility-pred}:
\begin{subequations} \label{eq:utility-pred-final}
\begin{align} 
     \E \|\mU_\star \mU_\star^\top - \rV \rV^\top\|^2 & = \min\left(1, \frac{H_{\mSigma(\mX)}(\lambda_\rnk)}{\beta} \right) + o(1), \\  \E \|\mU_\star \mU_\star^\top - \rV \rV^\top\|_{\fr}^2  &= 2 \sum_{i=1}^\rnk \min\left(1, \frac{H_{\Sigma(\mX)}(\lambda_i)}{\beta} \right) + o(1).
\end{align}
\end{subequations}
\begin{figure}
    \centering
    \includegraphics[width=0.70\linewidth]{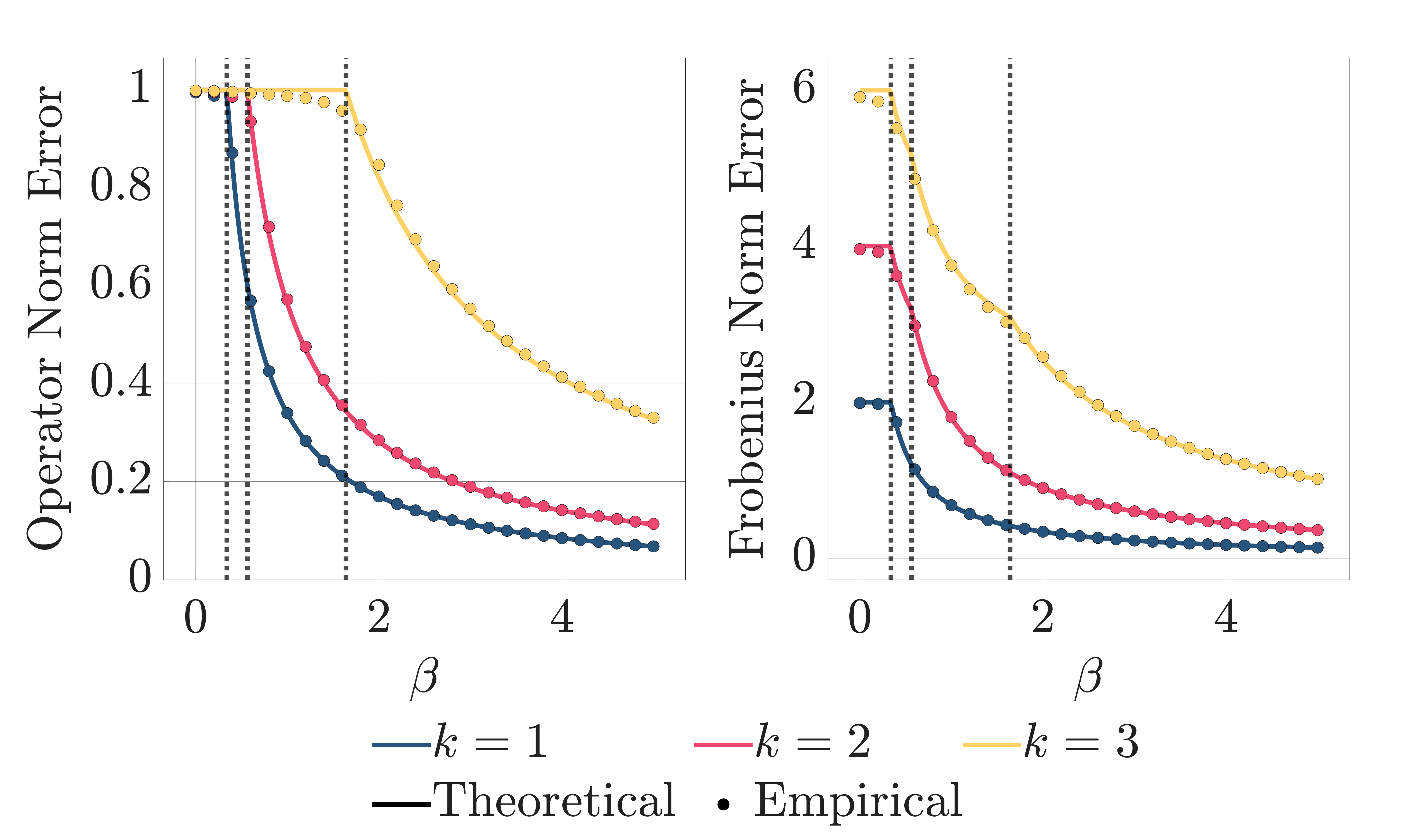}
    \caption{Estimation error of the exponential mechanism in operator norm (left panel) and Frobenius norm (right panel) as a function of $\beta$ for $\rnk = 1,2,3$ on the 1000 Genomes dataset ($n=2504,p=200$). Solid curves represent theoretical predictions derived from \Cref{thm:utility} and circular markers represent the empirical error, averaged over 100,000 Monte Carlo simulations. The dotted vertical lines represent the threshold values $\beta = H_\mu(\gamma_{1:3})$, at which phase transitions occur.}
    \label{fig:main-utility}
\end{figure}
\Cref{fig:main-utility} compares the finite-$\dim$ performance of the exponential mechanism as measured by $\E \|\mU_\star \mU_\star^\top - \rV \rV^\top\|^2$ and $\E \|\mU_\star \mU_\star^\top - \rV \rV^\top\|^2_{\fr}$ (estimated using 100,000 Monte-Carlo simulations) on the 1000 Genomes dataset with the asymptotic predictions on the RHS of \eqref{eq:utility-pred-final} for $\rnk = 1, 2, 3$. We find that the  finite-$\dim$ performance of the exponential mechanism closely tracks the asymptotic predictions derived from \Cref{thm:utility}. In \Cref{app:survey}, 
%the supplement \citep[Appendix H]{supplementary}, 
we report a similar agreement between the theoretical predictions of \Cref{thm:utility} and the empirical performance of the exponential mechanism on another dataset: the American National Election Studies (ANES) survey dataset \citep{ANES2024Pilot}.

\begin{remark} \label{rem:non-asymptotic} While we anticipate that Assumption 1(2) to be necessary for the utility of the exponential mechanism to converge to well-defined limits  as $\dim \rightarrow \infty$ (that is, for \eqref{eq:utility-pred} to hold), the close agreement between the RHS of \eqref{eq:utility-pred-final} and the empirical finite-$\dim$ performance of the exponential mechanism on real-world datasets (see \Cref{fig:main-utility})  suggests that the following non-asymptotic analog of \eqref{eq:utility-pred} could be true:
\begin{align*}
    \left| \E \|\mU_\star \mU_\star^\top - \rV \rV^\top\|_{\fr}^2  -  2 \sum_{i=1}^\rnk \min\left(1, \frac{H_{\Sigma(\mX)}(\lambda_i)}{\beta} \right)  \right| & \leq \varepsilon_\dim
\end{align*}
for some explicit error bound $\varepsilon_\dim$ which vanishes as $\dim \rightarrow \infty$. Such a non-asymptotic result would likely not require assuming the convergence of the empirical spectral measure  $\mu_{\Sigma(\mX)}$ or the leading eigenvalues $\lambda_{1:\rnk+1}(\mSigma(\mX))$ to well-defined limits. Indeed, such non-asymptotic analogs of high-dimensional asymptotic results have recently been explored in the literature (see, e.g., \citep{cheng2024dimension} for a discussion). This is a possible direction for future work. 
\end{remark}

% visualizes empirical estimates of the finite-sample and limiting  estimation error of the exponential mechanism run on the genomes dataset. Observe that, as implied by the corollary, the algorithm incurs a flat (asymptotic) estimation error of 1 in the regime where $\beta\leq H_\mu(\gamma_\rnk).$ In this regime, the utility of the algorithm is no greater than that of the algorithm run with $\beta=0,$ which corresponds to sampling $\rV$ from the uniform distribution $\xi_{\dim,\rnk}$ (see the right plot of \Cref{fig:genomes_vis}). Thus, in this work, we will limit our analysis of the algorithm to the case where $\beta>H_\mu(\gamma_\rnk).$
% \begin{figure}[H]
%     \centering
%     \includegraphics[width=0.8\linewidth]{plots/main-results/utility.pdf}
%     \caption{Visualization of estimation error.}
%     \label{fig:ee}
% \end{figure}

\paragraph{Phase Transitions} As depicted in \Cref{fig:main-utility}, the asymptotic predictions for estimation error (the RHS of \eqref{eq:utility-pred}), viewed as a function of $\beta$, exhibit some interesting phase transitions (derivative discontinuities). The limiting spectral norm estimation error $\lim_{\dim \rightarrow \infty} \E \|\mU_\star \mU_\star^\top - \rV \rV^\top\|^2$ exhibits a single phase transition at $\beta = H_\mu(\gamma_\rnk)$:
\begin{itemize}
    \item When $\beta \in [0, H_\mu(\gamma_\rnk)]$, the asymptotic spectral norm estimation error is constant and coincides with the estimation error at $\beta = 0$, when the privatized PCs $\rV \sim \unif{\O(\dim,\rnk})$ are pure noise. 
    \item In the regime $\beta >  H_\mu(\gamma_\rnk)$, the exponential mechanism has a non-trivial estimation error. 
\end{itemize}
The asymptotic Frobenius norm estimation error $\lim_{\dim \rightarrow \infty} \E \|\mU_\star \mU_\star^\top - \rV \rV^\top\|^2_{\fr}$ exhibits $\rnk$  phase transitions at $\beta = H_\mu(\gamma_1), \dotsc, H_\mu(\gamma_\rnk)$. One way of interpreting the $i$th phase transition point at $H_\mu(\gamma_i)$ is to observe that \Cref{thm:utility} implies that:
\begin{align*}
    \|\rV^\top \vu_i\|^2 = e_i^\top \mU_\star^\top \rV \rV^\top \mU_\star \ve_i \pc \begin{cases} 0  & \text{ if } \beta \leq  H_\mu(\gamma_i) \\ 1-\frac{H_\mu(\gamma_i)}{\beta} & \text{ if } \beta > H_\mu(\gamma_i) \end{cases} \quad \forall \; i \; \in \; [\rnk],
\end{align*}
where $\vu_1, \dotsc, \vu_k$ denote the true top $k$ PCs (eigenvectors of $\mSigma(\mX)$). Since $\|\rV^\top \vu_i\|^2$ is the squared norm of the projection of $\vu_i$ on the subspace spanned by the privatized PCs $\rV$, the equation above implies that $\vu_i$ is asymptotically orthogonal to this subspace when $\beta \leq H_\mu(\gamma_i)$, and has a non-trivial projection on this subspace when $\beta > H_\mu(\gamma_i)$. Hence, as $\beta$ increases, the exponential mechanism captures the $k$ target PCs one by one, leading to the $k$ phase transitions. In particular, $\beta = H_\mu(\gamma_\rnk)$ is the minimum value of $\beta$ needed to ``capture'' all of the $k$ target PCs $\vu_1, \dotsc, \vu_\rnk$.

% where the second term represents the estimation error since the integrand can be viewed as a distance between the subspace spanned by the first $\rnk$ PCs of $\mX$ and its private counterpart (see \citet[Section 2.2.2]{chen2021spectral}). The following result characterizes the asymptotic behavior of the estimation error and is obtained as a simple corollary of \Cref{thm:overlap}.

\subsection{Privacy Analysis}  \label{sec:privacy}
Next, we present our privacy result for the exponential mechanism. We will quantify the privacy loss of a privatized statistic $\rO(\cdot)$ on a dataset $\mX$ (such as the privatized PCs generated by the exponential mechanism) using the trade-off function and Rényi divergence between the output distribution of the statistic on the dataset $\mX$ and its \emph{worst-case} neighboring dataset:
\begin{align} \label{eq:priv-loss-quantities}
    \inf_{\tilde{\mX} \in \calN(\mX)} \tf{\rO(\mX)}{\rO(\tilde{\mX})}(\alpha) \quad \text{and} \quad \sup_{\tilde{\mX} \in \calN(\mX)} \rdv{\alpha}{\rO(\tilde{\mX})}{\rO(\mX)}.
\end{align}
Since neighboring datasets are constructed by adding or removing a single data point, these quantities measure how hard it is for an adversary to detect if a target individual is present or absent in the given dataset (for the worst-case target).  Several prior works have used quantities closely related to those in \eqref{eq:priv-loss-quantities} to measure the privacy loss of a mechanism on a given dataset:
\begin{itemize} 
\item \citet{papernot2016semi} and a follow-up work \citep{papernot2018scalable} develop bounds on a variant of the Rényi divergence between the output distribution of a mechanism on a dataset $\mX$ and its worst-case neighboring dataset to obtain improved data-dependent privacy analysis of differentially private deep learning systems.
\item A related framework is per-instance differential privacy \citep{wang2019per}, which measures the privacy loss incurred by an individual $x$ in the dataset $\mX$ using the approximate max-divergence between the output distribution of the statistic on the dataset $\mX$ and its neighboring dataset $\mX \backslash \{x\}$, where the individual $x$ has been removed.
\end{itemize}
At a high level, our privacy result shows that in the high-dimensional limit, detecting the presence or absence of the worst-case target individual based on the output of the exponential mechanism on a given dataset $\mX$ is \emph{exactly} as hard as distinguishing between $\gauss{0}{1}$ and $\gauss{\rho}{1}$, for some $\rho$ that depends on the asymptotic spectral properties of the dataset $\mX$. We will show that this asymptotic equivalence holds with respect to both trade-off functions and Rényi divergence-based notions of indistinguishability. We formally define our measure of privacy loss below. 

\begin{definition}[Asymptotic Gaussian Differential Privacy (AGDP) Loss] \label{def:AGDP} Consider a randomized algorithm which takes a dataset $\mX \subset \R^\dim$ as an input and returns an output denoted by $\rO(\mX)$. For $\rho\geq0,$ we say that the algorithm has a (high-dimensional) Asymptotic Gaussian Differential Privacy (AGDP) loss of at most $\rho$ on a dataset $\mX \subset \R^\dim$ if:
\begin{align*}
    \li \inf_{\tilde{\mX} \in \calN(\mX)} \tf{\rO(\mX)}{\rO(\tilde{\mX})}(\alpha) & \geq \tf{\gauss{0}{1}}{\gauss{\rho}{1}}(\alpha) \quad \forall \; \alpha \; \in \; [0,1], \\
     \ls\sup_{\tilde{\mX} \in \calN(\mX)} \rdv{\alpha}{\rO(\tilde{\mX})}{\rO(\mX)} & \leq \rdv{\alpha}{\gauss{\rho}{1}}{\gauss{0}{1}} \quad \forall \; \alpha > 1.
\end{align*}
We say the algorithm has a \underline{sharp} AGDP loss of $\rho$ on a dataset $\mX$ if the statements above hold with equality:
\begin{align*}
     \lm \inf_{\tilde{\mX} \in \calN(\mX)} \tf{\rO(\mX)}{\rO(\tilde{\mX})}(\alpha) & = \tf{\gauss{0}{1}}{\gauss{\rho}{1}}(\alpha) \quad \forall \; \alpha \; \in \; [0,1], \\
     \lm \sup_{\tilde{\mX} \in \calN(\mX)} \rdv{\alpha}{\rO(\tilde{\mX})}{\rO(\mX)} & = \rdv{\alpha}{\gauss{\rho}{1}}{\gauss{0}{1}} \quad \forall \; \alpha > 1.
\end{align*}
\end{definition}
%\begin{remark} Prior works have used quantities closely related to those in \eqref{eq:priv-loss-quantities} to measure the privacy loss of a mechanism on a given dataset. \citet{papernot2016semi} and follow-up work \citep{papernot2018scalable} develop bounds on a variant of the Rényi divergence between the output distribution of a mechanism on a dataset $\mX$ and its worst-case neighboring dataset to obtain improved data-dependent privacy analysis of differentially private deep learning systems. A related framework is per-instance differential privacy \citep{wang2019per}, which measures the privacy loss incurred by an individual $x$ in the dataset $\mX$ using the approximate max-divergence between the output distribution of the statistic on the dataset $\mX$ and its neighboring dataset $\mX \backslash \{x\}$, where the individual $x$ has been removed.
%Note that the privacy guarantee holds for each fixed dataset $\mX$ and hence is data-dependent. An analogous guarantee, in a finite-sample setting, is considered in \citet[Theorem 3]{papernot2016semi} for a different task. Another related notion is per-instance differential privacy \citep{wang2019per}, which provides a privacy guarantee for both a fixed dataset and a fixed individual in the given dataset. On the other hand, our privacy guarantee holds in the worst-case over all individuals present or absent in the given dataset, but its proof can be adapted to derive the privacy guarantee for a single individual.
%\end{remark}

The following theorem formally states our privacy result.
\begin{theorem}\label{thm:privacy}
   Consider a sequence of datasets $X \subset \R^\dim$ which satisfies \Cref{assump:data}, a sequence of noise parameters $\beta_{\dim} \rightarrow \beta\in(H_\mu(\gamma_k),\infty)$ as $\dim \rightarrow \infty$, and a fixed rank $\rnk \in \N$ (independent of $\dim$). Then, the exponential mechanism (\Cref{alg:ExpM}) run on inputs $(\mX, \beta_{\dim},\rnk)$ has a \underline{sharp}  AGDP loss of $\sigma_\beta$, where $\sigma_\beta$ is given by:
   \begin{align} \label{eq:priv-func}
        {\sigma^2_\beta}\bydef \begin{cases} \frac{1}{2 \Delta \theta^2} \cdot \frac{\left(\beta-H_{\mu}(\gamma_\rnk)\right)^2}{2({\beta-H_{\mu}(\gamma_\rnk)})+ \Delta H'_{\mu}(\gamma_\rnk)}&\text{if}\quad \beta \geq -\Delta H'_{\mu}(\gamma_\rnk)+H_{\mu}(\gamma_{\rnk})\\
        -\frac{1}{2\theta^2}H'_{\mu}(\gamma_\rnk)&\text{if}\quad \beta<-\Delta H'_{\mu}(\gamma_\rnk)+H_{\mu}(\gamma_{\rnk})
    \end{cases}.
    \end{align}
\end{theorem}
Next, we highlight some important features of \Cref{thm:privacy}. 
\paragraph{A Privacy Plateau} 
\begin{figure}
    \centering
    \includegraphics[width=\linewidth]{plots/1000Genomes_Figure4.pdf}
    \caption{Projections of the 1000 Genomes dataset onto the first $\rnk=2$ privatized PCs for $\beta=1,$ $\beta=2,$ $\beta=4,$ and $\beta=8.$ The privatized PCs with these $\beta$ values have reasonable estimated AGDP loss values of $\hat{\sigma}_\beta=0.46,$ $\hat{\sigma}_\beta=0.68,$ $\hat{\sigma}_\beta=1,$ and $\hat{\sigma}_\beta=1.46,$ respectively.} \label{fig:privacy-projections}
\end{figure}Note that \Cref{thm:privacy} applies only when $\beta > H_\mu(\gamma_\rnk)$. While it might be possible to extend our proof techniques to cover the $\beta < H_\mu(\gamma_\rnk)$ case, we restrict our focus to the $\beta > H_\mu(\gamma_\rnk)$ case because as discussed in \Cref{sec:utility-disc}, $\beta = H_\mu(\gamma_\rnk)$ is the minimum value of $\beta$ needed to ``capture'' all of the top $k$ target PCs. This means that the smallest AGDP loss that can be achieved by the exponential mechanism while still capturing all $k$ target PCs is $\sigma_{\min}$, where:
\begin{align*}
    \sigma_{\min}^2 \bydef -\tfrac{1}{2\theta^2}H'_{\mu}(\gamma_\rnk).
\end{align*}
% \begin{figure}
%     \centering
%     \includegraphics[width=0.42\linewidth]{plots/Figure4_privacy.pdf}
%   \caption{Noise parameter $\beta$ v.s. privacy parameter $\sigma_{\beta}$ for the 1000 Genomes dataset for $k = 1,2,3$.}
%   \label{fig:privacy-parameter}
% \end{figure}
\begin{wrapfigure}[17]{r}{0.4\linewidth}
  \centering
  \vspace{-5mm}
  \includegraphics[width=\linewidth]{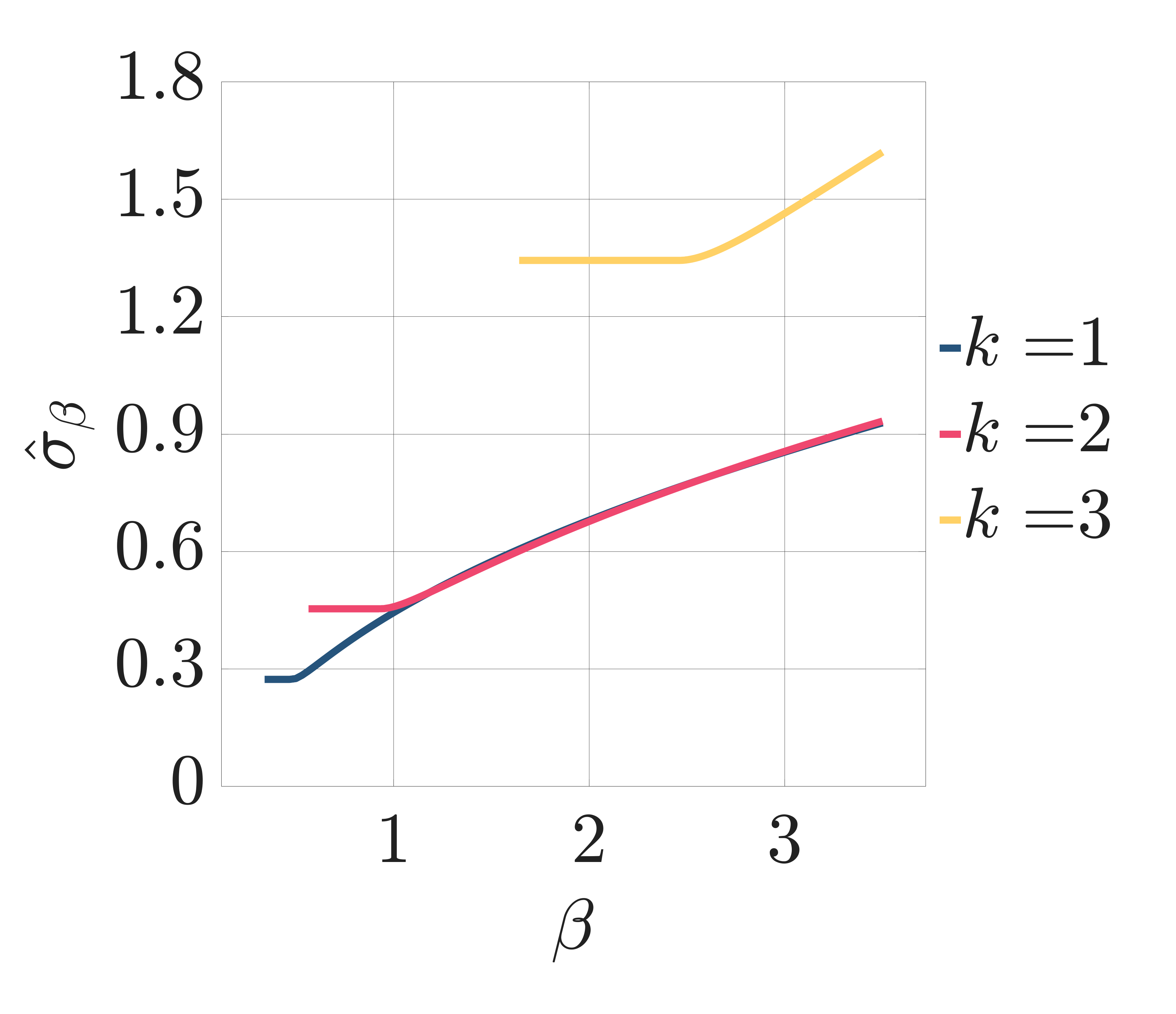}
  \caption{Noise parameter $\beta$ v.s. estimated asymptotic privacy loss $\hat{\sigma}_{\beta}$ for the 1000 Genomes dataset for $k = 1,2,3.$}
  \label{fig:privacy-parameter}
\end{wrapfigure}
For a given dataset, we estimate the AGDP loss $\sigma_{\beta}$ of the exponential mechanism by estimating the limiting Hilbert transform $H_\mu$ and its derivative using the empirical Hilbert transform \eqref{eq:empirical-hilb} and its derivative, respectively. \Cref{fig:privacy-parameter} plots this estimate of the AGDP loss, denoted by  $\hat{\sigma}_\beta$, as a function of $\beta$ for the 1000 Genomes dataset for $\rnk = 1, 2, 3$. A surprising feature of this plot is the presence of a \emph{``privacy plateau''} when $$\beta \in (H_\mu(\gamma_\rnk), -\Delta H'_{\mu}(\gamma_\rnk)+H_{\mu}(\gamma_{\rnk})].$$ 
In this regime, decreasing $\beta$ increases the noise in the privatized PCs, resulting in a decrease in the utility, but does not lead to reductions in the asymptotic privacy loss.  In \Cref{fig:privacy-projections}, we also plot the projections of the 1000 Genomes dataset on the privatized PCs with $\beta=1,$ $\beta=2,$ $\beta=4,$ and $\beta=8.$ At these values of $\beta$, \Cref{thm:privacy} shows that the privatized PCs have reasonable estimated AGDP loss values of $\hat{\sigma}_\beta=0.46,$ $\hat{\sigma}_\beta=0.68,$ $\hat{\sigma}_\beta=1,$ and $\hat{\sigma}_\beta=1.46,$ respectively, while still preserving the clustering structure in the dataset (compare with non-private PCs in \Cref{fig:proj}).
\paragraph{Asymptotic Theory v.s. Finite-$\dim$ Privacy} 
\begin{figure}
    \centering
    \includegraphics[width=0.9\linewidth]{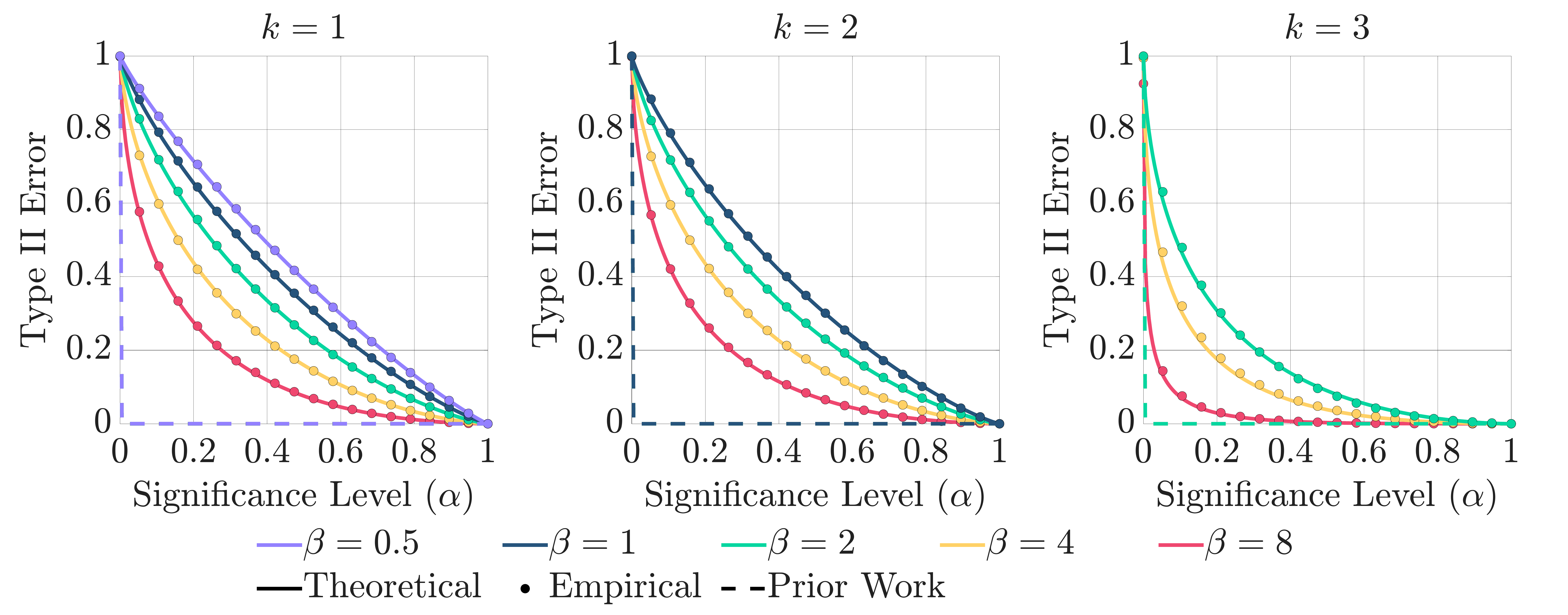}
    \caption{Trade-off functions for the exponential mechanism on the 1000 Genomes dataset for $\rnk \in \{1, 2, 3\}$ and $\beta\in\{0.5,1,2,4,8\}.$  Solid curves: theoretical predictions from \Cref{thm:privacy}. For $\rnk=2$ and $\rnk=3,$ \Cref{thm:privacy} applies only for $\beta\in\{1,2,4,8\}$ and $\beta\in\{2,4,8\},$ respectively, since other $\beta$ values are below the threshold value $H_\mu(\gamma_\rnk)$. Circular markers: empirically estimated trade-off functions (using 100,000 Monte Carlo samples). Dashed curves:  non-asymptotic privacy bounds from prior work \citep{chaudhuri2013near}, which are essentially vacuous for these values of $\beta$ and coincide with the axes. }\label{fig:privacy-tf}
\end{figure}
We also explore if the asymptotic predictions of \Cref{thm:privacy} provide good approximations to the finite-$\dim$ privacy loss of the exponential mechanism on the 1000 Genomes dataset. This requires us to estimate the trade-off function between the output distributions of the exponential mechanism on the 1000 Genomes dataset $\mX$ and its worst-case neighboring dataset:
\begin{align*}
    \inf_{\tilde{\mX} \in \calN(\mX)} \tf{\nu(\cdot \mid \mSigma(\mX), \beta,\rnk)}{\nu(\cdot \mid \mSigma(\tilde{\mX}), \beta,\rnk)}.
\end{align*}
% \begin{figure}[H]
%     \centering
%     \includegraphics[width=0.3\linewidth]{plots/main-results/main_priv.pdf}\hspace{-0.5cm}
%     \includegraphics[width=0.3\linewidth]{plots/main-results/main_priv_comp.pdf}\hspace{-0.5cm}
%     \includegraphics[width=0.3\linewidth]{plots/main-results/main_priv_vis.pdf}
%     \caption{Visualizations of trade-off functions (left and middle) and projections of the genomes dataset (right) run with the corresponding choice of $\beta.$}
%     \label{fig:main-privacy}
% \end{figure}
The optimization over neighboring datasets makes it challenging to derive empirical estimates of the above quantity. However, the proof of \Cref{thm:privacy} (see \Cref{prop:var-lim-v2} in \Cref{sec:privacy-proof}) shows that if $\{(\lambda_i, \vu_i): i \in [\rnk+1]\}$ denote the top $\rnk+1$ eigenpairs of $\mSigma(\mX),$ the neighboring dataset  $\tilde{\mX}_\star \bydef \mX \cup \{\vx_{\star}\}$ with:
\begin{align*}
   \vx_\star  &\bydef  \sqrt{\dim} \left(  \sqrt{t_\star}  \vu_{\rnk} + \sqrt{1-t_\star}  \vu_{\rnk+1} \right),\\
   t_\star &\explain{def}{=}\begin{cases} \frac{{\beta-H_{\mu}(\gamma_\rnk)}}{2(\beta-H_{\mu}(\gamma_\rnk)) + \Delta H'_{\mu}(\gamma_\rnk)}&\text{if}\quad \beta\geq H_{\mu}(\gamma_{\rnk})-\Delta H'_{\mu}(\gamma_\rnk)\\
   1&\text{if}\quad \beta < H_{\mu}(\gamma_{\rnk})-\Delta H'_{\mu}(\gamma_\rnk) \end{cases}
\end{align*}
is asymptotically the worst-case neighboring dataset. Hence, we expect that:
\begin{align} \label{eq:tf-opt-approx}
    \inf_{\tilde{\mX} \in \calN(\mX)} \tf{\nu(\cdot \mid \mSigma(\mX), \beta,\rnk)}{\nu(\cdot \mid \mSigma(\tilde{\mX}), \beta,\rnk)} & \approx \tf{\nu(\cdot \mid \mSigma(\mX), \beta,\rnk)}{\nu(\cdot \mid \mSigma(\tilde{\mX}_\star), \beta,\rnk)}.
\end{align}
The finite-$\dim$ trade-off function on the RHS can be easily estimated by simulating a large number of samples from $\nu(\cdot \mid \mSigma(\mX), \beta,\rnk)$ and $\nu(\cdot \mid \mSigma(\tilde{\mX}_\star), \beta,\rnk)$  (see \Cref{sec:figures-info}
%the supplement \citep[Appendix G]{supplementary}
for more details).  \Cref{fig:privacy-tf} compares this trade-off function (estimated using 100,000 Monte Carlo samples) on the 1000 Genomes dataset with the asymptotic predictions of \Cref{thm:privacy}. We find that the asymptotic predictions of \Cref{thm:privacy} closely track the empirical trade-off functions. To gauge if these asymptotic predictions provide additional information beyond the existing worst-case non-asymptotic privacy bounds, we also plot the non-asymptotic privacy bounds from prior work \citep{chaudhuri2013near} in \Cref{fig:privacy-tf}. We find that these non-asymptotic bounds are essentially vacuous for these values of $\beta$ and coincide with the axes.  Hence, \Cref{fig:privacy-tf} provides some evidence that the asymptotic predictions of \Cref{thm:privacy} can complement traditional worst-case non-asymptotic privacy bounds by providing reasonable approximations to the actual privacy loss of the exponential mechanism on the given dataset. 

\begin{remark}\Cref{fig:privacy-tf} suggests that the asymptotic predictions of the privacy cost derived from our \Cref{thm:privacy} are reasonably accurate for the 1000 Genomes dataset. However, on another dataset, these asymptotic predictions may be less accurate due to finite-$\dim$ effects. Simulations like those depicted in \Cref{fig:privacy-tf} can help detect such gaps between the asymptotic predictions of privacy loss and its finite counterpart. In \Cref{app:survey}
%the supplement \citep[Appendix H]{supplementary}
, we report a similar comparison on another dataset, the American National Election Studies (ANES) survey dataset \citep{ANES2024Pilot}, where we find that the asymptotic trade-off functions derived from our privacy result  (\Cref{thm:privacy}) closely approximate the empirical trade-off functions measuring the finite-$\dim$ privacy loss of the exponential mechanism, except when $\beta$ is very close to the critical value $H_\mu(\gamma_\rnk)$,  below which our privacy result (\Cref{thm:privacy}) is not applicable.  
\end{remark}
\begin{remark} Before running the exponential mechanism on the 1000 Genomes dataset, we pre-processed the dataset using a rank-transformation procedure (described in \Cref{app:rank} and \Cref{sec:figures-info})
%the supplement \citep[Appendices E and G]{supplementary}) 
to ensure that the dataset is centered and normalized to satisfy the requirement that $\|x\| \leq \sqrt{\dim}$ for any $x \in \mX$. We emphasize that the trade-off functions in \Cref{fig:privacy-tf} only measure the privacy loss of the exponential mechanism and does not account for the preprocessing step. An end-to-end asymptotic privacy loss bound for the exponential mechanism with rank-transformation preprocessing is provided in \Cref{app:rank}.
%the supplement \citep[Appendix E]{supplementary}. 
\end{remark}

\paragraph{An Important Caveat} It is important to note that releasing data-driven estimates of privacy loss (as depicted in \Cref{fig:privacy-tf}) or using them to calibrate the noise parameter $\beta$ can incur privacy loss. These data-driven estimates are intended as confidential information available only to the data curator to help them better understand the actual privacy loss of releasing the privatized PCs for a pre-chosen noise parameter $\beta$ on a given dataset, complementing existing worst-case, non-asymptotic privacy bounds. From a mathematical standpoint, sharp utility and privacy characterizations (such as those in \Cref{thm:utility} and \Cref{thm:privacy}) allow for a fine-grained comparison of different mechanisms. Such characterizations are also the first step towards understanding whether the privacy-utility trade-off of a mechanism is Pareto-optimal or whether another mechanism can improve both its utility and privacy loss on a given dataset. From a practical perspective, an interesting problem for future work is to develop differentially private estimates of the privacy loss that can be safely released. Prior works \citep{papernot2016semi,papernot2018scalable,redberg2021privately} on privatizing data-dependent privacy bounds serve as good starting points.

%% file: main_related.tex
\section{Related Work} \label{sec:related-work}

%Our analysis of the  exponential mechanism (\Cref{alg:ExpM}) adopts the differential privacy framework and relies on various asymptotic  properties of the Gibbs distribution (\Cref{def:Gibbs}). 
The problem of differentially private PCA (DP-PCA) has been studied by various works, which we discuss below. We will use the notations $\tilde{O}(\cdot)$ and $\tilde{\Omega}(\cdot)$ to suppress logarithmic factors while discussing prior work. 

\paragraph{Differentially Private PCA in the Empirical Setting} In the empirical setting for differentially private PCA studied in this paper, one is given a deterministic dataset, and the goal is to construct a private approximation to the leading eigenvectors of the empirical (or sample) second moment matrix of the dataset.  Several works \citep{chaudhuri2013near, kapralov2013differentially, amin2019differentially, leake2021sampling, mangoubi2022re, dwork2014analyze, wei2016analysis, gonem2018smooth, hardt2013beyond, hardt2012beating, hardt2014noisy, tran2025spectral,d2025tight} have studied differentially private PCA in the empirical setting, obtaining non-asymptotic bounds on the privacy and utility of various mechanisms. Specifically, under (pure) $\epsilon$-DP \cite{dwork2006calibrating}, several mechanisms \citep{chaudhuri2013near,kapralov2013differentially,amin2019differentially,leake2021sampling} give accurate approximations to the true PCs using $\ssize = \tilde{O}(\dim^2)$ samples, with a lower bound due to \citet{chaudhuri2013near} showing that this sample complexity is optimal in the sense that any $\epsilon$-DP mechanism must use $\ssize = \tilde{\Omega}(\dim^2)$ samples to achieve low error (assuming a constant spectral gap). In contrast, under the more relaxed privacy notion of $(\epsilon,\delta)$-DP \citep{dwork2006our}, some works \citep{dwork2014analyze,gonem2018smooth,tran2025spectral,mangoubi2022re} construct private approximations to PCs using $\ssize = \tilde{O}(\dim^{3/2})$ samples. This sample complexity is optimal for $(\epsilon,\delta)$-DP mechanisms for sufficiently small $\delta$ \citep{dwork2014analyze}. Notably, our work also considers the $\ssize \asymp \dim^{3/2}$ regime, which is information-theoretically (up to logarithmic factors) optimal for $(\epsilon,\delta)$-DP.  However, some caution is needed since the privacy notions used in our work correspond to asymptotic and data-dependent analogs of Gaussian \citep{dong2022gaussian} and Rényi Differential Privacy \citep{mironov2017renyi}, and it is not immediately clear whether the existing lower bounds extend to the former.

% {\color{magenta}
% \begin{itemize}
%     \item We're unsure whether to mention \cite{he2025differentially} since it considers a different task.
% \end{itemize}}

\paragraph{Differentially Private PCA in the Population Setting}   A different line of work \citep{cai2024optimal,liu2022dp,singhal2021privately,dungler2025iterative} has studied the problem of privatizing PCs in the population setting. Here, one assumes that the dataset $X$ consists of $\ssize$ samples drawn i.i.d. from some distribution with population second moment matrix $\Sigma_\star$, and the goal is to construct private approximations of the leading eigenvectors of $\Sigma_\star$. The state-of-the-art results \citep{cai2024optimal,liu2022dp,dungler2025iterative} in this population setting construct consistent and $(\epsilon,\delta$)-DP estimators for the population PCs (top $\rnk$ eigenvectors of $\mSigma_\star$) using $\ssize = \tilde{O}(\dim)$ samples. This sample complexity is known to be unimprovable for $(\epsilon, \delta)$-DP algorithms for extremely small $\delta \leq e^{-\Omega(\dim)}$  \citep{cai2024optimal,liu2022dp}. Note that the optimal sample complexity in the population setting is smaller than in the empirical setting. {Several algorithms for DP-PCA in the population setting seem to achieve the improved $n = \tilde{O}(p)$  sample complexity by filtering out (or clipping) data points that are ``high-leverage’’ in some suitable sense (e.g., those that lead to high stochastic gradients in SGD-based algorithms for DP-PCA) from the dataset \citep{liu2022dp,dungler2025iterative}. This can help reduce the sample complexity because high-leverage data points may require substantial noise to be introduced to guarantee privacy (since adding or removing them can significantly change the algorithm's output). On the other hand, in the population setting, under suitable model assumptions, such high-leverage data points are rare, so deleting them does not introduce significant error in estimating the leading eigenvectors of the population second moment matrix. In the empirical setting, deleting high-leverage data points can significantly alter the empirical second moment matrix and hence reduce utility, since in this setting the empirical second moment matrix itself is the target of estimation.   Additionally, the aforementioned lower bounds in the empirical setting suggest that a sample complexity of  $n = \tilde{O}(p)$ may not be possible for certain privacy notions \citep{dwork2014analyze}.}

%This difference can, in part, be attributed to the additional assumptions made on the dataset in the model-based setting.

\paragraph{High-dimensional Asymptotics in Differential Privacy} More recently, \citet{dwork2024differentially} and \citet{bombari2025better} have studied differentially private regression in the high-dimensional asymptotic setting. In this problem, one observes feature vectors $\mX = \{x_1, \dotsc, x_\ssize\} \subset \R^\dim$ and responses $y = \{y_1, \dotsc, y_\ssize\} \subset \R$. The feature vectors are drawn independently from either a (possibly correlated) Gaussian distribution or a sub-Gaussian product distribution. The responses are generated via the linear model $y_i = \ip{x_i}{\theta_\star} + \epsilon_i$, where $\theta_\star$ is the unknown coefficient vector and $\epsilon_1,\dotsc, \epsilon_\ssize$ represent measurement noise.  Under these model assumptions, these works characterize the precise asymptotic estimation error for popular differentially private estimators of $\theta_\star$ in the proportional high-dimensional asymptotic setting $\dim \rightarrow \infty, \ssize/\dim \rightarrow \delta$ for some $\delta \in (0,\infty)$. These papers primarily focus on obtaining sharp characterizations of the \emph{estimation error}. The privacy guarantees in these works are worst-case non-asymptotic privacy {bounds}. Our work takes a different, complementary focus. We are particularly interested in obtaining a \emph{sharp characterization} of the exact privacy loss in the high-dimensional limit $\dim \rightarrow \infty$ (in addition to a sharp characterization of the utility). Moreover, our privacy and utility results do not make distributional assumptions and apply to deterministic datasets. The interesting concurrent work of \citet{lin2025high}  studies noisy stochastic gradient descent for differentially private high-dimensional linear regression with the least squares loss and ridge penalty in the model-based setting. Like this paper, \citet{lin2025high} propose high-dimensional asymptotic privacy characterizations for noisy SGD, assuming as a conjecture that, in the high-dimensional limit ($\dim \to \infty$), the distribution of the noisy SGD iterates is indistinguishable from the law of a certain diffusion process called noisy homogenized SGD. Finally, we also refer the reader to the works of \citet{wei2016analysis} and \citet{wei2025asymptotic}, which also derive high-dimensional asymptotic approximations to the utility of the exponential mechanism. The scaling regimes and assumptions in these works differ significantly from those of our work, making a direct comparison of the utility results challenging.

%\RD{We also mention the work \citep{lin2025high}, which studies the estimation error and privacy of the stochastic gradient descent algorithm for linear regression through the analysis of a stochastic process as its proxy.}

%{\color{magenta} Unlike existing works in high-dimensional asymptotics, we do not impose distributional assumptions on the dataset, so our theoretical predictions apply to real-world datasets. Likewise, in our setting, there does not exist a population-level covariance matrix, and our goal is to estimate the the leading PCs of the sample covariance matrix $\mSigma(\mX).$}

\paragraph{Spherical Integrals and Spin Glasses} The Gibbs distribution in \eqref{eq:gibbs-intro} has also been studied in statistical physics and probability theory \cite{kosterlitz1976spherical,guionnet2005fourier,baik2016fluctuations, baik2017fluctuations, baik2021spherical,baik2018ferromagnetic,landon2020fluctuations,guionnet2021asymptotics,husson2025spherical}, where it is referred to as the spherical Sherrington–Kirkpatrick model \citep{kosterlitz1976spherical}, and its normalizing constant is called a spherical integral \citep{guionnet2005fourier}. In particular, our approach to analyzing this Gibbs distribution builds on the work of \citet{guionnet2005fourier}, and a subsequent refinement of their result by \citet{guionnet2021asymptotics} plays an important role in our proofs. A series of works \cite{hoff2009simulation,kume2006sampling,kapralov2013differentially,kent2018new,leake2021sampling,kent2004simulation} have also studied the problem of sampling from this Gibbs distribution. We have found that the Gibbs sampler developed by \citet{hoff2009simulation} works well in practice, and we use this sampler for all our experimental results. Unfortunately, a mixing time bound for this sampler is not known. The problem of designing samplers for the Gibbs distribution with polynomial run-time guarantees has been investigated \citep{kapralov2013differentially,ge2021efficient,leake2021sampling,power2026complexity}, but existing results are not immediately applicable to our setting as they either focus on the $k=1$ case \citep{kapralov2013differentially,ge2021efficient,power2026complexity} or study a complex analog of the Gibbs distribution \citep{leake2021sampling}, which is supported on the set of complex $ p\times k$ column-orthonormal matrices. 

{\paragraph{Central Limit Theorems for Composition} Our work characterizes the privacy loss of a mechanism by analyzing the asymptotics of the trade-off function associated with it. Asymptotics of trade-off functions have also been studied to analyze the overall privacy of the composition of a sequence of $t$ differentially private mechanisms $\rM_1, \rM_2, \dotsc, \rM_t$, where each mechanism takes as input a dataset $\mX$ and a tuning parameter $\tau$. \citet{dong2022gaussian} show that if each mechanism  satisfies the differential privacy guarantee of the form:
\begin{align*}
    \tf{\rM_i(\mX, \tau)}{\rM_i(\tilde{\mX}, \tau)} \geq \tf{\mu_i}{\nu_i} \quad \forall \; i \; \in \; [t] 
\end{align*}
for any two neighboring datasets $\mX,\tilde{\mX}$ and any choice of the tuning parameter $\tau$, then the composition of these mechanisms $\rM_1 \circ \rM_2 \circ \dotsb \circ \rM_t$ obtained by using the output of the first $i-1$ mechanisms to choose the tuning parameter of the $i$th mechanism satisfies the following overall privacy guarantee:
\begin{align*}
    \tf{\rM_1 \circ \dotsb \circ \rM_t (\mX)}{\rM_1 \circ \dotsb \circ \rM_t (\tilde{\mX})} \geq \tf{\mu_1 \otimes \mu_2 \otimes \dotsb \otimes \mu_t}{\nu_1 \otimes \nu_2 \otimes \dotsb \otimes \nu_t}
\end{align*}
for any two neighboring datasets $\mX,\tilde{\mX}$.  Moreover, \citet{dong2022gaussian} show that, under suitable assumptions, as $t \rightarrow \infty$, the trade-off function on the RHS converges to a limiting Gaussian trade-off function. This result holds because the optimal test to distinguish between the two product distributions $\mu_1 \otimes \dotsb \otimes \mu_t$ and $\nu_1 \otimes \dotsb \otimes \nu_t$ is based on the log-likelihood ratio, which is a sum of independent random variables under the null and the alternative, and hence satisfies a CLT that gives rise to a limiting Gaussian trade-off function. The follow-up work by \citet{pandey2025infinitely} characterizes all possible limiting trade-off functions that can arise in this composition setting. Our work shows that limiting Gaussian trade-off functions can arise in high dimensions beyond the composition setting, even for one-shot mechanisms like the exponential mechanism whose output distribution is not a product distribution.} 

\paragraph{High-Dimensional Central Limit Theorems in Differential Privacy} \citet{dong2021central} also study high-dimensional asymptotics of trade-off functions. They consider differentially private query answering, where the goal is to privately release a high-dimensional query $f(\cdot) \in \R^\dim$ with sensitivity  $\sup_{X,\tilde{X} \in \calN(X)} \|f(X) - f(\tilde{X})\|_2 = 1,$ using an additive noise mechanism $\rM(\cdot)$. Under suitable regularity conditions on the density of the added noise and its scale, they show that there exists a set $E_\dim$ of directions on the unit sphere, with Haar measure $1-o_\dim(1)$, such that:
\begin{align}\label{eq:clt-res}
    \underset{\dim\rightarrow\infty}{\liminf}\underset{\substack{
    \frac{f(X)-f(\tilde{X})}{\|f(X)-f(\tilde{X})\|_2}\in E_p}}{\inf}\tf{\rM(X)}{\rM(\tilde{X})}(\alpha) \geq \tf{\gauss{0}{1}}{\gauss{1}{1}}(\alpha) \quad \forall \; \alpha \; \in \; [0,1].
\end{align}
Thus, for neighboring datasets whose normalized query difference lies in $E_\dim$, distinguishing between the two datasets based on the mechanism's output is asymptotically at least as hard as distinguishing $\gauss{0}{1}$ from $\gauss{1}{1}$. While \citet{dong2021central} also study the high-dimensional asymptotics of trade-off functions, their work differs from ours in several ways. First, the Gaussian limit in \eqref{eq:clt-res} follows from Sudakov’s theorem \citep{sudakov1978typical}, which states that under weak assumptions, most (but not all) 1-dimensional projections of a high-dimensional random vector are asymptotically Gaussian. Consequently, their result in \eqref{eq:clt-res} excludes an exceptional set of directions of vanishing Haar measure.  In contrast, our work characterizes the worst-case trade-off over all neighboring datasets of a given dataset. The limiting Gaussian trade-off function in our work arises because the likelihood ratio between the output distributions of the exponential mechanism on two neighboring datasets can be approximated by a quadratic function of a high-dimensional Gaussian vector, which satisfies a CLT for \emph{any} neighboring dataset.

%% file: main_proof_overview.tex
\section{Proof Overview}
We now discuss a few key ideas involved in the proofs of our main results.

\subsection{Spherical Integrals}
Our starting point is a result of \citet{guionnet2005fourier} and a subsequent refinement of their result by \citet{guionnet2021asymptotics}, who study the asymptotics of $Z(\mSigma,\beta, \rnk)$, the normalizing constant of the Gibbs distribution $\nu(\cdot \mid \mSigma, \beta, \rnk)$ introduced in \Cref{def:Gibbs}:
\begin{align*}
Z(\mSigma,\beta,\rnk) \explain{def}{=}  \E_{\rV \sim \xi_{\dim,\rnk}} \left[ \exp\left( \frac{\dim\beta }{2}  \Tr[ \rV^\top \mSigma \rV]\right) \right].
\end{align*}
This is often called a spherical integral in this literature. As we will see shortly, the normalizing constant encodes fewer important properties of the Gibbs distribution. The works \citep{guionnet2005fourier,guionnet2021asymptotics} analyze the asymptotics of the log-normalizing constant under the following assumption on $\mSigma$ (which are compatible with the assumptions imposed on the empirical second moment matrix $\mSigma(\mX)$ in  \Cref{assump:data}). \begin{assumption} \label{assump:mat} 
The sequence of symmetric matrices $\mSigma$ satisfies the following requirements:
%The symmetric matrix $\Sigma \in \R^{\dim \times \dim}$ satisfies the following requirements:
\begin{enumerate}
    \item For some constant $\rnk \in \N$ (independent of $p$), the largest $\rnk+1$ eigenvalues of $\mSigma$ converge to some finite limits $\gamma_{1:\rnk+1}$ as $\dim \rightarrow \infty$:
    \begin{align*}
        \lambda_i(\mSigma) & \rightarrow \gamma_{i} \in [0,\infty) \quad \forall \; i \; \in \; [\rnk+1].
    \end{align*}
    Moreover, $\limsup_{\dim \rightarrow \infty} \|\mSigma\| < \infty,$ and the asymptotic spectral gap $\Delta \bydef \gamma_{\rnk} - \gamma_{\rnk+1}$ is strictly positive. 
     \item As $\dim \rightarrow \infty$, $\mu_{\mSigma}$, the empirical distribution of the smallest $\dim - \rnk$ eigenvalues $\mSigma$ converges weakly to some limiting spectral measure $\mu$ with bounded support:%\footnote{Observe that this assumption is naturally satisfied if $\mSigma$ is a positive semi-definite matrix.}
    \begin{align*}
        \mu_{\mSigma} & \bydef \frac{1}{\dim} \sum_{i=k+1}^{\dim} \delta_{\lambda_i(\mSigma)} \rightarrow \mu.
    \end{align*}
\end{enumerate}
\end{assumption}
The following is a restatement of the works \citep{guionnet2005fourier,guionnet2021asymptotics}.\footnote{In \citet[Proposition 2.1]{guionnet2021asymptotics}, $\beta_{\dim}$ does not change with $\dim.$ However, the result still holds as the dependence on $\dim$ can be absorbed into the matrix $\mA.$ Specifically,  $\exp(\dim\beta_{\dim}/2\Tr[\rU^\top\mA\rU])=\exp(\beta \dim/2\Tr[\rU^\top\beta_{\dim}/\beta\mA\rU]),$ and the sequence of matrices $\beta_{\dim}/\beta\mA$ satisfies the same set of assumptions (cf. \Cref{assump:mat}) as $\mSigma$ since $\|\beta_{\dim}/\beta\mA-\mA\|=o(1)$ (see \Cref{lem:misc_conv}).}
%the supplement \citep[Lemma H.1]{supplementary}).}
. 
\begin{fact}[{\citet[Proposition 2.1]{guionnet2021asymptotics}}]\label{fact:guionnet} Consider any fixed rank $\rnk\in\N$ (independent of $\dim$), a sequence of matrices $\mSigma \in \R^{\dim \times \dim}$ satisfying \Cref{assump:mat}, and a sequence of noise parameters $\beta_p \rightarrow \beta \in (0,\infty).$ Then, 
\begin{align}\label{eq:guionnet}
    \lim_{\dim \rightarrow \infty} \frac{\ln  Z(\mSigma, \beta_\dim, \rnk)}{\dim} & =  \frac{1}{2}\sum_{i=1}^\rnk f_\mu(\gamma_i),
\end{align}
where $f_\mu:(\gamma_{\rnk+1},\infty)\rightarrow\R$ is defined as:
\begin{align*}
    f_\mu(\gamma)=\begin{cases}\beta\gamma-\ln \beta-\int_{\R} \ln(\gamma-\lambda)\diff\mu(\lambda)-1&\text{if }H_\mu(\gamma)\leq\beta\\
    \beta H_\mu^{-1}(\beta)-\ln \beta - \int_{\R}\ln (H_\mu^{-1}(\beta)-\lambda)\diff\mu(\lambda)-1&\text{if }H_\mu(\gamma)>\beta.
    \end{cases}
\end{align*}
In the above equation $H_\mu^{-1}$ denotes the inverse function of $H_\mu$.
\end{fact}

\subsection{Utility Analysis}
Next, we show how our utility result follows from \Cref{fact:guionnet}.  Specifically, we prove the following result about the asymptotics of the overlap matrix $\mU_\star^\top \rV \rV^\top \mU_\star$ between the top $\rnk$ eigenvectors of $\mSigma$ and a sample $\rV$ drawn from the Gibbs distribution $\nu(\cdot \mid \mSigma, \beta, \rnk)$.
\begin{theorem}\label{thm:overlap} 

Consider a sequence of matrices $\mSigma\in\R^{\dim \times \dim}$ which satisfies \Cref{assump:mat}, a sequence of noise parameters $\beta_{\dim} \rightarrow \beta \in (0,\infty)$ as $\dim \rightarrow \infty$, and a fixed rank $\rnk \in \N$ (independent of $\dim$). Let $\mU_{\star} \in \R^{\dim \times \rnk}$ denote the matrix of the first $\rnk$ eigenvectors of $\mSigma,$ and the $\rV$ denote a sample from the Gibbs distribution $\nu( \cdot \mid \mSigma, \beta_\dim, \rnk)$. Then, as $\dim \rightarrow \infty,$
\begin{align*}
    \mU_{\star}^\top \rV \rV^\top \mU_{\star}  \pc \diag\left(1 - \frac{H_\mu(\gamma_1)}{\beta}, 1 - \frac{H_\mu(\gamma_2)}{\beta}, \dotsc, 1-\frac{H_\mu(\gamma_\rnk)}{\beta} \right)_+.
\end{align*}
\end{theorem}
Note that our utility result (\Cref{thm:utility}) immediately holds by applying the above theorem to the second moment matrix $\mSigma(\mX)$. 
\begin{proof}[Proof sketch.] We sketch the main ideas here and provide a complete proof in 
\Cref{app:overlap}.
%the supplement \citep[Appendix A]{supplementary}.
Note that it is sufficient to show that for any $\epsilon >0$ and any symmetric matrix $B \in \R^{\rnk \times \rnk}$,
\begin{align} \label{eq:target-concentration}
    \lim_{\dim \rightarrow \infty} \P \left( \left| \ip{\mU_{\star}^\top \rV \rV^\top \mU_{\star}}{\mB} - \ip{\mD}{\mB} \right| > \epsilon \right) & = 0 \quad \text{where}\quad\mD \bydef \diag\left(1 - \frac{H_\mu(\gamma_{1:\rnk})}{\beta}\right)_+.
\end{align}
Indeed, taking $B = (e_i e_j^\top + e_j e_i^\top)/2$ for the standard basis vectors $e_{1:\rnk} \in \R^\rnk$ shows the entry-wise convergence of the overlap matrix $\mU_{\star}^\top \rV \rV^\top \mU_{\star}$ to the limit $D$.  To show \eqref{eq:target-concentration}, the main idea is to study $M_\dim(t)$, the log-MGF (moment generating function) of the random variable $\ip{\mU_{\star}^\top \rV \rV^\top \mU_{\star}}{\mB}$: 
\begin{align*}
    M_\dim(t) \bydef \frac{1}{p} \ln \E_{\rV \sim \nu(\cdot \mid \mSigma, \beta_\dim, \rnk)} \left[ e^{\frac{t\dim\beta_{\dim}}{2}\langle\mU_{\star}^\top \rV \rV^\top \mU_{\star},\mB\rangle}  \right] \quad \forall \; t \; \in \; \R.
\end{align*}
Note that many useful properties of a random variable can be read off from the log-MGF:
\begin{itemize}
    \item \textit{Expectation.} Intuitively, we expect that $\ip{\mU_{\star}^\top \rV \rV^\top \mU_{\star}}{\mB}$ should concentrate to its expectation, which can be obtained from the derivative of the log-MGF at $0$:
    \begin{align*}
       M_\dim^\prime(0) & = \frac{\beta_\dim \cdot \E[\ip{\mU_{\star}^\top \rV \rV^\top \mU_{\star}}{\mB}]}{2} . % \frac{1}{p} \ln \E \left[ e^{\frac{\dim\beta_{\dim}}{2}\Tr[\rU^\top\mSigma(\mX)\rU]+\frac{t\dim\beta_{\dim}}{2}\langle\mU_{\star}^\top \rU \rU^\top \mU_{\star},\mB\rangle}  \right] \quad \forall \; t \; \in \; \R,\quad\text{where}\;\rU\sim \xi_{\dim,\rnk},\;\mB=\mI.
    \end{align*}
   Indeed, we arrived at the claimed formula for the limit of $\ip{\mU_{\star}^\top \rV \rV^\top \mU_{\star}}{\mB}$ (that is, $\ip{\mD}{\mB}$) by inspecting the derivative of the log-MGF. 
    \item \textit{Concentration Estimates.} The assertion \eqref{eq:target-concentration} is a concentration estimate, which can be readily derived from the log-MGF through the standard Chernoff argument. 
\end{itemize}
Consequently, the proof boils down to analyzing the asymptotics of the log-MGF. The key step in doing so is to note that the log-MGF is related to the normalizing constant $Z(\cdot, \beta_\dim,\rnk)$ of the Gibbs distribution. Indeed, recalling the density of the Gibbs measure from \Cref{def:Gibbs}, we find that:
\begin{align*}
    M_\dim(t) 
& \explain{Def. \ref{def:Gibbs}}{=} \frac{1}{\dim}\ln\E_{\rU \sim \xi_{\dim,\rnk}}\left[e^{\frac{\dim\beta_{\dim}}{2}\Tr[\rU^\top \mSigma\rU]} \cdot e^{\frac{t \dim\beta_{\dim}}{2}\langle\mU_{\star}^\top \rU \rU^\top \mU_{\star},\mB\rangle}\right] - \frac{1}{\dim}\ln\E_{\rU \sim \xi_{\dim,\rnk}}\left[e^{\frac{\dim\beta_{\dim}}{2}\Tr[\rU^\top \mSigma\rU]}\right]\\
    &\explain{}{=}{\frac{1}{\dim}\ln\E_{\rU \sim \xi_{\dim,\rnk}}\left[e^{\frac{\dim\beta_{\dim}}{2}\Tr[\rU^\top (\mSigma+t\mU_\star\mB\mU_\star^\top)\rU]}\right]} - {\frac{1}{\dim}\ln\E_{\rU \sim \xi_{\dim,\rnk}}\left[e^{\frac{\dim\beta_{\dim}}{2}\Tr[\rU^\top \mSigma\rU]}\right]} \notag \\
    & \explain{Def. \ref{def:Gibbs}}{=} \frac{\ln Z(\mSigma +t\mU_\star\mB\mU_\star^\top, \beta_\dim, \rnk) }{\dim} - \frac{\ln Z(\mSigma, \beta_\dim, \rnk) }{\dim}.
\end{align*}
Hence, the asymptotics of the log-MGF $M_\dim(t)$ can be readily obtained using the asymptotic formula for the log-normalizing constant $Z(\cdot , \beta_\dim, \rnk)$ from \Cref{fact:guionnet}.
\end{proof}

\subsection{Sampling Algorithm}
\begin{algorithm}
\caption{\textsc{Sampler}($\mSigma,\beta,\rnk$)}
\label{alg:sampler}
\begin{algorithmic}
\STATE {\textit{Input:}} Matrix $\mSigma \in \R^{\dim \times \dim}$ with eigendecomposition $\mSigma = \mU \diag(\lambda_{1:\dim}) \mU^\top$, noise parameter $\beta > 0$, rank $\rnk\in\N.$
\STATE {\textit{Output:}} An approximate sample $\rV$ from $\nu(\cdot \mid \mSigma, \beta, \rnk)$.
\begin{itemize}
\item Sample a uniformly random orthogonal matrix $\rQ \sim \xi_{\rnk,\rnk} \bydef \unif{\O(\rnk)}$. 
\item Generate a random matrix $\rZ \in \R^{(\dim-\rnk) \times \rnk}$ with entries: $$\rZ_{ij} \explain{}{\sim} \gauss{0}{\frac{1}{\beta \dim( \lambda_j - \lambda_{\rnk+i})}} \quad i \in [\dim-\rnk], \; j \; \in \; [\rnk],$$
sampled independently of each other and $\rQ$.
\item Set:
\begin{align*}
    \rV & \bydef \mU \begin{bmatrix} (I_k - \rZ^\top \rZ)_+^{1/2} \\ \rZ \end{bmatrix}  \rQ.
\end{align*}
\end{itemize}
\STATE \textit{Return:} Sample $\rV$.
\end{algorithmic}
\end{algorithm}
To prove our privacy result, we will also need to analyze fluctuations of certain statistics of the form $f(\rV \rV^\top)$ where $\rV$ is a random sample from the Gibbs measure $\nu(\cdot \mid \mSigma, \beta, \rnk)$. Unfortunately, such results do not follow readily from \Cref{fact:guionnet}. To obtain such results, we will rely on the sampling scheme given in \Cref{alg:sampler}, which constructs an approximate sample from the Gibbs distribution using independent Gaussian random variables and a uniformly random orthogonal matrix. The following result (proved in 
\Cref{app:sampling})
%the supplement \citep[Appendix B]{supplementary})
shows that the output distribution of the sampling algorithm approximates the Gibbs distribution in total variation distance.
\begin{theorem}\label{thm:sampling} Consider a sequence of matrices $\mSigma \in \R^{\dim \times \dim}$ which satisfies \Cref{assump:mat}, a sequence of noise parameters $\beta_\dim \rightarrow \beta \in (H_\mu(\gamma_k),\infty)$ as $\dim \rightarrow \infty$, and a fixed rank $\rnk \in \N$ (independent of $\dim$). Let $\hat{\nu}(\cdot \mid \mSigma, \beta_\dim, \rnk)$ denote the output distribution of the sampler (\Cref{alg:sampler}). Then,
\begin{align*}
    \lim_{\dim \rightarrow \infty} \tv[{\nu}(\cdot \mid \mSigma, \beta_\dim, \rnk), \; \hat{\nu}(\cdot \mid \mSigma, \beta_\dim, \rnk)] & = 0. 
\end{align*}
\end{theorem}
We use \Cref{alg:sampler} as a proof technique to derive limit distributions of statistics of the form $f(\rV \rV^\top)$ where $\rV \sim \nu(\cdot \mid \mSigma, \beta, \rnk)$. The basic idea is that since the output distribution of the sampler approximates the Gibbs distribution in total variation distance, it suffices to analyze the limit distribution of $f(\rV \rV^\top)$ when $\rV$ is the output of the sampler. Now, $f(\rV \rV^\top)$ can be expressed as a non-linear transformation of independent Gaussian random variables $\rZ$ from the sampling algorithm, and CLTs for non-linear functions of Gaussian random variables can be easily obtained using the second-order Poincaré inequality \citep{chatterjee2009fluctuations}. 

\paragraph{Understanding the Exponential Mechanism} The sampling scheme in \Cref{alg:sampler} also sheds some light on the nature of noise introduced by the exponential mechanism for privatization. For simplicity, let us consider the case when  $\rnk=1$ and assume that the eigenvalues $\lambda_{1:\dim}$ of the second moment matrix $\mSigma := \mSigma(\mX)$ are distinct. In this setting, \Cref{thm:sampling} suggests that the privatized PC generated by the exponential mechanism can be approximated in the total variation distance by $\rv \in \R^\dim$, the output of the sampling algorithm, which is given by:
\begin{align*}
    \rv\bydef \rs\sqrt{(1-\|\rz\|^2)_+} \cdot \vu_1^{} + \underbrace{\sum_{i=1}^{\dim-1}\rz_i \cdot \vu_{i+1}}_{\bydef \rw \text{ (noise)}},
\end{align*}
where
\begin{align*}
    \rs\sim\mathrm{Unif}(\{-1,1\}), \quad \rz_i\overset{\indep}{\sim} \gauss{0}{\frac{1}{\dim\beta(\lambda_1-\lambda_{i+1})}} \; \forall \;  i\;\in\;[\dim-1].
\end{align*}
In the above equations, $\lambda_1, \lambda_2,\dotsc, \lambda_\dim$ denote the eigenvalues of $\mSigma,$ and $\vu_1^{},\vu_2,\dotsc,\vu_{\dim}$ denote the corresponding eigenvectors. The second term in the above equations, denoted by $\rw$, can be interpreted as the noise introduced by the exponential mechanism for privatization. Note that this noise is Gaussian, but it is anisotropic.  The noise in the direction $\vu_{i+1}$ follows the distribution:
\begin{align} \label{eq:exp-mech-aniso}
    \ip{\rw}{\vu_{i+1}} \sim \gauss{0}{\frac{1}{\beta \dim(\lambda_1 - \lambda_{i+1})}} \quad \forall \; i \; \in \; [\dim-1].
\end{align}
The noise variance is largest in the direction $\vu_2$ and smallest in the direction $\vu_\dim$. Using anisotropic noise for privatization can be useful if the statistic of interest has different sensitivities in different directions. This is indeed the case for the leading principal component $\vu_1(\mSigma)$. One way of quantifying the local sensitivity of $\vu_1(\mSigma)$ is to fix a perturbation direction $\mE$ (which represents the perturbation introduced in the empirical second moment matrix by the addition/removal of a data point) with $\|\mE\|  = 1$ and to study the change in the leading PC, $\vu_1^{}(\mSigma + t \mE) - \vu_1^{}(\mSigma)$, via its first order Taylor series approximation around $t=0$. Using the Hadamard first variation formula for the eigenvector derivative (see e.g., \citep[Section 1.3.4]{tao2012topics}\footnote{While the Hadamard first variation formula in \citep[Section 1.3.4]{tao2012topics} provides a formula for the derivative of the eigenvalue, an analogous derivation gives the formula for the projection of the eigenvector along another eigenvector.}), 
\begin{align*}
    \vu_1^{}(\mSigma + t \mE) - \vu_1^{}(\mSigma) & \approx t \sum_{i= 1}^{\dim - 1} \frac{\vu_{i+1}^\top \mE\vu_1^{}}{\lambda_1-\lambda_{i+1}} \cdot \vu_{i+1}.
\end{align*}
Hence we expect that, for any $i \in [\dim -1]$, the maximum possible change in the direction $\vu_{i+1}$ is approximately:
\begin{align*}
    \sup_{\|\mE\| = 1} \ip{\vu_1^{}(\mSigma + t \mE) - \vu_1^{}(\mSigma)}{\vu_{i+1}} & \approx \sup_{\|\mE\| = 1} \frac{t\vu_{i+1}^\top \mE\vu_1^{}}{\lambda_1-\lambda_{i+1}} =  \frac{t}{\lambda_1 - \lambda_{i+1}}.
\end{align*}
This suggests that to privatize the leading PC, one should add more noise along $\vu_2$ than $\vu_\dim$ to hide (from the adversary) the larger perturbation introduced along $\vu_2$ compared to $\vu_{\dim}$ by the addition/removal of a data point. This is consistent with how the exponential mechanism calibrates the noise level in different directions (cf. \eqref{eq:exp-mech-aniso}). It would be interesting to study if the exponential mechanism satisfies any optimality properties  as a consequence of this observation.
\subsection{Contiguity Result}
To prove our privacy result, we need to understand the asymptotics of the trade-off function:
\begin{align*}
    \tf{\nu(\cdot \mid \mSigma(\mX), \beta, \rnk)}{\nu(\cdot \mid \mSigma(\tilde{\mX}), \beta, \rnk)}
\end{align*}
between the output distributions of the exponential mechanism on two neighboring datasets $\mX, \tilde{\mX}$ (along with the Rényi divergence between these distributions). We show that in the high-dimensional regime, these output distributions are mutually contiguous and hence, the corresponding trade-off function can be derived using classical techniques due to \citet{le2012asymptotic}. We begin by giving a brief overview of contiguity in \Cref{sec:contiguity-overview} and then present our main contiguity result in \Cref{sec:contiguity-result}.
\input{main_contiguity_overview}
\subsubsection{Contiguity Result}\label{sec:contiguity-result}
The following theorem states our contiguity result and shows that the Gibbs distributions $\nu( \cdot \mid \mSigma, \beta, \rnk)$ and $\nu(\cdot \mid \tilde{\mSigma}, \beta, \rnk)$ are mutually contiguous provided that the perturbation $\mSigma - \tilde{\mSigma}$ is sufficiently small. To state this result, for a matrix $\mSigma$ which satisfies \Cref{assump:mat}, we introduce the variance function $\sigma_{\mSigma}^2: \R^{\dim \times \dim} \times [0,\infty) \to \R$:
\begin{align}\label{eq:var-func}
    \sigma_{\mSigma}^2(\mE,\beta) \explain{def}{=} \frac{1}{2} \sum_{j, \ell = 1}^\rnk K_\mSigma(\lambda_j, \lambda_\ell) \cdot (\vu_j^\top E \vu_\ell)^2  + \sum_{j=1}^\rnk \sum_{i= 1}^{\dim - \rnk} \frac{\beta - {H_{\mSigma}(\lambda_j)}}{\lambda_j - \lambda_{\rnk+i}}  \cdot (\vu_{\rnk+i}^\top E \vu_j)^2,
\end{align}
where $\{(\lambda_{i}, \vu_i): i \in [\dim] \}$ denote the eigenvalues and corresponding eigenvectors of $\mSigma$ and the functions $H_\mSigma, K_\mSigma$ are defined, for all $\lambda, \lambda^\prime \; \in \; (\lambda_{\rnk+1}, \infty),$ as: 
\begin{align}\label{eq:empirical-HG}
    H_{\mSigma}(\lambda) &\explain{def}{=} \frac{1}{\dim} \sum_{i=1}^{\dim - \rnk} \frac{1}{\lambda - \lambda_{\rnk+i}}, \quad  K_{\mSigma}(\lambda,\lambda^\prime) \explain{def}{=} \frac{1}{\dim} \sum_{i=1}^{\dim-\rnk} \frac{1}{(\lambda - \lambda_{\rnk+i})(\lambda^\prime - \lambda_{\rnk+i})}.
\end{align}
Roughly speaking, $\sigma_{\mSigma}^2(\mE,\beta)$ can be interpreted as a proxy for the variance of the log-likelihood ratio between the Gibbs distributions $\nu( \cdot \mid \mSigma, \beta, \rnk)$ and $\nu(\cdot \mid \tilde{\mSigma}, \beta, \rnk)$ with $\tilde{\mSigma} = \mSigma + \tfrac{E}{\sqrt{\dim}}:$
\begin{align*}
    \sigma_{\mSigma}^2(\mE,\beta) & \approx \Var\left[ \ln \frac{\diff \nu(\rV \mid \tilde{\mSigma}, \beta, \rnk)}{\diff \nu( \rV \mid \mSigma, \beta, \rnk)} \right] \quad \text{when}\quad\rV \sim \nu( \cdot \mid \mSigma, \beta, \rnk).
\end{align*}
The following is the statement of the key contiguity result.
\begin{theorem} \label{thm:contiguity} Consider a sequence of matrices $\mSigma\in\R^{\dim \times \dim}$ which satisfies \Cref{assump:mat}, a sequence of noise parameters $\beta_{\dim} \rightarrow \beta \in (H_\mu(\gamma_\rnk),\infty)$ as $\dim \rightarrow \infty$, and a fixed rank $\rnk \in \N$ (independent of $\dim$). Let $\tilde{\mSigma} = \mSigma + \tfrac{1}{\sqrt{\dim}} \mE $ for a symmetric perturbation matrix $\mE \in \R^{\dim \times \dim}$ with $\|\mE\|_{\fr} \lesssim 1.$ Then, the following claims hold.
%$\|\mE\| \lesssim 1$ and $\|\mE\|_{\fr} \ll \sqrt{\dim}$. 
\begin{enumerate}
\item The Gibbs distributions $\nu(\cdot \mid \mSigma, \beta_\dim, \rnk)$ and $\nu(\cdot \mid \tilde{\mSigma}, \beta_\dim, \rnk)$ are mutually contiguous.
    \item If $\|\mE\|\ll1$ or $\sigma^2_{\mSigma}(\mE, \beta_\dim) \rightarrow 0,$ then:
    \begin{enumerate}
        \item $\lim_{\dim \rightarrow \infty} \tv(\nu(\cdot \mid \mSigma, \beta_\dim, \rnk), \nu(\cdot \mid \tilde{\mSigma}, \beta_\dim, \rnk)) = 0$. 
    \item For any fixed $\alpha > 1$ (independent of $\dim$), $\lim_{\dim \rightarrow \infty} \rdv{\alpha}{\nu(\cdot \mid \tilde{\mSigma}, \beta_\dim, \rnk)}{\nu(\cdot \mid {\mSigma}, \beta_\dim, \rnk)} = 0$.\end{enumerate}
    \item If $\sigma^2_{\mSigma}(\mE, \beta_\dim) \rightarrow v \in (0,\infty)$:
    \begin{enumerate}
       % \item  and the log-likelihood ratio between the two distributions satisfies:
        % \begin{align*}
        %     \ln \frac{\diff \nu(\rV \mid \tilde{\mSigma}, \beta_\dim, \rnk)}{ \diff \nu(\rV \mid \mSigma, \beta_\dim, \rnk)} \dc \gauss{-\frac{v}{2}}{v} \quad \text{when } \rV \sim \nu(\cdot \mid \mSigma, \beta_\dim, \rnk).
        % \end{align*}
        \item For any fixed $\alpha \in [0,1]$ (independent of $\dim$),
        $$\lim_{\dim \rightarrow \infty} \tf{\nu(\cdot \mid \mSigma, \beta_\dim, \rnk)}{\nu(\cdot \mid \tilde{\mSigma}, \beta_\dim, \rnk)}(\alpha) =  \tf{\gauss{0}{1}}{\gausss{\sqrt{v}}{1}}(\alpha).$$
        \item For any fixed $\alpha > 1$ (independent of $\dim$), $$\lim_{\dim \rightarrow \infty} \rdv{\alpha}{\nu(\cdot \mid \tilde{\mSigma}, \beta_\dim, \rnk)}{\nu(\cdot \mid {\mSigma}, \beta_\dim, \rnk)}=  \rdv{\alpha}{\gausss{\sqrt{v}}{1}}{\gauss{0}{1}}.$$
    \end{enumerate}
\end{enumerate}
\end{theorem}

\begin{proof}[Proof sketch.] We give a proof sketch highlighting the main ideas below, and the full proof is deferred to \Cref{app:contiguity}.
%the supplement \citep[Appendix C]{supplementary}. 
\paragraph{Showing Contiguity}
To show that the Gibbs distributions $\nu_{\dim}\bydef\nu(\cdot \mid \mSigma, \beta_\dim, \rnk),$ $\tilde{\nu}_{\dim}\bydef \nu(\cdot \mid \tilde{\mSigma}, \beta_\dim, \rnk)$ are contiguous, by Le Cam's first lemma (\Cref{fact:lecam-1}), we only need to check that the log-likelihood ratio between $\nu_{\dim}$ and $\tilde{\nu}_{\dim}$ satisfies:
\begin{align} \label{eq:lecam-criteria}
    \ln \frac{\diff\tilde{\nu}_{\dim}}{\diff\nu_{\dim}}(\rV) \dc \gauss{-\frac{v}{2}}{v}\quad \text{for some $v \geq 0$ when $\rV\sim\nu_{\dim}$}.
\end{align}
Recalling the formula for the densities of $\nu_{\dim}, \tilde{\nu}_{\dim}$ from \Cref{def:Gibbs}, we compute the log-likelihood ratio as:
\begin{align}\label{eq:pf_sketch_leCam}
    \ln \frac{\diff\tilde{\nu}_{\dim}}{\diff\nu_{\dim}}(\rV) & = \underbrace{\frac{\sqrt{\dim} \beta_\dim }{2} \Tr[\rV^\top \mE \rV]}_{(i)}  + \underbrace{\ln Z({\mSigma}, \beta_\dim, \rnk)-\ln Z(\tilde{\mSigma}, \beta_\dim, \rnk)}_{(ii)}.
\end{align}
To verify Le Cam's criterion in \eqref{eq:lecam-criteria}, we analyze the random term $(i)$ and the deterministic term $(ii)$. We discuss both of these terms below. 
\begin{itemize}
    \item \emph{Analysis of the Random Term $(i):$} We need to understand the limit distribution of random term $(i)$ when $\rV \sim \nu_\dim$ is sampled from the Gibbs distribution, which appears to be a complicated high-dimensional probability measure. We can simplify our problem by appealing to \Cref{thm:sampling}, which shows that it is sufficient to derive the limit distribution of the random term $(i)$ under the output distribution of the sampling algorithm (\Cref{alg:sampler}). When $\rV$ is the output generated by the sampler, $\Tr[\rV^\top \mE \rV]$ can be approximated by a quadratic form involving Gaussian random variables and hence, we can use the second-order Poincaré Inequality \citep{chatterjee2009fluctuations} to analyze its limit distribution.
    \item \emph{Analysis of Deterministic Term $(ii):$} To analyze the deterministic term $(ii)$, it is tempting to appeal to \Cref{fact:guionnet}, which implies that:
    \begin{align}\label{eq:pf_sketch_log_normalizing}
        \ln Z(\mSigma, \beta_\dim, \rnk)=\frac{\dim}{2}\sum_{i=1}^\rnk f_\mu(\gamma_i)+o(\dim),\quad \ln Z(\tilde{\mSigma}, \beta_\dim, \rnk)=\frac{\dim}{2}\sum_{i=1}^\rnk f_\mu(\gamma_i)+o(\dim),
 \end{align}
 where the function $f_\mu$ is as defined in \Cref{fact:guionnet}. Hence, the difference of the log-normalizing constants becomes:
\begin{align*}
    \ln Z(\mSigma, \beta_\dim, \rnk)-\ln Z(\tilde{\mSigma}, \beta_\dim, \rnk)=o(\dim).
\end{align*}
The $o(p)$ estimate on the asymptotic error in the above equation is too weak to verify Le Cam's criterion in \eqref{eq:lecam-criteria}, which requires a stronger result characterizing the asymptotic difference in the log-normalizing constants up to $o(1)$ error. Such results are unfortunately unavailable in the current literature. To bypass this difficulty, we use a contiguity criterion of \citet[Lemma 3.1]{mukherjee2013statistics}, which provides an analog of Le Cam's criterion for checking contiguity between distributions with complicated normalizing constants. This criterion helps us avoid a fine-grained asymptotic analysis of the log-normalizing constants of the Gibbs distribution. 
\end{itemize}
\paragraph{Asymptotics of the Trade-off Function} Once we have shown that the log-likelihood ratio satisfies Le Cam's criterion \eqref{eq:lecam-criteria} when $\rV \sim \nu_\dim$, Le Cam's third lemma (\Cref{eq:lecam-3}) automatically gives us the limit distribution of the log-likelihood ratio when $\rV \sim \tilde{\nu}_\dim$:
\begin{align*}
    \ln \frac{\diff\tilde{\nu}_{\dim}}{\diff\nu_{\dim}}(\rV) \dc \gauss{\frac{v}{2}}{v}\quad \text{ when\quad $\rV\sim\tilde{\nu}_{\dim}$}.
\end{align*}
Since the optimal test between $\nu_\dim$ and $\tilde{\nu}_\dim$ is based on the log-likelihood ratio, the trade-off function between $\nu_\dim,\tilde{\nu}_\dim$ converges to the trade-off function between the limit distributions of the log-likelihood ratio under these two probability measures (this statement can be formalized using standard arguments used in asymptotic theory of hypothesis testing \citep{lehmann2005testing}):
\begin{align*}
    \lim_{\dim \rightarrow \infty} \tf{\nu_\dim}{\tilde{\nu}_\dim}(\alpha) & = \tf{\gausss{-\tfrac{v}{2}}{v}}{\gausss{\tfrac{v}{2}}{v}}(\alpha) = \tf{\gausss{0}{1}}{\gausss{\sqrt{v}}{1}}(\alpha)\quad \forall\;\alpha\;\in\;[0,1],
\end{align*}
as claimed. 
\paragraph{Asymptotics of Rényi Divergence} Recall from \Cref{def:Rényi}, that the Rényi divergence $\rdv{\alpha}{\tilde{\nu}_\dim}{\nu_\dim}$ is related to the $\alpha$ moment of the likelihood ratio. We already know from \eqref{eq:lecam-criteria} that the log-likelihood ratio converges to a Gaussian distribution. We upgrade the weak convergence in \eqref{eq:lecam-criteria} to the convergence of moments via a uniform integrability argument, which gives us the claimed asymptotic formula for the Rényi divergence. 
\end{proof}

\subsection{Privacy Analysis (Proof of \texorpdfstring{\Cref{thm:privacy}}{thm:privacy})}\label{sec:privacy-proof}
We finish the overview of our proof techniques by deriving the sharp asymptotic privacy loss for the exponential mechanism given in \Cref{thm:privacy} using our contiguity result (\Cref{thm:contiguity}). \Cref{thm:privacy} characterizes the asymptotic privacy loss for the exponential mechanism in terms of trade-off functions, as well as Rényi divergence. We present the proof of the trade-off function characterization here, and the proof of the Rényi divergence characterization is presented in 
\Cref{app:privacy-renyi}.
%the supplement \citep[Appendix D]{supplementary}. 
The proof relies on a few intermediate results, which we introduce below. The proofs of these results are also deferred to 
\Cref{app:privacy-renyi}.
%the supplement \citep[Appendix D]{supplementary}.
\paragraph{Simplification Through an Intermediate Gibbs Measure} Characterizing the privacy of the exponential mechanism on a given dataset $\mX$ requires that we analyze the trade-off function between the output distributions of the exponential mechanism on $\mX$ and its worst-case neighboring dataset:
\begin{align*}
    \lim_{\dim \rightarrow \infty} \inf_{\tilde{\mX} \in \calN(\mX)}\tf{\nu(\cdot \mid \mSigma(\mX), \beta_\dim, \rnk)}{\nu(\cdot \mid \mSigma(\tilde{\mX}), \beta_\dim,\rnk)}. 
\end{align*}
We will compute these limits using  \Cref{thm:contiguity}, which tells us that the asymptotic values of the trade-off function depend on the magnitude of the perturbation matrix $\mSigma(\tilde{\mX})-\mSigma(\mX)$. A simple computation reveals:
\begin{align}\label{eq:privacy-perturb-orig}
    \Sigma(\tilde{\mX})-\Sigma(\mX)= -\underbrace{\frac{\Sigma(\mX)}{\ssize+1}}_{(\star)}+\frac{\vx\vx^\top}{\ssize+1}  \text{ if }\tilde{\mX} = \mX \cup \{\vx\}
\end{align}
and 
\begin{align*}
    \Sigma(\tilde{\mX})-\Sigma(\mX)=
    \underbrace{\frac{\Sigma(\mX)}{\ssize-1}}_{(\star)}-\frac{\vx\vx^\top}{\ssize-1}  \text{ if } \tilde{\mX} = \mX \backslash \{\vx\},
\end{align*}
% \begin{align}\label{eq:privacy-perturb-orig}
%     \Sigma(\tilde{\mX})-\Sigma(\mX)=\begin{dcases}\frac{\ssize\Sigma(\mX)+\vx\vx^\top}{\ssize+1}-\Sigma(\mX)=-\frac{\Sigma(\mX)}{\ssize+1}+\frac{\vx\vx^\top}{\ssize+1}&\text{if }\tilde{\mX} = \mX \cup \{\vx\}\\
%     \frac{\ssize\Sigma(\mX)-\vx\vx^\top}{\ssize-1}-\Sigma(\mX)=\frac{\Sigma(\mX)}{\ssize-1}-\frac{\vx\vx^\top}{\ssize-1}&\text{if } \tilde{\mX} = \mX \backslash \{\vx\}.
%     \end{dcases}
% \end{align}
where $\ssize\bydef |\mX|,$ and $\vx \in \R^\dim$ denotes the point added to or removed from $\mX$ to construct $\tilde{\mX}$. We will see that the terms marked $(\star)$ are asymptotically negligible and hence, to simplify our computations, it will be helpful to introduce the following simpler approximation to the second moment matrix, which removes these negligible terms: 
% so one might expect the terms marked $(\#)$ to contribute little to the the asymptotic privacy quantities. Based on this intuition, we define the matrix
\begin{align} \label{eq:simple-cov}
    \bar{\mSigma}(\tilde{\mX})\bydef\begin{dcases}
        \mSigma(\mX)+\frac{\vx\vx^\top}{\ssize}&\text{if }\tilde{\mX}\in\calN_+(\mX)\\
        \mSigma(\mX)-\frac{\vx\vx^\top}{\ssize}&\text{if }\tilde{\mX}\in\calN_-(\mX)
    \end{dcases}.
\end{align}
The following lemma shows that removing the terms marked $(\star)$ has a negligible effect and the Gibbs distributions $\nu(\cdot \mid \mSigma(\tilde{\mX}), \beta_\dim, \rnk)$ and $\nu(\cdot \mid \bar{\mSigma}(\tilde{\mX}), \beta_\dim, \rnk)$ are very close. 

\begin{lemma}\label{lem:simple-gibbs-approx} Under the assumptions of \Cref{thm:privacy}, for any neighboring dataset $\tilde{\mX} \in \calN(\mX)$,
\begin{align*}
    \lm \tv(\nu(\cdot \mid \mSigma(\tilde{\mX}), \beta_\dim, \rnk), \nu(\cdot \mid \bar{\mSigma}(\tilde{\mX}), \beta_\dim, \rnk)) &=0, \\ \lm \rdv{\alpha}{\nu(\cdot \mid \mSigma(\tilde{\mX}), \beta_\dim, \rnk)}{\nu(\cdot \mid \bar{\mSigma}(\tilde{\mX}), \beta_\dim, \rnk)} &=0\quad\forall\:\alpha>1, \\
    \lm  \rdv{\alpha}{\nu(\cdot \mid \bar{\mSigma}(\tilde{\mX}), \beta_\dim, \rnk)}{\nu(\cdot \mid \mSigma(\tilde{\mX}), \beta_\dim, \rnk)}&=0\quad\forall\:\alpha>1.
\end{align*}
\end{lemma}
Now, the proof boils down to understanding the asymptotics of the trade-off function between $\nu(\cdot \mid \mSigma({\mX}), \beta_\dim, \rnk)$ and the simplified Gibbs distribution $\nu(\cdot \mid \bar{\mSigma}(\tilde{\mX}), \beta_\dim, \rnk)$, and transferring the result to  $\nu(\cdot \mid \mSigma(\tilde{\mX}), \beta_\dim, \rnk)$ using the following perturbation bounds. 
\begin{lemma}[{\citet[Proposition 11]{mironov2017renyi}}] \label{lem:triangle-ineq} Let $\nu,$ $\tilde{\nu},$ and $\bar{\nu}$ be probability measures on a common space. Then:
\begin{enumerate}
    \item For any $\alpha\in[0,1],$ 
    \begin{align*}
        \left\vert\tf{\nu}{\tilde{\nu}}(\alpha)-\tf{\nu}{\bar{\nu}}(\alpha)\right\vert\leq \tv(\tilde{\nu},\bar{\nu}).
    \end{align*}
    \item For any $\alpha>1$ and $\epsilon\in(0,\alpha-1),$ 
    \begin{align*}
        \rdv{\alpha-\epsilon}{\bar{\nu}}{\nu}-\frac{\alpha}{\alpha-1}\rdv{\frac{(\alpha-\epsilon)(\alpha-1)}{\epsilon}}{\bar{\nu}}{\tilde{\nu}}\leq \rdv{\alpha}{\tilde{\nu}}{\nu}&\leq \rdv{\alpha+\epsilon}{\bar{\nu}}{\nu}+\frac{\alpha+\epsilon}{\alpha+\epsilon-1}\rdv{\frac{\alpha(\alpha+\epsilon-1)}{\epsilon}}{\tilde{\nu}}{\bar{\nu}}.
    \end{align*}
\end{enumerate}
\end{lemma}

% \begin{lemma}\label{lem:intermediate-Gibbs}
% Consider any three matrices $\mSigma,$ $\tilde{\mSigma},$ and $\bar{\mSigma}$ satisfying \Cref{assump:mat}. Suppose that 
% \begin{align*}
%     \|\tilde{\mSigma}-\bar{\mSigma}\|_{\fr}\lesssim\frac{1}{\sqrt{\dim}}\quad\text{and}\quad\sigma^2_{\bar{\mSigma}}(\sqrt{\dim}(\tilde{\mSigma}-\bar{\mSigma}), \beta_\dim) \rightarrow 0.
% \end{align*}
% Then:
% \begin{enumerate}
%     \item For any $\alpha\in[0,1],$
%     \begin{align*}
%         \li\tf{\nu(\cdot \mid \mSigma, \beta_{\dim}, \rnk)}{\nu(\cdot \mid \mSigma, \beta_{\dim}, \rnk)}(\alpha)=\li\tf{\nu(\cdot \mid \mSigma, \beta_{\dim}, \rnk)}{\nu(\cdot \mid \bar{\mSigma}, \beta_{\dim}, \rnk)}(\alpha).
%     \end{align*}
%     \item Further suppose that $\|\bar{\mSigma}-\mSigma\|_{\fr}\lesssim1/\sqrt{\dim}$ and $\lm\sigma^2_{\bar{\mSigma}}(\sqrt{\dim}(\bar{\mSigma}-\mSigma), \beta_\dim) $ exists. Then, for any $\alpha>1,$
%     \begin{align*}
%         \ls \rdv{\alpha}{\nu(\cdot \mid \tilde{\mSigma}, \beta_{\dim}, \rnk)}{\nu(\cdot \mid \mSigma, \beta_{\dim}, \rnk)}\leq\lm \rdv{\alpha}{\nu(\cdot \mid \bar{\mSigma}, \beta_{\dim}, \rnk)}{\nu(\cdot \mid \mSigma, \beta_{\dim}, \rnk)}.
%     \end{align*}
% \end{enumerate}
% \end{lemma}
\paragraph{Characterization of the Worst-Case Neighboring Dataset} 
To summarize our discussion so far, for a neighboring dataset $\tilde{\mX} \in \calN(\mX)$ constructed by adding or removing a data point $\vx \in \R^\dim$, we expect that:
\begin{align*}
    \tf{\nu(\cdot \mid \mSigma(\mX), \beta_\dim, \rnk)}{\nu(\cdot \mid \mSigma(\tilde{\mX}), \beta_\dim,\rnk)} \approx \tf{\nu(\cdot \mid \mSigma(\mX), \beta_\dim, \rnk)}{\nu(\cdot \mid \bar{\mSigma}(\tilde{\mX}), \beta_\dim,\rnk)}.
\end{align*}
Now, we can approximate the RHS  in the high-dimensional limit $\dim \rightarrow \infty$ using \Cref{thm:contiguity}:
\begin{align}
     &\tf{\nu(\cdot \mid \mSigma(\mX), \beta_\dim, \rnk)}{\nu(\cdot \mid \mSigma(\tilde{\mX}), \beta_\dim,\rnk)}\\
     &\quad\approx \tf{\gauss{0}{1}}{\gausss{\sigma_{\mSigma(\mX)}({\sqrt{\dim} (\bar{\mSigma}(\tilde{\mX})-\mSigma(\mX))},\beta_\dim)}{1}} \nonumber \\
     &\quad \explain{\eqref{eq:simple-cov}}{=} \tf*{\gauss{0}{1}}{\gauss{\sigma_{\mSigma(\mX)}\bigg(\pm \tfrac{ \sqrt{\dim} \vx \vx^\top}{\ssize},\beta_\dim\bigg)}{1}} \nonumber \\
     &\quad = \tf*{\gauss{0}{1}}{\gauss{\sigma_{\mSigma(\mX)}\bigg(\tfrac{ \sqrt{\dim} \vx \vx^\top}{\ssize},\beta_\dim\bigg)}{1}}, \label{eq:heursitic-plugin}
\end{align}
 where  $\sigma^2_{\mSigma(\mX)}(\cdot,\beta_{\dim})$ denotes the variance function introduced in  \Cref{thm:contiguity} (recall \eqref{eq:var-func}) and the last step follows because $\sigma^2_{\mSigma(\mX)}(\cdot,\beta_{\dim})$ is an even function. For notational convenience, we define the function $\sigma^2_{\mX}$:
\begin{align}\label{eq:var-func-x}
    \sigma^2_{\mX}(\vx,\beta)\bydef \sigma^2_{\mSigma(\mX)}\left(\frac{\sqrt{\dim}\vx\vx^\top}{\ssize},\beta\right) \quad \forall \; \beta > H_\mu(\gamma_\rnk), \; \vx \in \R^\dim.
\end{align}
Recall that our goal is to characterize the trade-off function for the \emph{worst-case} neighboring dataset. In light of \eqref{eq:heursitic-plugin}, we expect that:
\begin{align*}
    \inf_{\tilde{\mX} \in \calN(\mX)} \tf{\nu(\cdot \mid \mSigma(\mX), \beta_\dim, \rnk)}{\nu(\cdot \mid \mSigma(\tilde{\mX}), \beta_\dim,\rnk)} &\approx  \inf_{\substack{\vx \in \R^\dim \\ \|\vx\|^2 \leq \dim }}\tf*{\gauss{0}{1}}{\gauss{\sigma_{\mX}(\vx,\beta)}{1}}.
\end{align*}
Since $\gauss{0}{1}$ and $\gauss{\sigma}{1}$ are easier to distinguish for large values of $\sigma$,  the final ingredient to finish our privacy analysis is a characterization of the maximum value of the function $\sigma^2_{\mX}(\vx, \beta)$ over all valid possibilities of the added/removed data point $\vx$. The following proposition provides this characterization. 

\begin{proposition}\label{prop:var-lim-v2} Consider a sequence of datasets $\mX$ which satisfies \Cref{assump:data} and a sequence of noise parameters $\beta_\dim \rightarrow \beta \in (H_\mu(\gamma_\rnk), \infty)$. Let $\sigma^2_{\mX}:\R^\dim\times (H_{\mu}(\gamma_{\rnk}),\infty) \rightarrow \R$ denote the variance function defined in \eqref{eq:var-func-x}. Then,
\begin{align} \label{eq:var-func-max-v2}
    {\lim_{\dim \rightarrow \infty}} \sup_{\substack{\vx \in \R^p \\ \|x\|^2 \leq \dim}}\sigma_{\mX}^2(\vx,\beta_\dim)= \sigma_\beta^2 \bydef \begin{cases} \frac{1}{2 \Delta \theta^2} \cdot \frac{\left(\beta-H_{\mu}(\gamma_\rnk)\right)^2}{2({\beta-H_{\mu}(\gamma_\rnk)})+ \Delta H'_{\mu}(\gamma_\rnk)}&\text{if } \beta \geq -\Delta H'_{\mu}(\gamma_\rnk)+H_{\mu}(\gamma_{\rnk})\\
        -\frac{1}{2\theta^2}H'_{\mu}(\gamma_\rnk)&\text{if } \beta<-\Delta H'_{\mu}(\gamma_\rnk)+H_{\mu}(\gamma_{\rnk})
    \end{cases}.
\end{align}
Moreover, if $\vu_1, \vu_2, \dotsc, \vu_\dim$ denote the eigenvectors of $\mSigma(\mX)$, the data point:
\begin{align*}
    \vx_\star & = \sqrt{\dim} \left(  \sqrt{t_\star}  \vu_{\rnk} + \sqrt{1-t_\star}  \vu_{\rnk+1} \right) \quad \text{with}\quad t_\star \explain{def}{=}\begin{cases} \frac{{\beta-H_{\mu}(\gamma_\rnk)}}{2(\beta-H_{\mu}(\gamma_\rnk)) + \Delta H'_{\mu}(\gamma_\rnk)}&\text{if}\quad \beta\geq H_{\mu}(\gamma_{\rnk})-\Delta H'_{\mu}(\gamma_\rnk)\\
   1&\text{if}\quad \beta < H_{\mu}(\gamma_{\rnk})-\Delta H'_{\mu}(\gamma_\rnk)\end{cases}
\end{align*}
asymptotically achieves the maximum in \eqref{eq:var-func-max-v2}: 
\begin{align*}
    \lim_{\dim \rightarrow \infty} \sigma_{\mX}^2(\vx_\star,\beta_\dim) & = \lim_{\dim \rightarrow \infty} \sup_{\substack{\vx \in \R^{\dim} \\ \|x\|^2 \leq \dim}}\sigma_{\mX}^2(\vx,\beta).
\end{align*}
\end{proposition}
We now have all the ingredients needed to prove \Cref{thm:privacy}. We defer the proofs of the intermediate results to 
\Cref{app:privacy-renyi} 
%the supplement \citep[Appendix D]{supplementary}
and present the proof of \Cref{thm:privacy}, taking these results for granted. 
\begin{proof}[Proof of \Cref{thm:privacy}] Consider any sequence of datasets $\mX$ that satisfies \Cref{assump:data} and a sequence of noise parameters $\beta_{\dim} \rightarrow \beta\in(H_\mu(\gamma_k),\infty)$. Throughout the proof, for notational convenience, we use $\nu_\dim$ to denote the Gibbs distribution:
\begin{align*}
    \nu_\dim \bydef \nu(\cdot \mid \mSigma(\mX), \beta_\dim, \rnk), 
\end{align*}
which is the output distribution of the exponential mechanism on the dataset $\mX$. 
To show that the exponential mechanism (\Cref{alg:ExpM}) has an AGDP loss of $\sigma_\beta$, we need to prove that:
\begin{align} \label{eq:priv-thm-tf}
   \lm \inf_{\tilde{\mX}\in\calN(\mX)}\tf{\nu_{\dim}}{\nu(\cdot \mid \mSigma(\tilde{\mX}), \beta_{\dim}, \rnk)}(\alpha)=\tf{\gauss{0}{1}}{\gauss{\sigma_\beta}{1}}(\alpha) \quad \forall \; \alpha \; \in \; [0,1],
\end{align}
where $\sigma_\beta$ is as defined in \eqref{eq:var-func-max-v2}. We prove \eqref{eq:priv-thm-tf} in two steps. First, we show the lower bound: 
\begin{align}
   \li\inf_{\tilde{\mX}\in\calN(\mX)}\tf{\nu_{\dim}}{\nu(\cdot \mid \mSigma(\tilde{\mX}), \beta_{\dim}, \rnk)}(\alpha)\geq\tf{\gauss{0}{1}}{\gauss{\sigma_\beta}{1}}(\alpha).\label{eq:privacy-toshow1}
\end{align}
Then, we show that the lower bound above is asymptotically tight by showing that the worst-case data point $\vx_\star$ constructed in \Cref{prop:var-lim-v2} achieves it:
\begin{align}  \label{eq:achievability-priv-thm}  
\lim_{\dim \rightarrow \infty} \tf{\nu_{\dim}}{\nu(\cdot \mid \mSigma({\mX} \cup \{\vx_\star\}), \beta_{\dim}, \rnk)}(\alpha)&=\tf{\gauss{0}{1}}{\gauss{\sigma_\beta}{1}}(\alpha).
\end{align}
Note that \eqref{eq:privacy-toshow1} and \eqref{eq:achievability-priv-thm} together imply \eqref{eq:priv-thm-tf}. To show \eqref{eq:privacy-toshow1}, we start by letting $\tilde{\mX}_{\sharp}\in\calN(\mX)$ be an approximate minimizer of: $$ \inf_{\tilde{\mX}\in\calN(\mX)}\tf{\nu_{\dim}}{\nu(\cdot \mid \mSigma(\tilde{\mX}), \beta_{\dim}, \rnk)}(\alpha)$$ in the sense that:
\begin{align} \label{eq:priv-thm-interm}
    \tf{\nu_{\dim}}{\nu(\cdot \mid \mSigma(\tilde{\mX}_{\sharp}), \beta_{\dim}, \rnk)}(\alpha)\leq \inf_{\tilde{\mX}\in\calN(\mX)}\tf{\nu_{\dim}}{\nu(\cdot \mid \mSigma(\tilde{\mX}), \beta_{\dim}, \rnk)}(\alpha)+\frac{1}{\dim},
\end{align}
which guarantees that:
\begin{align*}
     \li \inf_{\tilde{\mX}\in\calN(\mX)}\tf{\nu_{\dim}}{\nu(\cdot \mid \mSigma(\tilde{\mX}), \beta_{\dim}, \rnk)}(\alpha)  &=\li \tf{\nu_{\dim}}{\nu(\cdot \mid \mSigma(\tilde{\mX}_{\sharp}), \beta_{\dim}, \rnk)}(\alpha)\notag\\& \explain{(a)}{=}\li \tf{\nu_{\dim}}{\nu(\cdot \mid \bar{\mSigma}(\tilde{\mX}_{\sharp}), \beta_{\dim}, \rnk)}(\alpha),
\end{align*}
where (a) follows from the first claim of \Cref{lem:simple-gibbs-approx} and \Cref{lem:triangle-ineq}(1). We compute the limit on the RHS using \Cref{thm:contiguity}. To ensure that the requirements of \Cref{thm:contiguity} are met, we pass to a subsequence along which the limit inferior is attained and then to a further subsequence\footnote{Bolzano–Weierstrass guarantees the existence of such a subsequence since $\lim\sup_{p\rightarrow\infty} \sigma^2_{\mSigma(\mX)}(\sqrt{\dim}(\bar{\mSigma}(\tilde{\mX}_{\sharp})-\mSigma({\mX})),\beta_{\dim})< \infty$, thanks to 
\Cref{prop:var-lim-v2}.} along which $$\sigma^2_{\mSigma(\mX)}(\sqrt{\dim}(\bar{\mSigma}(\tilde{\mX}_{\sharp})-\mSigma({\mX})),\beta_{\dim})\rightarrow v$$ for some $v\in[0,\infty)$. Now, applying \Cref{thm:contiguity}, we conclude that: 
\begin{align*}
     \li \inf_{\tilde{\mX}\in\calN(\mX)}\tf{\nu_{\dim}}{\nu(\cdot \mid \mSigma(\tilde{\mX}), \beta_{\dim}, \rnk)}(\alpha)  & = \tf{\gauss{0}{1}}{\gauss{\sqrt{v}}{1}}(\alpha)\\
     &\explain{(c)}{\geq}\tf*{\gauss{0}{1}}{\gauss{\sigma_\beta}{1}}(\alpha),
\end{align*}
where (c) holds by observing that if $\vx_\sharp$ denotes the data point added/removed from $\mX$ to construct $\mX_\sharp,$ then:
\begin{align*}
    v &\bydef \lm \sigma^2_{\mSigma(\mX)}(\sqrt{\dim}(\bar{\mSigma}(\tilde{\mX}_{\sharp})-\mSigma({\mX})),\beta_{\dim}) \\ &\explain{\eqref{eq:simple-cov}}{=} \lm \sigma^2_{\mSigma(\mX)}\left(\frac{\sqrt{\dim}\vx_\sharp\vx_\sharp^\top}{\ssize},\beta_{\dim}\right) \\&\bydef \lm\sigma^2_{\mX}(\vx_{\sharp},\beta)\\&\leq \lim_{\dim \rightarrow \infty} \sup_{\substack{\vx \in \R^\dim \\ \|\vx\|^2 \leq \dim}} \sigma^2_{\mX}(\vx,\beta_{\dim})\explain{Prop. \ref{prop:var-lim-v2}}{=} \sigma_\beta^2.
\end{align*}
This proves the lower bound \eqref{eq:privacy-toshow1}.  To complete the proof of the first claim of \Cref{thm:privacy}, we also need to prove \eqref{eq:achievability-priv-thm}. That is, the worst-case data point $\vx_{\star}$ constructed in \Cref{prop:var-lim-v2} achieves the lower bound proved above. By essentially following the same approach used for showing the lower bound, we have:
\begin{align*}
    \lim_{\dim \rightarrow \infty} \tf{\nu_{\dim}}{\nu(\cdot \mid \mSigma({\mX} \cup \{\vx_\star\}), \beta_{\dim}, \rnk)}(\alpha)
    &\explain{(a)}{=}\lim_{\dim \rightarrow \infty} \tf{\nu_{\dim}}{\nu(\cdot \mid \bar{\mSigma}({\mX} \cup \{\vx_\star\}), \beta_{\dim}, \rnk)}(\alpha)\\&\explain{(b)}{=}\tf{\gauss{0}{1}}{\gauss{\sigma_\beta}{1}}(\alpha),
\end{align*}
where step (a) follows from the first claim of \Cref{lem:simple-gibbs-approx}  and \Cref{lem:triangle-ineq}(1). In step (b) we applied \Cref{thm:contiguity} to compute the limiting trade-off function after observing that:
\begin{align*}
    \lm\sigma^2_{\mSigma(\mX)}(\sqrt{\dim}(\bar{\mSigma}({\mX} \cup \{\vx_\star\})-\mSigma(\mX)),\beta_{\dim})&=\lm \sigma^2_{\mSigma(\mX)}\left(\frac{\sqrt{\dim}\vx_\star\vx_\star^\top}{\ssize},\beta_{\dim}\right)\\
    &\bydef \lm\sigma^2_{\mX}(\vx_{\star},\beta)\\
    &\explain{Prop. \ref{prop:var-lim-v2}}{=} \sigma_\beta^2.
\end{align*}
This completes the proof of the asymptotic privacy loss characterization based on trade-off functions \eqref{eq:priv-thm-tf}.
\end{proof}

%% file: main_contiguity_overview.tex
\subsubsection{Overview of Contiguity} \label{sec:contiguity-overview}
\begin{definition}[{Mutual Contiguity \citep[Section 6.3]{van2000asymptotic}}] For each $\dim \in \N$, let  $\P_{\dim},\Q_{\dim}$ be two probability measures defined on some common space $\mathcal{V}_{\dim}$. The two sequences of probability measures $(\P_{\dim})_{\dim \in \N}$ and $(\Q_{\dim})_{\dim \in \N}$ are mutually contiguous if for any sequence of sets $(A_{\dim})_{\dim \in \N}$ with $A_\dim \subset \mathcal{V}_{\dim}$ for each $\dim \in \N$,
\begin{align*}
    \lm \P_{\dim}(A_{\dim})=0\iff \lm \Q_{\dim}(A_{\dim})=0.
\end{align*}
\end{definition}
Recall that a sequence of tests $\phi_\dim: \mathcal{V}_\dim \rightarrow [0,1]$ is consistent and asymptotically level $\alpha$ for distinguishing between $\P_\dim$ as the null hypothesis and $\Q_\dim$ as the alternative hypothesis if:
\begin{align*}
    \ls\E_{\rV \sim \P_\dim}[\phi_\dim(\rV)] \leq \alpha \quad \text{and} \quad \lim_{\dim \rightarrow \infty} \E_{\rV \sim \Q_\dim}[\phi_\dim(\rV)] = 1.
\end{align*}
That is, the sequence of tests asymptotically achieves full power while controlling the false rejection probability at level $\alpha$. An important consequence of mutual contiguity of $\P_\dim, \Q_\dim$ is that it rules out consistent tests for distinguishing between $\P_\dim$ as the null hypothesis and $\Q_\dim$ as the alternative hypothesis (or vice-versa) for any non-trivial level $\alpha \in [0,1)$. Hence, in differential privacy, it is natural to calibrate the noise level so that the output distributions of a mechanism on two neighboring datasets are mutually contiguous, preventing an adversary from consistently distinguishing between them. 

\paragraph{A Sufficient Criterion for Checking Contiguity} Le Cam's first lemma \citep[Section 6.4]{van2000asymptotic} provides several necessary and sufficient criteria for checking contiguity. The following special case will suffice for our purposes.
\begin{fact}[{Special Case of Le Cam's First Lemma \citep[Section 6.4]{van2000asymptotic}}]\label{fact:lecam-1} For each $\dim \in \N$, let  $\P_{\dim},\Q_{\dim}$ be two probability measures defined on some common space $\mathcal{V}_{\dim}$. Suppose that for each $p \in \N$, $\Q_\dim \ll \P_\dim$ and that: 
\begin{align}\label{eq:suff}
    \frac{\diff \Q_{\dim}}{\diff \P_{\dim}}(\rV) & \dc  e^{\rZ} \quad \text{with} \quad \rZ \sim \gauss{-\frac{1}{2}\sigma^2}{\sigma^2} \quad \text{when} \quad \rV \sim \P_\dim
\end{align}
for some $\sigma^2 \in [0,\infty)$. Then, $\P_\dim, \Q_\dim$ are mutually contiguous. 
\end{fact}
Note that when $\rV \sim \P_\dim$, the likelihood ratio $\tfrac{\diff \Q_{\dim}}{\diff \P_{\dim}}(\rV)$ has expectation $1$. The relationship between the mean and variance of the limiting Gaussian random variable $\rZ$ in \eqref{eq:suff} ensures that the weak limit of the likelihood ratio also has expectation $1$: $\E[e^{\rZ}] = 1$. In particular, this requirement guarantees that the sequence of likelihood ratios $\tfrac{\diff \Q_{\dim}}{\diff \P_{\dim}}(\rV)$ is uniformly integrable. 
\paragraph{Contiguity and Asymptotic Change of Measure} One of the main applications of mutual contiguity of two probability measures $\P_\dim, \Q_\dim$ is that if one knows the limit distribution of a particular statistic under $\P_\dim$, one can derive its limit distribution under $\Q_\dim$ via an asymptotic change of measure argument. This is the content of Le Cam's third lemma \citep[Section 6.4]{van2000asymptotic}. In our context, we will be interested in computing the limiting trade-off function $\lim_{\dim \rightarrow \infty} \tf{\P_\dim}{\Q_{\dim}}$. Recall that $\tf{\P_\dim}{\Q_{\dim}}(\alpha)$ is the Type II error of the optimal level-$\alpha$ test for distinguishing between $\P_\dim$ (as the null) and $\Q_\dim$ (as the alternative). Since the optimal test is based on the likelihood ratio, computing its Type II error requires understanding the limit distribution of the likelihood ratio under the alternative $\Q_\dim$. The following special case of Le Cam's third lemma will allow us to do so.
\begin{fact}[{Special Case of Le Cam's Third Lemma \citep[Section 6.4]{van2000asymptotic}}]\label{fact:lecam-3} Under the assumptions of \Cref{fact:lecam-1}, 

\begin{align}\label{eq:lecam-3}
    \frac{\diff \Q_{\dim}}{\diff \P_{\dim}}(\rV) & \dc  e^{\rZ} \quad \text{with} \quad \rZ \sim \gauss{\frac{1}{2}\sigma^2}{\sigma^2} \quad \text{when} \quad \rV \sim \Q_\dim.
\end{align}
\end{fact}
\begin{proof} We provide a proof to explain how the particular formula for the limit distribution arises. To identify the limit distribution of the likelihood ratio when $\rV \sim \Q_\dim$ we will compute: $$\lim_{\dim \rightarrow \infty} \E_{\rV \sim \Q_\dim}\left[f\left(\frac{\diff \Q_{\dim}}{\diff \P_{\dim}}(\rV)\right)\right]$$ for any bounded and continuous function $f: \R \rightarrow \R$. By change of measure,
\begin{align} \label{eq:exchange-E-lim}
   \lim_{\dim \rightarrow \infty} \E_{\rV \sim \Q_{\dim}}\left[f\left(\frac{\diff\Q_{\dim}}{\diff \P_{\dim}}(\rV)\right)\right]&=  \lim_{\dim \rightarrow \infty} \E_{\rV \sim \P_\dim}\left[f\left(\frac{\diff \Q_{\dim}}{\diff \P_{\dim}}(\rV)\right)\frac{\diff \Q_{\dim}}{\diff \P_{\dim}}(\rV)\right].
\end{align}
By the assumption \eqref{eq:suff} and the continuous mapping theorem, $f\left(\tfrac{\diff \Q_{\dim}}{\diff \P_{\dim}}(\rV)\right)\tfrac{\diff \Q_{\dim}}{\diff \P_{\dim}}(\rV) \dc f(e^{\rZ}) e^{\rZ},$ where $\rZ \sim \gauss{-\tfrac{\sigma^2}{2}}{\sigma^2}$ when $\rV \sim \P_\dim$. Moreover, since the sequence of likelihood ratios is uniformly integrable under $\P_\dim$ (recall the discussion below \Cref{fact:lecam-1}) and $f$ is bounded, the sequence of random variables $f\left(\tfrac{\diff \Q_{\dim}}{\diff \P_{\dim}}(\rV)\right)\tfrac{\diff \Q_{\dim}}{\diff \P_{\dim}}(\rV)$ is also uniformly integrable when $\rV \sim \P_\dim$. Hence, the limit and expectation on the RHS of \eqref{eq:exchange-E-lim} can be safely exchanged:
\begin{align*}
    \lim_{\dim \rightarrow \infty} \E_{\rV \sim \Q_{\dim}}\left[f\left(\frac{\diff\Q_{\dim}}{\diff \P_{\dim}}(\rV)\right)\right]&= \E[f(e^{\rZ}) e^{\rZ}] \quad \text{where} \quad \rZ \sim \gauss{-\tfrac{\sigma^2}{2}}{\sigma^2} \\
    & = \int_{\R} f(e^z)  \cdot \frac{e^z \cdot e^{-\frac{(z+\frac{\sigma^2}{2})^2}{2 \sigma^2}}}{\sqrt{2\pi \sigma^2}} \diff z  \\ &= \int_{\R} f(e^z)  \cdot \frac{ e^{-\frac{(z-\frac{\sigma^2}{2})^2}{2 \sigma^2}}}{\sqrt{2\pi \sigma^2}} \diff z & \text{[completing the square]} \\ &= \E[f(e^{\rZ})] \quad \text{where} \quad \rZ \sim \gauss{\tfrac{\sigma^2}{2}}{\sigma^2}.
\end{align*}
This proves that $\frac{\diff \Q_{\dim}}{\diff \P_{\dim}}(\rV)  \dc  e^{\rZ}$ with $\rZ \sim \gauss{\frac{1}{2}\sigma^2}{\sigma^2}$ when $\rV \sim \Q_\dim$, as claimed. 
\end{proof}

%% file: appendix_overlap.tex
\section{Proof of the Utility Theorem (Theorem \texorpdfstring{\mref{thm:overlap}}{thm:overlap})}\label{app:overlap}
This appendix is devoted to the proof of our utility theorem (Theorem \mref{thm:overlap}).
\begin{proof}[Proof of Theorem \mref{thm:overlap}]
Recall that for a sequence of matrices $\mSigma\in\R^{\dim\times\dim}$ that satisfies Assumption  \mref{assump:mat}, any fixed rank $\rnk \in \N,$ and any sequence of noise parameters $\beta_p \rightarrow \beta \in (0,\infty),$ our goal is to show:
\begin{align}\label{eq:utility-wts1}
    \mU_{\star}^\top \rV \rV^\top \mU_{\star}\pc \diag\left(1 - \frac{H_\mu(\gamma_1)}{\beta}, 1 - \frac{H_\mu(\gamma_2)}{\beta}, \dotsc, 1-\frac{H_\mu(\gamma_\rnk)}{\beta} \right)_+,
\end{align}
where $\mU_{\star} \in \R^{\dim \times \rnk}$ is the matrix of the top $\rnk$ eigenvectors of $\mSigma$ and $\rV \sim \nu(\cdot \mid \mSigma, \beta_{\dim}, \rnk)$.  For notational convenience, we define:
\begin{align*}
    \mD\bydef \diag\left(1 - \frac{H_\mu(\gamma_1)}{\beta}, 1 - \frac{H_\mu(\gamma_2)}{\beta}, \dotsc, 1-\frac{H_\mu(\gamma_\rnk)}{\beta} \right)_+.
\end{align*}
Observe that to show \eqref{eq:utility-wts1}, it suffices to show that for any non-zero symmetric matrix $\mB \in \R^{\rnk \times \rnk}$ and any $\epsilon > 0,$\footnote{Observe that setting $\mB=(\ve_i\ve_j^\top+\ve_j\ve_i^\top)/2$ recovers the $ij$th entry of $\mU_{\star}^\top \rV \rV^\top \mU_{\star}.$}
\begin{align} \label{eq:overlap-conc-B}
   \lm\P(\vert\langle\mU_{\star}^\top \rV \rV^\top \mU_{\star},\mB\rangle-\langle\mD,\mB\rangle\vert>\epsilon)  & = 0.
\end{align}
By the union bound,
\begin{multline*}
    \P(\vert\langle\mU_{\star}^\top \rV \rV^\top \mU_{\star},\mB\rangle-\langle\mD,\mB\rangle\vert>\epsilon)\leq \P(\langle\mU_{\star}^\top \rV \rV^\top \mU_{\star},\mB\rangle>\langle\mD,\mB\rangle+\epsilon)\\ + \P(\langle\mU_{\star}^\top \rV \rV^\top \mU_{\star},\mB\rangle<\langle\mD,\mB\rangle-\epsilon).
\end{multline*}
We will show the upper tail bound:
\begin{align}\label{eq:utility-wts1-general}
    \lm\P(\langle\mU_{\star}^\top \rV \rV^\top \mU_{\star},\mB\rangle>\langle\mD,\mB\rangle+\epsilon)=0,
\end{align}
and the lower tail bound $\lm\P(\langle\mU_{\star}^\top \rV \rV^\top \mU_{\star},\mB\rangle<\langle\mD,\mB\rangle-\epsilon)=0$ follows by the same argument. We prove \eqref{eq:utility-wts1-general} in several steps.

\paragraph{Step 1: Chernoff Argument} To prove the upper tail bound \eqref{eq:utility-wts1-general}, we use the standard Chernoff argument. Let $t_\star > 0$ be a tuning parameter which we will set appropriately. Then, we have the upper tail bound:
\begin{align}
    &\P(\langle\mU_{\star}^\top \rV \rV^\top \mU_{\star},\mB\rangle>\langle\mD,\mB\rangle+\epsilon)\\
    &\qquad = \P\left(e^{\frac{t_\star \dim\beta_{\dim}}{2}\langle\mU_{\star}^\top \rV \rV^\top \mU_{\star},\mB\rangle}>e^{\frac{t_\star \dim\beta_{\dim}}{2}(\langle\mD,\mB\rangle+\epsilon)}\right) \nonumber \\
    &\qquad \leq \E \left[ e^{\frac{t_\star \dim\beta_{\dim}}{2}\langle\mU_{\star}^\top \rV \rV^\top \mU_{\star},\mB\rangle} \right] \cdot e^{- \frac{t_\star \dim\beta_{\dim}}{2}(\langle\mD,\mB\rangle+\epsilon)} & \text{[Markov's Inequality]} \nonumber\\
    &\qquad \leq e^{ - \dim  \big(  \tfrac{t_\star \beta_{\dim}}{2}(\langle\mD,\mB\rangle+\epsilon) - M_\dim(t_\star) \big)}, \label{eq:chernoff}
\end{align}
where, in the last step, we used $M_\dim$ to denote the log-MGF of $\langle\mU_{\star}^\top \rV \rV^\top \mU_{\star},\mB\rangle$ (up to a reparametrization and normalization):
\begin{align} \label{eq:logmgf-def}
    M_\dim(t) \bydef \frac{1}{p} \ln \E \left[ e^{\frac{t \dim\beta_{\dim}}{2}\langle\mU_{\star}^\top \rV \rV^\top \mU_{\star},\mB\rangle}  \right] \quad \forall \; t \; \in \; \R.
\end{align}

\paragraph{Step 2:  Asymptotics of the log-MGF} To further simplify the upper tail bound \eqref{eq:chernoff}, we need to understand the asymptotic behavior of the log-MGF $M_\dim(t)$. Recall that the density of $\rV \sim \nu( \cdot \mid \mSigma, \beta_\dim, \rnk)$ with respect to $\xi_{\dim, \rnk} \bydef \unif{\O(\dim,\rnk)}$ is given by:
\begin{align} \label{eq:V-density-recall}
    \frac{\diff \nu( \mV | \mSigma, \beta_\dim, \rnk)}{\diff \xi_{\dim, \rnk}} & = \frac{e^{\frac{\dim \beta_\dim}{2}\Tr[\mV^\top \mSigma \mV ]}}{Z(\mSigma, \beta_\dim, \rnk)}\quad\forall\;\mV\;\in\;\O(\dim,\rnk),
\end{align}
where $Z(\mSigma, \beta_\dim, \rnk)$ denotes the normalizing constant:
\begin{align} \label{eq:log-Z}
    Z(\mSigma, \beta_\dim, \rnk) \bydef \E_{\rU \sim \xi_{\dim, \rnk}} \left[ e^{\frac{\dim \beta_\dim}{2}\Tr[\rU^\top \mSigma \rU ]} \right].
\end{align}
The key idea behind understanding the asymptotic behavior of the log-MGF $M_\dim(t)$ is to note that it can be expressed in terms of the log-normalizing constant $\ln Z(\cdot)$. Indeed, we have:
\begin{align}
M_\dim(t) & \explain{\eqref{eq:logmgf-def}}{=} \frac{1}{p} \ln \E \left[ e^{\frac{t \dim\beta_{\dim}}{2}\langle\mU_{\star}^\top \rV \rV^\top \mU_{\star},\mB\rangle}  \right] \notag \\
& \explain{\eqref{eq:V-density-recall}}{=} \frac{1}{\dim}\ln\E_{\rU \sim \xi_{\dim,\rnk}}\left[e^{\frac{\dim\beta_{\dim}}{2}\Tr[\rU^\top \mSigma\rU]} \cdot e^{\frac{t \dim\beta_{\dim}}{2}\langle\mU_{\star}^\top \rU \rU^\top \mU_{\star},\mB\rangle}\right] - \frac{1}{\dim}\ln\E_{\rU \sim \xi_{\dim,\rnk}}\left[e^{\frac{\dim\beta_{\dim}}{2}\Tr[\rU^\top \mSigma\rU]}\right]\\
    &\explain{}{=}{\frac{1}{\dim}\ln\E_{\rU \sim \xi_{\dim,\rnk}}\left[e^{\frac{\dim\beta_{\dim}}{2}\Tr[\rU^\top (\mSigma+t\mU_\star\mB\mU_\star^\top)\rU]}\right]} - {\frac{1}{\dim}\ln\E_{\rU \sim \xi_{\dim,\rnk}}\left[e^{\frac{\dim\beta_{\dim}}{2}\Tr[\rU^\top \mSigma\rU]}\right]} \notag \\
    & \explain{\eqref{eq:log-Z}}{=} \frac{\ln Z(\mSigma +t\mU_\star\mB\mU_\star^\top, \beta_\dim, \rnk) }{\dim} - \frac{\ln Z(\mSigma, \beta_\dim, \rnk) }{\dim}. \label{eq:mgf-to-logZ}
\end{align}
Combining the formula \eqref{eq:mgf-to-logZ} for the log-MGF $M_\dim(t)$ with Fact \mref{fact:guionnet} on the asymptotics of $\ln Z(\cdot, \beta_\dim, \rnk)$ immediately gives the following asymptotic formula for the log-MGF.
\begin{claim} \label{claim:log-MGF} For any $t \in (-\tfrac{\Delta}{2\|\mB\|}, \tfrac{\Delta}{2\|\mB\|})$, $M_\dim(t) \rightarrow M(t)$ as $\dim \rightarrow \infty$, where:
\begin{align*}
    M(t)  & \bydef \frac{1}{2}\sum_{i=1}^\rnk f_\mu(\lambda_i\!\left(\mGamma+t\mB\right))-\frac{1}{2}\sum_{i=1}^\rnk f_\mu(\gamma_i) \quad \forall \; |t|<\tfrac{\Delta}{2\|\mB\|},\quad \mGamma \bydef \diag(\gamma_1,\dotsb,\gamma_\rnk).
\end{align*}
\end{claim}
We will also find it helpful to record the value of the limiting log-MGF and its derivative at $t=0$.
\begin{claim} \label{claim:log-MGF-at0} The limiting log-MGF $M$ is continuously differentiable on $(-\tfrac{\Delta}{2\|\mB\|}, \tfrac{\Delta}{2\|\mB\|})$ with $M(0) = 0$ and $M^\prime(0) = \beta \ip{D}{B}/2$.
\end{claim}
We defer the detailed calculations involved in the derivation of \Cref{claim:log-MGF,claim:log-MGF-at0} to the end of this section, and continue with the proof of Theorem \mref{thm:overlap} assuming these claims.

\paragraph{Step 3: Choosing $t_\star$} To summarize our results from Steps 1-2, we have the following tail bound:
\begin{align*}
    &\P(\langle\mU_{\star}^\top \rV \rV^\top \mU_{\star},\mB\rangle>\langle\mD,\mB\rangle+\epsilon)\\
    & \quad \explain{\eqref{eq:chernoff}}{\leq} \quad  e^{ - \dim  \big(  \tfrac{t_\star \beta_{\dim}}{2}(\langle\mD,\mB\rangle+\epsilon) - M_\dim(t_\star) \big)} & \text{[assuming $t_\star > 0$]} \\
    & \quad \explain{\Cref{claim:log-MGF}}{=} \quad e^{ - \dim  \big(  \tfrac{t_\star \beta}{2}(\langle\mD,\mB\rangle+\epsilon) - M(t_\star)  - o(1) \big)} & \text{[assuming $|t_\star| <\tfrac{\Delta}{2\|\mB\|} $]}.
\end{align*}
Hence, to obtain the desired result $\P(\langle\mU_{\star}^\top \rV \rV^\top \mU_{\star},\mB\rangle>\langle\mD,\mB\rangle+\epsilon) \rightarrow 0$, we only need to find $t_\star \in (0, \tfrac{\Delta}{2\|\mB\|})$ such that the (limiting) exponent in our tail bound is positive:
\begin{align*}
    \frac{t_\star \beta(\langle\mD,\mB\rangle+\epsilon)}{2} - M(t_\star) > 0.
\end{align*}
The existence of such a $t_\star$ is guaranteed by \Cref{claim:log-MGF-at0} since the function $t \mapsto \beta t (\langle\mD,\mB\rangle+\epsilon)/2 - M(t)$ evaluates to $0$ at $t= 0$, and moreover, this function is strictly increasing in a neighborhood around $0$ since:
\begin{align*}
    \frac{\diff}{\diff t} \left( \frac{\beta t (\langle\mD,\mB\rangle+\epsilon)}{2} - M(t) \right) \bigg|_{t=0} = \frac{\beta (\ip{D}{B}+\epsilon)}{2} - M^\prime(0) \; \; \explain{\Cref{claim:log-MGF-at0}}{=} \; \;   \frac{\beta \epsilon}{2} > 0.
\end{align*}
This completes the proof of Theorem \mref{thm:overlap} taking \Cref{claim:log-MGF,claim:log-MGF-at0} for granted. The remainder of this section is devoted to the proofs of these claims.
\end{proof}
\subsection{Proof of \texorpdfstring{\Cref{claim:log-MGF}}{claim:log-MGF}}
\begin{proof} Recall from \eqref{eq:mgf-to-logZ} that:
\begin{align} \label{eq:mgf-recall}
    M_\dim(t) = \frac{\ln Z(\mSigma +t\mU_\star\mB\mU_\star^\top, \beta_\dim, \rnk) }{\dim} - \frac{\ln Z(\mSigma, \beta_\dim, \rnk) }{\dim}.
\end{align}
Our goal is to compute $\lim_{\dim \rightarrow \infty} M_\dim(t)$. For the second term, since $\mSigma$ satisfies Assumption \mref{assump:mat}, we know that $\mu_{\mSigma}\rightarrow\mu,$ $\lambda_i(\Sigma) \rightarrow \gamma_i$ for $i\in[\rnk+1],$ and $\Delta\bydef\gamma_\rnk-\gamma_{\rnk+1}>0$. Hence, Fact \mref{fact:guionnet} guarantees that 
\begin{align} \label{eq:mgf-part-1}
    \frac{\ln Z(\mSigma, \beta_\dim, \rnk) }{\dim} \rightarrow \frac{1}{2}\sum_{i=1}^\rnk f_\mu(\gamma_i).
\end{align}
To compute $$\lim_{\dim \rightarrow \infty} \frac{\ln Z(\mSigma +t\mU_\star\mB\mU_\star^\top, \beta_\dim, \rnk)}{\dim}$$  using Fact \mref{fact:guionnet}, we will need to understand the spectrum of the matrix $\mSigma+t\mU_\star\mB\mU_\star^\top.$ To do so, we first consider the eigendecomposition of $\mSigma$. Let $\lambda_1(\mSigma) \geq \lambda_2(\mSigma) \geq \dotsb \geq \lambda_\dim(\mSigma)$ denote the eigenvalues of $\mSigma$ in decreasing order. We can write the eigendecomposition of $\mSigma$ as:
\begin{align*}
    \mSigma & = \begin{bmatrix}\mU_\star&\mU_\perp\end{bmatrix}\begin{bmatrix}\Lambda_\star&0\\0&\Lambda_\perp\end{bmatrix}\begin{bmatrix}
        \mU_\star^\top\\\mU_\perp^\top
    \end{bmatrix},
\end{align*}
where $\Lambda_\star = \diag(\lambda_{1:\rnk}(\mSigma))$ is the diagonal matrix of the top $\rnk$ eigenvalues of $\mSigma$, $\Lambda_\perp = \diag(\lambda_{\rnk+1:\dim}(\mSigma))$ is the diagonal matrix of the remaining $\dim - \rnk$ eigenvalues of $\mSigma$ and $\mU_\star \in \R^{\dim \times \rnk}, \mU_\perp \in \R^{\dim \times (\dim - \rnk)}$ denote the matrices of the corresponding eigenvectors. We can now make the following observations regarding the spectrum of $\mSigma+t\mU_\star\mB\mU_\star^\top$.
\begin{itemize}
\item \emph{Eigenvalues:} Note that:
\begin{align*}
    \Sigma+t\mU_\star\mB\mU_\star^\top & =  \begin{bmatrix}\mU_\star&\mU_\perp\end{bmatrix}\begin{bmatrix}\Lambda_\star  &0\\0&\Lambda_\perp\end{bmatrix}\begin{bmatrix}
        \mU_\star^\top\\\mU_\perp^\top \end{bmatrix}+t\begin{bmatrix}\mU_\star&\mU_\perp\end{bmatrix}\begin{bmatrix}\mB&0\\0&0\end{bmatrix}\begin{bmatrix}
        \mU_\star^\top\\\mU_\perp^\top
    \end{bmatrix} \\ &= \begin{bmatrix}\mU_\star&\mU_\perp\end{bmatrix}\begin{bmatrix}\Lambda_\star+t\mB&0\\0&\Lambda_\perp\end{bmatrix} \begin{bmatrix}
        \mU_\star^\top\\\mU_\perp^\top
    \end{bmatrix}.
\end{align*}
Since $\begin{bmatrix}\mU_\star&\mU_\perp\end{bmatrix}$ is an orthogonal matrix, the eigenvalues of $ \Sigma+t\mU_\star\mB\mU_\star^\top$ are exactly the eigenvalues of the block-diagonal matrix:
\begin{align*}
    \begin{bmatrix}\Lambda_\star+t\mB&0\\0&\Lambda_\perp\end{bmatrix},
\end{align*}
which are:
\begin{align*}
    \{\lambda_1(\Lambda_\star+t\mB), \lambda_2(\Lambda_\star+t\mB), \dotsc, \lambda_\rnk(\Lambda_\star+t\mB), \lambda_{\rnk+1}(\mSigma), \dotsc, \lambda_{\dim}(\mSigma)\}.
\end{align*}
\item \emph{Limiting Spectral Measure:} Note that only $\rnk = O(1)$ eigenvalues of the matrices $\mSigma+t\mU_\star\mB\mU_\star^\top$ and $\mSigma$ differ, so the spectral measure of $\mSigma+t\mU_\star\mB\mU_\star^\top$ also converges weakly to $\mu.$ 
\item \emph{Limits of the Largest $\rnk+1$ Eigenvalues:} We also need to compute the limits of the largest $\rnk+1$ eigenvalues of $\mSigma+t\mU_\star\mB\mU_\star^\top$. We already know that the set of eigenvalues of $\mSigma+t\mU_\star\mB\mU_\star^\top$ is:
\begin{align*}
      \{\lambda_1(\Lambda_\star+t\mB), \lambda_2(\Lambda_\star+t\mB), \dotsc, \lambda_\rnk(\Lambda_\star+t\mB)\} \cup  \{\lambda_{\rnk+1}(\mSigma), \dotsc, \lambda_{\dim}(\mSigma)\}.
\end{align*}
Note that the elements of the two sets in the above union are already sorted in decreasing order, so to determine the largest $\rnk+1$ eigenvalues of $\mSigma+t\mU_\star\mB\mU_\star^\top$ we only need to determine the relative ordering between the elements of the two sets. For sufficiently small $t$, we expect that $\lambda_{\rnk}(\Lambda_\star+t\mB)$ will be close to $\lambda_{\rnk}(\Lambda_\star) = \lambda_\rnk(\mSigma)$, and hence greater than $\lambda_{\rnk+1}\left(\mSigma\right)$ for large enough $\dim$ (since  $\lambda_{\rnk}(\mSigma)-\lambda_{\rnk+1}(\mSigma)\rightarrow \Delta>0$). Indeed: 
\begin{align*} %\label{eq:utility-eigengap}
    &\lambda_{\rnk}(\Lambda_\star+t\mB) - \lambda_{\rnk+1}(\mSigma) \\
    &\quad= \lambda_{\rnk}(\Lambda_\star+t\mB)-\lambda_{\rnk}(\Lambda_*)+\lambda_{\rnk}(\mSigma) - \lambda_{\rnk+1}(\mSigma) \quad \text{[since $\lambda_{\rnk}(\Lambda_*) = \lambda_{\rnk}(\mSigma)$]}  \\
    &\quad\geq -|\lambda_{\rnk}(\Lambda_\star+t\mB)-\lambda_{\rnk}(\Lambda_\star)|+\lambda_{\rnk}(\mSigma) - \lambda_{\rnk+1}(\mSigma)\\
    &\quad \geq -|t|\|\mB\|+\lambda_{\rnk}(\mSigma) - \lambda_{\rnk+1}(\mSigma) \quad \text{[by Weyl's Inequality]} \\
    &\quad \geq -\frac{\Delta}{2} + \lambda_{\rnk}(\mSigma) - \lambda_{\rnk+1}(\mSigma) \quad \text{[since $|t| < \tfrac{\Delta}{2\|B\|}$ by assumption]} \\
    &\quad \rightarrow \frac{\Delta}{2} > 0  \;\;  \text{as $\dim \rightarrow \infty$} \quad \text{[since $\lambda_\rnk(\mSigma) - \lambda_{\rnk+1}(\mSigma) \rightarrow \Delta>0$]}.
\end{align*} 
%where (a) follows by Weyl's inequality (see e.g., \citep[Theorem 4.5.3]{vershynin2018high}). 
Hence, for sufficiently large $\dim$, the largest $\rnk+1$ eigenvalues of $\mSigma+t\mU_\star\mB\mU_\star^\top$ are:
\begin{align*}
      \lambda_1(\Lambda_\star+t\mB) \geq  \lambda_2(\Lambda_\star+t\mB) \geq  \dotsb \geq  \lambda_\rnk(\Lambda_\star+t\mB) >  \lambda_{\rnk+1}(\mSigma),
\end{align*}
which converge to the limit values:
\begin{align*}
\lambda_1(\mGamma+t\mB) \geq \lambda_2(\mGamma+t\mB) \geq \dotsb \geq  \lambda_\rnk(\mGamma+t\mB) >  \gamma_{\rnk+1}
\end{align*}
since $\Lambda_\star \bydef \diag(\lambda_{1:\rnk}(\mSigma)) \rightarrow \Gamma \bydef  \diag(\gamma_{1:\rnk})$ and $\lambda_{\rnk+1}(\mSigma) \rightarrow \gamma_{\rnk+1}$ (recall Assumption \mref{assump:mat}), and eigenvalues of a matrix are continuous functions of the matrix. 
% \begin{align}\label{eq:utility-min-eigval}
%     \lm \lambda_{\min}\!\left(\mSigma+t\mU_\star\mB\mU_\star^\top\right)>\lm \lambda_{\rnk+1}\!\left(\mSigma+t\mU_\star\mB\mU_\star^\top\right)
% \end{align}
% since taking limits on both sides of the inequality in \eqref{eq:utility-eigengap} gives:
% \begin{align*}
%     \lm\lambda_{\min}(\Lambda_\star+t\mB)-\lm\lambda_{\max}\left(\Lambda_{\perp}\right)\geq -|t|\|\mB\|+\lm\left(\lambda_{\min}(\Lambda_*)- \lambda_{\max}\left(\Lambda_{\perp}\right)\right)= -|t|\|\mB\|+\Delta\geq \frac{\Delta}{2}.
% \end{align*}
\end{itemize}
Now, we can apply Fact \mref{fact:guionnet} to conclude that: 
\begin{align*}
    \lim_{\dim \rightarrow \infty} \frac{\ln Z(\mSigma +t\mU_\star\mB\mU_\star^\top, \beta_\dim, \rnk)}{\dim} & =  \frac{1}{2}\sum_{i=1}^\rnk f_\mu(\lambda_i\!\left(\mGamma+t\mB\right)).
\end{align*}
Combining the above equation with \eqref{eq:mgf-recall} and \eqref{eq:mgf-part-1} gives us the claimed formula for the limiting log-MGF:
\begin{align*}
    \lim_{\dim \rightarrow \infty} M_\dim(t)=\frac{1}{2}\sum_{i=1}^\rnk f_\mu(\lambda_i\!\left(\mGamma+t\mB\right))-\frac{1}{2}\sum_{i=1}^\rnk f_\mu(\gamma_i)
\end{align*}
for any $t\in\R$ satisfying $|t|<\Delta/(2\|\mB\|).$
\end{proof}
\subsection{Proof of \texorpdfstring{\Cref{claim:log-MGF-at0}}{claim:log-MGF-at0}}
\begin{proof} Recall from \Cref{claim:log-MGF} that:
\begin{align*}
    M(t)  & \bydef \frac{1}{2}\sum_{i=1}^\rnk f_\mu(\lambda_i\!\left(\mGamma+t\mB\right))-\frac{1}{2}\sum_{i=1}^\rnk f_\mu(\gamma_i) \quad \forall \; |t|<\tfrac{\Delta}{2\|\mB\|},
\end{align*}
where:
\begin{subequations}\label{eq:f_mu-recall}
    
\begin{align} 
    \mGamma &\bydef \diag(\gamma_1, \dotsc, \gamma_\rnk),\\
    f_\mu(\gamma)&=\begin{cases}\beta\gamma-\ln \beta-\int_{\R} \ln(\gamma-\lambda)\diff\mu(\lambda)-1&\text{if }H_\mu(\gamma)\leq\beta\\
    \beta H_\mu^{-1}(\beta)-\ln \beta - \int_{\R}\ln (H_\mu^{-1}(\beta)-\lambda)\diff\mu(\lambda)-1&\text{if }H_\mu(\gamma)>\beta.
\end{cases}
\end{align}
\end{subequations}
Hence,
\begin{align*}
     M^\prime(t) = \frac{1}{2}\frac{\diff}{\diff t}\sum_{i=1}^\rnk f_\mu(\lambda_i\!\left(\Gamma+t\mB\right)).
\end{align*}
To compute the above derivative, we view the function $t\mapsto \sum_{i=1}^\rnk f_\mu(\lambda_i\!\left(\Gamma+t\mB\right))$ as a composition of the functions
\begin{align*}
    t\mapsto \Gamma+t\mB\quad \text{and}\quad \mA\mapsto \sum_{i=1}^\rnk f_\mu(\lambda_i(\mA)),
\end{align*}
and apply the chain rule. The derivative of the first function is simple: $\partial_t (\Gamma+t\mB) = \mB$. For the second function, we rely on a formula of \citet[Theorem 1.1]{lewis1996derivatives}:
\begin{align*}
     \nabla_A \sum_{i=1}^\rnk f_\mu(\lambda_i(\mA)) =\sum_{i=1}^\rnk f_\mu'(\lambda_i(\mA)) \cdot u_i(\mA)u_i(\mA)^\top,
\end{align*}
where for a matrix $\mA \in \R^{\rnk \times \rnk}$, $\lambda_{1:\rnk}(A)$ denote the eigenvalues of $A$ and $\vu_{1:\rnk}(A)$ denote the corresponding eigenvectors. By the chain rule, we have: 
\begin{align*}
    M^\prime(t) &= \frac{1}{2}\frac{\diff}{\diff t}\sum_{i=1}^\rnk f_\mu(\lambda_i\!\left(\Gamma+t\mB\right)) \\ &=\frac{1}{2}\Tr\left[\sum_{i=1}^\rnk f_\mu'(\lambda_i(\Gamma+t\mB))u_i(\Gamma+t\mB)u_i(\Gamma+t\mB)^\top\mB\right]\\
    &= \frac{1}{2}\sum_{i=1}^\rnk (\beta-H_\mu(\lambda_i(\mGamma + t\mB)))_+ \cdot u_i(\Gamma+t\mB)^\top\mB u_i(\Gamma+t\mB), 
\end{align*}
where the last equality holds since $f_\mu^\prime(\gamma) = (\beta - H_\mu(\gamma))_+$ (see \eqref{eq:f_mu-recall}).
We can now evaluate the derivative at $t = 0$. Since $\mGamma = \diag(\gamma_{1:\rnk})$,
\begin{align*}
    M^\prime(0) = \frac{1}{2}\sum_{i=1}^\rnk (\beta-H_\mu(\lambda_i(\mGamma)))_+ \cdot u_i(\Gamma)^\top\mB u_i(\Gamma) = \frac{1}{2}\sum_{i=1}^\rnk (\beta-H_\mu(\gamma_i))_+ \cdot B_{ii} = \frac{\beta \ip{D}{B}}{2},
\end{align*}
where the last equality follows by recalling:
\begin{align*}
     \mD\bydef \diag\left(1 - \frac{H_\mu(\gamma_1)}{\beta}, 1 - \frac{H_\mu(\gamma_2)}{\beta}, \dotsc, 1-\frac{H_\mu(\gamma_\rnk)}{\beta} \right)_+.
\end{align*}
\end{proof}

%% file: appendix_sampling.tex
\section{Analysis of the Sampling Algorithm (Proof of Theorem \texorpdfstring{\mref{thm:sampling}}{thm:sampling})}\label{app:sampling}
This appendix is devoted to the analysis of the sampling algorithm (Algorithm \mref{alg:sampler}) and the proof of Theorem \mref{thm:sampling}. Recall that to prove Theorem \mref{thm:sampling}, we need to show that:
\begin{align*}
    \lim_{\dim \rightarrow \infty} \tv \left( \nu(\cdot \mid \mSigma, \beta_\dim, \rnk), \hat{\nu}(\cdot \mid \mSigma, \beta_\dim, \rnk) \right) & = 0 \quad \text{ provided } \beta_\dim \rightarrow \beta > H_\mu(\gamma_\rnk),
\end{align*}
where $\nu(\cdot \mid \mSigma, \beta_\dim, \rnk)$ denotes the Gibbs distribution 
\begin{align*}
    \frac{\diff \nu(\mV \mid \Sigma, \beta, \rnk)}{\diff \xi_{\dim,\rnk}} & \propto \exp\left( \frac{\beta p}{2} \Tr[ \mV^\top \mSigma \mV]\right),
\end{align*}
%$\nu(\diff \mV \mid \mSigma, \beta_\dim, \rnk) \propto e^{{\beta_\dim \Tr[\mV^\top \mSigma \mV]}/{2}} \xi_{\dim,\rnk}(\diff \mV)$  
and $\hat{\nu}(\cdot \mid \mSigma, \beta_\dim, \rnk)$ the output distribution of the sampling algorithm in Algorithm \mref{alg:sampler} on inputs $(\mSigma, \beta_\dim, \rnk)$. In this section, we will often use the convenient shorthand notations:
\begin{align*}
    \nu_\dim \bydef \nu(\cdot \mid \mSigma, \beta_\dim, \rnk), \quad \hat{\nu}_\dim \bydef \hat{\nu}(\cdot \mid \mSigma, \beta_\dim, \rnk).
\end{align*}
We will introduce the main ideas involved in the proof of Theorem \mref{thm:sampling} in the form of some intermediate results, whose proofs are deferred to the end of this section. 
% The exponential mechanism (\Cref{alg:ExpM}) samples $\rV \in \O(\ssize,\dim)$ from the unnormalized density:
% \begin{align} \label{eq:exp-mech-density-recall}
% \diff \nu_{\dim}(\mV) \propto \exp\left( \frac{\beta p}{2} \cdot \Tr[ \mV^\top \mSigma \mV] \right) \diff\xi_{\dim,\rnk}(\mV),
% \end{align}
% where $\xi_{\dim,\rnk} \bydef \unif{\O(\dim,\rnk)}$. 

\paragraph{Two Convenient Statistics} To develop a sampling algorithm for the exponential mechanism, our strategy will be to express the output $\rV \sim \nu_\dim$ of the exponential mechanism as a function of two statistics, which have a convenient and easy-to-sample distribution. To define these statistics, we consider the eigendecomposition of the matrix $\mSigma$:
\begin{align*}
    \mSigma & = \mU \diag(\lambda_1, \dotsc, \lambda_{\dim}) \mU^\top.
\end{align*}
We will find it convenient to partition the eigenvector matrix into two blocks $U_\star$ and $U_\perp$ consisting of the first $\rnk$ eigenvectors and the remaining $\dim - \rnk$ eigenvectors:
\begin{align*}
    \mU & = \begin{bmatrix} \mU_\star & \mU_\perp \end{bmatrix}\quad \text{where}\quad \mU_\star \in \R^{\dim \times k},\quad \mU_\perp \in \R^{\dim \times (\dim - \rnk)}.
\end{align*}
For any $\mV \in \O(\dim,\rnk)$ such that $\mU_\star^\top \mV$ is invertible,  we define the following convenient statistics:
\begin{align} \label{eq:conv-stat-1}
    \calQ(\mV) \bydef (\bar{\mV}_{\star} \bar{\mV}_\star^\top)^{-1/2} \cdot \bar{\mV}_{\star}, \quad \calZ(\mV) = \bar{\mV}_\perp \cdot Q(\mV)^\top \quad \text{where}\quad \bar{\mV}_{\star} \bydef \mU_\star^\top \mV, \quad  \bar{\mV}_{\perp} \bydef \mU_\perp^\top \mV. 
\end{align}
The following lemma (proved in \Cref{app:decomp}) shows that $\rV$ can be expressed as a function of $\calQ(\rV), \calZ(\rV)$, reducing the problem of sampling $\rV$ to sampling $\calQ(\rV), \calZ(\rV)$.
\begin{lemma}\label{lem:decomp} For any $\mV \in \O(\dim,\rnk)$ such that $\mU_\star^\top \mV \in \R^{\rnk \times \rnk}$ is invertible, we have:
\begin{align*}
    \mV & = \mU \begin{bmatrix} (I_k - \calZ(\mV)^\top \calZ(\mV))^{1/2} \\ \calZ(\mV) \end{bmatrix}  \calQ(\mV).
\end{align*}
\end{lemma}

\paragraph{Some Notation} For a measure $\nu$ and a function $\phi$, we will find it convenient to use the notation $\phi \# \nu$ for the pushforward of $\nu$ by $\phi$: $\phi \# \nu$ denotes the law of $\phi(\rV)$ when $\rV \sim \nu$. Hence, $(\calQ, \calZ) \# \nu_\dim$ denotes
the joint law of $(\calQ(\rV), \calZ(\rV))$ when $\rV \sim \nu_\dim$ is the output of the exponential mechanism (Algorithm \mref{alg:ExpM}). 

\paragraph{Characterizing the Distribution of 
\texorpdfstring{$\calQ(\rV), \calZ(\rV)$}{calQcalZ}} The following proposition (proved in \Cref{app:stat-distr}) characterizes the distribution of the statistics $\calQ(\rV), \calZ(\rV)$ when $\rV$ denotes the output of the exponential mechanism. 

\begin{proposition} \label{prop:stat-distr} Let $\rV \sim \nu_\dim$ denote the output of the exponential mechanism. Then, $\calQ(\rV)$ and $\calZ(\rV)$ are independent. Moreover, $\calQ(\rV) \sim \xi_{\rnk,\rnk} \explain{def}{=} \unif{\O(\rnk)}$ and the unnormalized density of $\calZ(\rV)$ (with respect to the Lebesgue measure) is given by:
\begin{align} \label{eq:Z-dist}
    \frac{\diff (\calZ \#\nu_\dim) }{\diff \mZ} (\mZ) \propto \frac{1}{\sqrt{\det(\mI_k - \mZ^\top \mZ)}} \cdot e^{ - \frac{ \dim \beta_\dim}{2} \sum_{i=1}^{\dim - \rnk} \sum_{j=1}^{\rnk} (\lambda_{j} - \lambda_{i + \rnk}) Z_{ij}^2 } \cdot \ind\{ \mZ^\top \mZ \preceq \mI_\rnk \},
\end{align}
where $\ind\{ \mZ^\top \mZ \preceq \mI_\rnk \}$ denotes the indicator function which evaluates to $1$ if $\mZ^\top \mZ \preceq \mI_\rnk$ and is $0$ otherwise. 
\end{proposition}

\paragraph{Approximating the Distribution of \texorpdfstring{$\calZ(\rV)$}{calZ}} In light of \Cref{lem:decomp} and \Cref{prop:stat-distr}, a natural approach to sampling the output of the exponential mechanism $\rV\sim \nu_\dim$ is:
\begin{enumerate}
    \item Generate $(\rQ, \rZ)  \sim (\calQ, \calZ) \# \nu_\dim$ by:
    \begin{enumerate}
        \item Sampling $\rQ \sim \unif{\O(\rnk)}$.
        \item Sampling $\rZ \sim \calZ \# \nu_{\dim}$ from the unnormalized density \eqref{eq:Z-dist}, independent of $\rQ$.
    \end{enumerate}
    \item Set $\rV$ as (cf. \Cref{lem:decomp})\footnote{To express $\rV \sim \nu_\dim$ as a function of $(\rQ, \rZ) \bydef (\calQ(\rV), \calZ(\rV))$, \Cref{lem:decomp} requires that $\bar{\rV}_\star \bydef \mU_\star^\top \rV$ is invertible with probability $1$. This holds because $\nu_\dim$ is absolutely continuous with respect to the Haar measure $\xi_{\dim,\rnk}$ on $\O(\dim,\rnk)$. When $\rV \sim \xi_{\dim,\rnk}$, for large enough $\dim$ (if $\dim\geq2\rnk$), the distribution of the matrix $\bar{\rV}_\star$ is absolutely continuous with respect to the Lebesgue measure on $\R^{\rnk \times \rnk}$ \citep[Theorem 2.11]{meckes2019random}, and the set of non-invertible matrices have Lebesgue measure $0$.}:
    \begin{align} \label{eq:map-motivation}
         \rV & = \mU \begin{bmatrix} (I_k - \rZ^\top \rZ)^{1/2} \\ \rZ \end{bmatrix}  \rQ \explain{(a)}{=} \mU \begin{bmatrix} (I_k - \rZ^\top \rZ)^{1/2}_+ \\ \rZ \end{bmatrix}  \rQ
    \end{align}
    where (a) follows because $I_k - \rZ^\top \rZ \succeq 0_{\rnk\times\rnk}$ with probability $1$ (recall the density of $\rZ$ from \Cref{prop:stat-distr}).
\end{enumerate}
This provides an \emph{exact sampler} for $\rV \sim \nu_{\dim}$. However, since the unnormalized density in \eqref{eq:Z-dist} is high-dimensional and non-log-concave, sampling $\rZ$ from it is not straightforward. To address this issue, our sampler (Algorithm \mref{alg:sampler}) replaces Step 1(b) above by sampling the entries of $\rZ$ independently as follows:
\begin{align} \label{eq:Z-sampler}
\rZ_{ij} \explain{}{\sim} \gauss{0}{\frac{1}{ \dim \beta_\dim( \lambda_j - \lambda_{\rnk+i})}} \quad i \in [\dim-\rnk], \; j \; \in \; [\rnk] \quad \iff \quad \rZ \sim \zeta_\dim.
\end{align}
Let $\zeta_\dim$  denote the distribution of $\rZ$ when generated according to \eqref{eq:Z-sampler}. 
%(recall $\hat{\nu}_{\dim}$ denotes the output distribution of the sampler)\footnote{We emphasize that $\calZ \# \hat{\nu}_{\dim}$ simply denotes the distribution of $\rZ$ when generated according to \eqref{eq:Z-sampler}. Formally speaking, this is an abuse of notation since $\calZ \# \hat{\nu}_{\dim}$ is not the push-forward of $\hat{\nu}_\dim$ (the output distribution of the sampler) by the map $\calZ$ defined in \eqref{eq:conv-stat}. This is because the domain of $\calZ$ is $\O(\dim,\rnk)$ and in certain rare events, the output of the sampler might not be an element of $\O(\dim,\rnk)$.}. 
The replacement of step (1b) by \eqref{eq:Z-sampler} is justified by the following proposition (proved in \Cref{app:Z-dist-approx}), which shows that the two distributions are close in total variation distance.

\begin{proposition} \label{prop:Z-dist-approx} If $\beta \bydef \lim_{\dim \rightarrow \infty} \beta_\dim  > H_\mu(\gamma_\rnk),$ then  $\lim_{\dim \rightarrow \infty} \tv{(\calZ \# \nu_{\dim}, \zeta_{\dim})} = 0$.
\end{proposition} 
We now have all the ingredients to present the proof of Theorem \mref{thm:sampling}. 

\begin{proof}[Proof of Theorem \mref{thm:sampling}] Let $\calM: \O(\rnk) \times \R^{(\dim - \rnk) \times \rnk}  \rightarrow \R^{\dim \times \rnk}$ denote the map:
\begin{align*}
    \calM(\mQ,\mZ) \bydef \mU \begin{bmatrix} (I_k - \mZ^\top \mZ)^{1/2}_+ \\ \mZ \end{bmatrix}  \mQ \quad \forall \; \mZ \; \in \; \R^{(\dim - \rnk) \times \rnk}, \; \mQ \; \in \; \O(\rnk).
\end{align*}
From the above discussion, we know that:
\begin{align*}
    \nu_\dim & = \calM \# (\unif{\O(\rnk)} \otimes \calZ \# \nu_\dim), \quad \hat{\nu}_\dim = \calM \# (\unif{\O(\rnk)} \otimes \zeta_\dim).
    \end{align*}
By the data processing inequality for total variation distance,
\begin{align*}
    \tv(\nu_{\dim}, \hat{\nu}_{\dim})  \explain{}{\leq} \tv(\unif{\O(\rnk)} \otimes \calZ \# \nu_{\dim} , \unif{\O(\rnk)} \otimes \zeta_{\dim})  = \tv(\calZ \# \nu_{\dim} , \zeta_{\dim}) \explain{Prop. \ref{prop:Z-dist-approx}}{\rightarrow} 0. 
\end{align*}
This concludes the proof of Theorem \mref{thm:sampling}.
\end{proof}

The remainder of this section is devoted to the proofs of \Cref{lem:decomp} and \Cref{prop:stat-distr,prop:Z-dist-approx}.

\subsection{Proof of \texorpdfstring{\Cref{lem:decomp}}{lem:decomp}} \label{app:decomp}

\begin{proof}[Proof of \Cref{lem:decomp}] Consider any $\mV \in \O(\dim,\rnk)$.  Recall that the statistics $\calQ(\mV)$ and $\calZ(\mV)$ were defined as:
\begin{align} \label{eq:stat-recall}
     \calQ(\mV) \bydef (\bar{\mV}_{\star} \bar{\mV}_\star^\top)^{-1/2} \cdot \bar{\mV}_{\star}, \quad \calZ(\mV) \bydef \bar{\mV}_\perp \cdot Q(\mV)^\top \quad \text{where}\quad \bar{\mV}_{\star} \bydef \mU_\star^\top \mV, \quad  \bar{\mV}_{\perp} \bydef \mU_\perp^\top \mV,
\end{align}
and $\mU = \begin{bmatrix} \mU_\star & \mU_\perp \end{bmatrix}$ is the matrix of eigenvectors of $\mSigma$.
% \begin{align*}
%     \bar{\mV} & = \begin{bmatrix} \bar{\mV}_\star \\ \bar{\mV}_\perp \end{bmatrix} \quad \text{ with } 
%  \quad \bar{\mV}_\star \in \R^{\rnk \times \dim}, \quad  \bar{\mV}_\perp \in \R^{(\dim-\rnk) \times \dim}.
% \end{align*}
Our goal is to invert \eqref{eq:stat-recall} and express $\mV$ in terms of $\calQ(\mV)$ and $\calZ(\mV)$. Note that $\calQ(\mV)$ is a $\rnk \times \rnk$ orthogonal matrix:
\begin{align*}
\calQ(\mV) \calQ(\mV)^\top  =   (\bar{\mV}_\star \bar{\mV}_\star^\top)^{-\frac{1}{2}} \cdot \bar{\mV}_\star \bar{\mV}_\star^\top \cdot (\bar{\mV}_\star \bar{\mV}_\star^\top)^{-1/2} = \mI_{\rnk}.
\end{align*}
Hence, rearranging \eqref{eq:stat-recall} gives:
\begin{align*}
    \bar{\mV} &\bydef \mU^\top \mV = \begin{bmatrix} \bar{\mV}_\star \\ \bar{\mV}_\perp \end{bmatrix} =   \begin{bmatrix} (\bar{\mV}_\star \bar{\mV}_\star^\top)^{1/2} \\ \calZ(\mV)    \end{bmatrix} \cdot \calQ(\mV).
\end{align*}
Since $\bar{V} \explain{def}{=} \mU^\top \mV$ and $\mU$ is the orthogonal matrix of eigenvectors of $\mSigma$, we conclude that:
\begin{align} \label{eq:decomp-interm}
    \mV & = \mU \cdot    \begin{bmatrix} (\bar{\mV}_\star \bar{\mV}_\star^\top)^{1/2} \\ \calZ(\mV)    \end{bmatrix} \cdot \calQ(\mV).
\end{align}
Finally, we eliminate $\bar{\mV}_\star \bar{\mV}_\star^\top$ from the above equation. Since $\mV \in \O(\dim,\rnk)$, we have the constraint $\mV^\top \mV = \mI_\rnk$. Substituting the above formula for $\mV$ in the constraint gives:
\begin{align*}
    \calQ(\mV)^\top \cdot \left( \bar{\mV}_\star \bar{\mV}^\top_\star + \calZ(\mV)^\top \calZ(\mV) \right) \cdot \calQ(\mV) & = \mI_{\rnk}.
\end{align*}
Since $\calQ(\mV)$ is orthogonal, we can easily solve the above equation for $\bar{\mV}_\star \bar{\mV}^\top_\star$ and obtain: $\bar{\mV}_\star \bar{\mV}^\top_\star = \mI_\rnk - \calZ(\mV)^\top \calZ(\mV).$ Combining this with  \eqref{eq:decomp-interm} gives us the claim of the lemma:
\begin{align*}
\mV & = \mU \cdot    \begin{bmatrix} (\mI_\rnk - \calZ(\mV)^\top \calZ(\mV))^{1/2} \\ \calZ(\mV)    \end{bmatrix} \cdot \calQ(\mV).    
\end{align*}
\end{proof}
\subsection{Proof of \texorpdfstring{\Cref{prop:stat-distr}}{prop:stat-distr}} \label{app:stat-distr}
Recall that the unnormalized density of $\rV \sim \nu_\dim$ is given by:
\begin{align} \label{eq:exp-mech-density-recall}
    \frac{\diff \nu_\dim}{\diff \xi_{\dim,\rnk}}(\mV) \propto \exp\left( \frac{\dim \beta_\dim}{2} \cdot \Tr[ \mV^\top \mSigma \mV] \right),
\end{align}
where $\mSigma$ is a $\dim \times \dim$ matrix with eigendecomposition $\mSigma = \mU \diag(\lambda_{1:\dim}) \mU^\top$.  
Our goal is to characterize the distribution of the statistics:
\begin{align} \label{eq:conv-stat}
    \calQ(\mV) \bydef (\bar{\mV}_\star \bar{\mV}_\star^\top)^{-1/2} \cdot \bar{\mV}_\star, \quad \calZ(\mV) \bydef \bar{\mV}_\perp \cdot Q(\mV)^\top,
\end{align}
where:
\begin{align*}
    \bar{\mV} & = \begin{bmatrix} \bar{\mV}_\star \\ \bar{\mV}_\perp \end{bmatrix} \bydef  \mU^\top \mV \quad \text{ with } 
 \quad \bar{\mV}_\star \in \R^{\rnk \times \rnk}, \quad  \bar{\mV}_\perp \in \R^{(\dim-\rnk) \times \rnk}.
\end{align*}
We begin by characterizing the distribution of $(\calQ(\rV), \calZ(\rV))$ when $\rV$ is sampled from the base measure $\xi_{\dim,\rnk} \bydef \unif{\O(\dim,\rnk)}$.

\begin{lemma} \label{lem:QZ-base-measure} When $\rV \sim \xi_{\dim,\rnk} \bydef \unif{\O(\dim,\rnk)}$, $\calQ(\rV)$ and $\calZ(\rV)$ are independent. Moreover, $\calQ(\rV) \sim \xi_{\rnk,\rnk} \explain{def}{=} \unif{\O(\rnk)}$ and if $\dim \geq 2\rnk$, the unnormalized density of $\calZ(\rV)$ (with respect to the Lebesgue measure) is given by:
\begin{align} 
    \frac{\diff (\calZ \#\xi_{\dim,\rnk}) }{\diff \mZ} (\mZ) \propto \frac{1}{\sqrt{\det(\mI_k - \mZ^\top \mZ)}} \cdot  \ind\{ 0_{\rnk\times\rnk}\prec \mZ^\top \mZ \prec \mI_\rnk \},
\end{align}
where $\ind\{ 0_{\rnk\times\rnk}\prec \mZ^\top \mZ \prec \mI_\rnk \}$ denotes the indicator function which evaluates to $1$ if $0_{\rnk\times\rnk}\prec \mZ^\top \mZ \prec \mI_\rnk$ and is $0$ otherwise. 
\end{lemma}
We first prove \Cref{prop:stat-distr} taking this lemma for granted.

\begin{proof}[Proof of \Cref{prop:stat-distr}] Let $f: \R^{\rnk \times \rnk} \rightarrow \R$ and $g: \R^{(\dim-\rnk) \times \rnk} \rightarrow \R$ be arbitrary bounded measurable functions. It suffices to show that:
\begin{multline} \label{eq:QZ-distr-goal}
    \E_{\rV \sim \nu_\dim} \left[f(\calQ(\rV)) g(\calZ(\rV)) \right]  \propto\\  \E_{\rQ \sim \xi_{\rnk,\rnk}} \left[f(\rQ) \right] \cdot \int_{\R^{(\dim-\rnk) \times \rnk}}  g(Z) \cdot \frac{e^{ - \frac{\dim \beta_\dim}{2} \sum_{i=1}^{\dim - \rnk} \sum_{j=1}^{\rnk} (\lambda_{j} - \lambda_{i + \rnk}) Z_{ij}^2 }}{\sqrt{\det(\mI_k - \mZ^\top \mZ)}} \cdot  \ind\{0_{\rnk\times\rnk}\prec \mZ^\top \mZ \prec \mI_\rnk \} \diff Z,
\end{multline}
where $\propto$ hides the deterministic normalizing constant, which does not depend on $f,g$. Recalling the density of $\nu_\dim$ from \eqref{eq:exp-mech-density-recall}, we obtain the following formula for the expectation on the LHS:
\begin{align} \label{eq:distr-sub}
    \E_{\rV \sim \nu_\dim} \left[f(\calQ(\rV)) g(\calZ(\rV)) \right]  \propto \E_{\rV \sim \xi_{\dim,\rnk}} \left[f(\calQ(\rV)) g(\calZ(\rV)) \cdot e^{\frac{\dim\beta_\dim}{2}\Tr[\rV^\top \mSigma \rV ]}  \right].
\end{align}
The factor $e^{\frac{\dim \beta_\dim}{2}\Tr[\rV^\top \mSigma \rV ]}$ on the RHS can be written as a function of the random variables $\rQ \explain{def}{=} \calQ(\rV)$ and $\rZ \explain{def}{=} \calZ(\rV)$ using \Cref{lem:decomp} which shows that:
\begin{align*}
    \rV = \mU \begin{bmatrix} (I_k - \rZ^\top \rZ)^{1/2} \\ \rZ \end{bmatrix}  \rQ.
\end{align*}
Recalling that $\mSigma$ has the eigendecomposition  $\mSigma = \mU \diag(\lambda_{1:\dim}) \mU^\top$, we can compute:
\begin{align*}
    \Tr[\rV^\top \mSigma \rV ] & = \Tr[(I_k - \rZ^\top \rZ)^{1/2} \diag(\lambda_{1:\rnk}) (I_k - \rZ^\top \rZ)^{1/2}] +  \Tr[\rZ^\top \diag(\lambda_{\rnk+1:\dim}) \rZ]\\
    & = \Tr[ \diag(\lambda_{1:\rnk}) (I_k - \rZ^\top \rZ)] +  \Tr[\rZ^\top \diag(\lambda_{\rnk+1:\dim}) \rZ] \\
    & = \sum_{j=1}^\rnk \lambda_j - \sum_{i=1}^{\dim - \rnk} \sum_{j=1}^\rnk (\lambda_j - \lambda_{i + \rnk}) \cdot \rZ_{ij}^2.
\end{align*}
Substituting the above formula into \eqref{eq:distr-sub}, we obtain:
\begin{align*}
 &\E_{\rV \sim \nu_\dim} \left[f(\calQ(\rV)) g(\calZ(\rV)) \right]  \propto \E_{(\rQ,\rZ) \sim (\calQ,\calZ)\# \xi_{\dim,\rnk}} \left[f(\rQ) g(\rZ) \cdot e^{-\frac{\dim\beta_\dim}{2}\sum_{i=1}^{\dim - \rnk} \sum_{j=1}^\rnk (\lambda_j - \lambda_{i + \rnk}) \cdot \rZ_{ij}^2}  \right] \\
 & \hspace{1cm} \explain{Lem. \ref{lem:QZ-base-measure}}{=}  \E_{\rQ \sim \xi_{\rnk,\rnk}} \left[f(\rQ) \right] \cdot \int_{\R^{(\dim-\rnk) \times \rnk}}  g(Z) \cdot \frac{e^{ - \frac{\dim\beta_\dim}{2} \sum_{i=1}^{\dim - \rnk} \sum_{j=1}^{\rnk} (\lambda_{j} - \lambda_{i + \rnk}) Z_{ij}^2 }}{\sqrt{\det(\mI_k - \mZ^\top \mZ)}} \cdot  \ind\{ \mZ^\top \mZ \preceq \mI_\rnk \} \diff Z,
 \end{align*}
 which proves \eqref{eq:QZ-distr-goal} and the claim of \Cref{prop:stat-distr}.
\end{proof}
We complete the proof of \Cref{prop:stat-distr} by presenting the proof of \Cref{lem:QZ-base-measure}.

\begin{proof}[Proof of \Cref{lem:QZ-base-measure}] The proof of \Cref{lem:QZ-base-measure} relies on the following well-known characterization (see for e.g., \citep[Lemma 5 and Lemma 6]{dawid1977spherical}) of the uniform distribution on $\O(\dim,\rnk)$, stated below for convenience.
\begin{fact} \label{fact:haar} Consider any $\dim,\rnk \in \N$ with $\rnk \leq \dim$. \begin{enumerate}
\item A random matrix $\rO \sim \xi_{\dim, \rnk} \bydef \unif{\O(\dim,\rnk)}$ if and only if $\rO^\top \rO = \mI_{\rnk}$ and $\rO \explain{d}{=} \mA \rO$ for any deterministic $\mA \in \O(\dim)$.
\item Let $\rO \sim \unif{\O(\dim,\rnk)}$, then $\mA \rO \mB \explain{d}{=} \rO$ for any deterministic $\mA \in \O(\dim)$ and any $\mB \in \O(\rnk)$.
\end{enumerate}
\end{fact}

% \begin{proof} Recall the standard construction of $\unif{\O(\dim,\rnk)}$: we first sample $\rU  \sim \unif{\O(\dim,\rnk)}$ and partition $\rU = [\rU_1, \rU_2]$ into two blocks consisting of the first $\rnk$ columns and the last $\dim - \rnk$ columns. Then $\rU_1 \sim \unif{\O(\dim)}$. We will show that any random matrix $\rO$ which satisfies $\rO^\top \rO = \mI_\rnk$ and $\mA\rO  \explain{d}{=} \rO$ for any deterministic $\mA \in \O(\dim)$, must have the same distribution as $\rU_1$. Let us realize the matrices $\rO, \rU$ in the same probability space by sampling them independently. Let $\rB = [\rO, \tilde{\rO}] \in \O(\dim)$ be any $\rO$-measurable orthogonal matrix whose first $\rnk$ columns coincide with the columns of $\rO$, and the remaining columns are constructed by basis completion. Note that since the distribution of $\rU$ is invariant to right-multiplication by an independent orthogonal matrix:
% \begin{align*}
%    [\rU_1, \rU_2] =  \rU \explain{d}{=} \rU \rB = [\rU \rO, \rU \tilde{\rO}].
% \end{align*}
% In particular $\rU_1 \explain{d}{=} \rU \rO$. Furthermore, since the distribution of $\rO$ is invariant to left multiplication by an independent random orthogonal matrix, $\rU \rO \explain{d}{=} \rO$. Hence $\rU_1 \explain{d}{=} \rO$, as claimed. 
% \end{proof}
% \end{proof}
Let $\rV \sim \xi_{\dim,\rnk} \bydef \unif{\O(\dim,\rnk)}$. To characterize the distributions of the statistics $\calQ(\rV)$ and $\calZ(\rV)$ in \eqref{eq:conv-stat}, it will be helpful to view them as functions of $\bar{\rV} \bydef \mU^\top \rV$:
\begin{align} \label{j7}
    \calQ(\bar{\rV}) \bydef (\bar{\rV}_\star \bar{\rV}_\star^\top)^{-1/2} \cdot \bar{\rV}_\star, \quad \calZ(\bar{\rV}) = \bar{\rV}_\perp \cdot Q(\bar{\rV})^\top,
\end{align}
where $\bar{\rV}_\star \in \R^{\rnk \times \rnk}$ and $\bar{\rV}_\perp \in \R^{(\dim - \rnk) \times \rnk}$ are the submatrices of $\bar{\rV}$ formed by the first $\rnk$ rows and the last $(\dim - \rnk)$ rows respectively. Since the distribution of $\rV$ is invariant to left multiplication by an orthogonal matrix (\Cref{fact:haar} item (1)), $\bar{\rV} \sim \unif{\O(\dim,\rnk)}$.

\paragraph{Marginal Distribution of \texorpdfstring{$\calQ(\bar{\rV})$}{calQbar}} We begin by showing that $\calQ(\bar{\rV}) \sim \unif{\O(\rnk)}$. Thanks to \Cref{fact:haar} (item 1), it suffices to verify that $\calQ(\bar{\rV}) \explain{d}{=} \mA \calQ(\bar{\rV})$ for any deterministic orthogonal matrix $\mA \in \O(\rnk)$. Since $\bar{\rV} \sim \unif{\O(\dim,\rnk)}$, its distribution is invariant to left multiplication by an orthogonal matrix (\Cref{fact:haar}, item 2):
\begin{align*}
    \overline{\rV} \explain{d}{=} \begin{bmatrix} \mA & 0 \\ 0 & I_{\dim-\rnk} \end{bmatrix}\overline{\rV} = \begin{bmatrix} \mA \bar{\rV}_\star \\ \bar{\rV}_\perp \end{bmatrix}.
\end{align*}
Applying the function $\calQ(\cdot)$ to the above distributional equality we conclude that $ \calQ(\bar{\rV}) \explain{d}{=} \mA \calQ(\bar{\rV})$ for any deterministic orthogonal matrix $\mA$, and hence $\calQ(\bar{\rV}) \sim \unif{\O(\rnk)}$. 
\paragraph{Independence of \texorpdfstring{$\calQ(\bar{\rV})$}{calQrV} and \texorpdfstring{$\calZ(\bar{\rV})$}{calZrV}} Next, we show that $\calQ(\bar{\rV})$ and $\calZ(\bar{\rV})$ are independent. It suffices to show that for any two bounded measurable functions $f:\O(\rnk) \rightarrow \R$ and $g: \R^{(\dim-\rnk) \times \rnk} \rightarrow \R$, we have:
\begin{align} \label{eq:indp-check}
    \E[f(\calQ(\bar{\rV})) g(\calZ(\bar{\rV}))] = \E[f(\calQ(\bar{\rV}))] \E[g(\calZ(\bar{\rV}))].
\end{align}
We begin by observing that since the distribution of $\bar{\rV} \sim \unif{\O(\dim,\rnk)}$ is invariant to right multiplication by a deterministic orthogonal matrix $\mA \in \O(\rnk)$, $\bar{\rV} \explain{d}{=} \bar{\rV} \mA$. Applying the functions $\calQ(\cdot)$ and $\calZ(\cdot)$ to both sides of this distributional equality (recall \eqref{j7}), we conclude that:
\begin{align} \label{eq:key-dist-eq}
    (\calQ(\bar{\rV}), \calZ(\bar{\rV})) \explain{d}{=} (\calQ(\bar{\rV}) \mA, \calZ(\bar{\rV})).
\end{align}
Let $\rA \sim \unif{\O(\rnk)}$ be a uniformly random matrix sampled independently of $\bar{\rV}$. To prove \eqref{eq:indp-check}, we simplify the LHS using the Tower property to condition on $\rA$:
\begin{multline*}
    \E[f(\calQ(\bar{\rV})) g(\calZ(\bar{\rV}))]  = \E[\E[f(\calQ(\bar{\rV})) g(\calZ(\bar{\rV})) | \rA]] \explain{(a)}{=} \E[\E[f(\calQ(\bar{\rV})\rA) g(\calZ(\bar{\rV})) | \rA]]\\ = \E[f(\calQ(\bar{\rV})\rA) g(\calZ(\bar{\rV}))],
\end{multline*}
where step (a) follows by applying the distributional equality \eqref{eq:key-dist-eq} conditional on $\rA$ (which allows us to view $\rA$ as a deterministic matrix without changing the distribution of $\bar{\rV}$, which is independent of $\rA$). We further simplify the above expression by applying the Tower property again, this time to condition on $\bar{\rV}$:
\begin{multline*}
    \E[f(\calQ(\bar{\rV})) g(\calZ(\bar{\rV}))]  = \E[f(\calQ(\bar{\rV})\rA) g(\calZ(\bar{\rV}))]  = \E[g(\calZ(\bar{\rV}))\E[f(\calQ(\bar{\rV})\rA) | \overline{\rV}]]\\ \explain{(b)}{=} \E[g(\calZ(\bar{\rV}))] \E[f(\rA)].
\end{multline*}
where step (b) follows by recalling that $\rA \sim \unif{\O(\rnk)}$, and hence $\rA \explain{d}{=} \mQ \rA$ for any deterministic orthogonal matrix. In particular, applying this distributional equality conditional on $\bar{\rV}$ (which allows us to treat $\calQ(\bar{\rV})$ as a deterministic orthogonal matrix, without changing the distribution of $\rA$), we conclude that $\calQ(\bar{\rV})\rA | \bar{\rV} \explain{d}{=} \rA | \bar{\rV} \explain{d}{=} \rA$ which implies the equality marked (b). Lastly, we recall that we have already shown that $\calQ(\overline{\rV})\explain{d}{=} \rA$ (since both are $\unif{\O(\rnk)}$ distributed), and hence:
\begin{align*}
    \E[f(\calQ(\bar{\rV})) g(\calZ(\bar{\rV}))]  &= \E[f(\calQ(\bar{\rV}))] \E[g(\calZ(\bar{\rV}))],
\end{align*}
which proves that $\calQ(\bar{\rV}), \calZ(\bar{\rV})$ are independent. 
\paragraph{Marginal Density of \texorpdfstring{$\calZ(\overline{\rV})$}{calZrV}} We will show that:
\begin{align} \label{eq:Z-dist-eq}
    \calZ(\overline{\rV}) & \explain{d}{=} \bar{\rV}_\perp,
\end{align}
where $\bar{\rV}_\perp$ denotes the sub-matrix of $\bar{\rV}$ formed by the last $(\dim - \rnk)$ rows of $\bar{\rV}$. This is sufficient to identify the marginal PDF of $\calZ(\overline{\rV})$ because the marginal PDF of $\bar{\rV}_\perp$ (which is a sub-matrix of a uniformly random orthogonal matrix) is well-known \citep[Theorem 2.11]{meckes2019random} and is given by:
\begin{align*}
    h(\bar{\mV}_\perp) \propto \det(\mI_k - \bar{\mV}_\perp^\top \bar{\mV}_\perp)^{-1/2} \cdot  \ind\{ \bar{\mV}_\perp^\top \bar{\mV}_\perp \preceq \mI_\rnk \},
\end{align*}
when $\dim - \rnk \geq \rnk \Leftrightarrow \dim \geq 2 \rnk$, which immediately gives the claimed density formula for $\calZ(\bar{\rV})$. To show \eqref{eq:Z-dist-eq}, we begin by observing that since the distribution of $\bar{\rV} \sim \unif{\O(\dim,\rnk)}$ is invariant to left and right multiplication by orthogonal matrices, for any deterministic $\mA \in \O(\rnk), \mB \in \O(\dim - \rnk)$,
\begin{align} \label{eq:rV-dist-inv}
    \bar{\rV} \explain{d}{=} \begin{bmatrix} \mA^\top & 0 \\ 0 & \mB \end{bmatrix} \bar{\rV} \mA = \begin{bmatrix} \mA^\top \bar{\rV}_\star \mA \\ \mB \bar{\rV}_\perp \mA.
    \end{bmatrix}
\end{align}
This implies the following two distributional equalities for any $\mA \in \O(\rnk), \mB \in \O(\dim - \rnk)$:
\begin{align} \label{eq:2-dist-eq}
    \bar{\rV}_\perp \explain{d}{=} \mB \bar{\rV}_\perp \mA, \quad \calZ(\bar{\rV}) \explain{d}{=} \mB \calZ(\bar{\rV}) \mA.
\end{align}
The first follows by comparing the corresponding submatrices in \eqref{eq:rV-dist-inv} and the second follows by applying the function $\calZ(\cdot)$ to both sides of \eqref{eq:rV-dist-inv}. Hence, we have shown that the distributions of $\bar{\rV}_\perp$ and $\calZ(\bar{\rV})$ are both invariant to left and right multiplication by deterministic orthogonal matrices. Moreover, since $\calZ(\bar{\rV}) \calZ(\bar{\rV})^\top = \bar{\rV}_\perp \bar{\rV}_\perp^\top$ (cf. \eqref{j7}), the two random matrices $\bar{\rV}_\perp$ and $\calZ(\bar{\rV})$ share the same singular values. This is enough to guarantee that the two random matrices have identical distributions. We demonstrate this by showing that for any bounded measurable function $g: \R^{(\dim - \rnk) \times \rnk} \rightarrow \R$ we have:
\begin{align} \label{eq:Z-dist-id}
 \E[g(\bar{\rV}_\perp)] = \E[g(\calZ(\bar{\rV}))].
\end{align}
We will show this by using an argument similar to the one used to prove the independence of $\calQ(\bar{\rV}), \calZ(\bar{\rV})$: we consider two independent random orthogonal matrices $\rA \sim \unif{\O(\rnk)}$ and $\rB \sim \unif{\O(\dim - \rnk)}$, sampled independently of $\bar{\rV}$. We simplify the LHS of \eqref{eq:Z-dist-id} using the tower property to condition on $\rA,\rB$:
\begin{align} \label{eq:arg-rep-1}
    \E[g(\bar{\rV}_\perp)] & = \E[\E[g(\bar{\rV}_\perp)|\rA,\rB]] \explain{(a)}{=} \E[\E[g(\rB\bar{\rV}_\perp \rA)|\rA,\rB]] = \E[g(\rB\bar{\rV}_\perp \rA)],
\end{align}
where (a) follows by applying \eqref{eq:2-dist-eq} (conditioned on $\rA,\rB$). We simplify the above expression by considering the singular value decomposition of $\bar{\rV}_\perp = \rL \rS \rR^\top$ (where $\rL \in \O(\dim-\rnk), \rS \in \R^{(\dim-\rnk) \times \rnk}, \rR \in \O(\rnk)$ denote the left singular vectors, the singular values, and right singular vectors of $\bar{\rV}_\perp$), and applying the tower property again, this time conditioning on $\rL, \rS, \rR$:
\begin{multline} \label{eq:arg-rep-2}
    \E[g(\bar{\rV}_\perp)]  = \E[g(\rB \rL \rS \rR^\top \rA)] = \E[\E[g(\rB \rL \rS \rR^\top \rA) | \rL, \rS, \rR]]\\ \explain{(b)}{=} \E[\E[g(\rB \rS \rA) | \rL, \rS, \rR]] = \E[g(\rB \rS \rA)],
\end{multline}
where (b) follows by recalling that $\rA, \rB$ are independent uniformly random orthogonal matrices and hence $(\rB \mL, \mR^\top \rA) \explain{d}{=} (\rB,\rA)$ for any deterministic orthogonal matrices $\mL,\mR$. In particular, applying this distributional equality conditional on $\rL,\rS,\rR$ gives the equality marked (b). 

We can similarly simplify $\E[g(\calZ(\bar{\rV}))]$ by considering the singular value decomposition of $\calZ(\bar{\rV}) = \tilde{\rL} \rS \tilde{\rR}^\top$, where $\tilde{\rL} \in \O(\dim-\rnk), \rS \in \R^{(\dim-\rnk) \times \rnk}, \tilde{\rR} \in \O(\rnk)$ denote the left singular vectors, the singular values, and right singular vectors of $\calZ(\bar{\rV})$ (recall $\bar{\rV}_\perp$ and $\calZ(\bar{\rV})$ have the same singular values). Repeating exactly the same arguments used in \eqref{eq:arg-rep-1} and \eqref{eq:arg-rep-2} gives:
\begin{align*}
    \E[g(\calZ(\bar{\rV}))] & = \E[g(\rB \rS \rA)].
\end{align*}
Comparing the above equation with \eqref{eq:arg-rep-2} proves that $\bar{\rV}_\perp$ and $\calZ(\bar{\rV})$ have identical distributions and concludes the proof of \Cref{lem:QZ-base-measure}.
\end{proof}
\subsection{Proof of \texorpdfstring{\Cref{prop:Z-dist-approx}}{prop:Z-dist-approx}} \label{app:Z-dist-approx}
Recall that our goal is to show that $\tv{(\calZ \# \nu_{\dim}, \zeta_\dim)} \rightarrow 0.$ The PDF of $\rZ = \calZ(\rV)$ when $\rV$ is drawn from the Gibbs distribution $\nu_\dim \bydef \nu(\cdot \mid \mSigma,\beta_\dim,\rnk)$ was computed in \Cref{prop:stat-distr} (up to an unknown normalizing constant):
\begin{align} \label{eq:unnorm-1}
    \frac{\diff (\calZ \#\nu_\dim) }{\diff \mZ} (\mZ) \propto \frac{1}{\sqrt{\det(\mI_k - \mZ^\top \mZ)}} \cdot e^{ - \frac{\dim\beta_\dim}{2} \sum_{i=1}^{\dim - \rnk} \sum_{j=1}^{\rnk} (\lambda_{j} - \lambda_{i + \rnk}) Z_{ij}^2 } \cdot \ind\{ 0_{\rnk\times\rnk}\prec \mZ^\top \mZ \prec \mI_\rnk \}.
\end{align}
On the other hand, the sampling algorithm (Algorithm \mref{alg:sampler}) samples $\rZ \sim \zeta_\dim$ from a multivariate Gaussian distribution, with density:
\begin{align} \label{eq:unnorm-2}
    \frac{\diff \zeta_\dim }{\diff \mZ} (\mZ) & \propto  e^{ - \frac{\dim \beta_\dim}{2} \sum_{i=1}^{\dim - \rnk} \sum_{j=1}^{\rnk} (\lambda_{j} - \lambda_{i + \rnk}) Z_{ij}^2 }.
\end{align}
A natural approach to show that $\lim_{\dim \rightarrow \infty} \tv(\calZ\#\nu_\dim, \zeta_\dim)=0$ would be to use Scheffé's lemma, which would give us the desired conclusion if we can show ${\diff (\calZ \#\nu_\dim)} (\rZ)/{\diff \zeta_\dim}  \pc 1$ when $\rZ \sim \zeta_\dim$. The bottleneck in doing so is that we do not have a fine-grained understanding of the asymptotics of the unknown normalizing constant in \eqref{eq:unnorm-1}. To bypass this, we will rely on the following result due to \citet{mukherjee2013statistics}, which can be viewed as an analog of Scheffé's lemma for distributions with complicated normalizing constants.

\begin{fact}[{\citet[Lemma 3.1]{mukherjee2013statistics}}] \label{fact:tv-conv} For each $\dim \in \N$, let $\chi_{0,\dim}$ and $\chi_{1,\dim}$ be distributions on a common sample space $\mathscr{Z}_p$ such that {$\chi_{1,\dim} \ll \chi_{0,\dim}$}. Suppose that the log-likelihood ratio of $\chi_{1,\dim}$ with respect to $\chi_{0,\dim}$ has the form:
\begin{align*}
   \ln  \frac{\diff \chi_{1,\dim}}{\diff \chi_{0,\dim}}(\mZ) & = L_\dim(\mZ) + c_\dim \quad \forall \; \mZ \; \in \; \mathscr{Z}_\dim.
\end{align*}
for some function $L_\dim: \mathscr{Z}_\dim \rightarrow \R \cup \{-\infty\}$ and a constant $c_\dim$. Suppose that:
\begin{enumerate} \item  When $\rZ \sim \chi_{0,\dim}$, $L_\dim(\rZ) \pc c$ for some non-random constant $c$.  \item When $\rZ \sim \chi_{1,\dim}$, $L_p(\rZ) = O_\P(1)$. 
\end{enumerate}
Then $\lim_{\dim \rightarrow \infty} \tv(\chi_{0,\dim}, \chi_{1,\dim}) = 0$.
\end{fact}
From \eqref{eq:unnorm-1} and \eqref{eq:unnorm-2}, the log-likelihood ratio between $\calZ \# \nu_{\dim}$ and $\zeta_\dim$ is given by:
\begin{align} \label{eq:log-LLR}
    \ln \frac{\diff (\calZ \#\nu_\dim)}{\diff \zeta_\dim} (\rZ) & = - \ln \det(\mI_k - \rZ^\top \rZ)/2 + \ln \ind\{ 0_{\rnk\times\rnk}\prec \rZ^\top \rZ \prec \mI_\rnk \} + c_\dim,
\end{align}
where $c_\dim$ denotes a deterministic constant which is determined by the normalizing constants in \eqref{eq:unnorm-1} and \eqref{eq:unnorm-2}. Following the strategy outlined in \Cref{fact:tv-conv} requires us to understand the asymptotics of $\rZ^\top \rZ$ when $\rZ \sim \calZ \# \nu_\dim$ and when $\rZ \sim \zeta_\dim$, which is done in the following two lemmas. 
\begin{lemma} \label{lem:Z^T*Z-exp}When $\rZ \sim \calZ \# \nu_\dim$ and $\beta = \lim_{\dim \rightarrow \infty} \beta_\dim > H_\mu(\gamma_\rnk)$, 
\begin{align*}
    \rZ^\top \rZ &\pc \diag\left( \frac{H_\mu(\gamma_1)}{\beta}, \dotsc,  \frac{H_\mu(\gamma_\rnk)}{\beta}\right) \prec \mI_\rnk.
\end{align*}
\end{lemma}
% \begin{lemma} \label{lem:Z^T*Z-sampler} When $\rZ \sim \zeta_\dim$ and $\beta = \lim_{\dim \rightarrow \infty} \beta_\dim > H_\mu(\gamma_\rnk)$, 
% \begin{align*}
%     \rZ^\top \rZ &\pc \diag\left( \frac{H_\mu(\gamma_1)}{\beta}, \dotsc,  \frac{H_\mu(\gamma_\rnk)}{\beta}\right) \prec \mI_\rnk\quad\text{and}\quad \E[\rZ^\top \rZ]\rightarrow \diag\left( \frac{H_\mu(\gamma_1)}{\beta}, \dotsc,  \frac{H_\mu(\gamma_\rnk)}{\beta}\right).
% \end{align*}
% \end{lemma}

\begin{lemma}\label{lem:Z^T*Z-sampler} When $\rZ \sim \zeta_\dim$ and $\beta = \lim_{\dim \rightarrow \infty} \beta_\dim > H_\mu(\gamma_\rnk)$, 
\begin{align*}
    \E[\rZ^\top \rZ]\rightarrow \diag\left( \frac{H_\mu(\gamma_1)}{\beta}, \dotsc,  \frac{H_\mu(\gamma_\rnk)}{\beta}\right)\prec \mI_\rnk\quad\text{and}\quad\norm*{\rZ^\top \rZ - \E[\rZ^\top \rZ]}_{\fr}=O_\P\left(\frac{1}{\sqrt{\dim}}\right).
\end{align*}
\end{lemma}
Note that the claim of \Cref{prop:Z-dist-approx} is immediate from these two lemmas. Indeed, by the continuous mapping theorem, when $\rZ  \sim \calZ \# {\nu}_\dim$ or $\rZ \sim \zeta_\dim$, the random part of the log-likelihood ratio in \eqref{eq:log-LLR} satisfies:
\begin{align*}
    - \frac{\ln \det(\mI_k - \rZ^\top \rZ)}{2} + \ln \ind\{ 0_{\rnk\times\rnk}\prec \mZ^\top \mZ \prec \mI_\rnk \} \pc - \frac{1}{2}\sum_{i=1}^\rnk \ln\left(1 - \frac{H_\mu(\gamma_i)}{\beta} \right).
\end{align*}
%where the limit on the RHS is well-defined since \Cref{prop:Z-dist-approx} assumes that $\beta > H_\mu(\gamma_k)$, which guarantees that $\beta > H_\mu(\gamma_i)$ for any $i \in [\rnk]$ (recall that $H_\mu$ is a non-increasing function). 
This implies that $\tv{(\calZ \# \nu_{\dim}, \zeta_\dim)} \rightarrow 0$ by \Cref{fact:tv-conv}. Hence, the remainder of this section is devoted to the proof of \Cref{lem:Z^T*Z-exp,lem:Z^T*Z-sampler}. 

\begin{proof}[Proof of \Cref{lem:Z^T*Z-exp}] Recall that when $\rZ \sim \calZ \# \nu_\dim$, we can realize $\rZ$ as $\rZ = \calZ(\rV)$ where $\rV \sim \nu_\dim$. Recalling the formula for $\calZ: \O(\dim,\rnk) \rightarrow \R^{(\dim - \rnk) \times \rnk}$ from \eqref{eq:conv-stat-1}, we find that:
\begin{align*}
   \rZ {=} \bar{\rV}_\perp \bar{\rV}_{\star}^\top (\bar{\rV}_{\star} \bar{\rV}_\star^\top)^{-1/2} \quad \text{where}\quad \bar{\rV}_{\star} \bydef \mU_\star^\top \rV, \quad  \bar{\rV}_{\perp} \bydef \mU_\perp^\top \rV, 
\end{align*}
and $\mU = \begin{bmatrix} \mU_\star & \mU_\perp \end{bmatrix}$ denotes the matrix of eigenvectors of the second moment matrix $\mSigma$. We can now compute:
\begin{align*}
    \rZ^\top \rZ & =(\bar{\rV}_{\star} \bar{\rV}_\star^\top)^{-1/2} \bar{\rV}_{\star} \bar{\rV}_{\perp}^\top \bar{\rV}_\perp \bar{\rV}_{\star}^\top (\bar{\rV}_{\star} \bar{\rV}_\star^\top)^{-1/2} \\
    & = (\bar{\rV}_{\star} \bar{\rV}_\star^\top)^{-1/2} \bar{\rV}_{\star} (\mI_\rnk - \bar{\rV}_\star^\top \bar{\rV}_\star) \bar{\rV}_{\star}^\top (\bar{\rV}_{\star} \bar{\rV}_\star^\top)^{-1/2} \quad \text{[since $\bar{\rV}_\star^\top \bar{\rV}_\star + \bar{\rV}^\top_\perp \bar{\rV}_\perp = \mI_\rnk$]} \\
    & = \mI_\rnk - \bar{\rV}_\star \bar{\rV}_\star^\top \\
    & \pc \mI_\rnk - \diag\left(1 - \frac{H_\mu(\gamma_1)}{\beta}, 1 - \frac{H_\mu(\gamma_2)}{\beta}, \dotsc, 1-\frac{H_\mu(\gamma_\rnk)}{\beta} \right)_+  \quad\text{[by Theorem \mref{thm:overlap}]}\\
    & = \diag\left( \frac{H_\mu(\gamma_1)}{\beta}, \dotsc,  \frac{H_\mu(\gamma_\rnk)}{\beta}\right) \prec \mI_\rnk \quad \text{[since $\beta > H_\mu(\gamma_\rnk)$, $H_\mu$ is decreasing],}
\end{align*}
which completes the proof of \Cref{lem:Z^T*Z-exp}.
\end{proof}
\begin{proof}[Proof of \Cref{lem:Z^T*Z-sampler}] Letting $\rz_{1}, \dotsc, \rz_{\dim-\rnk}$ denote the rows of $\rZ$, our goal is to show: 
\begin{align*}
    \underbrace{\sum_{i=1}^{\dim - \rnk} \E[\rz_i \rz_i^\top]}_{(i)}\rightarrow \diag\left( \frac{H_\mu(\gamma_1)}{\beta}, \dotsc,  \frac{H_\mu(\gamma_\rnk)}{\beta}\right)\prec \mI_\rnk
\end{align*}
and 
\begin{align*}
    \underbrace{\norm*{\sum_{i=1}^{\dim - \rnk} (\rz_i \rz_i^\top - \E[\rz_i \rz_i^\top])}_{\fr}}_{(ii)}=O_\P\left(\frac{1}{\sqrt{\dim}}\right).
\end{align*}
% decompose $\rZ^\top \rZ$ as:
% \begin{align} \label{eq:Z^T*Z-decomp}
%     \rZ^\top \rZ  = \sum_{i=1}^{\dim - \rnk} \rz_i \rz_i^\top &= \underbrace{\sum_{i=1}^{\dim - \rnk} \E[\rz_i \rz_i^\top]}_{(i)} + \underbrace{\sum_{i=1}^{\dim - \rnk} (\rz_i \rz_i^\top - \E[\rz_i \rz_i^\top])}_{(ii)}.
% \end{align}
% \begin{align*}
%     \norm*{\rZ^\top \rZ - \E[\rZ^\top \rZ]}_{\fr}=O_\P\left(\frac{1}{\sqrt{\dim}}\right)\quad\text{and}\quad \E[\rZ^\top \rZ]\rightarrow \diag\left( \frac{H_\mu(\gamma_1)}{\beta}, \dotsc,  \frac{H_\mu(\gamma_\rnk)}{\beta}\right)\prec \mI_\rnk.
% \end{align*}
Recall the sampling algorithm (Algorithm \mref{alg:sampler}) generates $\rz_{1:(\dim-\rnk)}$ independently with:
\begin{align} \label{eq:z-dist-recall}
    \rz_i \sim \gauss{0}{\frac{\diag\left(\lambda_1 - \lambda_{\rnk+i}, \dotsc, \lambda_\rnk - \lambda_{\rnk+i}\right)^{-1}}{\dim \beta_\dim}}.
\end{align}
Hence, we find that the expectation term $(i)$ converges to:
\begin{align*}
    \sum_{i=1}^{\dim - \rnk} \E[\rz_i \rz_i^\top] & = \diag\left( \frac{1}{\dim\beta_\dim} \sum_{i = 1}^{\dim - \rnk} \frac{1}{\lambda_1 - \lambda_{k+i}}, \dotsc, \frac{1}{\dim\beta_{\dim}} \sum_{i = 1}^{\dim - \rnk} \frac{1}{\lambda_{\rnk} - \lambda_{k+i}}  \right) \\
    & \rightarrow \diag\left( \frac{H_\mu(\gamma_1)}{\beta}, \dotsc,  \frac{H_\mu(\gamma_\rnk)}{\beta}\right)\prec \mI_\rnk & \text{[since $\beta > H_\mu(\gamma_\rnk)$]},
\end{align*}
where the convergence in the final step follows by recalling that the empirical distribution of the eigenvalues $\lambda_{1:\dim}$ converges weakly to $\mu$ and $\lambda_{1:\rnk} \rightarrow \gamma_{1:\rnk}$ (see \Cref{lem:misc_HK}). Next, we show that the fluctuation term (ii) is $O_\P(1/\sqrt{\dim}).$ By Chebyshev's inequality, for any $\tau >0,$ we have:
\begin{align*}
    \P \left(\sqrt{\dim}\bigg\| \sum_{i=1}^{\dim - \rnk} (\rz_i \rz_i^\top - \E[\rz_i \rz_i^\top]) \bigg\|_{\fr}  > \tau\right) & \leq \frac{\dim}{\tau^2} \E \bigg\| \sum_{i=1}^{\dim - \rnk} (\rz_i \rz_i^\top - \E[\rz_i \rz_i^\top]) \bigg\|_{\fr}^2 \\
    & \explain{(a)}{=} \frac{\dim}{\tau^2} \sum_{i=1}^{\dim - \rnk} \E \| \rz_i \rz_i^\top - \E[\rz_i \rz_i^\top] \|_{\fr}^2 \\
    & = \frac{\dim}{\tau^2} \sum_{i=1}^{\dim - \rnk} (\E \| \rz_i \rz_i^\top\|_{\fr}^2 - \|\E [\rz_i \rz_i^\top]\|_{\fr}^2)\\
    & \leq \frac{\dim}{\tau^2} \sum_{i=1}^{\dim - \rnk} \E \| \rz_i\|^4,
\end{align*}
where (a) follows by expanding the squared norm and observing that the cross terms vanish because the individual terms in the sum are centered. We can further simplify the upper bound by recalling \eqref{eq:z-dist-recall}:
\begin{align}
     \P \left( \sqrt{\dim}\bigg\| \sum_{i=1}^{\dim - \rnk} (\rz_i \rz_i^\top - \E[\rz_i \rz_i^\top]) \bigg\|_{\fr}  > \tau\right) &\leq \frac{\dim}{\tau^2} \sum_{i=1}^{\dim - \rnk} \sum_{j,\ell = 1}^\rnk \E[\rz_{ij}^2 \rz_{i\ell}^2] \notag\\
     & \leq \frac{\dim}{\tau^2} \sum_{i=1}^{\dim - \rnk} \sum_{j,\ell = 1}^\rnk \sqrt{\E[\rz_{ij}^4]} \sqrt{\E[\rz_{i\ell}^4]}\notag\\
     & \explain{(a)}{=} \frac{3}{\tau^2 \dim \beta^2_\dim } \sum_{i=1}^{\dim - \rnk} \sum_{j,\ell = 1}^\rnk \frac{1}{ (\lambda_j - \lambda_{\rnk+i}) (\lambda_\ell - \lambda_{\rnk+i})}\notag\\
     & \explain{(b)}{\leq} \frac{3 \rnk^2}{\tau^2 \beta^2_\dim(\lambda_k - \lambda_{k+1})^2}\notag\\
     &\explain{(c)}{\leq} \frac{6 \rnk^2}{\tau^2 \beta^2\Delta^2}\quad\text{for large enough $\dim$}.\notag
\end{align}
In the above equations, (a) follows from \eqref{eq:z-dist-recall} and the fact that the 4th moment of a standard Gaussian random variable is $3$. In step (b) we noted that each term in the summation is bounded by $1/(\lambda_k - \lambda_{k+1})^2$ since the eigenvalues are sorted in decreasing order. Finally step (c) relied on our assumption that $\lambda_\rnk - \lambda_{\rnk+1} \rightarrow \gamma_\rnk - \gamma_{\rnk+1} \in (0,\infty)$ and $\beta_\dim \rightarrow \beta > H_\mu(\gamma_\rnk)$. Since the inequality holds for any $\tau>0,$ it can be chosen to make the upper bound in the last line be arbitrarily small. Hence, the fluctuation term $(ii)$ is $O_\P(1/\sqrt{\dim})$ and concludes the proof of \Cref{lem:Z^T*Z-sampler}.
\end{proof}

%% file: appendix_contiguity.tex
\section{Proof of Contiguity Result (Theorem \texorpdfstring{\mref{thm:contiguity})}{thm:contiguity}}\label{app:contiguity}
This appendix presents the proof of our contiguity result (Theorem \mref{thm:contiguity}). Recall that Theorem \mref{thm:contiguity} claims mutual contiguity for the two Gibbs distributions:
\begin{align*}
    \nu_\dim \bydef \nu( \cdot \mid \Sigma, \beta_\dim, \rnk), \quad \tilde{\nu}_\dim  \bydef \nu( \cdot  \mid  \tilde{\mSigma}, \beta_\dim, \rnk), \quad \text{where} \quad \tilde{\mSigma} \bydef \Sigma + \frac{\mE}{\sqrt{\dim}},
\end{align*}
for a sequence of matrices $\mSigma \in \R^{\dim \times \dim}$ which satisfies Assumption \mref{assump:mat}, sequence of noise parameters $\beta_\dim \rightarrow \beta > H_\mu(\gamma_\rnk),$ and a sequence of symmetric perturbation matrices $\mE\in\R^{\dim\times\dim}$ with $\|\mE\|_{\fr} \lesssim 1$. We will use $\nu_\dim, \tilde{\nu}_\dim$ as a convenient shorthand for these two Gibbs distributions throughout this section. 
%For any matrix $\mSigma\in \R^{\dim \times \dim}$ that satisfies \Cref{assump:mat}, the matrix $\tilde{\mSigma} \bydef \mSigma + \tfrac{1}{\sqrt{\dim}} \mE,$ where $\mE\in \R^{\dim\times\dim}$ is any symmetric perturbation matrix  satisfying $\|\mE\|_{\fr} \lesssim 1,$ and any sequence of noise parameters $\beta_{\dim} \rightarrow \beta \in (H_\mu(\gamma_\rnk),\infty),$ our goal is to show that the distributions $\nu(\cdot \mid \mSigma,\beta_\dim, \rnk)$ and $ \nu(\cdot \mid \tilde{\mSigma}, \beta_\dim, \rnk)$ are mutually contiguous and to asymptotically characterize the privacy of the exponential mechanism (\Cref{alg:ExpM}) with respect to the two distributions. For convenience, we will let $\nu_{\dim}\bydef \nu(\cdot \mid \mSigma,\beta_\dim, \rnk)$ and $\tilde{\nu}_{\dim}\bydef \nu(\cdot \mid \tilde{\mSigma}, \beta_\dim, \rnk)$ in the rest of this section.
We will begin by introducing the key ideas required to prove Theorem \mref{thm:contiguity} as a few intermediate results.

\paragraph{Preliminaries on Contiguity} Le Cam's first lemma (see \citet[Section 6.4]{van2000asymptotic}) provides a criterion for contiguity between two distributions $\nu_{\dim}$ and $\tilde{\nu}_{\dim}$ in terms of the limit distribution of the log-likelihood ratio between the two distributions. In our context, the log-likelihood ratio between $\tilde{\nu}_\dim, \nu_\dim$ is given by:
\begin{align} \label{eq:LLR-Z}
     \ln \frac{\diff\tilde{\nu}_{\dim}}{\diff\nu_{\dim}}(\mV) & = \frac{\sqrt{\dim} \beta_\dim }{2} \Tr[\mV^\top \mE \mV]  + \ln Z({\mSigma}, \beta_\dim, \rnk)-\ln Z(\tilde{\mSigma}, \beta_\dim, \rnk) \quad \forall \; \mV \; \in \; \O(\dim,\rnk),
\end{align}
% where $\ln Z(\mSigma, \beta_\dim, \rnk), \ln Z(\tilde{\mSigma}, \beta_\dim, \rnk)$ denote the log-normalizing constants of $\nu_\dim, \tilde{\nu}_\dim$. The likelihood ratio, additionally, plays a central role in characterizing the asymptotics of the privacy quantities of interest: 
% \begin{itemize}
%     % \item \textit{Total variation distance 
%     % $\tv(\nu_{\dim}, \tilde{\nu}_{\dim}).$} To show $\lm \tv(\nu_{\dim}, \tilde{\nu}_{\dim})=0,$ it suffices to show, by Scheffé's lemma, that $(\diff\tilde{\nu}_{\dim}/\diff\nu_{\dim})(\rV)\pc 1$ when $\rV\sim \nu_{\dim}.$ 
%     \item \textit{trade-off function $\tf{\nu_{\dim}}{\tilde{\nu}_{\dim}}.$} Since the Neyman-Pearson test for distinguishing between $\nu_{\dim}$ and $\tilde{\nu}_{\dim}$  is based on the likelihood ratio between these distributions, one expects that the asymptotic trade-off function can be derived from the limit distributions of the likelihood ratio under these two distributions.
%     \item \textit{Rényi divergence $\rdv{\alpha}{\tilde{\nu}_{\dim}}{\nu_{\dim}}.$} Recall that the definition of Rényi divergence of order $\alpha>1$  involves the statistic $(\diff\tilde{\nu}_{\dim}/\diff\nu_{\dim})(\rV),$ where $\rV\sim \nu_{\dim}:$
%     \begin{align*}
%         \rdv{\alpha}{\tilde{\nu}_\dim}{\nu_{\dim}} \bydef \frac{1}{\alpha - 1} \ln \E_{\rV \sim \nu_{\dim}} \left[ \left( \frac{\diff \tilde{\nu}_\dim}{\diff \nu_{\dim}}(\rV) \right)^\alpha \right].
%     \end{align*}
% \end{itemize}
However, because the normalizing constants of $\nu_{\dim}$ and $\tilde{\nu}_{\dim}$ take a complicated form, it is challenging to directly study the asymptotics of the likelihood ratio between the two distributions. Instead, we will rely on the following result, due to \citet{mukherjee2013statistics}, which provides a convenient criterion for contiguity between distributions with complicated normalizing constants.  

% \Cref{fact:Mukherjee} addresses this challenge by characterizing it in terms of the ratio between the un-normalized densities of the two distributions. 
%provides a way of bypassing the need to work with the normalizing constants, %characterizing the asymptotics of the likelihood ratio between any two distributions in terms of those of the ratio between their un-normalized densities. 
% \Cref{fact:Mukherjee}(2) and 
% \Cref{fact:Mukherjee}(3) serving as analogs of Le Cam's lemmas and Scheffé's lemma, respectively. 

\begin{fact}[{\citet[Lemma 3.1]{mukherjee2013statistics}\footnote{\citet[Lemma 3.1]{mukherjee2013statistics} consider the special case where $\mathcal{V}_{\dim}$ is a Cartesian product of intervals in $\R^{\dim}$ and the probability measures $\nu_{\dim}$ and $\tilde{\nu}_{\dim}$ are absolutely continuous with respect to the Lebesgue measure on $\R^{\dim}.$ The same proof holds for the more general setting in \Cref{fact:Mukherjee}. Additionally, the result for the constant shift $c_{\dim}$ is presented as an intermediate result in the proof of \citet[Lemma 3.1]{mukherjee2013statistics}.}}] \label{fact:Mukherjee} For each $\dim \in \N$, let $\nu_{\dim}$ and $\tilde{\nu}_{\dim}$ be two probability measures on a common sample space $\mathcal{\mV}_{\dim}$ such that $\tilde{\nu}_{\dim}$ is absolutely continuous with respect to $\nu_{\dim}.$ Suppose that the log-likelihood ratio of $\tilde{\nu}_{\dim}$ with respect to $\nu_{\dim}$ has the form:
\begin{align} \label{eq:LLR-random-nonrandom}
   \ln\frac{\diff \tilde{\nu}_{\dim}}{\diff \nu_{\dim}}(\mV) & = L_\dim(\mV) + c_\dim\quad \forall\: \mV\in \mathcal{\mV}_{\dim}
\end{align}
for some function $L_\dim:\mathcal{\mV}_{\dim}\rightarrow \R$ and a non-random constant $c_\dim\in\R.$ Suppose that $L_p(\rV)=O_\P(1)$ when $\rV\sim\tilde{\nu}_{\dim}$ and when $\rV\sim\nu_{\dim}.$ Then:
\begin{enumerate}
    \item $\nu_{\dim}$ and $\tilde{\nu}_{\dim}$ are mutually contiguous.
    % \item If $L_{\dim}(\rV)\dc \gauss{0}{\sigma^2}$ when  $\rV\sim\nu_{\dim},$ then $L_{\dim}(\rV)\dc\gauss{\sigma^2}{\sigma^2}$ when $\rV\sim\tilde{\nu}_{\dim},$ and $c_{\dim}\rightarrow -\sigma^2/2.$ Consequently, $\ln(\diff\tilde{\nu}_{\dim}/\diff\nu_{\dim})(\rV)\dc \gauss{-\sigma^2/2}{\sigma^2}$ when $\rV\sim\nu_{\dim}$ and $\ln(\diff\tilde{\nu}_{\dim}/\diff\nu_{\dim})(\rV)\dc \gauss{\sigma^2/2}{\sigma^2}$ when $\rV\sim\tilde{\nu}_{\dim}.$
    \item If $L_{\dim}(\rV)\dc \gauss{0}{\sigma^2}$ when  $\rV\sim\nu_{\dim},$ then: 
    \begin{align*}
        c_{\dim}\rightarrow -\frac{\sigma^2}{2}, \quad  \ln\frac{\diff\tilde{\nu}_{\dim}}{\diff\nu_{\dim}}(\rV)\dc \gauss{-\frac{\sigma^2}{2}}{\sigma^2}  \quad\text{when}\quad\rV\sim\nu_{\dim}.
    \end{align*}
    % \item If $L_\dim(\rV)\pc c$ for some constant $c\in\R$ when  $\rV\sim\nu_{\dim}$, then $\tv(\nu_{\dim}, \tilde{\nu}_{\dim}) \rightarrow 0,$ and $c_{\dim}\rightarrow-c.$ Consequently, $\ln(\diff\tilde{\nu}_{\dim}/\diff\nu_{\dim})(\rV)\pc0$ when $\rV\sim\nu_{\dim}$ and $\rV\sim\tilde{\nu}_{\dim}.$
    \item If $L_\dim(\rV)\pc c$ for some non-random constant $c\in\R$ when  $\rV\sim\nu_{\dim}$, then: 
    \begin{align*}
        c_{\dim}\rightarrow -c,\quad   \ln\frac{\diff\tilde{\nu}_{\dim}}{\diff\nu_{\dim}}(\rV)\pc 0 \quad\text{when}\quad \rV\sim\nu_{\dim},  \quad  \tv(\nu_{\dim}, \tilde{\nu}_{\dim}) \rightarrow 0.
    \end{align*}
\end{enumerate}
\end{fact}
The key point is that applying the above criterion for contiguity only requires us to analyze the random part $L_\dim(\rV)$ of the log-likelihood ratio in \eqref{eq:LLR-random-nonrandom}. This allows us to avoid a fine-grained analysis of the log-normalizing constants in \eqref{eq:LLR-Z}, as they can be absorbed into the non-random constant $c_\dim$ in \eqref{eq:LLR-random-nonrandom}. 
\paragraph{Trade-off Function and Limit Distribution of the Likelihood Ratio} While the first claim of \Cref{fact:Mukherjee} will help us prove the contiguity claim made in Theorem \mref{thm:contiguity}, the second and third claims (which characterize the limit distribution of the log-likelihood ratio) will be useful for deriving the asymptotic trade-off function between the Gibbs distributions $\nu_\dim, \tilde{\nu}_\dim$. This is due to a classical result in hypothesis testing \citep[Chapter 15]{lehmann2005testing}, summarized in the proposition below, which shows that the asymptotic trade-off function between two mutually contiguous distributions can be derived from the limit distribution of the log-likelihood ratio. 
\begin{proposition}\label{prop:le-cam} For each $\dim \in \N$, let $\nu_\dim$ and $\tilde{\nu}_\dim$ be two distributions on the same sample space $\mathcal{V}_{\dim}$ such that $\tilde{\nu}_{\dim}$ is absolutely continuous with respect to $\nu_{\dim}.$
\begin{enumerate}
    \item If $\ln (\diff\tilde{\nu}_{\dim}/\diff\nu_{\dim})(\rV)\dc \gauss{-\tfrac{v}{2}}{v}$ for some constant $v>0$ when $\rV\sim\nu_\dim$, then:
    \begin{align*}
        \lim_{\dim \rightarrow \infty} \tf{\nu_\dim}{\tilde{\nu}_\dim}(\alpha)=\tf{\gauss{0}{1}}{\gauss{\sqrt{v}}{1}}(\alpha)\quad\forall\:\alpha\in[0,1].
    \end{align*}
    \item If $\ln (\diff\tilde{\nu}_{\dim}/\diff\nu_{\dim})(\rV) \pc 0$ when $\rV\sim\nu_\dim$, then: 
    \begin{align*}
        \lim_{\dim \rightarrow \infty} \tf{\nu_\dim}{\tilde{\nu}_\dim}(\alpha)=1-\alpha\quad\forall\:\alpha\in[0,1].
    \end{align*}
\end{enumerate}
\end{proposition}
\paragraph{Limit Distribution of the Likelihood Ratio} To apply \Cref{fact:Mukherjee} and \Cref{prop:le-cam} to show that the Gibbs distributions $\nu_\dim, \tilde{\nu}_\dim$ are mutually contiguous and to derive their asymptotic trade-off function, we will need to characterize the limit distribution of the random part $\Tr[\rV^\top \mE \rV]$ of the log-likelihood ratio in \eqref{eq:LLR-Z}. The following proposition shows that the statistic $\Tr[\rV^\top \mE \rV]$ converges to a Gaussian limit,  after appropriate centering and rescaling. For the centering term, we will find it useful to introduce the function:
\begin{align}\label{eq:M-def}
    \calM_{\beta}(\mSigma)\bydef \sum_{j=1}^\rnk\left( 1-\frac{H_{\mSigma}(\lambda_j)}{\beta}\right)\vu_j\vu_j^\top+\frac{1}{\dim\beta}\sum_{j=1}^{\rnk}\sum_{i=1}^{\dim-\rnk}\frac{\vu_{\rnk+i}\vu_{\rnk+i}^\top}{\lambda_j-\lambda_{\rnk+i}} 
\end{align}
for any $\beta > 0$ and any symmetric matrix $\Sigma\in\R^{\dim\times\dim}$, where $\{(\lambda_i,\vu_i): i \in [\dim]\}$ denotes the eigenvalues and the corresponding eigenvectors of $\mSigma.$ We will also find it helpful to recall the variance function $\sigma_{\mSigma}^2: \R^{\dim \times \dim} \times [0,\infty) \to \R$ introduced in the statement of Theorem \mref{thm:contiguity}:
\begin{align*}
    \sigma_{\mSigma}^2(\mE,\beta) \explain{def}{=} \frac{1}{2} \sum_{j, \ell = 1}^\rnk K_\mSigma(\lambda_j, \lambda_\ell) \cdot (\vu_j^\top E \vu_\ell)^2  + \sum_{j=1}^\rnk \sum_{i= 1}^{\dim - \rnk} \frac{\beta - {H_{\mSigma}(\lambda_j)}}{\lambda_j - \lambda_{\rnk+i}}  \cdot (\vu_{\rnk+i}^\top E \vu_j)^2.
\end{align*}
In the above equations, the functions $H_\mSigma, K_\mSigma$ are defined as: 
\begin{align} \label{eq:H-K-def-recall}
    H_{\mSigma}(\lambda) &\explain{def}{=} \frac{1}{\dim} \sum_{i=1}^{\dim - \rnk} \frac{1}{\lambda - \lambda_{\rnk+i}}, \;   K_{\mSigma}(\lambda,\lambda^\prime) \explain{def}{=} \frac{1}{\dim} \sum_{i=1}^{\dim-\rnk} \frac{1}{(\lambda - \lambda_{\rnk+i})(\lambda^\prime - \lambda_{\rnk+i})} \;\forall \; \lambda, \lambda^\prime \; \in \; (\lambda_{\rnk+1}, \infty).
\end{align}
\begin{proposition}\label{prop:Lnull-dc} Consider a sequence of matrices $\mSigma\in\R^{\dim \times \dim}$ which satisfies Assumption \mref{assump:mat}, a sequence of noise parameters $\beta_{\dim} \rightarrow \beta \in (H_\mu(\gamma_\rnk),\infty)$ as $\dim \rightarrow \infty$, and a fixed rank $\rnk \in \N$ (independent of $\dim$). Let $\rV \sim \nu(\cdot \mid \Sigma, \beta_\dim, \rnk)$ denote a sample from the Gibbs distribution $\nu(\cdot \mid \Sigma, \beta_\dim, \rnk)$.  Suppose that the matrix $\mE \in \R^{\dim \times \dim}$ is symmetric and satisfies $\|\mE\|_{\fr} \lesssim 1.$ Then,
\begin{enumerate}
\item $\sigma^2_{\mSigma}(\mE, \beta_\dim) \lesssim 1$ and $\sqrt{\dim}\left(\Tr[\rV^\top\mE\rV]-\Tr[\mE\calM_{\beta_{\dim}}(\mSigma)]\right)=O_\P(1)$.
    \item If $\|\mE\|\ll1$ or $\sigma^2_{\mSigma}(\mE, \beta_\dim) \rightarrow 0,$ then:
    \begin{align*}
        \sqrt{\dim}(\Tr[\rV^\top\mE\rV]-\Tr[\mE\calM_{\beta_{\dim}}(\mSigma)])\pc 0. 
    \end{align*}
\item If $\sigma^2_{\mSigma}(\mE, \beta_\dim) \rightarrow v\in (0,\infty),$ then:
\begin{align*}
    \sqrt{\dim}(\Tr[\rV^\top\mE\rV]-\Tr[\mE\calM_{\beta_{\dim}}(\mSigma)])  
    \dc\gauss{0}{\frac{4v}{\beta^2}}.
\end{align*}
\end{enumerate}
\end{proposition} 
% In light of \Cref{prop:Lnull-dc}, we will decompose the log-likelihood ratio as:
% \begin{align}\label{eq:LLR-prob-decomp}
%      \ln \frac{\diff\tilde{\nu}_{\dim}}{\diff\nu_{\dim}}(\mV) \explain{\eqref{eq:LLR-prob}}{=}\frac{\sqrt{\dim} \beta_\dim }{2} \mV^\top \mE \mV+c_{\dim}= L_p(\mV) +  c_p^\prime \quad\forall\: \mV\in \R^{\dim},
% \end{align}
% where $L_\dim(\mV)$ and the constant $c_\dim^\prime$ are defined as:
% \begin{align*}
%   L_\dim(\mV) \bydef \frac{\sqrt{\dim}\beta_{\dim}}{2}\left(\mV^\top\mE\mV-\Tr[\mE M_{\beta_{\dim}}(\mSigma)]\right)\quad\text{and}\quad c_\dim^\prime &= c_\dim + \frac{\sqrt{\dim}\beta_{\dim}}{2}\Tr[\mE M_{\beta_{\dim}}(\mSigma)].\notag
% \end{align*}

\paragraph{A Perturbation Result} Note that \Cref{prop:Lnull-dc} requires us to appropriately center the statistic $\Tr[\rV^\top\mE\rV]$, where the centering term $\Tr[\mE\calM_{\beta_{\dim}}(\mSigma)]$ depends on the spectral properties of the matrix $\mSigma$. To handle the slight discrepancy in centering when analyzing the asymptotic distribution of $\Tr[\rV^\top\mE\rV]$ under $\rV \sim \nu_{\dim} \bydef \nu( \cdot \mid \mSigma, \beta_\dim, \rnk)$ versus $\rV \sim \tilde{\nu}_{\dim} \bydef \nu( \cdot \mid \tilde{\mSigma}, \beta_\dim, \rnk)$, we will use the following perturbation bound.
\begin{lemma}\label{lem:perturb} Consider a sequence of matrices $\mSigma\in\R^{\dim \times \dim}$ which satisfies Assumption \mref{assump:mat} and a sequence of noise parameters $\beta_{\dim}$ such that $\beta_{\dim} \rightarrow \beta \in (H_\mu(\gamma_\rnk),\infty)$ as $\dim \rightarrow \infty$. Let $\tilde{\mSigma} = \mSigma + \tfrac{1}{\sqrt{\dim}} \mE $ for any symmetric matrix $\mE \in \R^{\dim \times \dim}$ satisfying $\|\mE\|_{\fr} \lesssim 1.$ Then:
\begin{align*}
    \|\calM_{\beta_{\dim}}(\mSigma)-\calM_{\beta_{\dim}}(\tilde{\mSigma})\|_{\fr} \lesssim \frac{1}{\sqrt{\dim}},
\end{align*}
where $\calM_{\beta_{\dim}}(\cdot)$ is as defined in \eqref{eq:M-def}.
\end{lemma}

We now have all of the ingredients to complete the proof of Theorem \mref{thm:contiguity}. We postpone the proofs of the intermediate results introduced above to the end of this section and present the proof of Theorem \mref{thm:contiguity}, taking these intermediate results for granted. 

\begin{proof}[Proof of Theorem \mref{thm:contiguity}] We begin by decomposing the log-likelihood ratio between $\nu_\dim$ and $\tilde{\nu}_\dim$ (computed in \eqref{eq:LLR-Z}) as follows:
\begin{align} \label{eq:log-LLR-decomp}
     \ln \frac{\diff\tilde{\nu}_{\dim}}{\diff\nu_{\dim}}(\mV) = L_p(\mV) +  c_p \quad\forall\: \mV\in  \O(\dim,\rnk),
\end{align}
where 
\begin{align*}
  L_\dim(\mV) \bydef \frac{\sqrt{\dim}\beta_{\dim}}{2}(\Tr[\mV^\top\mE\mV]-\Tr[\mE\calM_{\beta_{\dim}}(\mSigma)]),
\end{align*}
and $c_\dim$ is a non-random constant (the exact formula for $c_\dim$ will not be important). We prove each claim of Theorem \mref{thm:contiguity} one by one. 
\paragraph{Mutual Contiguity of $\nu_\dim$, $\tilde{\nu}_\dim$ (Claim (1) in Theorem \mref{thm:contiguity})} We first prove that the Gibbs distributions $\nu_\dim, \tilde{\nu}_\dim$ are mutually contiguous. Using the contiguity criterion from \Cref{fact:Mukherjee} (Item (1)), it suffices to show that: \begin{align} \label{eq:log-LLR-O(1)}
    L_{\dim}(\rV)=O_{\P}(1) \quad \text{when} \quad \rV\sim\nu_{\dim} \quad \text{and when} \quad  \rV\sim\tilde{\nu}_{\dim}.
\end{align}
Note that \Cref{prop:Lnull-dc} guarantees that $L_{\dim}(\rV) = O_\P(1)$. On the other hand, when $\rV \sim \tilde{\nu}_{\dim}$, we can bound $L_\dim(\rV)$ as:
\begin{align*}
    |L_\dim(\rV)| &= \frac{\sqrt{\dim}\beta_{\dim}}{2}|\Tr[\rV^\top\mE\rV]-\Tr[\mE\calM_{\beta_{\dim}}(\mSigma)]|\\
    &\leq \underbrace{\frac{\sqrt{\dim}\beta_{\dim}}{2}|\Tr[\rV^\top\mE\rV]-\Tr[\mE\calM_{\beta_{\dim}}(\tilde{\mSigma})]|}_{(i)}+\underbrace{\frac{\sqrt{\dim}\beta_{\dim}}{2}|
    \Tr[\mE\calM_{\beta_{\dim}}(\tilde{\mSigma})]-\Tr[\mE\calM_{\beta_{\dim}}(\mSigma)]|}_{(ii)}.
\end{align*}
For the term $(i)$, we apply\footnote{Specifically, we apply \Cref{prop:Lnull-dc} with $\tilde{\mSigma}$ playing the role of the matrix $\mSigma$ in \Cref{prop:Lnull-dc}. Since $\|\tilde{\mSigma}-\mSigma\|\bydef\|\mE\|/\sqrt{\dim}\leq\|\mE\|_{\fr}/\sqrt{\dim}=o(1)$, $\tilde{\mSigma}$ also satisfies Assumption \mref{assump:mat} (see \Cref{lem:misc_conv}).} \Cref{prop:Lnull-dc} to conclude that this term is  $O_\P(1)$. Additionally, $(ii)\lesssim 1$ since
\begin{align*}
    |\Tr[\mE\calM_{\beta_{\dim}}(\tilde{\mSigma})]-\Tr[\mE\calM_{\beta_{\dim}}(\mSigma)]|\leq\|\mE\|_{\fr}\|\calM_{\beta_{\dim}}(\tilde{\mSigma})-\calM_{\beta_{\dim}}(\mSigma)\|_{\fr}\explain{\Cref{lem:perturb}}{\lesssim} \frac{1}{\sqrt{\dim}}.
\end{align*}
Hence, $L_{\dim}(\rV)  = O_\P(1)$ when $\rV \sim \tilde{\nu}_{\dim}$. This proves the first claim made in Theorem \mref{thm:contiguity}. To prove the remaining claims of Theorem \mref{thm:contiguity}, we split our analysis into two cases.   
\paragraph{Case 1: $\|\mE\|\ll1$ or $\sigma^2_{\mSigma}(\mE, \beta_\dim) \rightarrow 0$ (Claim (2) in Theorem \mref{thm:contiguity})} We first consider the case when $\|\mE\|\ll1$ or $\sigma^2_{\mSigma}(\mE, \beta_\dim) \rightarrow 0,$ and show: 
\begin{subequations}
    \begin{align}\label{eq:thm-contiguity-1a-toshow}
    \lim_{\dim \rightarrow \infty} \tv(\nu_{\dim}, \tilde{\nu}_{\dim}) &= 0, \\ \lim_{\dim \rightarrow \infty} \rdv{\alpha}{\tilde{\nu}_\dim}{\nu_\dim} &= 0 \quad \forall \; \alpha > 1. \label{eq:thm-contiguity-1b-toshow}
\end{align}
\end{subequations}
To prove \eqref{eq:thm-contiguity-1a-toshow}, we appeal to \Cref{fact:Mukherjee} (item (3)):
\begin{itemize}
    \item In \eqref{eq:log-LLR-decomp}, we decomposed the log-likelihood ratio between $\tilde{\nu}_\dim, \nu_\dim$ as $L_\dim(\mV) + c_\dim$ for any $\mV\in\O(\dim,\rnk).$
    \item We have already shown in \eqref{eq:log-LLR-O(1)} that $L_\dim(\rV) = O_\P(1)$ when $\rV \sim \nu_\dim$ and when $\rV \sim \tilde{\nu}_\dim$.
    \item Since $\|\mE\|\ll1$ or $\sigma^2_{\mSigma}(\mE, \beta_\dim) \rightarrow 0$, \Cref{prop:Lnull-dc} guarantees that $L_\dim(\rV) \bydef \sqrt{\dim}\beta_{\dim}(\Tr[\mV^\top\mE\mV]-\Tr[\mE\calM_{\beta_{\dim}}(\mSigma)])/2 \pc 0$ when $\rV \sim \nu_\dim$.
\end{itemize}
Since all the requirements of \Cref{fact:Mukherjee} (item (3)) are met, we conclude that:
\begin{align} \label{eq:mukherjee-fact-concl}
    \ln\frac{\diff\tilde{\nu}_{\dim}}{\diff\nu_{\dim}}(\rV)\pc 0 \quad\text{when}\quad \rV\sim\nu_{\dim},  \quad  \tv(\nu_{\dim}, \tilde{\nu}_{\dim}) \rightarrow 0.
\end{align}
This proves \eqref{eq:thm-contiguity-1a-toshow}. To show \eqref{eq:thm-contiguity-1b-toshow}, we recall the formula for $\rdv{\alpha}{\tilde{\nu}_\dim}{\nu_\dim}$: 
\begin{align} \label{eq:rdv-formula-recall}
   \rdv{\alpha}{\tilde{\nu}_{\dim}}{\nu_{\dim}}\bydef \frac{1}{\alpha - 1} \ln \E_{\rV \sim \nu_\dim}\left[ \left( \frac{\diff \tilde{\nu}_\dim}{\diff \nu_{\dim}}(\rV) \right)^\alpha \right].
\end{align}
The key idea to compute $\rdv{\alpha}{\tilde{\nu}_{\dim}}{\nu_{\dim}}$ will be to observe that the random variable $({\diff \tilde{\nu}_\dim}/{\diff \nu_{\dim}})(\rV)^\alpha$ is proportional to the likelihood ratio between $\nu_\dim$ and another Gibbs measure $\bar{\nu}_\dim$:
\begin{align*}
    \bar{\nu}_{\dim}\bydef\nu\left(\cdot\mid \mSigma+\frac{\alpha\mE}{\sqrt{\dim}},\beta_{\dim},\rnk\right).
\end{align*}
Indeed,
\begin{align} \label{eq:LR-moment}
    \left( \frac{\diff \tilde{\nu}_\dim}{\diff \nu_{\dim}}(\mV) \right)^\alpha & = \frac{Z({\mSigma}, \beta_\dim,\rnk)^\alpha}{Z(\tilde{\mSigma}, \beta_\dim,\rnk)^\alpha} \cdot \exp\left( \frac{\alpha \beta_\dim \sqrt{\dim}}{2} \Tr[\mV^\top \mE \mV] \right) = \bar{c}_\dim \cdot \frac{\diff\bar{\nu}_{\dim}}{\diff\nu_{\dim}}(\mV) \quad \forall \; \mV \; \in \; \O(\dim,\rnk),
\end{align}
where $\bar{c}_\dim$ is a non-random constant which absorbs the normalizing constants of $\nu_\dim, \tilde{\nu}_\dim, \bar{\nu}_\dim$. Using the above equation, we can compute the Rényi divergence in \eqref{eq:rdv-formula-recall}:
\begin{align} \label{eq:rdv-formula-cp}
    \rdv{\alpha}{\tilde{\nu}_{\dim}}{\nu_{\dim}} & = \frac{1}{\alpha-1} \ln \E_{\rV \sim \nu_\dim}\left[ \bar{c}_\dim \cdot \frac{\diff\bar{\nu}_{\dim}}{\diff\nu_{\dim}}(\rV) \right] = \frac{\ln \bar{c}_\dim}{\alpha -1 }
\end{align}
Now, we only need to compute $\lim_{\dim \rightarrow \infty} \ln \bar{c}_\dim$, which we can do by rearranging \eqref{eq:LR-moment}:
\begin{align} \label{eq:cp-formula}
    \ln \bar{c}_\dim & = \alpha \ln  \frac{\diff \tilde{\nu}_\dim}{\diff \nu_{\dim}}(\mV) -  \ln  \frac{\diff\bar{\nu}_{\dim}}{\diff\nu_{\dim}}(\mV) \quad \forall \; \mV \; \in \; \O(\dim,\rnk).
\end{align}
Note that:
\begin{itemize}
    \item \eqref{eq:mukherjee-fact-concl} guarantees that
    \begin{align} \label{eq:cp-formula-t1}
        \alpha \ln  \frac{\diff \tilde{\nu}_\dim}{\diff \nu_{\dim}}(\rV) \pc 0  \quad \text{when}\quad \rV \sim \nu_\dim.
    \end{align}
    \item The Gibbs measure $\bar{\nu}_\dim = \nu(\cdot \mid \mSigma + \tfrac{\alpha \mE}{\sqrt{\dim}}, \beta_\dim, \rnk)$ has the same form as the Gibbs measure $\tilde{\nu}_\dim = \nu(\cdot | \mSigma + \tfrac{ \mE}{\sqrt{\dim}}, \beta_\dim, \rnk)$ (the only change being replacement of $\mE$ by $\alpha \mE$). Hence, by repeating the arguments used in \eqref{eq:mukherjee-fact-concl} (or by appealing to the result of \eqref{eq:mukherjee-fact-concl} after replacing $\tilde{\nu}_\dim$ by $\bar{\nu}_\dim$), we conclude that:
    \begin{align} \label{eq:cp-formula-t2}
        \ln  \frac{\diff \bar{\nu}_\dim}{\diff \nu_{\dim}}(\rV) \pc 0  \quad \text{when}\quad \rV \sim {\nu}_\dim.
    \end{align}
\end{itemize}
Combining \eqref{eq:cp-formula}, \eqref{eq:cp-formula-t1}, and \eqref{eq:cp-formula-t2} using Slutsky's theorem, we conclude that:
\begin{align*}
    \ln \bar{c}_\dim \rightarrow 0 \implies   \rdv{\alpha}{\tilde{\nu}_{\dim}}{\nu_{\dim}}  \explain{\eqref{eq:rdv-formula-cp}}{=}  \frac{\ln \bar{c}_\dim}{\alpha -1 } \rightarrow  0.
\end{align*}
This completes the proof of \eqref{eq:thm-contiguity-1b-toshow} and the proof of the second claim in Theorem \mref{thm:contiguity}. 
\paragraph{Case 2: $\sigma^2_{\mSigma}(\mE, \beta_\dim) \rightarrow v \in (0,\infty)$ (Claim (3) in Theorem \mref{thm:contiguity})} Next, we consider the case when $\sigma_{\mSigma}^2(\mE,\beta_{\dim})\rightarrow v$ for some $v\in(0,\infty),$ and  show that:
\begin{subequations}
\begin{align}
    \lim_{\dim \rightarrow \infty} \tf{\nu_{\dim}}{\tilde{\nu}_{\dim}}(\alpha) &=  \tf{\gauss{0}{1}}{\gauss{\sqrt{v}}{1}}(\alpha), \label{eq:thm-contiguity-tf} \\
    \lim_{\dim \rightarrow \infty} \rdv{\alpha}{\tilde{\nu}_{\dim}}{\nu_{\dim}} &=  \rdv{\alpha}{\gauss{\sqrt{v}}{1}}{\gauss{0}{1}} \quad \forall \; \alpha >1. \label{eq:thm-contiguity-rdv}
\end{align}
\end{subequations}
To prove \eqref{eq:thm-contiguity-tf}, we appeal to \Cref{prop:le-cam}. This requires us to first understand the limit distribution of the log-likelihood ratio between $\nu_\dim, \tilde{\nu}_\dim$ which we can derive using \Cref{fact:Mukherjee} (item 2):
\begin{itemize}
    \item In \eqref{eq:log-LLR-decomp}, we decomposed the log-likelihood ratio between $\tilde{\nu}_\dim, \nu_\dim$ as $L_\dim(\rV) + c_\dim$.
    \item We have already shown in \eqref{eq:log-LLR-O(1)} that $L_\dim(\rV) = O_\P(1)$ when $\rV \sim \nu_\dim$ and when $\rV \sim \tilde{\nu}_\dim$.
    \item Since $\sigma^2_{\mSigma}(\mE, \beta_\dim) \rightarrow v$, \Cref{prop:Lnull-dc} guarantees:
    \begin{align*}
        L_\dim(\rV) \bydef \sqrt{\dim}\beta_{\dim}(\Tr[\rV^\top\mE\rV]-\Tr[\mE\calM_{\beta_{\dim}}(\mSigma)])/2 \dc \gauss{0}{v}\quad\text{when}\quad\rV \sim \nu_\dim.
    \end{align*}
\end{itemize}
Since all the requirements of \Cref{fact:Mukherjee} (item (2)) are met, we conclude that:
\begin{align} \label{eq:mukherjee-fact-concl-2}
   \ln\frac{\diff\tilde{\nu}_{\dim}}{\diff\nu_{\dim}}(\rV)\dc \gauss{-\frac{v}{2}}{v} \quad\text{when}\quad  \rV\sim\nu_{\dim}.
\end{align}
Combining \eqref{eq:mukherjee-fact-concl-2} with \Cref{prop:le-cam} immediately gives the claim \eqref{eq:thm-contiguity-tf}. To show \eqref{eq:thm-contiguity-rdv}, we recall the formula for $\rdv{\alpha}{\tilde{\nu}_\dim}{\nu_\dim}$: 
\begin{align} \label{eq:rdv-formula-recall-2}
   \rdv{\alpha}{\tilde{\nu}_{\dim}}{\nu_{\dim}}\bydef \frac{1}{\alpha - 1} \ln \E_{\rV \sim \nu_\dim}\left[ \left( \frac{\diff \tilde{\nu}_\dim}{\diff \nu_{\dim}}(\rV) \right)^\alpha \right].
\end{align}
As before, the main idea to compute $\rdv{\alpha}{\tilde{\nu}_{\dim}}{\nu_{\dim}}$ will be to observe that the random variable $({\diff \tilde{\nu}_\dim}/{\diff \nu_{\dim}})^\alpha(\rV)$ is proportional to the likelihood ratio between $\nu_\dim$ and the Gibbs measure $\bar{\nu}_\dim \bydef \nu(\cdot\mid \mSigma+\tfrac{\alpha\mE}{\sqrt{\dim}},\beta_{\dim},\rnk)$:
\begin{align} \label{eq:LR-moment-2}
    \left( \frac{\diff \tilde{\nu}_\dim}{\diff \nu_{\dim}}(\mV) \right)^\alpha & =  \bar{c}_\dim \cdot \frac{\diff\bar{\nu}_{\dim}}{\diff\nu_{\dim}}(\mV) \quad \forall \; \mV \; \in \; \O(\dim,\rnk),
\end{align}
where $\bar{c}_\dim$ is a non-random constant which absorbs the normalizing constants of $\nu_\dim, \tilde{\nu}_\dim, \bar{\nu}_\dim$. Using the above equation, we can compute the Rényi divergence in \eqref{eq:rdv-formula-recall-2}:
\begin{align} \label{eq:rdv-formula-cp-2}
    \rdv{\alpha}{\tilde{\nu}_{\dim}}{\nu_{\dim}} & = \frac{1}{\alpha-1} \ln \E_{\rV \sim \nu_\dim}\left[ \bar{c}_\dim \cdot \frac{\diff\bar{\nu}_{\dim}}{\diff\nu_{\dim}}(\rV) \right] = \frac{\ln \bar{c}_\dim}{\alpha -1 }.
\end{align}
Now, we only need to compute $\lim_{\dim \rightarrow \infty} \ln \bar{c}_\dim$, which we can do by rearranging \eqref{eq:LR-moment-2}:
\begin{align} \label{eq:cp-formula-2}
    \ln \bar{c}_\dim & = \alpha \ln  \frac{\diff \tilde{\nu}_\dim}{\diff \nu_{\dim}}(\mV) -  \ln  \frac{\diff\bar{\nu}_{\dim}}{\diff\nu_{\dim}}(\mV) \quad \forall \; \mV \; \in \; \O(\dim,\rnk).
\end{align}
Note that:
\begin{itemize}
    \item \eqref{eq:mukherjee-fact-concl-2} guarantees that
    \begin{align} \label{eq:cp-formula-t1-2}
        \alpha \ln  \frac{\diff \tilde{\nu}_\dim}{\diff \nu_{\dim}}(\rV) \dc  \gauss{-\frac{\alpha v}{2}}{\alpha^2 v} \quad\text{when}\quad  \rV \sim \nu_\dim.
    \end{align}
    \item The Gibbs measure $\bar{\nu}_\dim = \nu(\cdot | \mSigma + \tfrac{\alpha \mE}{\sqrt{\dim}}, \beta_\dim, \rnk)$ has the same form as the Gibbs measure $\tilde{\nu}_\dim = \nu(\cdot | \mSigma + \tfrac{ \mE}{\sqrt{\dim}}, \beta_\dim, \rnk)$ (the only change being replacement of $\mE$ by $\alpha \mE$). Hence, by repeating the arguments used in \eqref{eq:mukherjee-fact-concl-2} (or by appealing to the result of \eqref{eq:mukherjee-fact-concl-2} after replacing $\tilde{\nu}_\dim$ by $\bar{\nu}_\dim$), we conclude that:
    \begin{align} \label{eq:cp-formula-t2-2}
        \ln  \frac{\diff \bar{\nu}_\dim}{\diff \nu_{\dim}}(\rV) \dc \gauss{-\frac{\alpha^2 v}{2}}{\alpha^2 v}\quad\text{when}\quad  \rV \sim {\nu}_\dim
    \end{align}
\end{itemize}
Combining \eqref{eq:cp-formula-2}, \eqref{eq:cp-formula-t1-2}, and \eqref{eq:cp-formula-t2-2} using Slutsky's theorem (see \Cref{lem:Slutsky}), we conclude that:
\begin{align*}
    \ln \bar{c}_\dim \rightarrow \frac{\alpha(\alpha - 1)v}{2} \implies   \rdv{\alpha}{\tilde{\nu}_{\dim}}{\nu_{\dim}}  \explain{\eqref{eq:rdv-formula-cp-2}}{=}  \frac{\ln \bar{c}_\dim}{\alpha -1 } \rightarrow  \frac{\alpha v}{2} = \rdv{\alpha}{\gauss{\sqrt{v}}{1}}{\gauss{0}{1}}.
\end{align*}
This completes the proof of \eqref{eq:thm-contiguity-rdv} and the proof of the third and last claim in Theorem \mref{thm:contiguity}. 
\end{proof}

\subsection{Proof of \texorpdfstring{\Cref{prop:Lnull-dc}}{prop:Lnull-dc}} \label{app:Lnull-dc} 
Recall that \Cref{prop:Lnull-dc} analyzes the asymptotic properties of the random variable:
\begin{align*}
    \sqrt{\dim}(\Tr[\rV^\top\mE\rV]-\Tr[\mE\calM_{\beta_{\dim}}(\mSigma)]), \quad \text{where} \quad  \rV \sim \nu(\cdot \mid \mSigma, \beta_\dim, \rnk). 
\end{align*}
Let $\hat{\nu}(\cdot\mid\mSigma,\beta_{\dim},\rnk)$ denote the output distribution of the sampler (Algorithm \mref{alg:sampler}) run on $(\mSigma,\beta_{\dim},\rnk)$.  Since Theorem \mref{thm:sampling} guarantees that the total variation distance between $\nu(\cdot\mid\mSigma,\beta_{\dim},\rnk)$ and $\hat{\nu}(\cdot\mid\mSigma,\beta_{\dim},\rnk)$ tends to 0 as $\dim\rightarrow\infty$, it is enough to prove the claims of \Cref{prop:Lnull-dc} for the random variable:
\begin{align} \label{eq:key-stat}
     \sqrt{\dim}(\Tr[\rV^\top\mE\rV]-\Tr[\mE\calM_{\beta_{\dim}}(\mSigma)]), \quad \text{where} \quad  \rV \sim \hat{\nu}(\cdot \mid \mSigma, \beta_\dim, \rnk). 
\end{align}

\paragraph{Some Intermediate Results} The proof of \Cref{prop:Lnull-dc} relies on two intermediate results. Recall the sampling algorithm (Algorithm \mref{alg:sampler}) constructs $\rV \sim  \hat{\nu}(\cdot \mid \mSigma, \beta_\dim, \rnk)$ as:
\begin{subequations}\label{eq:sampler-recall}
\begin{align}
    \rV\explain{}{=}\mU \begin{bmatrix} (I_k - \rZ^\top \rZ)_+^{1/2} \\ \rZ \end{bmatrix}  \rQ,\quad\text{where }\quad\rQ\sim\xi_{\rnk,\rnk}
\end{align}
and 
\begin{align}
    \rZ_{ij}\overset{\indep}{\sim}\gauss{0}{\frac{1}{\dim\beta_{\dim}(\lambda_j-\lambda_{\rnk+i})}},\; i\in[\dim-\rnk],\;j\in[\rnk],
\end{align}
\end{subequations}
where $\mU \in \O(\dim)$ denotes the matrix of eigenvectors of $\mSigma$ and $\lambda_{1:\dim}$ are the corresponding eigenvalues. Plugging \eqref{eq:sampler-recall} into \eqref{eq:key-stat}, together with some simple asymptotic estimates, gives us the following simplified asymptotic approximation for the statistic 
\begin{align*}
    \sqrt{\dim}(\Tr[\rV^\top\mE\rV]-\Tr[\mE\calM_{\beta_{\dim}}(\mSigma)]).
\end{align*}
\begin{lemma}\label{lem:Lnull-dc-w}
When $\rV\sim \hat{\nu}(\cdot\mid\mSigma,\beta_{\dim},\rnk)$, $\sqrt{\dim}\left(\Tr[\rV^\top\mE\rV]-\Tr[\mE\calM_{\beta_{\dim}}(\mSigma)]\right)=\rw + o_\P(1)$, where the random variable $\rw$ is defined as:
\begin{align*}
    \rw\bydef \sqrt{\dim}(\Tr[\mU_\star^\top\mE\mU_\star(\E[\rZ^\top\rZ]-\rZ^\top\rZ)]+2\Tr[\mU_\perp^\top\mE\mU_\star(\mI_\rnk-\E[\rZ^\top\rZ])_+^\frac{1}{2}\rZ^\top]).
\end{align*}
In the above equation, $\mU_\star \in \R^{\dim \times \rnk}$ is the matrix of the leading $\rnk$ eigenvectors of $\mSigma$ and $\mU_\perp \in\R^{\dim \times (\dim - \rnk)}$ is the matrix of the remaining $\dim - \rnk$ eigenvectors. 
\end{lemma}
We will also find it useful to compute the variance of the random variable $\rw$ introduced above.
\begin{lemma}\label{lem:Lnull-dc-var}
    For sufficiently large $\dim$, $\Var[\rw] = \tfrac{4 \sigma_{\mSigma}^2(\mE,\beta_{\dim})}{\beta_{\dim}^2}$, where the variance function  $\sigma_{\mSigma}^2(\mE,\beta_{\dim})$ is as defined in the statement of \Cref{prop:Lnull-dc}. %$\footnote{Note that $\sigma_{\mSigma}^2$ is non-negative for large enough $\dim$ since $\beta_\dim\rightarrow\beta>H_{\mu}(\gamma_\rnk)$ and $H_{\mSigma}(\lambda_j)\rightarrow H_\mu(\gamma_j)$ for each $j\in[\rnk]$ (see \Cref{lem:misc_HK}).}
\end{lemma}
We postpone the derivations of the two intermediate lemmas introduced above to the end of this section and present the proof of \Cref{prop:Lnull-dc}, taking these intermediate results for granted. 
\begin{proof}[Proof of \Cref{prop:Lnull-dc}] We prove each of the three claims made in \Cref{prop:Lnull-dc} one by one.
\paragraph{Proof of Claim (1) in \Cref{prop:Lnull-dc}}  We first prove that:
\begin{subequations}
\begin{align} 
    \sigma^2_{\mSigma}(\mE, \beta_\dim)&\lesssim 1, \label{eq:var-func-O1} \\ \quad\sqrt{\dim}\left(\Tr[\rV^\top\mE\rV]-\Tr[\mE\calM_{\beta_{\dim}}(\mSigma)]\right)&=O_\P(1). \label{eq:key-stat-O1}
\end{align}
\end{subequations}
To prove \eqref{eq:var-func-O1}, we recall that the variance function $ \sigma^2_{\mSigma}(\mE, \beta_\dim)$ was defined as:
\begin{align*}
    \sigma^2_{\mSigma}(\mE, \beta_\dim)&= \frac{1}{2} \sum_{j, \ell = 1}^\rnk K_\mSigma(\lambda_j, \lambda_\ell) \cdot (\vu_j^\top E \vu_\ell)^2  + \sum_{j=1}^\rnk \sum_{i= 1}^{\dim - \rnk} \frac{\beta_{\dim} - {H_{\mSigma}(\lambda_j)}}{\lambda_j - \lambda_{\rnk+i}}  \cdot (\vu_{\rnk+i}^\top E \vu_j)^2,
\end{align*}
where $\{(\lambda_i, \vu_i): i \in [\dim]\}$ denote the eigenvalues and the corresponding eigenvectors of $\mSigma,$ and the functions $H_\mSigma, K_\mSigma$ were defined as: 
\begin{align*}
    H_{\mSigma}(\lambda) &\explain{def}{=} \frac{1}{\dim} \sum_{i=1}^{\dim - \rnk} \frac{1}{\lambda - \lambda_{\rnk+i}}, \quad   K_{\mSigma}(\lambda,\lambda^\prime) \explain{def}{=} \frac{1}{\dim} \sum_{i=1}^{\dim-\rnk} \frac{1}{(\lambda - \lambda_{\rnk+i})(\lambda^\prime - \lambda_{\rnk+i})} \quad \forall \; \lambda, \lambda^\prime \; \in \; (\lambda_{\rnk+1}, \infty).
\end{align*}
Hence,
\begin{align}
    \sigma^2_{\mSigma}(\mE, \beta_\dim)&\leq \left(\frac{1}{(\lambda_{\rnk}-\lambda_{\rnk+1})^2} + \frac{\beta_\dim}{\lambda_{\rnk}-\lambda_{\rnk+1}} \right) \cdot \sum_{j=1}^\rnk \sum_{i=1}^{\dim} \ip{\vu_i}{\mE \vu_j}^2 \nonumber \\
    &\lesssim \sum_{j=1}^\rnk \sum_{i=1}^{\dim} \ip{\vu_i}{\mE \vu_j}^2\quad\text{[since $\mSigma$ satisfies Assumption \mref{assump:mat} and $\beta_\dim \rightarrow \beta$]}\nonumber\\
    &= \sum_{j=1}^\rnk \|\mE\vu_j\|^2 \quad\text{[since $\vu_{1}, \dotsc, \vu_\dim$ form an orthonormal basis]}\nonumber\\
    & \lesssim \|\mE\|^2  \label{eq:var-E} \\
    & \lesssim 1 \quad \text{[since \Cref{prop:Lnull-dc} assumes $\|\mE\|_{\fr} \lesssim 1$].} \nonumber
\end{align}
This verifies \eqref{eq:var-func-O1}. To prove \eqref{eq:key-stat-O1}, we first observe that \Cref{lem:Lnull-dc-w} guarantees that $\sqrt{\dim}(\Tr[\rV^\top\mE\rV]-\Tr[\mE\calM_{\beta_{\dim}}(\mSigma)]) = O_\P(1)$ if $\rw = O_\P(1)$. On the other hand, from \Cref{lem:Lnull-dc-var}, we know that $$\E[\rw^2] = \Var[\rw] = \frac{4 \sigma_{\mSigma}^2(\mE,\beta_{\dim})}{\beta_{\dim}^2} \explain{\eqref{eq:var-func-O1}}{\lesssim} 1.$$ Hence, by Chebyshev's inequality, $\rw = O_\P(1)$, which proves \eqref{eq:key-stat-O1} and completes the proof of the first claim of \Cref{prop:Lnull-dc}.  
\paragraph{Proof of Claim (2) in \Cref{prop:Lnull-dc}} Next, we consider the situation when $\|\mE\|\ll1$ or $\sigma^2_{\mSigma}(\mE, \beta_\dim) \rightarrow 0,$ and prove the second claim of \Cref{prop:Lnull-dc}:
    \begin{align}
        \sqrt{\dim}(\Tr[\rV^\top\mE\rV]-\Tr[\mE\calM_{\beta_{\dim}}(\mSigma)])  \pc 0. 
    \end{align}
Since $\sqrt{\dim}\left(\Tr[\rV^\top\mE\rV]-\Tr[\mE\calM_{\beta_{\dim}}(\mSigma)]\right) = \rw + o_\P(1)$ by \Cref{lem:Lnull-dc-w}, it is enough to show:
\begin{align} \label{eq:Lnull-dc-toshow-w-1}
    \rw & = o_\P(1).
\end{align}
To show \eqref{eq:Lnull-dc-toshow-w-1}, we observe that \eqref{eq:var-E} implies that if $\|\mE\|\ll1,$ then $\sigma_{\mSigma}^2(\mE,\beta_{\dim}) \rightarrow 0$ as $\dim \rightarrow \infty.$ It is therefore enough to show \eqref{eq:Lnull-dc-toshow-w-1} assuming that $\sigma_{\mSigma}^2(\mE,\beta_{\dim}) \rightarrow 0.$ In this situation \eqref{eq:Lnull-dc-toshow-w-1} follows by Chebyshev's inequality. Indeed,
for any $\epsilon>0,$ 
\begin{align*}
    \lm \P\left(|\rw|>\epsilon\right)&\leq \lm \frac{4\cdot\sigma^2_{\mSigma}(\mE,\beta_{\dim})}{\beta^2_{\dim}\epsilon^2}&&\text{[by Chebyshev's inequality and \Cref{lem:Lnull-dc-var}]}\\
    &=0&&\text{[since $\sigma^2_{\mSigma}(\mE,\beta_{\dim})\rightarrow0$ and $\beta_{\dim}\rightarrow\beta$]}.
\end{align*}
\paragraph{Proof of Claim (3) in \Cref{prop:Lnull-dc}} Finally, we consider the situation when $\sigma^2_{\mSigma}(\mE, \beta_\dim) \rightarrow v \in (0,\infty)$,  and prove the last claim made in \Cref{prop:Lnull-dc}:
\begin{align}
    \sqrt{\dim}(\Tr[\rV^\top\mE\rV]-\Tr[\mE\calM_{\beta_{\dim}}(\mSigma)])  
    \dc\gauss{0}{\frac{4v}{\beta^2}}.
\end{align}
Again thanks to \Cref{lem:Lnull-dc-w}, it is sufficient to show:
\begin{align}\label{eq:Lnull-dc-toshow-w-2}
    \rw  
    \dc\gauss{0}{\frac{4v}{\beta^2}}.
\end{align}
At a high-level, our strategy would be to view $\rw$ as a non-linear function of the Gaussian vectors $\rz_1, \dotsc, \rz_{\dim - \rnk} \in \R^\rnk$, which form the rows of the matrix $\rZ$ from \eqref{eq:sampler-recall}, and apply Chatterjee's second-order Poincaré inequality \citep{chatterjee2009fluctuations} to obtain \eqref{eq:Lnull-dc-toshow-w-2}. In the following fact, we recall the statement of this result for the reader's convenience. 
\begin{fact}[{Second-Order Poincaré Inequality \citep[Theorem 2.2]{chatterjee2009fluctuations}}] \label{fact:chatterjee} 
    Let $\rz$ be a  Gaussian random vector with $\E[\rz] = 0$ and consider any twice continuously differentiable function $g$ such that $\E[g(\rz)] = 0$ and $\E[g(\rz)^4]<\infty.$ Then, 
    \begin{align}\label{eq:fact:chatterjee-ub}
        \tv\left(g(\rz),\gauss{0}{\mathrm{Var}[g(\rz)]}\right)\leq  2\sqrt{5} \cdot \frac{\|\Cov[\rz]\|^{\frac{3}{2}}\left(\E\left[\|\nabla g(\rz)\|^4\right]\right)^{\frac{1}{4}}\left(\E\left[\|\nabla^2 g(\rz)\|^4\right]\right)^{\frac{1}{4}}}{\Var[g(\rz)]}.
    \end{align}
\end{fact}
To apply \Cref{fact:chatterjee} to our setting, we define the function $g$ as: 
\begin{align} \label{eq:g-def}
    g(\vz_1,...,\vz_{\dim-\rnk})=\sqrt{\dim}\sum_{i=1}^{\dim-\rnk}\left\{\Tr\left[\mU_\star^\top\mE\mU_\star\left(\E[\rz_i\rz_i^\top]-\vz_i\vz_i^\top\right)\right]+2\vu_{\rnk+i}^\top\mE\mU_\star\left(\mI_\rnk-\E[\rZ^\top\rZ]\right)_+^{\frac{1}{2}}\vz_i\right\}
\end{align}
so that $g(\rz_1,...,\rz_{\dim-\rnk})=\rw.$ Since \Cref{lem:Lnull-dc-var} guarantees that $\Var\left[\rw\right]\rightarrow 4v/\beta^2,$ \eqref{eq:Lnull-dc-toshow-w-2} follows if we show that 
\begin{align*}
    \lm\tv\left(\rw,\gauss{0}{\mathrm{Var}[\rw]}\right)=0.
\end{align*}
By the second-order Poincaré inequality,
\begin{align}
    &\tv\left(\rw,\gauss{0}{\mathrm{Var}[\rw]}\right)\\ &\quad\leq 2\sqrt{5} \cdot \frac{\|\Cov[\rz_{1:\dim-\rnk}]\|^{\frac{3}{2}}\left(\E\left[\|\nabla^2 g(\rz_{1:\dim-\rnk})\|^4\right]\right)^{\frac{1}{4}}\left(\E\left[\|\nabla g(\rz_{1:\dim-\rnk})\|^4\right]\right)^{\frac{1}{4}}}{\Var[g(\rz_{1:\dim-\rnk})]} \nonumber\\
    &\quad \explain{(a)}{\lesssim} \|\Cov[\rz_{1:\dim-\rnk}]\|^{\frac{3}{2}}\left(\E\left[\|\nabla^2 g(\rz_{1:\dim-\rnk})\|^4\right]\right)^{\frac{1}{4}}\left(\E\left[\|\nabla g(\rz_{1:\dim-\rnk})\|^4\right]\right)^{\frac{1}{4}} \label{eq:TV-plughere},
\end{align}
where (a) follows from \Cref{lem:Lnull-dc-var}, $\Var[g(\rz_{1:\dim-\rnk})] \rightarrow \tfrac{4v}{\beta^2}.$ We bound the covariance, Hessian, and gradient terms in the above expression one by one. 
\begin{itemize}
    \item First, we consider the covariance term. Recall from \eqref{eq:sampler-recall} that:
    \begin{align*}
       \|\Cov[\rz_{1:\dim-\rnk}] \|^{3/2}&=\left\| \frac{1}{\dim \beta_\dim} \diag\left( \frac{1}{\lambda_1 - \lambda_{\rnk+1}}, \dotsc, \frac{1}{\lambda_\rnk - \lambda_{\dim}} \right)\right\|^{3/2} \\ & \leq \left(\frac{1}{\dim \beta_\dim (\lambda_\rnk -\lambda_{\rnk+1})} \right)^{3/2} \\ &\explain{(a)}{\lesssim} \dim^{-3/2} \qquad \text{[since $\beta_\dim \rightarrow \beta$ and $\mSigma$ satisfies Assumption \mref{assump:mat}]}.
    \end{align*}
    \item Next, we bound the Hessian term by explicitly computing the Hessian of $g$ from \eqref{eq:g-def}:
    \begin{align*}
    \nabla^2g(\vz_1,...,\vz_{\dim-\rnk})=-2\sqrt{\dim}\mU_\star^\top\mE\mU_\star\otimes \mI_{\dim-\rnk},
\end{align*}
where $\otimes$ denotes the Kronecker product. Hence,
\begin{align*}
    \left(\E\left[\|\nabla^2 g(\rz_{1:\dim-\rnk})\|^4\right]\right)^{\frac{1}{4}}=2\sqrt{\dim}\|\mU_\star^\top\mE\mU_\star\|\| \mI_{\dim-\rnk}\|\leq 2\sqrt{\dim}\|\mE\|
    \lesssim \sqrt{\dim},
\end{align*}
where the last estimate follows by recalling the assumption that $\|\mE\|_{\fr} \lesssim 1$. 
\item Finally, we consider the gradient term. We first compute the gradient of $g$ from \eqref{eq:g-def}:
\begin{align*}
    \nabla_{z_i}g(\vz_1,...,\vz_{\dim-\rnk})=\sqrt{\dim}(-2\mU_\star^\top\mE\mU_\star\vz_i+2(\mI_\rnk-\E[\rZ^\top\rZ])_+^{\frac{1}{2}}\mU_\star^\top \mE\vu_{\rnk+i}).
\end{align*}
Hence,
\begin{align*}
   & \E\left[\|\nabla g(\rz_{1:\dim-\rnk})\|^4\right] \\
   &\quad=\E\left[\left(\sum_{i=1}^{\dim-\rnk}\|\nabla_{z_i}g(\rz_1,...,\rz_{\dim-\rnk})\|^2\right)^2\right]\\
    &\quad\lesssim \dim^2\E\left[\left(\sum_{i=1}^{\dim-\rnk}\left\{\|\mU_\star^\top\mE\mU_\star\rz_i\|^2+\|(\mI_\rnk-\E[\rZ^\top\rZ])_+^{\frac{1}{2}}\mU_\star^\top \mE\vu_{\rnk+i}\|^2\right\}\right)^2\right] \\ %\qquad \text{[since $\|a+b\|^2\leq 2(\|a\|^2+\|b\|^2)$]}\\
    &\quad\lesssim \dim^2\E\Bigg[\Bigg(\sum_{i=1}^{\dim-\rnk}\|\rz_i\|^2+\sum_{i=1}^{\dim-\rnk}\|\mE\vu_{\rnk+i}\|^2\Bigg)^2\Bigg]\qquad \text{[since $\|(I_\rnk-\E[\rZ^\top\rZ])_+^{\frac{1}{2}}\|, \|\mE\| \lesssim 1$]}\\
    &\quad\lesssim \dim^2{\E\Bigg[\left(\sum_{i=1}^{\dim-\rnk}\|\rz_i\|^2\right)^2\Bigg]}+\dim^2\Bigg({\sum_{i=1}^{\dim}\|\mE\vu_{i}\|^2}\Bigg)^2\qquad\text{[since $(a+b)^2\leq 2(a^2+b^2)$]} \\
     &\quad = \dim^2\E\Bigg[\left(\sum_{i=1}^{\dim-\rnk}\|\rz_i\|^2\right)^2\Bigg]+\dim^2 \|\mE\|^4_{\fr}\qquad\text{[since $\vu_{1:\dim}$ are orthonormal]}\\
     &\quad \lesssim \dim^2\E\Bigg[\left(\sum_{i=1}^{\dim-\rnk}\|\rz_i\|^2\right)^2\Bigg]+\dim^2.
\end{align*}
We can further simplify our estimate above:
\begin{align*}
    \E\left[\|\nabla g(\rz_{1:\dim-\rnk})\|^4\right] &\lesssim \dim^3\E\Bigg[\sum_{i=1}^{\dim-\rnk}\|\rz_i\|^4\Bigg] + \dim^2 \qquad \text{[since $\|\va\|_2^4 \leq \dim \|\va\|^4_4$ ]} \\
    & = \dim^3\sum_{i=1}^{\dim-\rnk}\left(\mathrm{Var}[\|\rz_i\|^2]+(\E[\|\rz_i\|^2])^2\right) +\dim^2 \\ &\explain{(a)}{=}\dim^3\sum_{i=1}^{\dim-\rnk}\left(2\Tr[\mathrm{Cov}[\rz_i]^2]+(\Tr[\mathrm{Cov}[\rz_i]])^2\right) + \dim^2  \\
    &\explain{\eqref{eq:sampler-recall}}{=}\frac{\dim}{\beta_{\dim}^2}\sum_{i=1}^{\dim-\rnk}\left(\sum_{j=1}^\rnk \frac{1}{(\lambda_j-\lambda_{\rnk+i})^2}+\left(\sum_{j=1}^\rnk \frac{1}{\lambda_j-\lambda_{\rnk+i}}\right)^2\right) + \dim^2 \\ &\lesssim \dim^2 \qquad \text{[since $\mSigma$ satisfies Assumption \mref{assump:mat}]},
\end{align*}
where (a) follows from the fact that if $\ry \sim \gauss{0}{\mS}$, then,  $\E[\|\ry\|^2]=\Tr[\mS]$ and $\Var[\|\ry\|^2]=2\Tr[\mS^2]$ (see \citet[Theorem 1 and Corollary 1.3]{searle1997linear}).
\end{itemize}
Plugging these estimates on the covariance, Hessian, and gradient terms into \eqref{eq:TV-plughere}, we conclude that:
\begin{align*}
     \tv\left(\rw,\gauss{0}{\mathrm{Var}[\rw]}\right)\lesssim \frac{1}{\sqrt{\dim}} &\explain{\Cref{lem:Lnull-dc-var}}{\implies} \rw  
    \dc\gauss{0}{\frac{4v}{\beta^2}} \\
    &\explain{\Cref{lem:Lnull-dc-w}}{\implies}\sqrt{\dim}\left(\Tr[\rV^\top\mE\rV]-\Tr[\mE\calM_{\beta_{\dim}}(\mSigma)]\right)  
    \dc\gauss{0}{\frac{4v}{\beta^2}}.
\end{align*}
This completes the proof of the last claim made in \Cref{prop:Lnull-dc}.
\end{proof}
We will conclude this section by presenting the proofs of \Cref{lem:Lnull-dc-w} and \Cref{lem:Lnull-dc-var}. 
% We will start by recalling the statement of \Cref{lem:Z^T*Z-sampler}, which is used often in the proofs of both lemmas: 
% \begin{align}\label{eq:lem-Z^T*Z-sampler-recall}
%     \rZ^\top \rZ &\pc \diag\left( \frac{H_\mu(\gamma_1)}{\beta}, \dotsc,  \frac{H_\mu(\gamma_\rnk)}{\beta}\right) \prec \mI_\rnk\quad\text{and}\quad \E[\rZ^\top \rZ]\rightarrow \diag\left( \frac{H_\mu(\gamma_1)}{\beta}, \dotsc,  \frac{H_\mu(\gamma_\rnk)}{\beta}\right).
% \end{align}
\subsubsection{Proof of \texorpdfstring{\Cref{lem:Lnull-dc-w}}{lem:Lnull-dc-w}}
\begin{proof}[Proof of \Cref{lem:Lnull-dc-w}.]
Recall that our sampling algorithm (Algorithm \mref{alg:sampler}) generates $\rV$ as follows:\begin{align}\label{eq:V-dist}
    \rV\explain{d}{=}\mU \begin{bmatrix} (I_k - \rZ^\top \rZ)_+^{1/2} \\ \rZ \end{bmatrix}  \rQ,
\end{align}
where $\rQ \sim \xi_{\rnk,\rnk}$, and the rows $\rz_1, \dotsc, \rz_{\dim - \rnk}$ of the matrix $\rZ$ are sampled independently with:
\begin{align} \label{eq:z-dist-recalled}
    \rz_i \overset{\indep}{\sim}\gauss{0}{\frac{1}{\dim\beta_{\dim}} \cdot \diag\left( \frac{1}{\lambda_1-\lambda_{\rnk+i}}, \dotsc, \frac{1}{\lambda_\rnk-\lambda_{\rnk+i}} \right)}\quad \forall \;  i\in[\dim-\rnk].
\end{align}
Let us partition the matrix of eigenvectors $\mU$ of $\mSigma$ as: $$\mU = \begin{bmatrix} \mU_\star & \mU_\perp \end{bmatrix},$$ where $\mU_\star \in \R^{\dim \times \rnk}$ is the matrix of the top $\rnk$ eigenvectors and $\mU_\perp \in \R^{\dim \times (\dim-\rnk)}$ is the matrix of the remaining $\dim - \rnk$ eigenvectors. Now, we can decompose the statistic $\Tr[\rV^\top \mE \rV]$ as:
\begin{align} \label{eq:TrVTEV}
    \Tr[\rV^\top \mE \rV]  = \Tr[\mU_\star^\top\mE\mU_\star(\mI_\rnk-\rZ^\top\rZ)_+] &+ 2\Tr[\mU_\perp^\top \mE\mU_\star\left(\mI_\rnk-\rZ^\top\rZ\right)_+^{\frac{1}{2}}\rZ^\top]\\
    &\hspace{3cm}+\Tr[\mU_\perp^\top\mE\mU_\perp\rZ\rZ^\top].\notag
\end{align}
We also recall the formula for the centering term $\Tr[\mE\calM_{\beta_{\dim}}(\mSigma)]$ and express it in the following convenient form:
\begin{align} \label{eq:centering-term}
    \Tr[\mE\calM_{\beta_{\dim}}(\mSigma)]&\bydef \sum_{j=1}^\rnk\left( 1-\frac{H_{\mSigma}(\lambda_j)}{\beta_\dim}\right) \cdot  \vu_j^\top \mE \vu_j + \frac{1}{\dim\beta_\dim}\sum_{j=1}^{\rnk}\sum_{i=1}^{\dim-\rnk}\frac{\vu_{\rnk+i}^\top \mE \vu_{\rnk+i}}{\lambda_j-\lambda_{\rnk+i}}\nonumber \\
    & \explain{\eqref{eq:z-dist-recalled}}{=} \Tr[\mU_\star^\top\mE\mU_\star(\mI_\rnk-\E[\rZ^\top\rZ])] + \Tr[\mU_\perp^\top\mE\mU_\perp \E[\rZ\rZ^\top]],
\end{align}
where
\begin{align*}
    H_{\mSigma}(\lambda) \bydef \frac{1}{\dim} \sum_{i = 1}^{\dim - \rnk} \frac{1}{\lambda - \lambda_{\rnk+i}}.
\end{align*}
Combining \eqref{eq:TrVTEV} and \eqref{eq:centering-term} gives the following formula for the centered and rescaled statistic:
\begin{align} \label{eq:remove-pos}
     &\sqrt{\dim}(\Tr[\rV^\top\mE\rV]-\Tr[\mE\calM_{\beta_{\dim}}(\mSigma)])\\
     &\qquad =\sqrt{\dim}\{\Tr[\mU_\star^\top\mE\mU_\star(\mI_\rnk-\rZ^\top\rZ)_+] - \Tr[\mU_\star^\top\mE\mU_\star(\mI_\rnk-\E[\rZ^\top\rZ])] \notag\\
    &\qquad\qquad+2\Tr[\mU_\perp^\top \mE\mU_\star(\mI_\rnk-\rZ^\top\rZ)_+^{\frac{1}{2}}\rZ^\top]+\Tr[\mU_\perp^\top\mE\mU_\perp\rZ\rZ^\top]-\Tr[\mU_\perp^\top\mE\mU_\perp \E[\rZ\rZ^\top]]\}. \nonumber
\end{align}
From \Cref{lem:Z^T*Z-sampler}, we know that:
\begin{align*}
    \rZ^\top \rZ &\pc \diag\left( \frac{H_\mu(\gamma_1)}{\beta}, \dotsc,  \frac{H_\mu(\gamma_\rnk)}{\beta}\right) \prec \mI_\rnk \implies \P\left( (\mI_\rnk-\rZ^\top\rZ)_+ = \mI_\rnk-\rZ^\top\rZ \right)  \rightarrow 1.
\end{align*}
Hence, we can remove the positive parts $(\cdot)_+$ from \eqref{eq:remove-pos} at the cost of introducing $o_\P(1)$ error:
\begin{align}
     &\sqrt{\dim}(\Tr[\rV^\top\mE\rV]-\Tr[\mE\calM_{\beta_{\dim}}(\mSigma)])\notag\\
     &\quad= \sqrt{\dim}\{\Tr[\mU_\star^\top\mE\mU_\star(\E[\rZ^\top\rZ]-\rZ^\top\rZ)] +2\Tr[\mU_\perp^\top \mE\mU_\star\left(\mI_\rnk-\rZ^\top\rZ\right)_+^{\frac{1}{2}}\rZ^\top] \nonumber\notag \\ 
     & \hspace{7cm} +\Tr[\mU_\perp^\top\mE\mU_\perp(\rZ\rZ^\top-\E[\rZ\rZ^\top])]\} + o_{\P}(1) \nonumber \notag\\
     &\quad = \rw + 2\underbrace{\sqrt{\dim}\Tr[\mU_\perp^\top\mE\mU_\star\{(\mI_\rnk-\rZ^\top\rZ)_+^{\frac{1}{2}}-(\mI_\rnk-\E[\rZ^\top\rZ])_+^{\frac{1}{2}}\}\rZ^\top]}_{(i)}\notag\\ &\quad\qquad+\underbrace{\sqrt{\dim}\Tr[\mU_\perp^\top\mE\mU_\perp(\rZ\rZ^\top-\E[\rZ\rZ^\top])]}_{(ii)} + o_{\P}(1), \label{eq:w-err-decomp}
\end{align}
where the last step follows by recalling the definition of $\rw$ from the statement of \Cref{lem:Lnull-dc-w}:
\begin{align*}
    \rw\bydef \sqrt{\dim}\{\Tr[\mU_\star^\top\mE\mU_\star(\E[\rZ^\top\rZ]-\rZ^\top\rZ)]+2\Tr[\mU_\perp^\top\mE\mU_\star(\mI_\rnk-\E[\rZ^\top\rZ])_+^\frac{1}{2}\rZ^\top]\}.
\end{align*}
Since our goal is to show that $\sqrt{\dim}\left(\Tr[\rV^\top\mE\rV]-\Tr[\mE\calM_{\beta_{\dim}}(\mSigma)]\right) = \rw + o_\P(1)$, we need to prove that the two error terms $(1), (2)$ in \eqref{eq:w-err-decomp} are $o_\P(1)$. We analyze each term separately.
\begin{itemize}
    \item We can bound the first term as:
 \begin{align*}
   |(i)|&=\sqrt{\dim} \left| \Tr[\{(\mI_\rnk-\rZ^\top\rZ)^{\frac{1}{2}}_+-(\mI_\rnk-\E[\rZ^\top\rZ])^{\frac{1}{2}}_+\}\rZ^\top\mU_\perp^\top\mE\mU_\star] \right|\\
   & \explain{(a)}{\leq} \rnk \sqrt{\dim} \| \{(\mI_\rnk-\rZ^\top\rZ)^{\frac{1}{2}}_+-(\mI_\rnk-\E[\rZ^\top\rZ])^{\frac{1}{2}}_+\}\rZ^\top\mU_\perp^\top\mE\mU_\star]\|  \\
   &\explain{(b)}{\leq} \rnk\underbrace{\|(\mI_\rnk-\rZ^\top\rZ)^{\frac{1}{2}}_+-(\mI_\rnk-\E[\rZ^\top\rZ])_+^{\frac{1} {2}}\|}_{(i.a)}\underbrace{\|\sqrt{\dim}\rZ^\top\mU_\perp^\top\mE\mU_\star\|}_{(i.b)},
\end{align*}
where (a) holds since $|\Tr[\mA]| \leq \rnk \|\mA\|$ for $A \in \R^{\rnk \times \rnk},$ and (b) follows from the fact that $\|\mA \mB\| \leq \|\mA\| \|\mB\|.$ To control the term $(i.a)$, we recall from \Cref{lem:Z^T*Z-sampler} that:
\begin{align*}
    \rZ^\top \rZ &\pc \diag\left( \frac{H_\mu(\gamma_1)}{\beta}, \dotsc,  \frac{H_\mu(\gamma_\rnk)}{\beta}\right), \quad \E[\rZ^\top \rZ]\rightarrow \diag\left( \frac{H_\mu(\gamma_1)}{\beta}, \dotsc,  \frac{H_\mu(\gamma_\rnk)}{\beta}\right).
\end{align*}
Hence, by the continuous mapping theorem, $(i.a)  = o_\P(1)$. For the term $(i.b)$, we will show that $(i.b)  = O_\P(1)$ by bounding its second moment:
\begin{align*}
\E[\|\sqrt{\dim}\rZ^\top\mU_\perp^\top\mE\mU_\star\|^2] &\leq \E[\|\sqrt{\dim}\rZ^\top\mU_\perp^\top\mE\mU_\star\|_{\fr}^2] \\ &=\dim\Tr[\mU_\star^\top\mE\mU_\perp\E[\rZ\rZ^\top]\mU_\perp^\top\mE\mU_\star] \\ &\explain{(a)}{\leq} \rnk\dim \|\mU_\star^\top\mE\mU_\perp\E[\rZ\rZ^\top]\mU_\perp^\top\mE\mU_\star\|  \\
    &\lesssim \dim \|\E[\rZ\rZ^\top]\| \quad\text{[since $\|\mE\| \leq \|\mE\|_{\fr} \lesssim 1$ by assumption]} \\ &\explain{\eqref{eq:z-dist-recalled}}{=}\frac{1}{\beta_\dim} \left\| \diag\left(\sum_{j=1}^\rnk \frac{1}{\lambda_j-\lambda_{\rnk+1}}, \dotsc,\sum_{j=1}^\rnk \frac{1}{\lambda_j-\lambda_{\dim}} \right) \right\|
    \\&\leq \frac{1}{\beta_{\dim}}\sum_{j=1}^\rnk \frac{1}{\lambda_j-\lambda_{\rnk+1}}\lesssim 1,
\end{align*}
where (a) holds since $|\Tr[\mA]| \leq \rnk \|\mA\|$ for $A \in \R^{\rnk \times \rnk},$ and the last step follows thanks to our assumption that $\beta_{\dim}\rightarrow\beta$ and $\lambda_j-\lambda_{\rnk+1} \rightarrow \gamma_j - \gamma_{\rnk+1} > 0$ for all $j\in[\rnk].$ Hence, $(i.b)=O_\P(1)$. Combining our analysis of the terms $(i.a), (i.b)$  using Slutsky's theorem, we conclude that $(i)  = o_\P(1)$.
\item Next, we show that the error term $(ii)$ in \eqref{eq:w-err-decomp} is $o_\P(1)$ by computing its second moment and showing that it vanishes. Let $\rZ^j$ denote the $j$th column of $\rZ,$ and observe that the term $(ii)$ can be expressed as a sum of independent random variables:
\begin{align*}
    (ii) &\bydef \sqrt{\dim}(\Tr[\rZ^\top \mU_\perp^\top\mE\mU_\perp\rZ]-\E[\Tr[\rZ^\top \mU_\perp^\top\mE\mU_\perp\rZ]])\\
    &= \sqrt{\dim}\sum_{j=1}^{\rnk}(\rZ^{j^\top}\mU_\perp^\top\mE\mU_\perp\rZ^j-\E[\rZ^{j^\top}\mU_\perp^\top\mE\mU_\perp\rZ^j]).
\end{align*}
Hence, \begin{align*}
    &\E[(ii)^2]=\Var[(ii)]=\dim\sum_{j=1}^\rnk \mathrm{Var}[\rZ^{j^\top}\mU_\perp^\top\mE\mU_\perp\rZ^j]\\ &\explain{(a)}{=}2\dim\sum_{j=1}^\rnk \Tr[(\mU_\perp^\top\mE\mU_\perp\mathrm{Cov}[\rZ^j])^2]\\
    &=2\dim\sum_{j=1}^\rnk\|\mathrm{Cov}[\rZ^j]^{\frac{1}{2}}\mU_\perp^\top\mE\mU_\perp\mathrm{Cov}[\rZ^j]^{\frac{1}{2}}\|^2_{\fr}\\&\explain{}{\leq} 2\dim \sum_{j=1}^\rnk\|\mathrm{Cov}[\rZ^j]^{\frac{1}{2}}\|^4 \|\mE\|_{\fr}^2 \hspace{2cm} \text{[since $\|AB\|_{\fr} \leq \|A\| \|B\|_{\fr}$]}\\& \explain{\eqref{eq:z-dist-recalled}}{=}2\|\mE\|_{\fr}^2\sum_{j=1}^\rnk\frac{1}{\dim\beta_{\dim}^2(\lambda_j-\lambda_{\rnk+1})^2}\lesssim \frac{1}{\dim},
\end{align*}
where (a) holds since $\Var[\ry^\top A \ry] = 2 \Tr[(\mA \mS)^2]$ if $\ry \sim \gauss{0}{\mS}$ (see \citep[Cor. 1.3]{searle1997linear}), and 
the last step follows thanks to our assumptions that $\|\mE\|_{\fr} \lesssim 1$ and $\lambda_j-\lambda_{\rnk+1} \rightarrow \gamma_j - \gamma_{\rnk+1} > 0$ for all $j\in[\rnk].$ This verifies that $(ii) = o_\P(1)$.
\end{itemize}
Combining our analysis of the terms $(i),(ii)$ above with \eqref{eq:w-err-decomp}, we conclude that
\begin{align*}
    \sqrt{\dim}(\Tr[\rV^\top\mE\rV]-\Tr[\mE\calM_{\beta_{\dim}}(\mSigma)]) = \rw + o_\P(1),
\end{align*}
which completes the proof of \Cref{lem:Lnull-dc-w}.
\end{proof}
\subsubsection{Proof of \texorpdfstring{\Cref{lem:Lnull-dc-var}}{lem:Lnull-dc-var}}
\begin{proof}[Proof of \Cref{lem:Lnull-dc-var}.]
Recall that the random variable $\rw$ was defined as:
\begin{align*}
    \rw&\bydef \sqrt{\dim}\{\Tr[\mU_\star^\top\mE\mU_\star(\E[\rZ^\top\rZ]-\rZ^\top\rZ)]+2\Tr[\mU_\perp^\top\mE\mU_\star(\mI_\rnk-\E[\rZ^\top\rZ])_+^\frac{1}{2}\rZ^\top]\}\\
    &= \sqrt{\dim}\sum_{i=1}^{\dim-\rnk}\{\Tr[\mU_\star^\top\mE\mU_\star(\E[\rz_i\rz_i^\top]-\rz_i\rz_i^\top)]+2\vu_{\rnk+i}^\top\mE\mU_\star(\mI_\rnk-\E[\rZ^\top\rZ])_+^{\frac{1}{2}}\rz_i\},
\end{align*}
where $\rz_1, \dotsc, \rz_{\dim - \rnk}$ are the rows of the matrix $\rZ$, which were sampled independently with:
\begin{align} \label{eq:z-dist-recall2}
    \rz_i & \sim \gauss{0}{\frac{1}{\dim \beta_\dim} \diag\left( \frac{1}{\lambda_1 - \lambda_{\rnk+i}}, \dotsc, \frac{1}{\lambda_\rnk - \lambda_{\rnk+i}} \right)} \quad \forall \; i\; \in \; [\dim - \rnk].
\end{align}
Since $\rw$ is a sum of independent random variables, we can compute its variance as:
\begin{align}
    \mathrm{Var}[\rw]
    &= \dim\sum_{i=1}^{\dim-\rnk}\mathrm{Var}[-\rz_i^\top\mU_\star^\top\mE\mU_\star\rz_i+2\vu_{\rnk+i}^\top\mE\mU_\star(\mI_\rnk-\E[\rZ^\top\rZ])_+^{\frac{1}{2}}\rz_i] \nonumber\\
    &= \dim\sum_{i=1}^{\dim-\rnk}\mathrm{Var}[\rz_i^\top\mU_\star^\top\mE\mU_\star\rz_i]+4\dim\sum_{i=1}^{\dim-\rnk}\mathrm{Var}[\vu_{\rnk+i}^\top\mE\mU_\star(\mI_\rnk-\E[\rZ^\top\rZ])_+^{\frac{1}{2}}\rz_i] \nonumber \\
    &\qquad-4\dim\sum_{i=1}^{\dim-\rnk}\underbrace{\mathrm{Cov}[\rz_i^\top\mU_\star^\top\mE\mU_\star\rz_i,\vu_{\rnk+i}^\top\mE\mU_\star(\mI_\rnk-\E[\rZ^\top\rZ])_+^{\frac{1}{2}}\rz_i]}_{(\#)} \nonumber \\
    &\explain{(a)}{=} \underbrace{\dim\sum_{i=1}^{\dim-\rnk}\mathrm{Var}[\rz_i^\top\mU_\star^\top\mE\mU_\star\rz_i]}_{(i)}+\underbrace{4\dim\sum_{i=1}^{\dim-\rnk}\mathrm{Var}[\vu_{\rnk+i}^\top\mE\mU_\star(\mI_\rnk-\E[\rZ^\top\rZ])_+^{\frac{1}{2}}\rz_i]}_{(ii)}, \label{eq:var-w-plug-here}
\end{align}
where (a) holds since the term $(\#)$ is a covariance between an even and odd function of a Gaussian vector with mean zero, and hence vanishes. Next, we simplify the terms $(i)$ and $(ii)$ in the above equation.  
\begin{itemize} \item Starting with term $(i)$, we find that:
\begin{align*}
    (i) &\bydef \dim\sum_{i=1}^{\dim-\rnk}\mathrm{Var}[\rz_i^\top\mU_\star^\top\mE\mU_\star\rz_i] \\ &\explain{(a)}{=}2\dim\sum_{i=1}^{\dim-\rnk}\Tr[(\mU_\star^\top\mE\mU_\star\mathrm{Cov}[\rz_i])^2]\\ &\explain{\eqref{eq:z-dist-recall2}}{=}\frac{2}{\beta_{\dim}}\sum_{i=1}^{\dim-\rnk}\sum_{j=1}^\rnk \frac{\vu_j^\top \mE\mU_\star\mathrm{Cov}[\rz_i]\mU_\star^\top\mE \vu_j}{\lambda_j-\lambda_{\rnk+i}}\\
    &=\frac{2}{\beta_{\dim}^2\dim}\sum_{i=1}^{\dim-\rnk}\sum_{j,\ell=1}^\rnk \frac{(\vu_\ell^\top\mE \vu_j)^2}{(\lambda_j-\lambda_{\rnk+i})(\lambda_\ell-\lambda_{\rnk+i})} \\&\bydef\frac{2}{\beta_{\dim}^2}\sum_{j,\ell=1}^\rnk K_{\Sigma}(\lambda_j,\lambda_\ell)\cdot(\vu_\ell^\top\mE \vu_j)^2,
\end{align*}
where (a) holds since  $\Var[\ry^\top A \ry] = 2 \Tr[(\mA \mS)^2]$ if $\ry \sim \gauss{0}{\mS}$ (see \citep[Cor. 1.3]{searle1997linear}), and the last step follows by recalling from \eqref{eq:H-K-def-recall} that the functions $H_{\mSigma}, K_\mSigma$ were defined, for all $\lambda, \lambda^\prime\in(\lambda_{\rnk+1}, \infty),$ as:
\begin{align} \label{eq:H-K-def-recall-recall}
     H_{\mSigma}(\lambda) &\explain{def}{=} \frac{1}{\dim} \sum_{i=1}^{\dim - \rnk} \frac{1}{\lambda - \lambda_{\rnk+i}}, \quad K_{\mSigma}(\lambda,\lambda^\prime) \explain{def}{=} \frac{1}{\dim} \sum_{i=1}^{\dim-\rnk} \frac{1}{(\lambda - \lambda_{\rnk+i})(\lambda^\prime - \lambda_{\rnk+i})}.
\end{align}
\item Next, we analyze term $(ii)$. We can first compute:
\begin{align} \label{eq:E-Z^TZ}
    \E[\rZ^\top\rZ] & = \sum_{i=1}^{\dim - \rnk} \E[\rz_i \rz_i^\top]\notag\\ &\explain{\eqref{eq:z-dist-recall2}}{=} \frac{1}{\beta_\dim} \diag\left( \frac{1}{\dim}\sum_{i=1}^{\dim - \rnk} \frac{1}{\lambda_1 - \lambda_{\rnk+i}}, \cdots,  \frac{1}{\dim}\sum_{i=1}^{\dim - \rnk} \frac{1}{\lambda_\rnk - \lambda_{\rnk+i}}\right)\notag\\
    &\explain{\eqref{eq:H-K-def-recall-recall}}{=} \frac{1}{\beta_\dim} \diag(H_{\mSigma}(\lambda_{1:\rnk})). 
\end{align}
Now we can simplify the term $(ii)$:
\begin{align*}
    (ii)&= 4\dim \sum_{i=1}^{\dim-\rnk}\vu_{\rnk+i}^\top\mE\mU_\star\left(\mI_\rnk-\E[\rZ^\top\rZ]\right)^\frac{1}{2}_+\mathrm{Cov}[\rz_i]\left(\mI_\rnk-\E[\rZ^\top\rZ]\right)^{\frac{1}{2}}_+\mU_\star^\top \mE\vu_{\rnk+i}\\
    & \explain{\eqref{eq:z-dist-recall2},\eqref{eq:E-Z^TZ}}{=} \frac{4}{\beta_{\dim}^2}\sum_{i=1}^{\dim-\rnk}\sum_{j=1}^\rnk \frac{(\beta_{\dim}-H_{\mSigma}(\lambda_j))_+}{\lambda_j-\lambda_{\rnk+i}}\cdot (\vu_j^\top\mE\vu_{\rnk+i})^2\\
    &= \frac{4}{\beta_{\dim}^2}\sum_{i=1}^{\dim-\rnk}\sum_{j=1}^\rnk \frac{\beta_{\dim}-H_{\mSigma}(\lambda_j)}{\lambda_j-\lambda_{\rnk+i}}\cdot (\vu_j^\top\mE\vu_{\rnk+i})^2\quad\text{for large enough $\dim$},
\end{align*}
where the last equality holds since $\beta_\dim\rightarrow\beta>H_{\mu}(\gamma_\rnk)$ (by assumption) and $H_{\mSigma}(\lambda_j)\rightarrow H_\mu(\gamma_j) \leq H_{\mu}(\gamma_\rnk)$ for each $j\in[\rnk]$ (see \Cref{lem:misc_HK}). 
\end{itemize}
Plugging the above formulas for the terms $(i), (ii)$ into \eqref{eq:var-w-plug-here}, we conclude that for sufficiently large $\dim$, \begin{align*}
    \mathrm{Var}[\rw]&= \frac{2}{\beta_{\dim}^2}\sum_{j,\ell=1}^\rnk K_{\Sigma}(\lambda_j,\lambda_\ell)\cdot(\vu_j^\top\mE \vu_\ell)^2+\frac{4}{\beta_{\dim}^2}\sum_{j=1}^\rnk\sum_{i=1}^{\dim-\rnk} \frac{\beta_{\dim}-H_{\mSigma}(\lambda_j)}{\lambda_j-\lambda_{\rnk+i}}\cdot (\vu_{\rnk+i}^\top\mE\vu_j)^2\\
    &=\frac{4\sigma_{\mSigma}^2(\mE,\beta_\dim)}{\beta_{\dim}^2},
\end{align*}
where the last step follows by recalling that the variance function $\sigma_{\mSigma}^2(\mE,\beta_\dim)$ was defined as:
\begin{align*}
     \sigma^2_{\mSigma}(\mE, \beta_\dim)&= \frac{1}{2} \sum_{j, \ell = 1}^\rnk K_\mSigma(\lambda_j, \lambda_\ell) \cdot (\vu_j^\top E \vu_\ell)^2  + \sum_{j=1}^\rnk \sum_{i= 1}^{\dim - \rnk} \frac{\beta_{\dim} - {H_{\mSigma}(\lambda_j)}}{\lambda_j - \lambda_{\rnk+i}}  \cdot (\vu_{\rnk+i}^\top E \vu_j)^2.
\end{align*}
This completes the proof of \Cref{lem:Lnull-dc-var}. 
\end{proof}

\subsection{Proof of \texorpdfstring{\Cref{lem:perturb}}{lem:perturb}}
\begin{proof}[Proof of \Cref{lem:perturb}]
Recall that for a sequence of matrices $\mSigma\in\R^{\dim\times\dim}$ which satisfies Assumption \mref{assump:mat}, the sequence of matrices $\tilde{\mSigma}\bydef \mSigma+\frac{1}{\sqrt{\dim}}\mE,$ where  $\mE \in \R^{\dim \times \dim}$ is any symmetric matrix satisfying $\|\mE\|_{\fr} \lesssim 1,$ and a sequence of noise parameters $\beta_{\dim} \rightarrow \beta \in (H_\mu(\gamma_\rnk),\infty),$ our goal is to show that:
\begin{align}\label{eq:perturb-goal}
    \|\calM_{\beta_{\dim}}(\mSigma)-\calM_{\beta_{\dim}}(\tilde{\mSigma})\|_{\fr} & = O\left( \frac{1}{\sqrt{\dim}} \right),
\end{align}
where, for any symmetric matrix $\mSigma\in\R^{\dim\times\dim},$ $\calM_{\beta_{\dim}}(\mSigma)$ was defined as:
\begin{align*}
   \calM_{\beta_{\dim}}(\mSigma)\bydef  \sum_{j=1}^\rnk \left(1-\frac{H_{\mSigma}(\lambda_j)}{\beta_{\dim}}\right) \cdot \vu_j\vu_j^\top+\frac{1}{\dim\beta_{\dim}}\sum_{j=1}^{\rnk}\sum_{i=1}^{\dim-\rnk}\frac{\vu_{\rnk+i}\vu_{\rnk+i}^\top}{\lambda_j-\lambda_{\rnk+i}}, 
\end{align*}
where $\{(\lambda_i,\vu_i): i \in [\dim]\}$ denotes the eigenvalues and the corresponding eigenvectors of $\mSigma.$ Let $\{(\lambda_i,\vu_i): i \in [\dim]\}$ and $\{(\tilde{\lambda}_i,\tilde{\vu}_i): i \in [\dim]\}$ denote the eigenvalues and the corresponding eigenvectors of $\mSigma$ and $\tilde{\mSigma},$ respectively, and observe that by the triangle inequality we have:
\begin{multline*}
    \|\calM_{\beta_{\dim}}(\mSigma)-\calM_{\beta_{\dim}}(\tilde{\mSigma})\|_{\fr}\leq {\norm*{\sum_{i=1}^\rnk \left(1-\frac{1}{\beta_{\dim}}H_{\mSigma}(\lambda_i)\right) \vu_i \vu_i^\top  -  \sum_{i=1}^\rnk\left(1-\frac{1}{\beta_{\dim}}H_{\tilde{\mSigma}}(\tilde{\lambda}_i)\right) \tilde{\vu}_i \tilde{\vu}_i^\top}_{\fr}}\\
    +\underbrace{\frac{1}{\dim\beta_{\dim}}\norm*{\sum_{j=1}^{\rnk}\sum_{i=1}^{\dim-\rnk}\frac{\vu_{\rnk+i}\vu_{\rnk+i}^\top}{\lambda_j-\lambda_{\rnk+i}}}_{\fr}}_{(a)}+\underbrace{\frac{1}{\dim\beta_{\dim}}\norm*{\sum_{j=1}^{\rnk}\sum_{i=1}^{\dim-\rnk}\frac{\tilde{\vu}_{\rnk+i}\tilde{\vu}_{\rnk+i}^\top}{\tilde{\lambda}_j-\tilde{\lambda}_{\rnk+i}}}_{\fr}}_{(b)},
\end{multline*}
where the terms marked $(a)$ and $(b)$ are $O(1/\sqrt{\dim})$ since:
\begin{align*}
    \frac{1}{\dim\beta_{\dim}}\norm*{\sum_{j=1}^{\rnk}\sum_{i=1}^{\dim-\rnk}\frac{\vu_{\rnk+i}\vu_{\rnk+i}^\top}{\lambda_j-\lambda_{\rnk+i}}}_{\fr}\leq\frac{\rnk}{\dim\beta_{\dim}(\lambda_\rnk-\lambda_{\rnk+1})}\norm*{\sum_{i=1}^{\dim-\rnk}\vu_{\rnk+i}\vu_{\rnk+i}^\top}_{\fr} = \frac{\rnk\sqrt{\dim-\rnk}}{\dim \beta_\dim (\lambda_\rnk-\lambda_{\rnk+1})} \lesssim \frac{1}{\sqrt{\dim}},
\end{align*}
where the last step follows from our assumption that $\beta_\dim \rightarrow \beta \in (0,\infty)$ and $\lambda_\rnk - \lambda_{\rnk+1} \rightarrow \gamma_\rnk - \gamma_{\rnk+1} > 0$. An analogous estimate holds for the term $(b)$. Hence,
\begin{subequations} \label{eq:mat-pert-plugin}
\begin{align}
      &\|\calM_{\beta_{\dim}}(\mSigma)-\calM_{\beta_{\dim}}(\tilde{\mSigma})\|_{\fr}\\
      &\qquad\lesssim \underbrace{\norm*{\sum_{i=1}^\rnk \left(1-\frac{1}{\beta_{\dim}}H_{\mSigma}(\lambda_i)\right) \vu_i \vu_i^\top  -  \sum_{i=1}^\rnk\left(1-\frac{1}{\beta_{\dim}}H_{\tilde{\mSigma}}(\tilde{\lambda}_i)\right) \tilde{\vu}_i \tilde{\vu}_i^\top}_{\fr}}_{(\star)} + \frac{1}{\sqrt{\dim}}. 
\end{align}

We will now show that $(\star)\lesssim \dim^{-1/2}.$ Observe that, by subtracting and adding $\frac{1}{\beta_\dim}\sum_{i=1}^\rnk H_{\mSigma}(\tilde{\lambda}_i)\tilde{\vu}_i \tilde{\vu}_i^\top$\footnote{Recall that $H_{\mSigma}$ was defined on the domain $(\lambda_{\rnk+1},\infty).$ Since $\lambda_{\rnk+1}\rightarrow\gamma_{\rnk+1}$ and $\tilde{\lambda}_i\rightarrow \gamma_i$ for any $i\in[\rnk],$ the assumption that $\Delta\bydef\gamma_\rnk-\gamma_{\rnk+1}>0$ implies that $H_{\mSigma}(\tilde{\lambda_i})$ for $i\in[\rnk]$ is well-defined for large enough $\dim.$} and applying the triangle inequality, we can upper bound $(\star)$ by:
\begin{align}
     \underbrace{\norm*{\sum_{i=1}^\rnk\vu_i\vu_i^\top-\sum_{i=1}^\rnk\tilde{\vu}_i\tilde{\vu}_i^\top}_{\fr}}_{(i)} &+ \frac{1}{\beta_{\dim}}\underbrace{\norm*{\sum_{i=1}^\rnk\left(H_{\tilde{\mSigma}}(\tilde{\lambda}_i)-H_{\mSigma}(\tilde{\lambda}_i)\right)\tilde{\vu}_i\tilde{\vu}_i^\top}_{\fr}}_{(ii)}\\
     &\hspace{1cm}+\frac{1}{\beta_{\dim}}\underbrace{\norm*{\sum_{i=1}^\rnk H_{\mSigma}(\tilde{\lambda}_i)\tilde{\vu}_i\tilde{\vu}_i^\top-\sum_{i=1}^\rnk H_{\mSigma}(\lambda_i)\vu_i\vu_i^\top}_{\fr}}_{(iii)}.
\end{align}
\end{subequations}

We will show that each of the terms $(i)$, $(ii)$, and $(iii)$ is $O(1/\sqrt{\dim}),$ which, together with the assumption that $\beta_{\dim}\rightarrow\beta,$ implies \eqref{eq:perturb-goal} and proves \Cref{lem:perturb}.
\paragraph{Step 1: Controlling Term $(i)$} We can bound $(i)$ using the Davis-Kahan sin$\Theta$ theorem (see e.g., \citep[Theorem 2.7]{chen2021spectral}):
\begin{align*}
    (i)&\bydef \norm*{\sum_{i=1}^\rnk\vu_i\vu_i^\top-\sum_{i=1}^\rnk\tilde{\vu}_i\tilde{\vu}_i^\top}_{\fr} 
\\
&\leq \frac{\sqrt{2{\rnk}}\|\tilde{\mSigma}-\mSigma\|}{\lambda_\rnk-\tilde{\lambda}_{\rnk+1}}\\
    &\lesssim \|\tilde{\mSigma}-\mSigma\|_{\fr}\quad\text{[since $\lambda_\rnk\rightarrow\gamma_\rnk,$  $\tilde{\lambda}_{\rnk+1}\rightarrow\gamma_{\rnk+1}$ (by \Cref{lem:misc_conv}), and $\gamma_{\rnk}-\gamma_{\rnk+1}>0$]}\\
    &\lesssim \frac{1}{\sqrt{\dim}}.
\end{align*}
Hence, $(i)\lesssim 1/\sqrt{\dim}.$

\paragraph{Step 2: Controlling Term $(ii)$} We can bound term $(ii)$ as:
\begin{align*}
    (ii) &\bydef \norm*{\sum_{i=1}^\rnk\left(H_{\tilde{\mSigma}}(\tilde{\lambda}_i)-H_{\mSigma}(\tilde{\lambda}_i)\right)\tilde{\vu}_i\tilde{\vu}_i^\top}_{\fr} \\
    & \leq \sum_{i=1}^\rnk |H_{\tilde{\mSigma}}(\tilde{\lambda}_i)-H_{\mSigma}(\tilde{\lambda}_i)| \quad \text{[by triangle inequality]} \\
    &\explain{(a)}{\leq} \frac{1}{\dim}\sum_{j=1}^\rnk\sum_{i=1}^{\dim-\rnk}\left\vert \frac{1}{\tilde{\lambda}_j-\tilde{\lambda}_{\rnk+i}}-\frac{1}{\tilde{\lambda}_j-\lambda_{\rnk+i}}\right\vert,
\end{align*}
where (a) follows from recalling that $H_{\Sigma}(\lambda) \bydef \tfrac{1}{\dim} \textstyle \sum_{i=1}^{\dim - \rnk} \tfrac{1}{\lambda - \lambda_{\rnk+i}}, \; H_{\tilde{\Sigma}}(\lambda) \bydef \tfrac{1}{\dim} \textstyle \sum_{i=1}^{\dim - \rnk} \tfrac{1}{\lambda - \tilde{\lambda}_{\rnk+i}}.$ We can simplify the above bound using Taylor's theorem:
\begin{align*}
    (ii) & \leq \frac{1}{\dim}\sum_{j=1}^\rnk\sum_{i=1}^{\dim-\rnk}\left\vert \frac{1}{\tilde{\lambda}_j-\tilde{\lambda}_{\rnk+i}}-\frac{1}{\tilde{\lambda}_j-\lambda_{\rnk+i}}\right\vert \\
    &  =  \frac{1}{\dim}\sum_{j=1}^\rnk\sum_{i=1}^{\dim-\rnk} \frac{| \lambda_{\rnk+i}-\tilde{\lambda}_{\rnk+i}|}{\gamma^2_{ij}}\quad\text{[for some $\gamma_{ij}$ between $\tilde{\lambda}_j-\tilde{\lambda}_{\rnk+i}$ and $\tilde{\lambda}_j-\lambda_{\rnk+i}$]} \\
    & \leq \frac{1}{\dim}\sum_{j=1}^\rnk\sum_{i=1}^{\dim-\rnk}  \frac{| \lambda_{\rnk+i}-\tilde{\lambda}_{\rnk+i}|}{(\min\{\tilde{\lambda}_\rnk-\tilde{\lambda}_{\rnk+1},\tilde{\lambda}_\rnk-\lambda_{\rnk+1}\})^2} \\
    &\explain{(a)}{\leq}\frac{1}{\dim}\sum_{j=1}^\rnk\sum_{i=1}^{\dim-\rnk}  \frac{\|\mSigma - \tilde{\mSigma}\|}{(\min\{\tilde{\lambda}_\rnk-\tilde{\lambda}_{\rnk+1},\tilde{\lambda}_\rnk-\lambda_{\rnk+1}\})^2}\\
    & \leq \frac{\rnk \|\mSigma - \tilde{\mSigma}\|}{(\min\{\tilde{\lambda}_\rnk-\tilde{\lambda}_{\rnk+1},\tilde{\lambda}_\rnk-\lambda_{\rnk+1}\})^2} \\
    & \lesssim \frac{1}{\sqrt{\dim}} \cdot \frac{1}{(\min\{\tilde{\lambda}_\rnk-\tilde{\lambda}_{\rnk+1},\tilde{\lambda}_\rnk-\lambda_{\rnk+1}\})^2} \quad \text{[by the assumption $\|\mSigma - \tilde{\mSigma}\|_{\fr} \lesssim \tfrac{1}{\sqrt{\dim}}$]}\\
    & \lesssim \frac{1}{\sqrt{\dim}},
\end{align*}
where (a) follows from Weyl's inequality (see e.g., \citep[Theorem 4.5.3]{vershynin2018high}), and 
the last steps holds since $\tilde{\lambda}_\rnk-\tilde{\lambda}_{\rnk+1},\tilde{\lambda}_{\rnk}-\lambda_{\rnk+1}\rightarrow\gamma_\rnk - \gamma_{\rnk+1}>0$ by \Cref{lem:misc_conv}. This proves that $(ii)\lesssim 1/\sqrt{\dim}.$
\paragraph{Step 3: Controlling Term $(iii)$} Recall that
\begin{align*}
    (iii) \bydef \norm*{\sum_{i=1}^\rnk H_{\mSigma}(\tilde{\lambda}_i)\tilde{\vu}_i\tilde{\vu}_i^\top-\sum_{i=1}^\rnk H_{\mSigma}(\lambda_i)\vu_i\vu_i^\top}_{\fr},
\end{align*}
and observe that the two terms in the expression above can be viewed as the function 
\begin{align} \label{eq:gen-mat-func}
    \mA\mapsto \sum_{i=1}^\rnk H_\mSigma(\lambda_i(A))u_i(\mA)u_i(\mA)^\top 
\end{align} 
applied to $\tilde{\mSigma}$ and $\mSigma,$ which we know are close. To estimate the error $(iii)$ in terms of $\|\mSigma - \tilde{\mSigma}\|$, we will use a technique of \citet{kato2013perturbation}  which uses Cauchy's residue theorem to derive perturbation bounds for complicated matrix functions like \eqref{eq:gen-mat-func} from classical perturbation bounds for the matrix inverse function. This technique has been recently used in other works in statistics (see e.g., \citep{koltchinskii2016asymptotics,xia2021normal}). 

The basic idea is to express the matrix function in \eqref{eq:gen-mat-func} as a contour integral by considering a contour $\Gamma$ on the complex plane $\mathbb{C}$ constructed so that the leading $\rnk$ eigenvalues $\lambda_{1:\rnk}(\mA)$ of $\mA$ (the argument of the function in \eqref{eq:gen-mat-func}) lie inside the contour, while the remaining eigenvalues $\lambda_{\rnk+1:\dim}$ lie outside the contour. Then, by Cauchy's residue theorem, for any symmetric matrix $\mA$ and any contour $\Gamma$ which satisfies this requirement,
\begin{align}
    &\sum_{i=1}^\rnk H_\mSigma(\lambda_i(A))u_i(\mA)u_i(\mA)^\top\notag\\
    &\qquad= \frac{1}{2\pi i}\sum_{i=1}^\rnk \oint_\Gamma \frac{H_\mSigma(z)\diff z}{z-\lambda_i(\mA)}u_i(\mA)u_i(\mA)^\top \quad\text{[since $\lambda_{1:\rnk}(\mA)$ are inside $\Gamma$]} \nonumber \\
    &\qquad=\frac{1}{2\pi i}\sum_{i=1}^{\dim} \oint_\Gamma \frac{H_\mSigma(z)\diff z}{z-\lambda_i(\mA)}u_i(\mA)u_i(\mA)^\top\quad\text{[since $\lambda_{\rnk+1:\dim}(\mA)$ are outside of $\Gamma$]} \nonumber\\
    &\qquad=\frac{1}{2\pi i} \oint_\Gamma \sum_{i=1}^{\dim}\frac{H_\mSigma(z)}{z-\lambda_i(\mA)}u_i(\mA)u_i(\mA)^\top \diff z \nonumber \\
    &\qquad=\frac{1}{2\pi i} \oint_\Gamma H_\mSigma(z)(zI_{\dim}-\mA)^{-1}\diff z. \label{eq:CRT}
\end{align}
\begin{figure}[ht]
    \centering
    \includegraphics[width=0.5\linewidth]{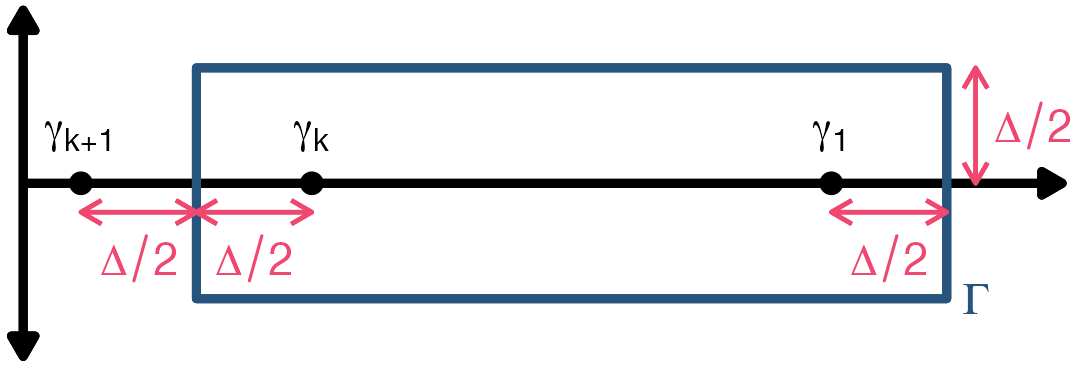}
    \caption{Visualization of the contour $\Gamma.$}
    \label{fig:contour}
\end{figure}
To apply the formula \eqref{eq:CRT} to $\mSigma$ and $\tilde{\mSigma},$ we will choose the contour $\Gamma$ as shown in \Cref{fig:contour} so that for large enough $\dim,$ $\lambda_i,\tilde{\lambda}_i$ for $i\in[\rnk]$ are strictly inside $\Gamma$ and $\lambda_i,\tilde{\lambda}_i$ for $i\in[\dim]\backslash[\rnk]$ are strictly outside $\Gamma,$ which is guaranteed since $\lambda_i,\tilde{\lambda}_i\rightarrow \gamma_i$ for $i\in\{1,\rnk,\rnk+1\}$ and $\Delta\bydef\gamma_{\rnk}-\gamma_{\rnk+1}>0.$ 
Moreover, this choice of $\Gamma$ ensures that
% \begin{align}\label{eq:CRT-min-gap}
%     \underset{\substack{z\in\Gamma\\i\in[\dim]}}{\min}\;|z-\lambda_i|\gtrsim\frac{\Delta}{4}\quad\text{and}\quad\underset{\substack{z\in\Gamma\\i\in[\dim]}}{\min}\;|z-\tilde{\lambda}_i|\gtrsim\frac{\Delta}{4}
% \end{align}
% since, for large enough $\dim,$
\begin{subequations}\label{eq:CRT-min-gap}
\begin{align} 
    \underset{\substack{z\in\Gamma\\i\in[\dim]}}{\min}\;|z-\lambda_i|&=\min\left\{\lambda_\rnk-\left(\gamma_{\rnk+1}+\frac{\Delta}{2}\right), \left(\gamma_{\rnk+1}+\frac{\Delta}{2}\right)-\lambda_{\rnk+1}, \left(\gamma_1+\frac{\Delta}{2}\right)-\lambda_1, \frac{\Delta}{2}\right\} \rightarrow \frac{\Delta}{2},
\end{align}
and by the same argument:
\begin{align}
    \underset{\substack{z\in\Gamma\\i\in[\dim]}}{\min}\;|z-\tilde{\lambda}_i|& \rightarrow \frac{\Delta}{2}.
\end{align}
\end{subequations}
We can now estimate the term $(iii)$ as follows:
\begin{align*}
    (iii)&\bydef \norm*{\sum_{i=1}^\rnk H_{\mSigma}(\tilde{\lambda}_i)\tilde{\vu}_i\tilde{\vu}_i^\top-\sum_{i=1}^\rnk H_{\mSigma}(\lambda_i)\vu_i\vu_i^\top}_{\fr} \\
    & \explain{\eqref{eq:CRT}}{=}\norm*{\frac{1}{2\pi i} \oint_\Gamma H_{\mSigma}(z)(zI_{\dim}-\mSigma)^{-1}\diff z-\frac{1}{2\pi i} \oint_\Gamma H_{\mSigma}(z)(zI_{\dim}-\tilde{\mSigma})^{-1}\diff z}_{\fr}\\
    &\leq \frac{1}{2\pi} \oint_\Gamma |H_{\mSigma}(z)|\cdot\norm{(zI_{\dim}-\mSigma)^{-1}-(zI_{\dim}-\tilde{\mSigma})^{-1}}_{\fr}\diff z\quad\text{[by triangle inequality]}\\
    & \explain{(a)}{\leq} \frac{1}{2\pi \min_{z \in \Gamma, i \in [\dim]}|z-\lambda_i|} \oint_\Gamma \norm{(zI_{\dim}-\mSigma)^{-1}-(zI_{\dim}-\tilde{\mSigma})^{-1}}_{\fr}\diff z\\
    &\explain{\eqref{eq:CRT-min-gap}}{\lesssim}  \oint_\Gamma \norm{(zI_{\dim}-\mSigma)^{-1}-(zI_{\dim}-\tilde{\mSigma})^{-1}}_{\fr}\diff z,
\end{align*}
where (a) holds since $H_{\mSigma}(z) \bydef \tfrac{1}{\dim} \textstyle \sum_{i=1}^{\dim-\rnk} \tfrac{1}{z - \lambda_{\rnk+i}}.$ To estimate $\norm{(zI_{\dim}-\mSigma)^{-1}-(zI_{\dim}-\tilde{\mSigma})^{-1}}_{\fr}$, we use a classical matrix perturbation bound for matrix inverse due to \citet{wedin1973perturbation}, which states that  for two invertible matrices $\mA$ and $\mB$:
\begin{align*}
    \|\mB^{-1}-\mA^{-1}\|_{\fr}\leq \|\mB^{-1}\|\|\mA^{-1}\|\|\mB-\mA\|_{\fr}.
\end{align*}
Hence,
\begin{align*}
    (iii)& \lesssim \oint_\Gamma \norm{(zI_{\dim}-\mSigma)^{-1}}\norm{(zI_{\dim}-\tilde{\mSigma})^{-1}}\norm{\tilde{\mSigma}-\mSigma}_{\fr}\diff z\\
    &= \oint_{\Gamma} \frac{1}{\min_{i\in[\dim]}|z-\lambda_i|}\frac{1}{\min_{i\in[\dim]}|z-\tilde{\lambda}_i|}\norm{\tilde{\mSigma}-\mSigma}_{\fr}\diff z\\
    &\lesssim \frac{1}{\sqrt{\dim}}\quad\text{[by \eqref{eq:CRT-min-gap} and the assumption $\|\mSigma - \tilde{\mSigma}\|_{\fr} \lesssim \tfrac{1}{\sqrt{\dim}}$]}.
\end{align*}
This gives our desired estimate $(iii)\lesssim 1/\sqrt{\dim}$.
\paragraph{Conclusion} Plugging in our estimates $(i), (ii), (iii) \lesssim 1/\sqrt{\dim}$ into \eqref{eq:mat-pert-plugin}, we obtain $\|\calM_{\beta_{\dim}}(\mSigma)-\calM_{\beta_{\dim}}(\tilde{\mSigma})\|_{\fr} \lesssim \tfrac{1}{\sqrt{\dim}}$, which completes the proof of \Cref{lem:perturb}.
\end{proof}

\subsection{Proof of \texorpdfstring{\Cref{prop:le-cam}}{prop:le-cam}} \label{app:le-cam}
\begin{proof}[Proof of \Cref{prop:le-cam}]
Fix any $\alpha\in[0,1],$ and recall that our goal is to characterize the asymptotics of the trade-off function $\tf{\nu_\dim}{\tilde{\nu}_\dim},$ which is defined as:
\begin{align*}
    \tf{\nu_{\dim}}{\tilde{\nu}_{\dim}}(\alpha)=\underset{\phi:\mathcal{V}_{\dim}\rightarrow[0,1]}{\min}\left\{1-\E_{\rV \sim \tilde{\nu}_\dim}[\phi(\rV)]\;:\;\E_{\rV \sim \nu_\dim}[\phi(\rV)]\leq\alpha\right\}\quad\forall\;\alpha\;\in\;[0,1].
\end{align*}
By the Neyman-Pearson theorem (see e.g., \citep[Theorem 3.2.1]{lehmann2005testing}), the optimal test which  minimizes  the RHS is given by:
\begin{align} \label{eq:NP-test}
    \phi_{\dim}^{\star}(\mV)=\ind\left\{\ln\frac{\diff\tilde{\nu}_{\dim}}{\diff\nu_{\dim}}(\mV)>\lambda_{\dim}\right\}+q_{\dim}\ind\left\{\ln\frac{\diff\tilde{\nu}_{\dim}}{\diff\nu_{\dim}}(\mV)=\lambda_{\dim}\right\},
\end{align}
where $\lambda_{\dim}\in\R \cup \{\pm \infty\}$ and $q_{\dim}\in[0,1]$ are chosen\footnote{The Neyman-Pearson theorem ensures the existence of such $\lambda_{\dim}$ and $q_{\dim}.$} so that $\phi_{\dim}^{\star}$ satisfies the constraint:
\begin{align}\label{eq:tf-typeI}
    \E_{\rV \sim \nu_\dim}[\phi_{\dim}^{\star}(\rV)]=\alpha \iff \P_{\rV \sim \nu_\dim}\left\{\ln\frac{\diff\tilde{\nu}_{\dim}}{\diff\nu_{\dim}}(\rV)>\lambda_{\dim}\right\}+q_{\dim}\P_{\rV \sim \nu_\dim}\left\{\ln\frac{\diff\tilde{\nu}_{\dim}}{\diff\nu_{\dim}}(\rV)=\lambda_{\dim}\right\}=\alpha.
\end{align}
This allows us to express the trade-off function $    \tf{\nu_{\dim}}{\tilde{\nu}_{\dim}}(\alpha)$ as: 
\begin{align}\label{eq:tf-expr}
    \tf{\nu_{\dim}}{\tilde{\nu}_{\dim}}(\alpha)&\bydef 1-\E_{\rV \sim \tilde{\nu}_\dim}[\phi_{\dim}^{\star}(\rV)]\notag\\
    &=1-\P_{\rV \sim \tilde{\nu}_\dim}\left\{\ln\frac{\diff\tilde{\nu}_{\dim}}{\diff\nu_{\dim}}(\rV)>\lambda_{\dim}\right\}-q_{\dim}\P_{\rV \sim \tilde{\nu}_\dim}\left\{\ln\frac{\diff\tilde{\nu}_{\dim}}{\diff\nu_{\dim}}(\rV)=\lambda_{\dim}\right\}\notag\\
    &=\P_{\rV \sim \tilde{\nu}_\dim}\left\{\ln\frac{\diff\tilde{\nu}_{\dim}}{\diff\nu_{\dim}}(\rV)\leq \lambda_{\dim}\right\}-q_{\dim}\P_{\rV \sim \tilde{\nu}_\dim}\left\{\ln\frac{\diff\tilde{\nu}_{\dim}}{\diff\nu_{\dim}}(\rV)=\lambda_{\dim}\right\},
\end{align}
which makes explicit the connection between the asymptotics of $\tf{\nu_{\dim}}{\tilde{\nu}_{\dim}}(\alpha)$ and the log-likelihood ratio. We will characterize the asymptotic trade-off function by dividing into cases based on the asymptotic behavior of $\ln(\diff\tilde{\nu}_{\dim}/\diff\nu_{\dim})(\rV)$ when $\rV\sim\nu_{\dim}.$ 

\paragraph{Case 1: The Log-likelihood Ratio Has a Non-Degenerate Limit Distribution} We first consider the case \begin{align*}
    \ln \frac{\diff\tilde{\nu}_{\dim}}{\diff\nu_{\dim}}(\rV)\dc \gauss{-\frac{v}{2}}{v} \quad\text{when}\quad   \rV\sim\nu_{\dim}
\end{align*} for some constant $v>0$ and prove the first claim made in \Cref{prop:le-cam}:
\begin{align}
    \lm\tf{\nu_{\dim}}{\tilde{\nu}_{\dim}}(\alpha)=\tf{\gauss{0}{1}}{\gauss{\sqrt{v}}{1}}(\alpha).
\end{align}
Since:
\begin{subequations} 
\begin{align} \label{eq:lecam-non-degen}
    \ln(\diff\tilde{\nu}_{\dim}/\diff\nu_{\dim})(\rV)\dc \gauss{-v/2}{v} \quad \text{ when } \rV\sim\nu_{\dim},
\end{align}
by Le Cam's first and third lemmas \citep[Section 6.4]{van2000asymptotic}, 
\begin{align} \label{eq:lecam-non-degen-alt}
    \ln \frac{\diff\tilde{\nu}_{\dim}}{\diff\nu_{\dim}}(\rV)\dc\gauss{\frac{v}{2}}{v} \quad\text{when}\quad  \rV\sim \tilde{\nu}_{\dim}.
\end{align}  
\end{subequations}
We can use \eqref{eq:lecam-non-degen} and \eqref{eq:lecam-non-degen-alt} to derive the asymptotic behavior of the cutoff $\lambda_\dim$ in the optimal test $\phi_\dim^\star$ in \eqref{eq:NP-test}, which will help us compute the limiting trade-off function using the formula in \eqref{eq:tf-expr}. From \eqref{eq:lecam-non-degen} and \eqref{eq:lecam-non-degen-alt}, we know that the limit distribution of $\ln(\diff\tilde{\nu}_{\dim}/\diff\nu_{\dim})(\rV)$ when $\rV \sim \nu_\dim$ and when $\rV \sim \tilde{\nu}_\dim$ is continuous. Hence, we expect the probability that $\ln(\diff\tilde{\nu}_{\dim}/\diff\nu_{\dim})(\rV)$ realizes any specific value tends to $0$. This indeed holds \citep[Problem 11.42]{lehmann2005testing}:
\begin{align}\label{eq:L=lambda}
    \lm \P_{\rV \sim \nu_\dim}\left(\ln\frac{\diff\tilde{\nu}_{\dim}}{\diff\nu_{\dim}}(\rV)=\lambda_{\dim}\right)=0\quad\text{and}\quad \lm \P_{\rV \sim \tilde{\nu}_\dim}\left(\ln\frac{\diff\tilde{\nu}_{\dim}}{\diff\nu_{\dim}}(\rV)=\lambda_{\dim}\right)=0.
\end{align}
Furthermore, recall from \eqref{eq:tf-typeI} that 
\begin{align*}
    \alpha&=\lm\P_{\rV \sim \nu_\dim}\left(\ln\frac{\diff\tilde{\nu}_{\dim}}{\diff\nu_{\dim}}(\rV)>\lambda_{\dim}\right)+\lm q_{\dim}\P_{\rV \sim \nu_\dim}\left(\ln\frac{\diff\tilde{\nu}_{\dim}}{\diff\nu_{\dim}}(\rV)=\lambda_{\dim}\right)\\
    &\explain{\eqref{eq:L=lambda}}{=}\lm\P_{\rV \sim \nu_\dim}\left(\ln\frac{\diff\tilde{\nu}_{\dim}}{\diff\nu_{\dim}}(\rV)>\lambda_{\dim}\right).
\end{align*}In light of the above equation and the fact that  $\ln(\diff\tilde{\nu}_{\dim}/\diff\nu_{\dim})(\rV)\dc\gauss{-v/2}{v}$ when $\rV\sim\nu_{\dim}$, we expect that $\lambda_{\dim}$ should converge to the $1-\alpha$ quantile of $\gauss{-v/2}{v},$ which is indeed true \citep[Problem 11.42]{lehmann2005testing}:
\begin{align}\label{eq:lambda-conv}
    \lambda_{\dim}\rightarrow\sqrt{v} \Phi^{-1}(1-\alpha)-\frac{v}{2},
\end{align}
where $\Phi$ denotes the CDF of $\gauss{0}{1}$. 
We can now compute the asymptotic trade-off function between $\nu_\dim, \tilde{\nu}_\dim$: 
\begin{align*}
    \lm\tf{\nu_{\dim}}{\tilde{\nu}_{\dim}}(\alpha) &\explain{\eqref{eq:tf-expr}}{=}\lm\P_{\rV \sim \tilde{\nu}_\dim}\left(\ln\frac{\diff\tilde{\nu}_{\dim}}{\diff\nu_{\dim}}(\rV)\leq \lambda_{\dim}\right)-\lm q_{\dim}\P_{\rV \sim \tilde{\nu}_\dim}\left(\ln\frac{\diff\tilde{\nu}_{\dim}}{\diff\nu_{\dim}}(\rV)=\lambda_{\dim}\right)\\
    &=\lm\P_{\rV \sim \tilde{\nu}_\dim}\left(\ln\frac{\diff\tilde{\nu}_{\dim}}{\diff\nu_{\dim}}(\rV)\leq \lambda_{\dim}\right)\quad \text{[by \eqref{eq:L=lambda} and $q_{\dim}\in[0,1]$]}\\
    &\explain{(a)}{=} \P\left(\gauss{\frac{v}{2}}{v}\leq \sqrt{v} \Phi^{-1}(1-\alpha)-\frac{v}{2}\right)\\
    &=\Phi\left( \Phi^{-1}(1-\alpha)-\sqrt{v}\right) \\
    &=\tf{\gauss{0}{1}}{\gauss{\sqrt{v}}{1}}(\alpha),
\end{align*}
where (a) follows from  \eqref{eq:lambda-conv} and the fact that when $\rV\sim\tilde{\nu}_{\dim}$, the CDF of $\ln(\diff\tilde{\nu}_{\dim}/\diff\nu_{\dim})(\rV)$  converges to the CDF of $\mathcal{N}(v/2,v)$ point-wise (recall \eqref{eq:lecam-non-degen-alt}). Since the limit CDF is continuous, this convergence is in fact uniform \citep[Theorem 11.2.9]{lehmann2005testing}. This proves the first claim of \Cref{prop:le-cam}.

\paragraph{Case 2: The Log-likelihood Ratio has a Degenerate Limit Distribution} Next, we consider the case \begin{align*}
    \ln \frac{\diff\tilde{\nu}_{\dim}}{\diff\nu_{\dim}}(\rV)\pc 0 \quad\text{when}\quad  \rV\sim\nu_{\dim},
\end{align*} and prove the second claim made in \Cref{prop:le-cam}:
\begin{align}\label{eq:lecam-degen}
    \lm\tf{\nu_{\dim}}{\tilde{\nu}_{\dim}}(\alpha)=1-\alpha.
\end{align}
Recall from \eqref{eq:tf-expr} that $\lm\tf{\nu_\dim}{\tilde{\nu}_\dim}(\alpha)=1-\lm\E_{\rV \sim \tilde{\nu}_\dim}[\phi_{\dim}^{\star}(\rV)],$ and note that 
\begin{align*}
    \lm\tf{\nu_\dim}{\tilde{\nu}_\dim}(\alpha)&=1-\lm\left(\E_{\rV \sim \tilde{\nu}_\dim}[\phi_{\dim}^{\star}(\rV)]-\E_{\rV \sim \nu_\dim}[\phi_{\dim}^{\star}(\rV)]\right)-\lm\E_{\rV \sim \nu_\dim}[\phi_{\dim}^{\star}(\rV)]\\
    &\explain{(a)}{=}1-\alpha -\lm\left(\E_{\rV \sim \tilde{\nu}_\dim}[\phi_{\dim}^{\star}(\rV)]-\E_{\rV \sim \nu_\dim}[\phi_{\dim}^{\star}(\rV)]\right)\\
    &\explain{(b)}{=}1-\alpha,
\end{align*}
where (a) holds since $\E_{\rV \sim \nu_\dim}[\phi_{\dim}^{\star}(\rV)]=\alpha$ for any $\dim\in\N$ (see \eqref{eq:tf-typeI}), and (b) holds since: 
\begin{align*}
    &\lm\left\vert\E_{\rV \sim \tilde{\nu}_\dim}[\phi_{\dim}^{\star}(\rV)]-\E_{\rV \sim \nu_\dim}[\phi_{\dim}^{\star}(\rV)]\right\vert\\
    &\qquad\explain{\eqref{eq:NP-test}}{\leq}  \lm\left\vert\P_{\rV \sim \tilde{\nu}_\dim}\left(\ln\frac{\diff\tilde{\nu}_{\dim}}{\diff\nu_{\dim}}(\rV)>\lambda_{\dim}\right)-\P_{\rV \sim \nu_\dim}\left(\ln\frac{\diff\tilde{\nu}_{\dim}}{\diff\nu_{\dim}}(\rV)>\lambda_{\dim}\right)\right\vert\\
    &\qquad\qquad+\lm q_{\dim}\cdot\left\vert \P_{\rV \sim \tilde{\nu}_\dim}\left(\ln\frac{\diff\tilde{\nu}_{\dim}}{\diff\nu_{\dim}}(\rV)=\lambda_{\dim}\right)-\P_{\rV \sim \nu_\dim}\left(\ln\frac{\diff\tilde{\nu}_{\dim}}{\diff\nu_{\dim}}(\rV)=\lambda_{\dim}\right)\right\vert\\
    &\qquad\explain{(a)}{\leq} 2\lm \tv\left(\nu_{\dim},\tilde{\nu}_{\dim}\right)\\
    &\qquad\explain{(b)}{=}0,
\end{align*}
where (a) follows from the fact that  $q_{\dim}\in[0,1]$ and the variational definition of the total variation distance. Step (b) follows by recalling the case assumption that $(\diff\tilde{\nu}_{\dim}/\diff\nu_{\dim})(\rV)\pc 1$, which combined with Scheffé's lemma guarantees that $\lm \tv\left(\nu_{\dim},\tilde{\nu}_{\dim}\right) = 0$. Thus, we have that $\lim_{\dim \rightarrow \infty} \tf{\nu_\dim}{\tilde{\nu}_\dim}(\alpha)=1-\alpha$, which proves the second claim of \Cref{prop:le-cam}. 
\end{proof}

%% file: appendix_privacy.tex
\section{Proof of the Privacy Theorem (Theorem \mref{thm:privacy})}\label{app:privacy-renyi}
This appendix completes the proof of our privacy result for the exponential mechanism (Theorem \mref{thm:privacy}) by proving the claimed Rényi differential privacy guarantee and also provides the proofs of the various intermediate results introduced in Section \mref{sec:privacy-proof}. We begin by completing the proof of Theorem \mref{thm:privacy}.

\begin{proof}[Proof of Theorem \mref{thm:privacy}] Since we have already proved the trade-off-function-based privacy guarantee in Section \mref{sec:privacy-proof}, we only need to show the Rényi differential privacy guarantee  required from a $\sigma_\beta$-AGDP algorithm (recall Definition \mref{def:AGDP}):
\begin{align} \label{eq:priv-thm-renyi}
    \ls\sup_{\tilde{\mX}\in\calN(\mX)} \rdv{\alpha}{\nu(\cdot \mid \mSigma(\tilde{\mX}), \beta_{\dim}, \rnk)}{\nu_{\dim}}=\rdv{\alpha}{\gauss{\sigma_\beta}{1}}{\gauss{0}{1}} \quad \forall \; \alpha > 1,
\end{align}
where we use the shorthand notation $\nu_{\dim}\bydef\nu(\cdot \mid \mSigma(\mX), \beta_{\dim}, \rnk).$ As before, we will prove this in two steps. First, we show the upper bound:
\begin{align}\label{eq:privacy-rdp-ub}
    \ls\sup_{\tilde{\mX}\in\calN(\mX)} \rdv{\alpha}{\nu(\cdot \mid \mSigma(\tilde{\mX}), \beta_{\dim}, \rnk)}{\nu_{\dim}}\leq\rdv{\alpha}{\gauss{\sigma_\beta}{1}}{\gauss{0}{1}}
\end{align}
and then show that the worst-case data point $\vx_\star$ constructed in Proposition \mref{prop:var-lim-v2} achieves this upper bound:
\begin{align}   \label{eq:achieve-renyi} 
\li\rdv{\alpha}{\nu(\cdot \mid \mSigma(\mX\cup \{\vx_\star\}), \beta_{\dim}, \rnk)}{\nu_{\dim}} \geq \rdv{\alpha}{\gauss{\sigma_\beta}{1}}{\gauss{0}{1}}.
\end{align}
Note that \eqref{eq:privacy-rdp-ub} and \eqref{eq:achieve-renyi} immediately imply \eqref{eq:priv-thm-renyi}.  As before, we let $\tilde{\mX}_{\sharp}\in\calN(\mX)$ be an approximate worst-case neighboring dataset which satisfies:
\begin{align*}
    \rdv{\alpha}{\nu(\cdot \mid \mSigma(\tilde{\mX}_{\sharp}), \beta_{\dim}, \rnk)}{\nu_{\dim}}\geq\sup_{\tilde{\mX}\in\calN(\mX)} \rdv{\alpha}{\nu(\cdot \mid \mSigma(\tilde{\mX}), \beta_{\dim}, \rnk)}{\nu_{\dim}}-\frac{1}{\dim}
\end{align*}
so that:
\begin{align*}
    \ls\sup_{\tilde{\mX}\in\calN(\mX)} \rdv{\alpha}{\nu(\cdot \mid \mSigma(\tilde{\mX}), \beta_{\dim}, \rnk)}{\nu_{\dim}}&= \ls \rdv{\alpha}{\nu(\cdot \mid \mSigma(\tilde{\mX}_{\sharp}), \beta_{\dim}, \rnk)}{\nu_{\dim}}.
    \end{align*}
Next, note that Lemma \mref{lem:simple-gibbs-approx} (item (2)) and Lemma \mref{lem:triangle-ineq} (item (2)) imply that for any $\epsilon\in(0,\alpha-1),$ we have:
\begin{align}\label{eq:wti-application}
    \ls\rdv{\alpha}{\tilde{\nu}_{\dim}}{\nu_{\dim}}&\explain{Lem. \mref{lem:triangle-ineq}}{\leq} \ls \left( \rdv{\alpha+\epsilon}{\bar{\nu}_{\dim}}{\nu_{\dim}} +  \frac{\alpha+\epsilon}{\alpha+\epsilon-1}\rdv{\frac{\alpha(\alpha+\epsilon-1)}{\epsilon}}{\tilde{\nu}_\dim}{\bar{\nu}_\dim} \right)\notag\\
    &\hspace{7cm}\explain{Lem. \mref{lem:simple-gibbs-approx}}{=}\ls\rdv{\alpha+\epsilon}{\bar{\nu}_{\dim}}{\nu_{\dim}},
\end{align}
where $\tilde{\nu}_{\dim}\bydef\nu(\cdot \mid \mSigma(\tilde{\mX}_{\sharp}), \beta_{\dim}, \rnk)$ and $\bar{\nu}_{\dim}\bydef\nu(\cdot \mid \bar{\mSigma}(\tilde{\mX}_{\sharp}), \beta_{\dim}, \rnk)$. Next, we compute the RHS using Theorem \mref{thm:contiguity}. To ensure that the requirements of Theorem \mref{thm:contiguity} are met, we compute the RHS by passing to a subsequence achieving the limit superior and then to a further subsequence\footnote{Bolzano–Weierstrass guarantees the existence of such a subsequence since $\ls \sigma^2_{\mSigma(\mX)}(\sqrt{\dim}(\bar{\mSigma}(\tilde{\mX}_{\sharp})-\mSigma(\mX)),\beta_{\dim})< \infty$, thanks to  Proposition~\mref{prop:var-lim-v2}.} along which $$\sigma^2_{\mSigma(\mX)}(\sqrt{\dim}(\bar{\mSigma}(\tilde{\mX}_{\sharp})-\mSigma(\mX)),\beta_{\dim})\rightarrow v$$ for some $v\in[0,\infty).$ Applying Theorem \mref{thm:contiguity}, we conclude that:
\begin{align*}
     \ls\sup_{\tilde{\mX}\in\calN(\mX)} \rdv{\alpha}{\nu(\cdot \mid \mSigma(\tilde{\mX}), \beta_{\dim}, \rnk)}{\nu_{\dim}}& \explain{\eqref{eq:wti-application}}{\leq} \lm\rdv{\alpha+\epsilon}{\bar{\nu}_\dim}{\nu_\dim}\\
     &= 
        \rdv{\alpha+\epsilon}{\gauss{\sqrt{v}}{1}}{\gauss{0}{1}}. 
\end{align*}
Since the above estimate holds for any $\epsilon \in (0, \alpha-1)$, we can let $\epsilon \rightarrow 0$ and conclude that:  
\begin{align*}
   \ls\sup_{\tilde{\mX}\in\calN(\mX)} \rdv{\alpha}{\nu(\cdot \mid \mSigma(\tilde{\mX}), \beta_{\dim}, \rnk)}{\nu_{\dim}} &\leq  \rdv{\alpha}{\gauss{\sqrt{v}}{1}}{\gauss{0}{1}}\\
   &\explain{(a)}{\leq} \rdv{\alpha}{\gauss{\sigma_\beta}{1}}{\gauss{0}{1}}.
\end{align*} 
In the above equations, step (a) follows by observing that if $\vx_\sharp$ denotes the data point added/removed from $\mX$ to obtain $\mX_\sharp$, we have:
\begin{align*}
    v  \bydef \lm \sigma^2_{\mSigma(\mX)}(\sqrt{\dim}(\bar{\mSigma}(\tilde{\mX}_{\sharp})-\mSigma(\mX)),\beta_{\dim})&\explain{\meqref{eq:simple-cov}}{=} \lm \sigma^2_{\mSigma(\mX)}\left(\frac{\sqrt{\dim}\vx_\sharp\vx_\sharp^\top}{\ssize},\beta_{\dim}\right)\\
    &\bydef \lm\sigma^2_{\mX}(\vx_{\sharp},\beta_{\dim})\\
    &\leq \lim_{\dim \rightarrow \infty} \sup_{\substack{\vx \in \R^\dim \\ \|\vx\|^2 \leq \dim}} \sigma^2_{\mX}(\vx,\beta_{\dim})\explain{Prop. \mref{prop:var-lim-v2}}{=} \sigma_\beta^2.
\end{align*}
This proves the upper bound \eqref{eq:privacy-rdp-ub}. Next, we prove the matching lower bound in \eqref{eq:achieve-renyi} (that is, the worst-case data point $\vx_{\star}$ constructed in Proposition \mref{prop:var-lim-v2} achieves the upper bound above). Using the shorthands $\tilde{\nu}_{\dim}\bydef\nu(\cdot \mid \mSigma({\mX} \cup \{\vx_\star\}), \beta_{\dim}, \rnk)$ and $\bar{\nu}_{\dim}\bydef\nu(\cdot \mid \bar{\mSigma}({\mX} \cup \{\vx_\star\}), \beta_{\dim}, \rnk),$ note that, for any $\epsilon\in(0,\alpha-1),$ we have:
\begin{align*}
    \li \rdv{\alpha}{\tilde{\nu}_{\dim}}{\nu_{\dim}}& \explain{Lem.  \mref{lem:triangle-ineq}}{\geq} \li \left(\rdv{\alpha-\epsilon}{\bar{\nu}_{\dim}}{\nu_{\dim}} - \frac{\alpha}{\alpha-1}\rdv{\frac{(\alpha-\epsilon)(\alpha-1)}{\epsilon}}{\bar{\nu}_\dim}{\tilde{\nu}_\dim}\right) \\ &  \explain{Lem. \mref{lem:simple-gibbs-approx}}{=} \li \rdv{\alpha-\epsilon}{\bar{\nu}_{\dim}}{\nu_{\dim}}
    \explain{(a)}{=}\rdv{\alpha-\epsilon}{\gauss{\sigma_\beta}{1}}{\gauss{0}{1}},
\end{align*}
where (a) follows by applying Theorem \mref{thm:contiguity} to compute the limiting Rényi divergence after observing that:
\begin{align*}
    \lm\sigma^2_{\mSigma(\mX)}(\sqrt{\dim}(\bar{\mSigma}({\mX} \cup \{\vx_\star\})-\mSigma(\mX)),\beta_{\dim}) &\explain{\meqref{eq:simple-cov}}{=}\lm \sigma^2_{\mSigma(\mX)}\left(\frac{\sqrt{\dim}\vx_\star\vx_\star^\top}{\ssize},\beta_{\dim}\right)\\
    &\hspace{2cm}\bydef \lm\sigma^2_{\mX}(\vx_{\star},\beta_{\dim})\explain{Prop. \mref{prop:var-lim-v2}}{=} \sigma_\beta^2.
\end{align*}
Finally, taking $\epsilon\rightarrow0,$ we conclude that:
\begin{align*}
    \li \rdv{\alpha}{\tilde{\nu}_{\dim}}{\nu_{\dim}} \geq \rdv{\alpha}{\gauss{\sigma_\beta}{1}}{\gauss{0}{1}},
\end{align*}
which proves the lower bound in \eqref{eq:achieve-renyi}. This completes the Rényi differential privacy analysis and the proof of Theorem \mref{thm:privacy}. 
\end{proof}

We now present the proofs of the intermediate results Lemma \mref{lem:simple-gibbs-approx}, Lemma \mref{lem:triangle-ineq}, and Proposition \mref{prop:var-lim-v2}, which were introduced in Section \mref{sec:privacy} and used in our privacy proof. 

\subsection{Proof of Lemma \mref{lem:simple-gibbs-approx}}
\begin{proof}[Proof of Lemma \mref{lem:simple-gibbs-approx}] Recall that our goal is to show that the total variation distance and Rényi divergence between the Gibbs distributions $\nu(\cdot \mid \mSigma(\tilde{\mX}), \beta_\dim, \rnk)$ and ${\nu(\cdot \mid \bar{\mSigma}(\tilde{\mX}), \beta_\dim, \rnk)}$ vanish in the limit as $\dim \rightarrow \infty$. To show this, we will appeal to Theorem \mref{thm:contiguity} (item (2)). Note that all the assumptions of Theorem \mref{thm:contiguity} are met:
\begin{itemize}
    \item The perturbation matrix $\mE \bydef \sqrt{\dim} \cdot ( \bar{\mSigma}(\tilde{\mX}) -  \mSigma(\tilde{\mX}))$ vanishes in operator norm and is bounded in Frobenius norm. Indeed, recall that from \meqref{eq:privacy-perturb-orig} and \meqref{eq:simple-cov} that:
    \begin{subequations}
    \begin{align}  \label{eq:recall-all-cov}
         \Sigma(\tilde{\mX}) & = \begin{dcases} \Sigma(\mX)-{\frac{\Sigma(\mX)}{\ssize+1}}+\frac{\vx\vx^\top}{\ssize+1} & \text{ if }\tilde{\mX} = \mX \cup \{\vx\}, \\ 
    \mSigma(\mX) + {\frac{\Sigma(\mX)}{\ssize-1}} -\frac{\vx\vx^\top}{\ssize-1}  &\text{ if } \tilde{\mX} = \mX \backslash \{\vx\},
    \end{dcases}
    \end{align}
    and 
    \begin{align}
    \bar{\mSigma}(\tilde{\mX})\bydef\begin{dcases}
        \mSigma(\mX)+\frac{\vx\vx^\top}{\ssize}&\text{if }\tilde{\mX} = \mX \cup \{\vx\}\\
        \mSigma(\mX)-\frac{\vx\vx^\top}{\ssize}&\text{if }\tilde{\mX} = \mX \backslash \{\vx\}.
    \end{dcases}
    \end{align}
    \end{subequations}
    Hence, 
    \begin{align} \label{eq:E-bound}
    \|\mE \|_{\fr} \leq \sqrt{\dim} \cdot \left( \frac{\sqrt{\dim}\lambda_1(\mSigma(\mX))}{\ssize-1}+\frac{\|\vx\|^2}{\ssize(\ssize-1)} \right) \lesssim \frac{{\dim}}{\ssize} + \frac{\dim \sqrt{\dim}}{\ssize^2}\lesssim \frac{1}{\sqrt{\dim}} \ll 1. 
    \end{align}
    \item The matrices $\mSigma(\tilde{\mX})$ and $\bar{\mSigma}(\tilde{\mX})$ both satisfy Assumption \mref{assump:mat}. This is because the matrix $\mSigma(\mX)$ satisfies Assumption \mref{assump:mat} (since the dataset $\mX$ satisfies Assumption \mref{assump:data}) and both $\bar{\mSigma}(\tilde{\mX})$ and $\mSigma(\tilde{\mX})$ are close to $\mSigma(\mX)$ in operator norm. Indeed,
    \begin{align}
        \|\mSigma(\tilde{\mX})-\mSigma(\mX)\| & \explain{\eqref{eq:recall-all-cov}}{\leq} \frac{\lambda_1(\mSigma(\mX))}{\ssize-1}+\frac{\|\vx\|^2}{\ssize-1} \lesssim \frac{\dim}{\ssize} \ll 1, \label{eq:sigma-tilde-sigma}\\  \|\bar{\mSigma}(\tilde{\mX})-\mSigma(\mX)\| &\leq   \|\bar{\mSigma}(\tilde{\mX})-\mSigma(\tilde{\mX})\| +  \|\mSigma(\tilde{\mX})-\mSigma(\mX)\| \explain{\eqref{eq:E-bound}, \eqref{eq:sigma-tilde-sigma}}{\ll} 1. 
    \end{align}
Hence, \Cref{lem:misc_conv} guarantees that  the matrices $\mSigma(\tilde{\mX})$ and $\bar{\mSigma}(\tilde{\mX})$ also satisfy Assumption \mref{assump:mat}.
\end{itemize}
Since all the requirements of Theorem \mref{thm:contiguity} (item (2)) are met, we conclude that
\begin{align*}
     \lm \tv(\nu(\cdot \mid \mSigma(\tilde{\mX}), \beta_\dim, \rnk), \nu(\cdot \mid \bar{\mSigma}(\tilde{\mX}), \beta_\dim, \rnk)) &=0, \\ \lm \rdv{\alpha}{\nu(\cdot \mid \mSigma(\tilde{\mX}), \beta_\dim, \rnk)}{\nu(\cdot \mid \bar{\mSigma}(\tilde{\mX}), \beta_\dim, \rnk)} &=0\quad\forall\:\alpha>1, \\
    \lm  \rdv{\alpha}{\nu(\cdot \mid \bar{\mSigma}(\tilde{\mX}), \beta_\dim, \rnk)}{\nu(\cdot \mid \mSigma(\tilde{\mX}), \beta_\dim, \rnk)}&=0\quad\forall\:\alpha>1,
\end{align*}
as claimed in Lemma \mref{lem:simple-gibbs-approx}. 
\end{proof}

\subsection{Proof of Lemma \mref{lem:triangle-ineq}}\label{app:triangle-ineq}
\begin{proof}[Proof of Lemma \mref{lem:triangle-ineq}] Let $\nu,$ $\tilde{\nu},$ and $\bar{\nu}$ be three probability measures on a common sample space. Let $\chi$ be a measure on the same space such that all three measures $\nu,$ $\tilde{\nu},$ and $\bar{\nu}$ are absolutely continuous with respect to $\chi$ (for instance, take  $\chi = (\nu + \tilde{\nu} + \bar{\nu})/3$). We consider and prove each claim made in Lemma \mref{lem:triangle-ineq} individually. 

\paragraph{Proof of Claim (1) in Lemma \mref{lem:triangle-ineq}} We first prove the perturbation bound for the trade-off function: for any $\alpha\in[0,1],$ 
    \begin{align}\label{eq:wti-toshow1}
        \left\vert\tf{\nu}{\tilde{\nu}}(\alpha)-\tf{\nu}{\bar{\nu}}(\alpha)\right\vert\leq \tv(\tilde{\nu},\bar{\nu}).
    \end{align}
Fix any $\alpha\in[0,1],$ and let $\tilde{\phi}$ and $\bar{\phi}$ denote the (optimal) level-$\alpha$ Neyman-Pearson tests for the hypothesis testing problems 
\begin{align}\label{eq:ht1}
    H_0:\rV\sim \nu\quad\text{v.s.}\quad\tilde{H}_1:\rV\sim\tilde{\nu}
\end{align}
and
\begin{align}\label{eq:ht2}
    H_0:\rV\sim \nu\quad\text{v.s.}\quad\bar{H}_1:\rV\sim\bar{\nu},
\end{align}
respectively. Let $\E_{\tilde{\nu}}$ and $\E_{\bar{\nu}}$ denote that the expectation is taken with respect to $\tilde{\nu}$ and $\bar{\nu},$ respectively. Then, by the definition of the trade-off function, we have:
\begin{align*}
    \tf{\nu}{\tilde{\nu}}(\alpha)=1-\E_{\tilde{\nu}}[\tilde{\phi}]\quad\text{and}\quad \tf{\nu}{\bar{\nu}}(\alpha)=1-\E_{\bar{\nu}}[\bar{\phi}].
\end{align*}
Observe that
\begin{align*}
    \E_{\tilde{\nu}}[\tilde{\phi}]-\E_{\bar{\nu}}[\bar{\phi}]=\E_{\tilde{\nu}}[\tilde{\phi}]-\E_{\bar{\nu}}[\tilde{\phi}]+\underbrace{\E_{\bar{\nu}}[\tilde{\phi}]-\E_{\bar{\nu}}[\bar{\phi}]}_{(\#)}\explain{(a)}{\leq} \E_{\tilde{\nu}}[\tilde{\phi}]-\E_{\bar{\nu}}[\tilde{\phi}]\explain{(b)}{\leq} \tv(\tilde{\nu},\bar{\nu})
\end{align*}
and 
\begin{align*}
    \E_{\bar{\nu}}[\bar{\phi}]-\E_{\tilde{\nu}}[\tilde{\phi}]=\E_{\bar{\nu}}[\bar{\phi}]-\E_{\tilde{\nu}}[\bar{\phi}]+\underbrace{\E_{\tilde{\nu}}[\bar{\phi}]-\E_{\tilde{\nu}}[\tilde{\phi}]}_{(\#)}
    \explain{(a)}{\leq} \E_{\bar{\nu}}[\bar{\phi}]-\E_{\tilde{\nu}}[\bar{\phi}]\explain{(b)}{\leq} \tv(\tilde{\nu},\bar{\nu}),
\end{align*}
where the inequalities marked (a) hold since the terms marked $(\#)$ are non-positive by the definition of $\tilde{\phi}$ and $\bar{\phi}$ being Neyman-Pearson tests for \eqref{eq:ht1} and \eqref{eq:ht2}, respectively, and the inequalities marked (b) hold since for any map $\psi:\mathcal{V}\rightarrow[0,1],$ we have: 
\begin{align*}
    |\E_{\bar{\nu} }[\psi ]-\E_{\tilde{\nu} }[\psi ]|=\left\vert\int \psi \diff\bar{\nu}-\int\psi \diff \tilde{\nu} \right\vert\explain{def}{\leq}\tv(\bar{\nu},{\tilde{\nu}})&&\text{[since $\psi\in[0,1]$]}.
\end{align*}
Hence, we conclude that:
\begin{align*}
    \left\vert\tf{\nu}{\tilde{\nu}}(\alpha)-\tf{\nu}{\bar{\nu}}(\alpha)\right\vert=\vert\E_{\tilde{\nu}}[\tilde{\phi}]-\E_{\bar{\nu}}[\bar{\phi}]\vert\leq \tv(\bar{\nu},{\tilde{\nu}}),
\end{align*}
as desired. 
\paragraph{Proof of Claim (2) in Lemma \mref{lem:triangle-ineq}} Next, we prove the claimed perturbation bound for Rényi divergence: for any $\alpha>1$ and $\epsilon\in(0,\alpha-1),$ 
    \begin{align}\label{eq:wti-toshow2}
        \rdv{\alpha-\epsilon}{\bar{\nu}}{\nu}-\frac{\alpha}{\alpha-1}\rdv{\frac{(\alpha-\epsilon)(\alpha-1)}{\epsilon}}{\bar{\nu}}{\tilde{\nu}}\leq \rdv{\alpha}{\tilde{\nu}}{\nu}\leq \rdv{\alpha+\epsilon}{\bar{\nu}}{\nu}+\frac{\alpha+\epsilon}{\alpha+\epsilon-1}\rdv{\frac{\alpha(\alpha+\epsilon-1)}{\epsilon}}{\tilde{\nu}}{\bar{\nu}}.
    \end{align}
This bound is essentially due to \citet{mironov2017renyi}.
\begin{fact}[{\citet[Proposition 11]{mironov2017renyi}}]\label{fact:weak-triangle}
    Let $\nu,$ $\tilde{\nu},$ and $\bar{\nu}$ be probability measures on some common space $\mathcal{V}.$ Then, for any $\alpha>1,$ 
    \begin{align*}
    \rdv{\alpha}{\tilde{\nu}}{\nu}\leq\rdv{\alpha+\epsilon}{\bar{\nu}}{\nu}+\frac{\alpha+\epsilon}{\alpha+\epsilon-1}\rdv{\frac{\alpha(\alpha+\epsilon-1)}{\epsilon}}{\tilde{\nu}}{\bar{\nu}}\quad\forall\:\alpha>1\quad\forall\:\epsilon>0.
\end{align*}
\end{fact} 
Observe that \Cref{fact:weak-triangle} gives the upper bound in \eqref{eq:wti-toshow2}. Applying \Cref{fact:weak-triangle} by swapping $\tilde{\nu}$ and $\bar{\nu}$ and  reparameterizing $\alpha+\epsilon$ as $\alpha,$ we have:  
\begin{align*}
    \rdv{\alpha}{\tilde{\nu}}{\nu}\geq \rdv{\alpha-\epsilon}{\bar{\nu}}{\nu}-\frac{\alpha}{\alpha-1}\rdv{\frac{(\alpha-\epsilon)(\alpha-1)}{\epsilon}}{\bar{\nu}}{\tilde{\nu}}\quad\forall\:\alpha>1\quad\forall\: \epsilon\in(0,\alpha-1),
\end{align*}
which then implies the lower bound in \eqref{eq:wti-toshow2}.
\end{proof}

\subsection{Proof of \texorpdfstring{Proposition \mref{prop:var-lim-v2}}{prop:var-lim}}
\begin{proof}[Proof of Proposition \mref{prop:var-lim-v2}] Recall that for a sequence of datasets $\mX$ which satisfies Assumption \mref{assump:data} and any sequence $\beta_\dim \rightarrow \beta \in (H_\mu(\gamma_\rnk), \infty)$, our goal is to show that:
\begin{align} \label{eq:var-func-max-v2-recall}
        \lim_{\dim \rightarrow \infty} \sup_{\substack{\vx \in \R^p \\ \|\vx\|^2 \leq \dim}}\sigma_{\mX}^2(\vx,\beta_\dim)=\begin{cases} \frac{1}{2 \Delta \theta^2} \cdot \frac{\left(\beta-H_{\mu}(\gamma_\rnk)\right)^2}{2({\beta-H_{\mu}(\gamma_\rnk)})+ \Delta H'_{\mu}(\gamma_\rnk)}&\text{if}\quad \beta \geq -\Delta H'_{\mu}(\gamma_\rnk)+H_{\mu}(\gamma_{\rnk})\\
        -\frac{1}{2\theta^2}H'_{\mu}(\gamma_\rnk)&\text{if}\quad \beta<-\Delta H'_{\mu}(\gamma_\rnk)+H_{\mu}(\gamma_{\rnk})
    \end{cases}.
\end{align}
where the variance function $\sigma_{X,\beta}^2: \R^{\dim}\times (H_{\mu}(\gamma_{\rnk}),\infty) \rightarrow \R$ is given by:
\begin{align}
    \sigma_{\mX}^2(\vx,\beta) &\explain{def}{=}\sigma^2_{\mSigma(\mX)}\left(\frac{\sqrt{\dim}\vx\vx^\top}{\ssize},\beta\right)\notag\\
    &=\frac{\dim}{2\ssize^2} \sum_{j, \ell = 1}^\rnk K_X(\lambda_j, \lambda_\ell)\ip{\vx}{\vu_j}^2 \ip{\vx}{\vu_\ell}^2\notag\\
    &\qquad+ \frac{\dim}{\ssize^2}\sum_{j=1}^\rnk \sum_{i= 1}^{\dim - \rnk} \frac{\beta - {H_{\mSigma(\mX)}(\lambda_j)}}{\lambda_j - \lambda_{\rnk+i}}  \ip{\vx}{\vu_j}^2 \ip{\vx}{\vu_{\rnk+i}}^2.\label{eq:var-func-recall} 
\end{align} Additionally,  if $\vu_1, \vu_2, \dotsc, \vu_\dim$ denote the eigenvectors of $\mSigma(\mX)$, we need to show that the data point:
\begin{align*}
    \vx_\star & = \sqrt{\dim} \left(  \sqrt{t_\star}  \vu_{\rnk} + \sqrt{1-t_\star}  \vu_{\rnk+1} \right) \; \text{ with } \;  t_\star \explain{def}{=}\begin{cases} \frac{{\beta-H_{\mu}(\gamma_\rnk)}}{2(\beta-H_{\mu}(\gamma_\rnk)) + \Delta H'_{\mu}(\gamma_\rnk)}&\text{if}\quad \beta\geq H_{\mu}(\gamma_{\rnk})-\Delta H'_{\mu}(\gamma_\rnk)\\
   1&\text{if}\quad \beta < H_{\mu}(\gamma_{\rnk})-\Delta H'_{\mu}(\gamma_\rnk).
   \end{cases}
\end{align*}
asymptotically achieves the maximum in \eqref{eq:var-func-max-v2-recall}: 
\begin{align*}
    \lim_{\dim \rightarrow \infty} \sigma_{\mX}^2(\vx_\star,\beta_\dim) & = \lim_{\dim \rightarrow \infty} \sup_{\substack{\vx \in \R^p \\ \|x\|^2 \leq \dim}}\sigma_{\mX}^2(\vx,\beta).
\end{align*}
We will prove these claims in three steps. First, we will show that the high-dimensional optimization problem in \eqref{eq:var-func-max-v2-recall} over $\vx \in \R^\dim$ can be reduced to a simpler low-dimensional optimization problem over $\R^\rnk$. Next, we will use this simplification to show the upper bound:
\begin{align} \label{eq:UB}
     \lim_{\dim \rightarrow \infty} \sup_{\substack{\vx \in \R^p \\ \|x\|^2 \leq \dim}}\sigma_{\mX}^2(\vx,\beta_\dim) \leq \begin{cases} \frac{1}{2 \Delta \theta^2} \cdot \frac{\left(\beta-H_{\mu}(\gamma_\rnk)\right)^2}{2({\beta-H_{\mu}(\gamma_\rnk)})+ \Delta H'_{\mu}(\gamma_\rnk)}&\text{if}\quad \beta \geq -\Delta H'_{\mu}(\gamma_\rnk)+H_{\mu}(\gamma_{\rnk})\\
        -\frac{1}{2\theta^2}H'_{\mu}(\gamma_\rnk)&\text{if}\quad \beta<-\Delta H'_{\mu}(\gamma_\rnk)+H_{\mu}(\gamma_{\rnk}).
    \end{cases}
\end{align}
Lastly, we will prove that the above upper bound is asymptotically tight by showing that the worst-case data point $\vx_{\star}$ attains it:
\begin{align} \label{eq:achievability}
    \lim_{\dim \rightarrow \infty} \sigma_{\mX}^2(\vx_\star,\beta_\dim) & = \begin{cases} \frac{1}{2 \Delta \theta^2} \cdot \frac{\left(\beta-H_{\mu}(\gamma_\rnk)\right)^2}{2({\beta-H_{\mu}(\gamma_\rnk)})+ \Delta H'_{\mu}(\gamma_\rnk)}&\text{if}\quad \beta \geq -\Delta H'_{\mu}(\gamma_\rnk)+H_{\mu}(\gamma_{\rnk})\\
        -\frac{1}{2\theta^2}H'_{\mu}(\gamma_\rnk)&\text{if}\quad \beta<-\Delta H'_{\mu}(\gamma_\rnk)+H_{\mu}(\gamma_{\rnk}).
    \end{cases}
\end{align}
Observe that the claims \eqref{eq:UB} and \eqref{eq:achievability} together imply Proposition \mref{prop:var-lim-v2}. 

\paragraph{Step 1: Simplification to a Low-Dimensional Optimization Problem} Recalling the formula for $\sigma^2_{\mX}(\vx,\beta)$ from \eqref{eq:var-func-recall}, we find that:
\begin{align} \label{eq:opt-problem}
  &\sup_{\substack{\vx \in \R^p \\ \|x\|^2 \leq \dim}}\sigma_{\mX}^2(\vx,\beta_\dim)  = \sup_{\substack{\vx \in \R^p \\ \|x\| \leq 1 }}  \frac{\dim^3}{2\ssize^2}\sum_{j, \ell = 1}^\rnk K_X(\lambda_j, \lambda_\ell)\vx_j^2 \vx_{\ell}^2  + \frac{\dim^3}{\ssize^2}\sum_{j=1}^\rnk \vx_j^2 \underbrace{\sum_{i= 1}^{\dim - \rnk} \frac{\beta_\dim - H_{\mSigma(\mX)}(\lambda_j)}{\lambda_j - \lambda_{\rnk+i}}   \vx_{\rnk+i}^2}_{(\#)},
\end{align}
where we reformulated the optimization problem using the one-to-one change of variables:
\begin{align*}
    \vx = (\vx_1, \vx_2, \dotsc, \vx_\dim)^\top \mapsto (\ip{\vx}{\vu_1}, \ip{\vx}{\vu_2}, \dotsc, \ip{\vx}{u_\dim})^\top/\sqrt{\dim}.
\end{align*}
Note that once $\vx_1, \vx_2, \dotsc, \vx_\rnk$ are fixed the objective function in \eqref{eq:opt-problem} depends on $(\vx_{\rnk+1}, \vx_{\rnk+2}, \dotsb \vx_{\dim})$ only via the term $(\#)$. For a fixed choice of $\vx_1, \vx_2, \dotsc, \vx_\rnk$, the term $(\#)$ is maximized by setting:
\begin{align*}
    (x_{\rnk+1}, x_{\rnk+2}, \dotsc, x_{\dim}) = \left(1-\sum_{i=1}^k \vx_i^2 \right)^{1/2} \cdot  (1, 0, 0, \dotsc, 0),
\end{align*}
since for any $j \in [\rnk]:$\footnote{Note that when $\beta > H_\mu(\gamma_k)$ (as assumed in Proposition \mref{prop:var-lim-v2}), for sufficiently large $\dim$, $\beta_{\dim}-H_{\mSigma(\mX)}(\lambda_j) \geq 0$ for any $j \in [k]$ since $\beta_{\dim}\rightarrow\beta$ and $\lim_{\dim \rightarrow \infty} H_{\mSigma(\mX)}(\lambda_j) \rightarrow H_{\mu}(\gamma_j) \leq H_{\mu}(\gamma_\rnk)$ (see \Cref{lem:misc_HK}).}
\begin{align*}
    \max_{i \in [\dim - \rnk]}  \frac{\beta_\dim - H_{\mSigma(\mX)}(\lambda_j)}{\lambda_j - \lambda_{\rnk+i}}  & =  \frac{\beta_\dim - H_{\mSigma(\mX)}(\lambda_j)}{\lambda_j - \lambda_{\rnk+1}}. 
\end{align*}
Hence,
\begin{align} \label{eq:reduced-opt}
    \lim_{\dim \rightarrow \infty}\sup_{\substack{\vx \in \R^{\dim} \\ \|x\|^2 \leq \dim}} \sigma_{\mX}^2(\vx,\beta) &=   \lim_{\dim \rightarrow \infty}\sup_{\substack{\vx \in \R^\rnk \\ \|x\| \leq 1 }} \frac{\dim^3}{2\ssize^2}\sum_{j, \ell = 1}^\rnk K_X(\lambda_j, \lambda_\ell)\vx_j^2 \vx_{\ell}^2 \notag\\
    &\hspace{4cm}+ \frac{\dim^3}{\ssize^2}(1-\|\vx\|^2)\cdot \sum_{j=1}^\rnk \vx_j^2\cdot{\frac{\beta_\dim - H_{\mSigma(\mX)}(\lambda_j)}{\lambda_j - \lambda_{\rnk+1}}}. 
\end{align}
We can further simplify the above optimization problem by letting $\dim \rightarrow \infty$. Recall that $\beta_\dim \rightarrow \beta$ and for $j,\ell \in [\rnk]$ (see \Cref{lem:misc_HK}),
\begin{subequations}\label{eq:lim-H-K}
\begin{align} 
H_{\mSigma(\mX)}(\lambda_j) &\explain{def}{=} \frac{1}{\dim} \sum_{i = 1}^{\dim - \rnk} \frac{1}{\lambda_j - \lambda_{\rnk+i}} \rightarrow H_\mu(\gamma_j) \explain{def}{=} \int_{\R} \frac{\diff \mu(\lambda) }{\gamma_j - \lambda}, \\
    K_{\mX}(\lambda_j, \lambda_\ell) &\explain{def}{=} \frac{1}{\dim} \sum_{i = 1}^{\dim - \rnk} \frac{1}{(\lambda_j - \lambda_{\rnk+i})(\lambda_\ell - \lambda_{\rnk+i})} \rightarrow K_\mu(\gamma_j, \gamma_\ell) \explain{def}{=} \int_{\R} \frac{\diff \mu(\lambda) }{(\gamma_j - \lambda)(\gamma_\ell - \lambda)}.
\end{align}
\end{subequations}
Hence, the objective function from the optimization problem \eqref{eq:reduced-opt}:
\begin{align*}
    f_{X,\beta_\dim}(\vx) \explain{def}{=}  \frac{\dim^3}{2\ssize^2}\sum_{j, \ell = 1}^\rnk K_X(\lambda_j, \lambda_\ell)\vx_j^2 \vx_{\ell}^2 + \frac{\dim^3}{\ssize^2}(1-\|\vx\|^2) \cdot \sum_{j=1}^\rnk \vx_j^2 \cdot {\frac{\beta_\dim - H_{\mSigma(\mX)}(\lambda_j)}{\lambda_j - \lambda_{\rnk+1}}}, 
\end{align*}
converges point-wise to the limiting objective function:
\begin{align*}
    f_{\beta}(\vx) \explain{def}{=}  \frac{1}{2\theta^2}\sum_{j, \ell = 1}^\rnk K_\mu(\gamma_j, \gamma_\ell)\vx_j^2 \vx_{\ell}^2 + \frac{1}{\theta^2}(1-\|\vx\|^2) \cdot\sum_{j=1}^\rnk \vx_j^2 \cdot {\frac{\beta - H_{\mu}(\gamma_j)}{\gamma_j - \gamma_{\rnk+1}}}.
\end{align*}
In fact, this convergence is uniform over the domain of the optimization problem in \eqref{eq:reduced-opt} since by the triangle inequality:
\begin{align*}
    \sup_{\substack{\vx \in \R^p \\ \|x\| \leq 1 }}\;\left\vert f_{\mX,\beta_\dim}(\vx)-f_\beta(\vx)\right\vert &\leq \sup_{\substack{\vx \in \R^p \\ \|x\| \leq 1 }}\; \Bigg(  \frac{\dim^3}{2\ssize^2}\sum_{j, \ell = 1}^\rnk x_j^2 x_\ell^2 \left\vert K_X(\lambda_j, \lambda_\ell)-K_\mu(\gamma_j, \gamma_\ell)\right\vert   \\ &\qquad+ \frac{\dim^3}{\ssize^2}(1-\|x\|^2)\cdot\sum_{j=1}^\rnk x_j^2 \left\vert\frac{\beta_\dim - H_{\mSigma(\mX)}(\lambda_j)}{\lambda_j - \lambda_{\rnk+1}}-\frac{\beta - H_{\mu}(\gamma_j)}{\gamma_j - \gamma_{\rnk+1}}\right\vert \Bigg) \\
    & \explain{(a)}{\leq} \frac{\dim^3}{2\ssize^2}\sum_{j, \ell = 1}^\rnk  \left\vert K_X(\lambda_j, \lambda_\ell)-K_\mu(\gamma_j, \gamma_\ell)\right\vert \\
    &\qquad+ \frac{\dim^3}{\ssize^2}\sum_{j=1}^\rnk  \left\vert\frac{\beta_\dim - H_{\mSigma(\mX)}(\lambda_j)}{\lambda_j - \lambda_{\rnk+1}}-\frac{\beta - H_{\mu}(\gamma_j)}{\gamma_j - \gamma_{\rnk+1}}\right\vert \\
    & \explain{\eqref{eq:lim-H-K}}{\rightarrow} 0 \quad \text{ as } \dim \rightarrow \infty,
\end{align*}
where (a) holds by observing that $|\vx_j| \leq \|\vx\| \leq 1$ for any $j\in[\rnk]$. Thanks to this uniform convergence, we can safely replace the objective function in \eqref{eq:reduced-opt} by its limit and obtain:
\begin{align} \label{eq:high-to-low}
     &\lim_{\dim \rightarrow \infty}\sup_{\substack{\vx \in \R^p \\ \|\vx\|^2 \leq \dim}} \sigma_{\mX}^2(\vx,\beta_\dim) = \sup_{\substack{\vx \in \R^\rnk \\ \|x\| \leq 1 }}  \frac{1}{2\theta^2}\sum_{j, \ell = 1}^\rnk K_\mu(\gamma_j, \gamma_\ell)\vx_j^2 \vx_{\ell}^2+\frac{1}{\theta^2}(1-\|\vx\|^2)\cdot \sum_{j=1}^\rnk {\frac{\beta - H_{\mu}(\gamma_j)}{\gamma_j - \gamma_{\rnk+1}}} \vx_j^2,
\end{align}
We have now reduced the high-dimensional optimization problem on the LHS of \eqref{eq:high-to-low} to a low-dimensional one over $\R^\rnk$ on the RHS. 
% \begin{lemma}\label{lem:high-to-low} We have,
% \begin{align} \label{eq:high-to-low}
%     \lim_{\dim \rightarrow \infty} \sup_{\substack{\vx \in \R^p \\ \|x\|^2 \leq \dim}}\sigma_{\mX}^2(\vx,\beta) = \sup_{\substack{\vx \in \R^\rnk \\ \|x\| \leq 1 }}  \frac{1}{4\theta^2} \sum_{j, \ell = 1}^\rnk K_\mu(\gamma_j, \gamma_\ell)\vx_j^2 \vx_{\ell}^2  + \frac{1}{4\theta^2} \sum_{j=1}^\rnk  K_\mu(\gamma_j, \gamma_j) \vx_j^4 + \frac{1-\|\vx\|^2}{\theta^2}\sum_{j=1}^\rnk {\frac{\beta - H_{\mu}(\gamma_j)}{\gamma_j - \gamma_{\rnk+1}}} \vx_j^2,
% \end{align}
% where the functions $H_\mu, K_\mu$ are defined as:
% \begin{align} \label{eq:K-def}
%      H_\mu(\gamma) \explain{def}{=} \int_{\R} \frac{\diff \mu(\lambda) }{\gamma - \lambda}, \quad K_\mu(\gamma, \gamma^\prime) \explain{def}{=} \int_{\R} \frac{\diff \mu(\lambda) }{(\gamma - \lambda)(\gamma^\prime - \lambda)} \quad \forall \; \gamma,\gamma^\prime \in [\gamma_{\rnk},\infty).
% \end{align}
% \end{lemma}

% To analyze the low-dimensional optimization problem on the RHS of \eqref{eq:high-to-low}, we will rely on the following monotonicity property of the function $\gamma \mapsto (\beta - H_\mu(\gamma)/(\gamma - \gamma_{\rnk+1})$.

% \begin{lemma}\label{lem:mono} When $\beta \geq H_{\mu}(\gamma_{\rnk})-\Delta H'_{\mu}(\gamma_\rnk)$, the function  $\gamma \mapsto (\beta - H_\mu(\gamma))/(\gamma - \gamma_{\rnk+1})$ is a non-increasing function on the domain $[\gamma_\rnk,\infty)$.
% \end{lemma}

% We now present the proof of \Cref{prop:var-lim-v2} taking these intermediate results for granted. 
\paragraph{Step 2: Proof of the Upper Bound \eqref{eq:UB}} We will split our proof of the upper bound in \eqref{eq:UB} into two cases: when $\beta \geq H_{\mu}(\gamma_{\rnk})-\Delta H'_{\mu}(\gamma_\rnk)$ and when $\beta < H_{\mu}(\gamma_{\rnk})-\Delta H'_{\mu}(\gamma_\rnk)$.
\begin{description}
    \item[Case 1: $\beta \geq H_{\mu}(\gamma_{\rnk})-\Delta H'_{\mu}(\gamma_\rnk)$.]  Thanks to \eqref{eq:high-to-low}, we have:
\begin{align}
     &\lim_{\dim \rightarrow \infty}\sup_{\substack{\vx \in \R^p \\ \|x\|^2 \leq \dim}} \sigma_{\mX}^2(\vx,\beta) = \sup_{\substack{\vx \in \R^\rnk \\ \|x\| \leq 1 }}  \frac{1}{2\theta^2}\sum_{j, \ell = 1}^\rnk K_\mu(\gamma_j, \gamma_\ell)\vx_j^2 \vx_{\ell}^2  + \frac{1}{\theta^2}(1-\|\vx\|^2)\cdot \sum_{j=1}^\rnk {\frac{\beta - H_{\mu}(\gamma_j)}{\gamma_j - \gamma_{\rnk+1}}} \vx_j^2 \nonumber \\
     & = \sup_{t \in [0,1]} \sup_{\substack{\vx \in \R^\rnk \\ \|x\| = 1}} \frac{1}{2\theta^2}t^2 \underbrace{\sum_{j, \ell = 1}^\rnk K_\mu(\gamma_j, \gamma_\ell)\vx_j^2 \vx_{\ell}^2}_{(i)}  +  \frac{1}{\theta^2}t(1-t)\underbrace{\sum_{j=1}^\rnk {\frac{\beta - H_{\mu}(\gamma_j)}{\gamma_j - \gamma_{\rnk+1}}} \vx_j^2}_{(ii)}, \label{eq:cont-analysis}
\end{align}
where in the last step we split the maximization over $\vx \in \R^k$ with $\|\vx\|\leq 1$ into a maximization over the squared norm of $\vx$ (denoted by $t \in [0,1]$) and a maximization over the unit vector in the direction of $\vx$ (denoted again by $\vx$ for convenience).  We will bound each term of the two summations $(i),(ii)$ by the maximum term involved in the summation. For term $(ii)$, we will find the following intermediate claim useful.
\begin{claim}\label{claim:mono} When $\beta \geq H_{\mu}(\gamma_{\rnk})-\Delta H'_{\mu}(\gamma_\rnk)$, the function  $\gamma \mapsto (\beta - H_\mu(\gamma))/(\gamma - \gamma_{\rnk+1})$ is a non-increasing function on the domain $[\gamma_\rnk,\infty)$.
\end{claim}
The proof of the above claim follows by studying the first two derivatives of the function $\gamma \mapsto (\beta - H_\mu(\gamma))/(\gamma - \gamma_{\rnk+1})$. We defer its proof to the end, and continue our analysis of \eqref{eq:cont-analysis} assuming this claim. The maximum terms involved in the summations $(i), (ii)$ in \eqref{eq:cont-analysis} are given by:
\begin{align*}
        \max_{j,\ell \in [\rnk]} K_\mu(\gamma_j, \gamma_\ell) & \explain{(a)}{=} K_\mu(\gamma_\rnk, \gamma_\rnk), \quad \max_{j \in [\rnk]} \frac{\beta - H_{\mu}(\gamma_j)}{\gamma_j - \gamma_{\rnk+1}} \explain{(b)}{=} \frac{\beta - H_{\mu}(\gamma_\rnk)}{\gamma_\rnk - \gamma_{\rnk+1}},
     \end{align*}
     where (a) follows by observing that $K_\mu$ is a coordinate-wise decreasing function (recall \eqref{eq:lim-H-K}) and (b) follows from \Cref{claim:mono}. Hence,
     \begin{align}
         &\lim_{\dim \rightarrow \infty}\sup_{\substack{\vx \in \R^p \\ \|x\|^2 \leq \dim}} \sigma_{\mX}^2(\vx,\beta)  \nonumber \\ 
         &\quad  \leq \sup_{t \in [0,1]} \sup_{\substack{\vx \in \R^\rnk \\ \|x\| = 1}}\frac{1}{2\theta^2}t^2 K_\mu(\gamma_\rnk,\gamma_\rnk)\cdot \|\vx\|^4+ \frac{1}{\theta^2}t(1-t) \cdot \frac{\beta - H_{\mu}(\gamma_\rnk)}{\gamma_\rnk - \gamma_{\rnk+1}} \cdot \|x\|^2 \nonumber\\
         & \quad = \sup_{t \in [0,1]} \frac{1}{2\theta^2}t^2 K_\mu(\gamma_\rnk,\gamma_\rnk)  + \frac{1}{\theta^2}t(1-t) \cdot \frac{\beta - H_{\mu}(\gamma_\rnk)}{\gamma_\rnk - \gamma_{\rnk+1}} \nonumber\\
         &\quad \explain{(a)}{=} \sup_{t \in [0,1]}  -\frac{1}{\theta^2}t^2 \cdot \left( \frac{\beta - H_\mu(\gamma_k)}{\Delta} + \frac{H_\mu^\prime(\gamma_k)}{2} \right) + \frac{1}{\theta^2}t \cdot \frac{\beta - H_\mu(\gamma_k)}{\Delta},\label{eq:qp}
     \end{align}
     where (a) holds since $\Delta = \gamma_{\rnk} - \gamma_{\rnk+1}, H_\mu^\prime(\gamma) = - K_\mu(\gamma,\gamma).$ Note that the above upper bound is tight, as the inequality in the equations is attained when $\vx = (0, \dotsc, 0, 1)^\top$. When $\beta \geq H_{\mu}(\gamma_{\rnk})-\Delta H'_{\mu}(\gamma_\rnk)$,
\begin{align*}
    \frac{\beta - H_\mu(\gamma_k)}{\Delta} + \frac{H_\mu^\prime(\gamma_k)}{2} \geq - \frac{H_\mu^\prime(\gamma_k)}{2} \bydef \frac{1}{2} \int_{\R} \frac{\diff \mu(\lambda)}{(\lambda - \gamma_k)^2} \geq 0.
\end{align*}
Hence, the objective function on the RHS of \eqref{eq:qp} is a concave quadratic function of $t$. We can further upper-bound it by maximizing this function over the entire real line instead of the interval $[0,1]$, which yields:
\begin{align*}
    \lim_{\dim \rightarrow \infty}\sup_{\substack{\vx \in \R^p \\ \|x\|^2 \leq \dim}} \sigma_{\mX}^2(\vx,\beta)  & \leq  \sup_{t \in [0,1]}  -\frac{1}{\theta^2}t^2 \cdot \left( \frac{\beta - H_\mu(\gamma_k)}{\Delta} + \frac{H_\mu^\prime(\gamma_k)}{2} \right) + \frac{1}{\theta^2}t \cdot \frac{\beta - H_\mu(\gamma_k)}{\Delta} \\
    & \leq  \sup_{t \in \R}  -\frac{1}{\theta^2}t^2 \cdot \left( \frac{\beta - H_\mu(\gamma_k)}{\Delta} + \frac{H_\mu^\prime(\gamma_k)}{2} \right) + \frac{1}{\theta^2}t \cdot \frac{\beta - H_\mu(\gamma_k)}{\Delta} \\
    & = \frac{1}{2 \Delta \theta^2} \cdot \frac{\left(\beta-H_{\mu}(\gamma_\rnk)\right)^2}{2({\beta-H_{\mu}(\gamma_\rnk)})+ \Delta H'_{\mu}(\gamma_\rnk)},
\end{align*}
 where the last step follows from the standard formula for the maximum of a concave quadratic function. This proves the upper bound \eqref{eq:UB} in the first case when $\beta \geq H_{\mu}(\gamma_{\rnk})-\Delta H'_{\mu}(\gamma_\rnk)$.      
\item[Case 2: $\beta < H_{\mu}(\gamma_{\rnk})-\Delta H'_{\mu}(\gamma_\rnk)$.] When $\beta$ is below the critical threshold $\beta_c \bydef H_{\mu}(\gamma_{\rnk})-\Delta H'_{\mu}(\gamma_\rnk)$, the proof of the upper bound \eqref{eq:UB} follows by simply observing that for any $\vx \in \R^\dim$, the function $\beta \mapsto \sigma_{\mX}^2(\vx,\beta)$ is a non-decreasing function of $\beta$ (see the formula for $\sigma_{\mX}^2(\vx,\beta)$ in \eqref{eq:var-func-recall}). Hence for any $\beta < \beta_c$,
\begin{align*}
   \lim_{\dim \rightarrow \infty}\sup_{\substack{\vx \in \R^p \\ \|x\|^2 \leq \dim}} \sigma_{\mX}^2(\vx,\beta) & \leq \lim_{\dim \rightarrow \infty}\sup_{\substack{\vx \in \R^p \\ \|x\|^2 \leq \dim}} \sigma_{\mX,\beta_c}^2(\vx)\\
   &\explain{Case 1}{=} \frac{1}{2 \Delta \theta^2} \cdot \frac{\left(\beta_c-H_{\mu}(\gamma_\rnk)\right)^2}{2({\beta_c-H_{\mu}(\gamma_\rnk)})+ \Delta H'_{\mu}(\gamma_\rnk)}\\
   &= -\frac{1}{2\theta^2}H'_{\mu}(\gamma_\rnk),
\end{align*}
which proves the upper bound \eqref{eq:UB} in the second case. 

\end{description}
\paragraph{Step 3: Achievability of the Upper Bound (Proof of \eqref{eq:achievability})} To complete the proof of Proposition \mref{prop:var-lim-v2}, we verify that the worst-case data point $\vx_\star = \sqrt{\dim} \cdot (\sqrt{t_\star} \vu_\rnk + \sqrt{1-t_\star} \vu_{\rnk+1})$ attains the upper bound in \eqref{eq:UB}:
\begin{align} \label{eq:achievability-recall}
    \lim_{\dim \rightarrow \infty} \sigma_{\mX}^2(\vx_\star,\beta_\dim) & = \begin{cases} \frac{1}{2 \Delta \theta^2} \cdot \frac{\left(\beta-H_{\mu}(\gamma_\rnk)\right)^2}{2({\beta-H_{\mu}(\gamma_\rnk)})+ \Delta H'_{\mu}(\gamma_\rnk)}&\text{if}\quad \beta \geq -\Delta H'_{\mu}(\gamma_\rnk)+H_{\mu}(\gamma_{\rnk})\\
        -\frac{1}{2\theta^2}H'_{\mu}(\gamma_\rnk)&\text{if}\quad \beta<-\Delta H'_{\mu}(\gamma_\rnk)+H_{\mu}(\gamma_{\rnk}).
    \end{cases}
\end{align}
Recalling the formula for the variance function $\sigma_{\mX,\beta}^2$ from \eqref{eq:var-func-recall}, we find that:
\begin{align}
\lim_{\dim\rightarrow\infty}\sigma_{\mX}^2(\vx_\star,\beta_\dim)&\explain{}{=}\lim_{\dim\rightarrow\infty}\frac{\dim^3}{2\ssize^2}t^2_\star K_X(\lambda_\rnk, \lambda_\rnk)+\frac{\dim^3}{\ssize^2}t_\star(1-t_\star) \cdot \frac{\beta_{\dim} - H_{\mSigma(\mX)}(\lambda_\rnk)}{\lambda_\rnk - \lambda_{\rnk+1}} \nonumber\\&\explain{(a)}{=}\frac{1}{2\theta^2}t^2_\star K_\mu(\gamma_\rnk, \gamma_\rnk)+\frac{1}{\theta^2}t_\star(1-t_\star) \cdot \frac{\beta - H_{\mu}(\gamma_\rnk)}{\Delta} \nonumber \\
& \explain{\eqref{eq:lim-H-K}}{=} -\frac{1}{2\theta^2}t^2_\star H_\mu^\prime(\gamma_\rnk)+\frac{1}{\theta^2}t_\star(1-t_\star)\cdot \frac{\beta - H_{\mu}(\gamma_\rnk)}{\Delta} \nonumber \\
& = - \frac{1}{2\theta^2}t_\star^2  \left(H_\mu^\prime(\gamma_\rnk) + \frac{2(\beta - H_\mu(\gamma_\rnk))}{\Delta} \right) + \frac{1}{\theta^2} t_\star \cdot  \frac{\beta - H_{\mu}(\gamma_\rnk)}{\Delta},  \label{eq:sub-here} 
\end{align}
where (a) follows from the fact that $\beta_{\dim}\rightarrow\beta$ and $\mX$ satisfies Assumption \mref{assump:data}, which guarantees $\lambda_\rnk - \lambda_{\rnk+1} \rightarrow \Delta$, $H_{\mSigma(\mX)}(\lambda_\rnk) \rightarrow H_\mu(\gamma_\rnk)$ and $K_{\mX}(\lambda_\rnk, \lambda_\rnk) \rightarrow K_\mu(\gamma_\rnk, \gamma_\rnk)$ (see \Cref{lem:misc_HK}).  Recall that: 
\begin{align*}
    t_\star &\explain{def}{=}\begin{cases} \frac{{\beta-H_{\mu}(\gamma_\rnk)}}{2(\beta-H_{\mu}(\gamma_\rnk)) + \Delta H'_{\mu}(\gamma_\rnk)}&\text{if}\quad \beta\geq H_{\mu}(\gamma_{\rnk})-\Delta H'_{\mu}(\gamma_\rnk)\\
   1&\text{if}\quad \beta < H_{\mu}(\gamma_{\rnk})-\Delta H'_{\mu}(\gamma_\rnk).
   \end{cases}
\end{align*}
Plugging the above formula in \eqref{eq:sub-here} immediately yields the desired conclusion \eqref{eq:achievability-recall}, completing the proof of Proposition \mref{prop:var-lim-v2}. 
\paragraph{Proof of \Cref{claim:mono}} Lastly, we prove \Cref{claim:mono}. For convenience, we define the function  of interest as $g_{\beta}$: 
\begin{align*}
    g_{\beta}(\gamma) \bydef \frac{\beta - H_{\mu}(\gamma)}{\gamma - \gamma_{\rnk+1}} \quad \forall \; \gamma \; \in \; [\gamma_\rnk,\infty).
\end{align*} We will show that $g_{\beta}$ is non-increasing by showing that its derivative:
\begin{align*}
    g^\prime_{\beta}(\gamma)&=\frac{H_\mu(\gamma)-\beta -(\gamma - \gamma_{\rnk+1})H^\prime_\mu(\gamma)}{(\gamma-\gamma_{\rnk+1})^2}
\end{align*}
is non-positive  for any $\gamma \geq \gamma_{\rnk}$. Since the denominator in the above formula is non-negative, it suffices to show that the numerator satisfies:
\begin{align} \label{eq:deriv-toshow}
    H_\mu(\gamma)-\beta -(\gamma - \gamma_{\rnk+1})H^\prime_\mu(\gamma) \leq 0 \quad \forall \; \gamma \geq \gamma_{\rnk}.
\end{align}
Note that the function on the LHS,  $\gamma \mapsto H_\mu(\gamma)-\beta -(\gamma - \gamma_{\rnk+1})H^\prime_\mu(\gamma)$ is non-increasing, which can be seen by inspecting its derivative:
\begin{align*}
    \frac{\diff}{\diff \gamma} \left( H_\mu(\gamma)-\beta -(\gamma - \gamma_{\rnk+1})H^\prime_\mu(\gamma) \right) &= - (\gamma - \gamma_{\rnk+1}) H_\mu^{\prime \prime}(\gamma)\\
    &= -  2(\gamma - \gamma_{\rnk+1}) \int_{\R} \frac{\diff \mu (\lambda)}{(\gamma - \lambda)^3} \\
    &\leq 0 
\end{align*}
for all $\gamma \geq \gamma_\rnk.$ Hence, \eqref{eq:deriv-toshow} holds for any $\gamma \geq \gamma_\rnk$ if it holds at $\gamma = \gamma_\rnk$, which is guaranteed by the assumption of the claim: $\beta \geq H_\mu(\gamma_\rnk) - \Delta H_\mu^\prime(\gamma_\rnk)$ (recall $\Delta \bydef \gamma_\rnk - \gamma_{\rnk+1}$). This proves \Cref{claim:mono}.
\end{proof}

%% file: appendix_rank.tex
\section{Data Normalization via Ranks}\label{app:rank}

In this appendix, we address data normalization, a practical aspect that arises while applying our results to real-world datasets. Our privacy result requires that the dataset is normalized so that every data point has a bounded norm (Assumption \mref{assump:data}). To apply such procedures to real-world datasets, one often preprocesses them to satisfy these assumptions. However, preprocessing can incur additional privacy loss as it is often data-dependent. In this section, we study a natural data normalization procedure based on rank transformation and provide an end-to-end privacy loss characterization for the resulting algorithm (\Cref{thm:privacy-rank}). 

Before running PCA on a dataset, one usually normalizes the dataset so that \citep[Section 10.2.1]{james2013introduction}:
\begin{enumerate}
    \item Each feature is centered (has zero mean).
    \item All the features have the same scale (e.g., unit variance).
\end{enumerate}
For differentially private PCA, we also need to ensure that:
\begin{enumerate}
\setcounter{enumi}{2}
    \item Each datapoint $\vx$ satisfies the norm constraint $\|\vx\| \leq \sqrt{\dim}$.
\end{enumerate}
To determine the privacy loss of the overall (end-to-end) algorithm, one must also analyze the privacy loss in the normalization step. Some natural ways to accomplish the three desiderata above run into issues. For instance, one could ensure requirements (1-2) by subtracting the mean from each feature and rescaling it by its standard deviation (perhaps using privatized estimators of the mean and standard deviation). However, there is no guarantee that the resulting dataset would satisfy requirement (3). Moreover, any further preprocessing to satisfy (3) might disturb the requirements (1-2). While there has been some work on differentially private preprocessing \citep{mcsherry2009differentially,hu2024provable, dungler2025iterative,he2025differentially}, we were unable to find a clean approach that would work well in our setting. To address this issue, we study a natural normalization procedure based on rank transformations in this section, which achieves goals (1-3) while ensuring that the end-to-end privacy loss of the overall algorithm can be controlled. 

\paragraph{Normalization via Rank Transformation} For a given dataset $\mX = \{\vx_1, \dotsc, \vx_{\ssize} \}$, we define the raw rank-transformed dataset $\mR_{\mathrm{raw}}(\mX)$ as the dataset:
\begin{subequations} \label{eq:rank-cov}
    \begin{align}
    \mR_{\mathrm{raw}}(\mX) & = \{\vr^{\mathrm{raw}}_1, \vr^{\mathrm{raw}}_2, \dotsc, \vr^{\mathrm{raw}}_\ssize\},
\end{align}
where for each data point $i \in [\ssize]$ and each feature $j \in [\dim]$, $\vr^{\mathrm{raw}}_{ij} \in [1,\ssize]$ denotes the rank of $\vx_{ij}$ among $\{\vx_{1j}, \vx_{2j}, \dotsc, \vx_{\ssize j}\}$ (in the case of ties, we average the ranks). After rank transformation, it is straightforward to normalize the rank-transformed dataset to satisfy the requirements (1-3) discussed above:
\begin{enumerate}
    \item By construction, the raw rank-transformed features have mean $\tfrac{\ssize+1}{2}$. Hence, they can be centered by simply subtracting this value. 
    \item After centering, the raw rank-transformed features lie in the range $[-\tfrac{\ssize-1}{2}, \tfrac{\ssize - 1}{2}]$, and hence have the same scale. Moreover, if the number of ties is not significant, they are roughly uniformly distributed in this range. 
    \item Finally, since the centered raw rank-transformed features are bounded in absolute value by $\frac{\ssize-1}{2}$, we can ensure that the transformed data points satisfy the $\ell_2$-norm constraint simply by rescaling the rank-transformed features by $\frac{2}{\ssize - 1}$.
\end{enumerate}
Hence, we define the normalized rank-transformed dataset $\mR(X)$ as:
\begin{align}
    \mR(X) = \{\vr_1, \vr_2, \dotsc, \vr_\ssize\}, \quad \text{ where } \quad \vr_i = \frac{2}{\ssize-1} \cdot \left( \vr_i^{\mathrm{raw}} - \frac{\ssize + 1}{2} \cdot 1_\dim  \right).
\end{align}
We use $ \mSigma_{\mathrm{rk}}(\mX)$ to denote the rank-transformed second moment matrix, which is simply the second moment matrix of the normalized rank-transformed dataset:
\begin{align}
    \mSigma_{\mathrm{rk}}(\mX) \bydef \mSigma(\mR(X)) = \frac{1}{\ssize} \sum_{i=1}^\ssize \vr_i \vr_i^\top.
\end{align}
\end{subequations}
As illustrated in \Cref{fig:projections-preprocessing}, the principal components of the rank-transformed second moment matrix can provide very similar information as the principal components of the usual second moment matrix obtained after the standard normalization steps of centering each feature by its mean and rescaling it by its standard deviation.
\begin{figure}[H]
    \centering
    \includegraphics[width=0.7\linewidth]{plots/Figure7.pdf}
    \caption{Visualizations of projections of the 1000 Genomes dataset onto its first two PCs of the second moment matrix obtained after the standard normalization (left) and rank normalization (right).}
    \label{fig:projections-preprocessing}
\end{figure}

% In practice, one would run the adaptive exponential mechanism on a preprocessed dataset to ensure that the requirements of PCA and the norm constraint ($\|\vx\|\leq \dim$) are satisfied. The privacy guarantee in \Cref{thm:composition} still holds for the mechanism that combines the adaptive exponential mechanism with the rank-transformed exponential mechanism (\Cref{alg:rank-ExpM}), as discussed in the footnote under \Cref{claim:diff-means-gen}.

\paragraph{Exponential Mechanism with rank transformation} Combining the rank-transformation-based preprocessing with the exponential mechanism results in the following algorithm, which we call the rank-transformed exponential mechanism.
%which transforms the given dataset $\mX$ into a dataset consisting of the ranks of the entries across each feature.  
\begin{algorithm}[H]
\caption{\textsc{Rank-Transformed Exponential Mechanism}($X, \beta,  \rnk$)}
\label{alg:rank-ExpM}
\begin{algorithmic}
\STATE {\textit{Input:}} Dataset $\mX \subset \R^{p}$, noise parameter $\beta\geq0,$ rank (number of PCs) $\rnk\in\N.$
\STATE {\textit{Output:}} Privatized PCs $\rV \in \O(\dim,\rnk).$
\begin{itemize}
\item Compute the rank-transformed second moment matrix $\mSigma_{\mathrm{rk}}(\mX)$ of the dataset, as defined in \eqref{eq:rank-cov}.
\item Sample  $\rV$ from $\nu(\cdot \mid \mSigma_{\mathrm{rk}}(\mX), \beta, \rnk )$.
\end{itemize}

\STATE \textit{Return:} Privatized PCs $\rV$.
\end{algorithmic}
\end{algorithm} 

\paragraph{Privacy Analysis} The asymptotic privacy loss characterization for the rank-transformed exponential mechanism (\Cref{alg:rank-ExpM}) does not follow immediately from the privacy loss characterization for the exponential mechanism in Theorem \mref{thm:privacy} since the rank-based normalization step is data-dependent: the rank of a particular data point depends on the other data points in the dataset, and will change if one adds or removes a data point from the dataset. Hence, we will need to properly account for the privacy loss due to the preprocessing step. Our privacy result for the rank-transformed exponential mechanism is based on the following assumption, which is a natural analog of Assumption \mref{assump:data} and requires that the relevant spectral properties of the rank-transformed second moment matrix $\mSigma_{\mathrm{rk}}(\mX)$ stabilize and converge to well-defined limits as $\dim \rightarrow \infty$.
\begin{assumption} \label{assump:data-rank} We observe a sequence of datasets $\mX \subset \R^{\dim}$ of increasing dimension $\dim$ and sample size $|\mX|$ which satisfies:
\begin{enumerate}
    \item $|\mX|/\dim^{3/2} \rightarrow \theta$ for some $\theta>0$ as $\dim \rightarrow \infty$.
    \item For some constant $\rnk \in \N$ (independent of $p$), the largest $\rnk+1$ eigenvalues of $\mSigma_{\mathrm{rk}}(\mX)$ converge to finite limits $\gamma_{1:\rnk+1}$ as $\dim \rightarrow \infty$:
    \begin{align*}
        \lambda_i(\mSigma_{\mathrm{rk}}(\mX)) & \rightarrow \gamma_{i} \in [0,\infty) \quad \forall \; i \; \in \; [\rnk+1].
    \end{align*}
    Moreover, the asymptotic spectral gap $\Delta \bydef \gamma_{\rnk} - \gamma_{\rnk+1}$ is strictly positive. 
    \item As $\dim \rightarrow \infty$, $\mu_{\mSigma_{\mathrm{rk}}(\mX)}$, the empirical distribution of the smallest $\dim - \rnk$ eigenvalues of $\mSigma_{\mathrm{rk}}(\mX)$ converges weakly to a limiting spectral measure $\mu$:
    \begin{align*}
        \mu_{\mSigma_{\mathrm{rk}}(\mX)} & \bydef \frac{1}{\dim-\rnk} \sum_{i=k+1}^{\dim} \delta_{\lambda_i(\mSigma_{\mathrm{rk}}(\mX))} \rightarrow \mu.
    \end{align*}
\end{enumerate}
\end{assumption}
Note that, crucially, \Cref{assump:data-rank} does not impose a norm constraint on the data points, as the rank-transformation-based preprocessing automatically ensures this constraint. We have the following result regarding the asymptotic privacy loss for the rank-transformed exponential mechanism.

\begin{theorem}\label{thm:privacy-rank}
   Consider any sequence of datasets $\mX$ which satisfies \Cref{assump:data-rank}, a sequence of noise parameters $\beta_{\dim} \rightarrow \beta\in(H_\mu(\gamma_k),\infty)$ as $\dim \rightarrow \infty$, and a fixed rank $\rnk \in \N$ (independent of $\dim$). Then, the AGDP loss of the rank-transformed exponential mechanism (\Cref{alg:rank-ExpM}) on dataset $\mX$ is \underline{at most} $\sigma_\beta$, where $\sigma_\beta$ is given by: \begin{align*} 
        {\sigma^2_\beta}\bydef \begin{cases} \frac{1}{2 \Delta \theta^2} \cdot \frac{\left(\beta-H_{\mu}(\gamma_\rnk)\right)^2}{2({\beta-H_{\mu}(\gamma_\rnk)})+ \Delta H'_{\mu}(\gamma_\rnk)}&\text{if}\quad \beta \geq -\Delta H'_{\mu}(\gamma_\rnk)+H_{\mu}(\gamma_{\rnk})\\
        -\frac{1}{2\theta^2}H'_{\mu}(\gamma_\rnk)&\text{if}\quad \beta<-\Delta H'_{\mu}(\gamma_\rnk)+H_{\mu}(\gamma_{\rnk}).
    \end{cases}
    \end{align*}
\end{theorem}
We provide the proof of \Cref{thm:privacy-rank} in the following subsection.

\begin{remark} For clarity, we emphasize that Theorem \mref{thm:privacy} claims that the output of the exponential mechanism $\rV(\cdot)$ satisfies:
\begin{align*}
    \li \inf_{\tilde{\mX} \in \calN(\mX)} \tf{\rV(\mX)}{\rV(\tilde{\mX})}(\alpha) & \geq \tf{\gauss{0}{1}}{\gauss{\sigma_\beta}{1}}(\alpha) \quad \forall \; \alpha \; \in \; [0,1], \\
     \ls\sup_{\tilde{\mX} \in \calN(\mX)} \rdv{\alpha}{\rV(\tilde{\mX})}{\rV(\mX)} & \leq \rdv{\alpha}{\gauss{\sigma_\beta}{1}}{\gauss{0}{1}} \quad \forall \; \alpha > 1,
\end{align*}
where $\calN(X) = \{\mX \cup \{x\}: x \in \R^\dim\} \cup \{X \backslash \{x\}: x \in \mX\}$ denotes all neighboring datasets obtained by adding or removing a single data point $x$ from $\mX$, where no $\ell_2$ norm constraint is imposed on $x$.
\end{remark}

\subsection{Proof of \texorpdfstring{\Cref{thm:privacy-rank}}{thm:privacy-rank}}
We consider any dataset $\mX$ that satisfies \Cref{assump:data-rank} and a sequence of noise parameters $\beta_{\dim} \rightarrow \beta\in(H_\mu(\gamma_k),\infty)$, and  bound the trade-off function and Rényi divergence between the output distributions of the rank-transformed exponential mechanism on the given dataset $\mX$ and its worst-case neighboring dataset:
\begin{align*}
    &{\li} \inf_{\tilde{\mX}\in\calN(\mX)}\tf{\nu(\cdot \mid {\mSigmaR(\mX)}, \beta_{\dim}, \rnk)}{\nu(\cdot \mid \mSigmaR(\tilde{\mX}), \beta_{\dim}, \rnk)}(\alpha)\quad\forall\:\alpha\in[0,1]\\
    &\ls\sup_{\tilde{\mX}\in\calN(\mX)} \rdv{\alpha}{\nu(\cdot \mid \mSigmaR(\tilde{\mX}), \beta_{\dim}, \rnk)}{\nu(\cdot \mid \mSigmaR(\mX), \beta_{\dim}, \rnk)}\quad\forall\;\alpha>1.
\end{align*}
As before, Theorem \mref{thm:contiguity} suggests that these privacy quantities are determined by the magnitude of the perturbation matrix $\mSigmaR(\tilde{\mX})-\mSigmaR(\mX)$ for any given neighboring dataset $\tilde{\mX}\in\calN(\mX)$.  For convenience, let us represent the dataset $X, \tilde{X}$ as:
\begin{align*}
    X = \{x_1, x_2, \dotsc, x_\ssize\}, \quad \tilde{X} = \begin{cases} \{x_1, \dotsc, x_{\ssize}, x_{\ssize+1}\} &: \text{ if } X \in \calN_+(X) \\ \{x_1, \dotsc, x_{\ssize-1}\} &: \text{ if } X \in \calN_-(X) \end{cases},
\end{align*}
and the rank-transformed datasets $R(X), R(\tilde{X})$ as:
\begin{align*}
     R(X) = \{r_1, r_2, \dotsc, r_\ssize\}, \quad R(\tilde{X}) = \begin{cases} \{\tilde{r}_1, \dotsc, \tilde{r}_{\ssize}, \tilde{r}_{\ssize+1}\} &: \text{ if } X \in \calN_+(X) \\ \{\tilde{r}_1, \dotsc, \tilde{r}_{\ssize-1}\} &: \text{ if } X \in \calN_-(X) \end{cases},
\end{align*}
where for each $i$, $r_i$ denotes the rank-transformed representation of $x_i$ in dataset $X$ and $\tilde{r}_i$ denotes the rank-transformed representation of $x_i$ in dataset $\tilde{X}$. 

Recall that $\mSigmaR({\mX}) \bydef \mSigma(R(\mX))$ is simply the second moment matrix of the rank-transformed dataset $R(\mX)$ and analogously $\mSigmaR(\tilde{\mX}) \bydef \mSigma(R(\tilde{\mX}))$ is the second moment matrix of the rank-transformed neighboring dataset. Hence, the error $\mSigmaR(\tilde{\mX})-\mSigmaR(\mX)$ can be attributed to two reasons:
\begin{enumerate}
    \item One of the two rank-transformed datasets $\mR(X), \mR(\tilde{X})$ contains an extra data point not contained in the other, which we define as $\vr(\tilde{\mX})$:
    \begin{align*}
        \vr(\tilde{\mX}) \bydef \begin{cases} \tilde{r}_{n+1} &:\text{ if } X \in \calN_+(X) \\ r_{n} &:  \text{ if } X \in \calN_-(X),\end{cases}
    \end{align*}and so the difference between their second moment matrices $\mSigmaR(\tilde{\mX})-\mSigmaR({\mX})$ will have a rank-$1$ term proportional to $\vr(\tilde{X}) \vr(\tilde{X})^\top$. 
    \item For any shared data point $x_i$, the rank-transformed representations of $x_i$ in the two datasets $\tilde{\mX}, \mX$  are also slightly different due to the addition/removal of a data point (that is, $r_i \neq  \tilde{r}_i$). 
\end{enumerate}
Our strategy will be to argue that the error due to the second reason is asymptotically negligible.  To do so, we introduce an intermediate second moment matrix: 
\begin{align} \label{eq:interm-rank-cov}
    \mSigmaRb(\tilde{\mX})\bydef\begin{dcases}
        \mSigma(R(\mX)\cup \{\vr(\tilde{\mX})\})&\text{if }\tilde{\mX}\in\calN_+(\mX)\\
        \mSigma(R(\mX)\backslash \{\vr(\tilde{\mX})\})&\text{if }\tilde{\mX}\in\calN_-(\mX),
    \end{dcases}
\end{align}
Then, we can write: 
\begin{align*}
    \mSigmaR(\tilde{\mX})-\mSigmaR(\mX)=\underbrace{\mSigmaRb(\tilde{\mX})-\mSigmaR(\mX)}_{(i)} + \underbrace{\mSigmaR(\tilde{\mX})-\mSigmaRb(\tilde{\mX})}_{(ii)}.
\end{align*}
The term (i) captures the error attributed to reason (1) mentioned above, and the term (ii) captures the error attributed to reason (2). 
% A key observation of our proof is that the 
% the data points that are shared between the rank-transformed datasets $R_0(\mX)$ 
% and $R_0(\tilde{\mX})$ differ by at most 1 in each of their entries, which, as we will see, 
The following lemma shows that the second term $\mSigmaR(\tilde{\mX})-\mSigmaRb(\tilde{\mX})$ is asymptotically negligible, and hence the distributions $\nu(\cdot \mid \mSigmaR(\tilde{\mX}), \beta_{\dim}, \rnk)$ and $\nu(\cdot \mid \mSigmaRb(\tilde{\mX}), \beta_{\dim},\rnk)$ are very close. This will allow us to argue that:
\begin{align*}
\tf{\nu(\cdot \mid {\mSigmaR(\mX)}, \beta_{\dim}, \rnk)}{\nu(\cdot \mid \mSigmaR(\tilde{\mX}), \beta_{\dim}, \rnk)}(\alpha) &\approx \tf{\nu(\cdot \mid {\mSigmaR(\mX)}, \beta_{\dim}, \rnk)}{\nu(\cdot \mid \mSigmaRb(\tilde{\mX}), \beta_{\dim}, \rnk)}(\alpha), \\   \rdv{\alpha}{\nu(\cdot \mid \mSigmaR(\tilde{\mX}), \beta_{\dim}, \rnk)}{\nu(\cdot \mid \mSigmaR(\mX), \beta_{\dim}, \rnk)}  & \approx \rdv{\alpha}{\nu(\cdot \mid \mSigmaRb(\tilde{\mX}), \beta_{\dim}, \rnk)}{\nu(\cdot \mid \mSigmaR(\mX), \beta_{\dim}, \rnk)}.\end{align*}
We will see shortly that the quantities on the RHS can be easily computed by applying our privacy result (Theorem \mref{thm:privacy}) for the vanilla exponential mechanism (Algorithm \mref{alg:ExpM}). 
\begin{lemma}\label{lem:rank_tv}
    Under the assumptions of \Cref{thm:privacy-rank}, for any neighboring dataset $\tilde{\mX}\in\calN(\mX),$
    \begin{align*}
        \lm \tv(\nu(\cdot \mid \mSigmaR(\tilde{\mX}), \beta_{\dim}, \rnk),\nu(\cdot \mid \mSigmaRb(\tilde{\mX}), \beta_{\dim},\rnk))&=0,\\
        \lm \rdv{\alpha}{\nu(\cdot \mid \mSigmaR(\tilde{\mX}), \beta_{\dim}, \rnk)}{\nu(\cdot \mid \mSigmaRb(\tilde{\mX}), \beta_{\dim},\rnk)}&=0\quad\forall\:\alpha>1,\\
        \lm\rdv{\alpha}{\nu(\cdot \mid \mSigmaRb(\tilde{\mX}), \beta_{\dim},\rnk)}{\nu(\cdot \mid \mSigmaR(\tilde{\mX}), \beta_{\dim}, \rnk)}&=0\quad\forall\:\alpha>1.
    \end{align*}
\end{lemma}
We defer the proof of \Cref{lem:rank_tv} to the end of this section and present the proof of \Cref{thm:privacy-rank}, taking this result for granted.
% Analogous to \Cref{lem:simple-gibbs-approx} in the proof of \Cref{thm:privacy}, \Cref{lem:rank_tv} allows us to analyze the asymptotic privacy of the algorithm based on $\nu(\cdot \mid {\mSigmaR(\mX)}, \beta_{\dim}, \rnk)$ and the simplified Gibbs distribution $\nu(\cdot \mid \mSigmaRb(\tilde{\mX}), \beta_{\dim},\rnk).$ For notational convenience, we will let 
% \begin{align*}
%     \nu_{\dim}\bydef\nu(\cdot \mid {\mSigmaR(\mX)}, \beta_{\dim}, \rnk)
% \end{align*}
% in the rest of the proof.
\begin{proof}[Proof of \Cref{thm:privacy-rank}] We prove the trade-off function and the Rényi divergence privacy loss characterization individually. For notational convenience, we will let $\nu_\dim$ denote:
\begin{align*}
    \nu_{\dim}\bydef\nu(\cdot \mid {\mSigmaR(\mX)}, \beta_{\dim}, \rnk).
\end{align*}
\paragraph{Privacy Analysis in Terms of the Trade-off Function} Let us first prove the claimed trade-off function-based privacy loss characterization: 
\begin{align*}
    {\li} \inf_{\tilde{\mX}\in\calN(\mX)}\tf{\nu(\cdot \mid {\mSigmaR(\mX)}, \beta_{\dim}, \rnk)}{\nu(\cdot \mid \mSigmaR(\tilde{\mX}), \beta_{\dim}, \rnk)}(\alpha)\geq\tf{\gauss{0}{1}}{\gauss{\sigma_\beta}{1}}(\alpha)
\end{align*}
for all $\alpha\in[0,1].$ Fix any $\alpha\in[0,1],$ and let $\tilde{\mX}_{\#}\in\calN(\mX)$ be an approximate worst-case dataset satisfying: 
\begin{align*}
    \tf{\nu_\dim}{\nu(\cdot \mid \mSigmaR(\tilde{\mX}_{\#}), \beta_{\dim}, \rnk)}(\alpha)\leq \inf_{\tilde{\mX}\in\calN(\mX)}\tf{\nu_{\dim}}{\nu(\cdot \mid \mSigmaR(\tilde{\mX}), \beta_{\dim}, \rnk)}(\alpha)+\frac{1}{\dim}
\end{align*}
so that:
\begin{align*}
    &\li \inf_{\tilde{\mX}\in\calN(\mX)}\tf{\nu_{\dim}}{\nu(\cdot \mid \mSigmaR(\tilde{\mX}), \beta_{\dim}, \rnk)}(\alpha)\\
    &\quad=\li\tf{\nu_{\dim}}{\nu(\cdot \mid \mSigmaR(\tilde{\mX}_{\#}), \beta_{\dim}, \rnk)}(\alpha)\\
    &\quad\explain{(a)}{=} \li\tf{\nu_{\dim}}{\nu(\cdot \mid \mSigmaRb(\tilde{\mX}_{\#}), \beta_{\dim}, \rnk)}(\alpha)\\
    &\quad\explain{(b)}{\geq} \lm \inf_{\tilde{\mR}\in\calN(R(\mX))}\tf{\nu_{\dim}}{\nu(\cdot \mid \mSigma(\tilde{\mR}), \beta_{\dim}, \rnk)}(\alpha)\\
    &\quad\explain{Thm. \mref{thm:privacy}}{=} \tf{\gauss{0}{1}}{\gauss{\sigma_\beta}{1}}(\alpha),
    \end{align*}
    where step (a) holds by \Cref{lem:rank_tv} (item (1)) and Lemma \mref{lem:triangle-ineq} (item (1)), and step (b) follows by recalling from \eqref{eq:interm-rank-cov} that $\mSigmaRb(\tilde{\mX}_{\#})$ was defined as:
    \begin{align*}
         \mSigmaRb(\tilde{\mX}_\#)\bydef \mSigma(\tilde{R}_\#), \quad \text{where}\quad \tilde{\mR}_\# \bydef \begin{dcases}
        R(\mX)\cup \{\vr(\tilde{\mX}_\#)\}&\text{if }\tilde{\mX}_\#\in\calN_+(\mX)\\
        R(\mX)\backslash \{\vr(\tilde{\mX}_\#)\}&\text{if }\tilde{\mX}_\#\in\calN_-(\mX),
    \end{dcases}
    \end{align*}
    and observing that $\tilde{\mR}_\#$ can be viewed as a neighboring dataset of the rank-transformed dataset $R(\mX)$. 
\paragraph{Rényi Differential Privacy Analysis} We next prove  the claimed  Rényi divergence-based asymptotic privacy loss characterization:
\begin{align*}
    &\ls\sup_{\tilde{\mX}\in\calN(\mX)} \rdv{\alpha}{\nu(\cdot \mid \mSigmaR(\tilde{\mX}), \beta_{\dim}, \rnk)}{\nu(\cdot \mid \mSigmaR(\mX), \beta_{\dim}, \rnk)}\\
    &\hspace{8cm}\leq\rdv{\alpha}{\gauss{\sigma_\beta}{1}}{\gauss{0}{1}}\quad\forall\;\alpha>1.
\end{align*}
Fix any $\alpha>1,$ and let $\tilde{\mX}_{\#}\in\calN(\mX)$ be the approximate worst-case dataset satisfying:
\begin{align*}
    \rdv{\alpha}{\nu(\cdot \mid \mSigmaR(\tilde{\mX}_{\#}), \beta_{\dim}, \rnk)}{\nu_{\dim}}\geq \sup_{\tilde{\mX}\in\calN(\mX)} \rdv{\alpha}{\nu(\cdot \mid \mSigmaR(\tilde{\mX}), \beta_{\dim}, \rnk)}{\nu_{\dim}} -\frac{1}{\dim}
\end{align*}
so that: 
\begin{align}\label{eq:rank-approx-dataset}
    \ls\sup_{\tilde{\mX}\in\calN(\mX)} \rdv{\alpha}{\nu(\cdot \mid \mSigmaR(\tilde{\mX}), \beta_{\dim}, \rnk)}{\nu_{\dim}}&=\ls  \rdv{\alpha}{\nu(\cdot \mid \mSigmaR(\tilde{\mX}_{\#}), \beta_{\dim}, \rnk)}{\nu_{\dim}}.
\end{align}
Next, observe that \Cref{lem:rank_tv} and  Lemma \mref{lem:triangle-ineq} imply that for any $\epsilon\in(0,\alpha-1)$:
\begin{align*}
    \ls \rdv{\alpha}{\tilde{\nu}_{\dim}}{\nu_{\dim}}&\explain{Lem. \mref{lem:triangle-ineq}}{\leq} \ls \left(\rdv{\alpha+\epsilon}{\bar{\nu}_{\dim}}{\nu_{\dim}}+\frac{\alpha+\epsilon}{\alpha+\epsilon-1} \rdv{\frac{\alpha(\alpha+\epsilon-1)}{\epsilon}}{\tilde{\nu}_{\dim}}{\bar{\nu}_{\dim}}\right)\\
    &\explain{Lem. \ref{lem:rank_tv}}{=}\ls \rdv{\alpha+\epsilon}{\bar{\nu}_{\dim}}{\nu_{\dim}}\\
    &\leq \lm\sup_{\tilde{\mR}\in\calN(R(\mX))} \rdv{\alpha+\epsilon}{\nu(\cdot \mid \mSigma(\tilde{\mR}), \beta_{\dim}, \rnk)}{\nu_{\dim}}\\
    &\explain{Thm. \mref{thm:privacy}}{=}\rdv{\alpha+\epsilon}{\gauss{\sigma_\beta}{1}}{\gauss{0}{1}},
\end{align*}
where $\tilde{\nu}_{\dim}\bydef \nu(\cdot \mid \mSigmaR(\tilde{\mX}_{\#}), \beta_{\dim}, \rnk)$ and $\bar{\nu}_{\dim}\bydef \nu(\cdot \mid \mSigmaRb(\tilde{\mX}_{\#}), \beta_{\dim},\rnk).$ Noting that the above inequality holds for any $\epsilon\in(0,\alpha-1),$ we let $\epsilon\rightarrow0$ to obtain:
\begin{align*}
    \ls\sup_{\tilde{\mX}\in\calN(\mX)} \rdv{\alpha}{\nu(\cdot \mid \mSigmaR(\tilde{\mX}), \beta_{\dim}, \rnk)}{\nu_{\dim}}&\explain{\eqref{eq:rank-approx-dataset}}{=} \ls \rdv{\alpha}{\tilde{\nu}_{\dim}}{\nu_{\dim}}\\
    &\leq \rdv{\alpha}{\gauss{\sigma_\beta}{1}}{\gauss{0}{1}},
\end{align*}
as desired. This concludes the proof of \Cref{thm:privacy-rank}.
\end{proof}

The rest of this section is dedicated to the proof of \Cref{lem:rank_tv}. 

\subsubsection{Proof of \texorpdfstring{\Cref{lem:rank_tv}}{lem:ranktv}}\label{sec:proof_lem_rank_tv} 
\begin{proof} Fix any neighboring dataset $\tilde{\mX}\in\calN(\mX),$ and recall that our goal is to show that the Gibbs distributions $\nu(\cdot \mid \mSigmaR(\tilde{\mX}), \beta_{\dim}, \rnk)$ and $\nu(\cdot \mid \mSigmaRb(\tilde{\mX}), \beta_{\dim},\rnk)$ are asymptotically close in the total variation distance and Rényi divergence. As in the proof of Lemma \mref{lem:simple-gibbs-approx}, these results are a consequence of Theorem \mref{thm:contiguity} (item (2)). Before applying this result, let us check that the matrices $\mSigmaR(\tilde{\mX})$ and $\mSigmaRb(\tilde{\mX})$ satisfy the assumptions of Theorem \mref{thm:contiguity} (item (2)). For simplicity, we consider the situation where $\tilde{X}$ contains one additional data point compared to $\mX$:
\begin{align*}
    \mX = \{\vx_1, \dotsc, \vx_\ssize\}, \quad \tilde{\mX} = \{\vx_1, \dotsc, \vx_{\ssize}, \vx_{\ssize + 1}\},
\end{align*}
and the case when $\tilde{\mX}$ is constructed by removing a data point from $\mX$ is exactly analogous. 
\begin{itemize}
    \item We first verify that the perturbation matrix $\mE\bydef\sqrt{\dim}(\mSigmaR(\tilde{\mX})-\mSigmaRb(\tilde{\mX}))$ satisfies the requirements $\|\mE\|\ll1$ and $\|\mE\|_{\fr}\lesssim1$. We will in fact show that $\|\mE\|_{\fr} \ll 1$. Let $\mR(X)$ and $\mR(\tilde{\mX})$ denote the rank transformed versions of $\mX, \tilde{\mX}$:
    \begin{align*}
        \mR(\mX) = \{\vr_1, \dotsc, \vr_\ssize\}, \quad \mR(\tilde{\mX})  = \{\tilde{r}_1, \tilde{r}_2, \dotsc, \tilde{r}_\ssize, \tilde{r}_{\ssize+1}\}.
    \end{align*}
    Recalling the definitions of $\mSigmaR(\tilde{\mX})$ and $\mSigmaRb(\tilde{\mX})$ we find that:
    \begin{align*}
        \mSigmaR(\tilde{\mX}) & \bydef \mSigma(R(\tilde{\mX})) = \frac{1}{\ssize+1} \sum_{i=1}^{\ssize+1} \tilde{r}_i \tilde{r}_i^\top\\
        \mSigmaRb(\tilde{\mX})  &\bydef \mSigma(R({\mX}) \cup\{\tilde{r}_{\ssize + 1}\}) = \frac{1}{\ssize+1} \sum_{i=1}^\ssize {r}_i {r}_i^\top + \frac{\tilde{r}_{\ssize + 1} \tilde{r}_{\ssize + 1}^\top}{\ssize + 1}.
    \end{align*}
    Hence,
    \begin{align*}
        \mE & \bydef \sqrt{\dim}(\mSigmaR(\tilde{\mX})-\mSigmaRb(\tilde{\mX})) = \frac{\sqrt{\dim}}{\ssize+1} \left(  \sum_{i=1}^\ssize \tilde{\vr}_i \tilde{\vr}_i^\top -  \sum_{i=1}^\ssize \vr_i \vr_i^\top \right).
    \end{align*}
    For the subsequent analysis, we will find it helpful to define the matrices $\mR(X), \mM$ as follows:
    \begin{align*}
        \mR(\mX) \bydef \begin{bmatrix} \vr_1 & \vr_2 & \dotsb & \vr_\ssize \end{bmatrix}, \quad \mM \bydef \begin{bmatrix} \tilde{\vr}_1 - \vr_1 & \tilde{\vr}_2 - \vr_2 & \dotsb & \tilde{\vr}_{\ssize} - \vr_\ssize \end{bmatrix}.
    \end{align*}
    Note the slight abuse in notation where we are using $\mR(\mX)$ to denote the set $\{r_1, \dotsc, \vr_\ssize\}$ as well as the matrix with columns $\{r_1, \dotsc, r_\ssize\}$. Note that $\mE$ can be rewritten as:
    \begin{align*}
        \mE & = \frac{\sqrt{\dim}}{\ssize+1} (  (\mR(X) + \mM)(\mR(X) + \mM)^\top - \mR(X) \mR(X)^\top)\\
        &=  \frac{\sqrt{\dim}}{\ssize+1} (  \mR(X) \mM^\top + \mM \mR(X)^\top + \mM \mM^\top ). 
    \end{align*}
    Hence, we can bound $\|\mE\|_{\fr}$ as:
    \begin{align}\label{eq:rank-i}
        \|\mE\|_{\fr}\explain{(a)}{\leq} \sqrt{\dim}\cdot\frac{2\|\mR(X)\|\|\mM\|_{\fr}+\|\mM\|_{\fr}^2}{\ssize+1}, %\explain{(b)}{\lesssim} \sqrt{\dim}\cdot\frac{\sqrt{\ssize }\|\mM_\pm\|_{\fr}+\|\mM_\pm\|_{\fr}^2}{\ssize}\lesssim \frac{\dim}{\ssize}\lesssim\frac{1}{\sqrt{\dim}}\ll1,
    \end{align}
     where (a) holds since $\|\mR(\mX)\mM^\top\|_{\fr}\leq \|\mR(\mX)\|\|\mM\|_{\fr}$ (see e.g., \citep[Lemma 2.1]{chen2021spectral}).
    % \begin{align*}
    %     |\mM_{+_{ij}}|&\bydef|\mR(\tilde{\mX})_{ij}-R(\mX)_{ij} |\leq \frac{2}{\ssize(\ssize-1)}R_0(\tilde{\mX})_{ij}+\frac{2}{\ssize-1}|R_0(\tilde{\mX})_{ij}-R_0(\mX)_{ij}|+\frac{2}{\ssize(\ssize-1)}\explain{(a)}{\lesssim} \frac{1}{\ssize},
    % \end{align*}
    To control $\|\mM\|_{\fr}$, we will simply bound each entry of $\mM$. 
    % For any $i \in [\ssize]$ and any $j \in [\dim]$, $\mM_{ji}$ is the difference in the normalized rank of the data point $\vx_i$ for the feature $j$ in the datasets $\mX = \{\vx_{1:\ssize}\}, \tilde{\mX} = \{\vx_{1:\ssize+1}\}$. Hence,
    % \begin{align*}
    %     |\mM_{ji}| & = |\vr_{ij} - \tilde{\vr}_{ij}| =  \left| \frac{1}{\ssize} \sum_{\ell = 1}^{\ssize} \ind\{\vx_{\ell,j} \leq \vx_{i,j}\} - \frac{1}{\ssize+1} \sum_{\ell = 1}^{\ssize+1} \ind\{\vx_{\ell,j} \leq \vx_{i,j}\}  \right| \leq \frac{2}{\ssize + 1},
    % \end{align*}
    % which implies,
    % \begin{align}  \label{eq:rank-ii}
    %    \|\mM\|_{\fr} \leq \frac{2 \sqrt{\ssize\dim}}{\ssize + 1} \lesssim \frac{\sqrt{\dim}}{\sqrt{\ssize}}.
    % \end{align}
    Recall from the rank transformed exponential mechanism that $r_{ij}$ and $\tilde{r}_{ij}$ are the centered and normalized versions of the raw ranks of $\vx_{ij}$ and $\tilde{\vx}_{ij},$ respectively, for each $i\in[\ssize]$ and $j\in[\dim].$ We will find it helpful to define the raw ranks of $\vx_{ij}$ and $\tilde{\vx}_{ij}$ as $r^{\mathrm{raw}}_{ij}$ and $\tilde{r}^{\mathrm{raw}}_{ij},$ respectively, so that $r_{ij}$ and $\tilde{r}_{ij}$ can be expressed as:
    \begin{align*}
        r_{ij}=\frac{2}{\ssize-1}\left(r^{\mathrm{raw}}_{ij}-\frac{\ssize+1}{2}\right),\quad \tilde{r}_{ij}=\frac{2}{\ssize}\left(\tilde{r}^{\mathrm{raw}}_{ij}-\frac{\ssize+2}{2}\right)\quad\forall  i\;\in\;[\ssize],\;j\;\in\;[\dim].
    \end{align*}
    Having defined the raw ranks, we can bound 
    $\mM_{ji}$ as follows:
    \begin{align*}
        |\mM_{ji}| & = |\vr_{ij} - \tilde{\vr}_{ij}|\\
        &=  \left| \frac{2}{\ssize-1}\left(r^{\mathrm{raw}}_{ij}-\frac{\ssize+1}{2}\right)-\frac{2}{\ssize}\left(\tilde{r}^{\mathrm{raw}}_{ij}-\frac{\ssize+2}{2}\right)\right|\\
        &\leq \frac{2}{\ssize(\ssize-1)}\tilde{r}^{\mathrm{raw}}_{ij}+\frac{2}{\ssize-1}|\tilde{r}^{\mathrm{raw}}_{ij}-r^{\mathrm{raw}}_{ij}|+\frac{2}{\ssize(\ssize-1)}\\
        %&\leq \frac{2}{\ssize-1}+\frac{2}{\ssize-1}+\frac{2}{\ssize(\ssize-1)}\\
        &\lesssim \frac{1}{\ssize} \qquad\text{ [since $\tilde{r}^{\mathrm{raw}}_{ij}\leq \ssize+1,$ $|\tilde{r}^{\mathrm{raw}}_{ij}-r^{\mathrm{raw}}_{ij}|\leq1$]},
    \end{align*}
    which implies:
    \begin{align}  \label{eq:rank-ii}
       \|\mM\|_{\fr} \lesssim \frac{\sqrt{\ssize\dim}}{\ssize} = \frac{\sqrt{\dim}}{\sqrt{\ssize}}.
    \end{align}
    To control $\|\mR(X)\|$, we exploit \Cref{assump:data-rank}:
    \begin{align} \label{eq:rank-iv}
        \|R(\mX)\|=\sqrt{\|R(\mX) R(\mX)^\top\|}=\sqrt{\ssize}\sqrt{\|\mSigmaR(\mX)\|} \; \; \explain{Assump. \ref{assump:data-rank}}{\lesssim} \; \; \sqrt{\ssize}.
    \end{align}
    Plugging the estimates \eqref{eq:rank-ii} and \eqref{eq:rank-iv} into \eqref{eq:rank-i}, we conclude that:
    \begin{align}\label{eq:rank-i-}
        \|\mE\|_{\fr} \lesssim \sqrt{\dim}\cdot\frac{\sqrt{\dim}+\tfrac{\dim}{\ssize}}{\ssize}\lesssim \frac{\dim}{\ssize}\lesssim\frac{1}{\sqrt{\dim}}\ll1.
    \end{align}
    This verifies that $\mE$ satisfies $\|\mE\|_{\fr} \lesssim 1$ and $\|\mE\| \ll 1$.
    \item The matrices $\mSigmaR(\tilde{\mX})$ and $\mSigmaRb(\tilde{\mX})$ satisfy Assumption \mref{assump:mat} since they are close (in operator norm) to the matrix $\mSigmaR(\mX),$ which satisfies Assumption \mref{assump:mat} because the data matrix $\mX$ satisfies \Cref{assump:data-rank} (see \Cref{lem:misc_conv}). Indeed, observe that 
    \begin{align*}
        \mSigmaRb(\tilde{\mX})= \frac{\ssize\mSigmaR(\mX)}{\ssize+1}+\frac{\vr(\tilde{\mX})\vr(\tilde{\mX})^\top}{\ssize+1},  % \begin{dcases}
        %\frac{\ssize\mSigmaR(\mX)}{\ssize+1}+\frac{\vr\vr^\top}{\ssize+1}&\text{if }\tilde{\mX}\in \calN_+(\mX)\\
        %\frac{\ssize\mSigmaR(\mX)}{\ssize-1}-\frac{\vr\vr^\top}{\ssize-1}&\text{if }\tilde{\mX}\in \calN_-(\mX)
    %\end{dcases},
    \end{align*}
    which implies: 
    \begin{align}
        \|\mSigmaRb(\tilde{\mX})-\mSigmaR(\mX)\|_{\fr}&\leq \frac{\|\mSigmaR(\mX)\|_{\fr}}{\ssize+1}+\frac{\|\vr(\tilde{\mX})\vr(\tilde{\mX})^\top\|_{\fr}}{\ssize+1}\leq \frac{\dim}{\ssize}+o\left(\frac{1}{\sqrt{\dim}}\right)\lesssim \frac{1}{\sqrt{\dim}},\label{eq:rank-ii-}\\
        \|\mSigmaR(\tilde{\mX})-\mSigmaR(\mX)\|&\leq\|\mSigmaR(\tilde{\mX})-\mSigmaRb(\tilde{\mX})\|+\|\mSigmaRb(\tilde{\mX})- \mSigmaR(\mX)\|
        \explain{\eqref{eq:rank-i-}, \eqref{eq:rank-ii-}}{\ll}1.\notag
    \end{align}
    %It thus follows from  that the matrices $\mSigmaR(\tilde{\mX})$ and $\mSigmaRb(\tilde{\mX})$ satisfy \Cref{assump:data-rank}.
\end{itemize}
Since all the requirements of Theorem \mref{thm:contiguity} (item (2)) are met, we conclude that:
\begin{align*}
    \lm \tv(\nu(\cdot \mid \mSigmaR(\tilde{\mX}), \beta_{\dim}, \rnk),\nu(\cdot \mid \mSigmaRb(\tilde{\mX}), \beta_{\dim},\rnk))&=0,\\
    \lm \rdv{\alpha}{\nu(\cdot \mid \mSigmaR(\tilde{\mX}), \beta_{\dim}, \rnk)}{\nu(\cdot \mid \mSigmaRb(\tilde{\mX}), \beta_{\dim},\rnk)}&=0\quad\forall\:\alpha>1,\\
    \lm\rdv{\alpha}{\nu(\cdot \mid \mSigmaRb(\tilde{\mX}), \beta_{\dim},\rnk)}{\nu(\cdot \mid \mSigmaR(\tilde{\mX}), \beta_{\dim}, \rnk)}&=0\quad\forall\:\alpha>1,
\end{align*}
as desired. 
\end{proof}

%% file: appendix_misc.tex
\section{Miscellaneous Lemmas}\label{app:misc}

\begin{lemma}\label{lem:misc_conv}
    Suppose $\mSigma_{\dim}\in\R^{\dim\times\dim}$ and $\tilde{\mSigma}_{\dim}\in\R^{\dim\times\dim}$ are sequences of symmetric matrices satisfying $\|\mSigma_{\dim}-\tilde{\mSigma}_{\dim}\|=o(1).$ Suppose that $\mSigma_{\dim}$ satisfies Assumption \mref{assump:mat}:
    \begin{enumerate}
        \item As $\dim\rightarrow\infty,$ $\lambda_i(\mSigma_{\dim})\rightarrow \gamma_i$ for all $i\in[\rnk+1]$ for some $\rnk\in\N$ (independent of $\dim$) and $\gamma_{1:\rnk+1}\in\R.$
        \item As $\dim\rightarrow\infty,$ $\mu_{\dim}\bydef\frac{1}{\dim-\rnk}\sum_{i=\rnk+1}^{\dim}\delta_{\lambda_i(\mSigma_{\dim})}\rightarrow\mu$ for some measure $\mu.$
    \end{enumerate} 
    Then:
    \begin{enumerate}
        \item As $\dim\rightarrow\infty,$ $\lambda_i(\tilde{\mSigma}_{\dim})\rightarrow \gamma_i$ for all $i\in[\rnk+1].$ 
        \item As $\dim\rightarrow\infty,$ $\frac{1}{\dim-\rnk}\sum_{i=\rnk+1}^{\dim}\delta_{\lambda_i(\tilde{\mSigma}_{\dim})}\rightarrow\mu.$ 
    \end{enumerate} 
\end{lemma}

\begin{proof}
For convenience, let $\lambda_i$ and $\tilde{\lambda}_i$ denote $\lambda_i(\mSigma_{\dim})$ and $\lambda_i(\tilde{\mSigma}_{\dim}),$ respectively, for each $i\in[\dim].$  
Recall that our goal is to show 
\begin{align*}
    \tilde{\lambda}_i\rightarrow \gamma_i\quad\forall \;i\;\in[\rnk+1]\quad\text{and}\quad \frac{1}{\dim-\rnk}\sum_{i=\rnk+1}^{\dim}\delta_{\tilde{\lambda}_i}\rightarrow\mu.
\end{align*}
Since the difference between the matrices $\mSigma_{\dim}$ and $\tilde{\mSigma}_{\dim}$ is asymptotically small in operator norm, we expect the largest $\rnk+1$ eigenvalues of the two matrices to converge to the same limit values. Indeed, by Weyl's inequality \citep[Theorem 4.5.3]{vershynin2018high}, we have: 
\begin{align}\label{eq:misc_conv_Weyl}
    \underset{i\in[\dim]}{\max}\;|\tilde{\lambda}_i-\lambda_i|\leq \|\tilde{\mSigma}_{\dim}-\mSigma_{\dim}\|=o(1),
\end{align}
which immediately implies that $\lm \tilde{\lambda_i} =  \lm (\tilde{\lambda}_i - \lambda_i) + \lm \lambda_i=\gamma_i$ for each $i\in[\rnk+1].$ We now show $\frac{1}{\dim-\rnk}\sum_{i=\rnk+1}^{\dim}\delta_{\tilde{\lambda}_i}\rightarrow\mu.$ Let $F_{\dim}$ and $\tilde{F}_{\dim}$ denote the CDFs of $\mu_{\dim}$ and $\tilde{\mu}_{\dim},$ respectively, and let $F$ denote the CDF of $\mu.$ Fix any continuity point $t$ of $F.$ We will show that $\tilde{F}_{\dim}(t)$ converges to $F(t)$ by relating $\tilde{F}_{\dim}$ to $F_{\dim}.$ Observe that 
\begin{align*}
    \tilde{F}_p(t)\bydef \frac{1}{\dim-\rnk}\sum_{i=\rnk+1}^{\dim}\ind_{(-\infty,t]}(\tilde{\lambda}_i)&= \frac{1}{\dim-\rnk}\sum_{i=\rnk+1}^{\dim}\ind_{(-\infty,t]}(\tilde{\lambda}_i-\lambda_i+\lambda_i)\\
    &\hspace{3cm}=\frac{1}{\dim-\rnk}\sum_{i=\rnk+1}^{\dim}\ind_{(-\infty,t-\tilde{\lambda}_i+\lambda_i]}(\lambda_i).
\end{align*}
Then, since $\underset{i\in[\dim]}{\max}\;|\tilde{\lambda}_i-\lambda_i|\leq\|\tilde{\mSigma}_{\dim}-\mSigma_{\dim}\|$ by Weyl's inequality, we have:
\begin{align*}
   \frac{1}{\dim-\rnk}\sum_{i=\rnk+1}^{\dim}\ind_{(-\infty,t-\|\tilde{\mSigma}_{\dim}-\mSigma_{\dim}\|]}(\lambda_i)\leq  \tilde{F}_p(t)\leq \frac{1}{\dim-\rnk}\sum_{i=\rnk+1}^{\dim}\ind_{(-\infty,t+\|\tilde{\mSigma}_{\dim}-\mSigma_{\dim}\|]}(\lambda_i),
\end{align*}
or equivalently, 
\begin{align*}
    F_{\dim}(t-\|\tilde{\mSigma}_{\dim}-\mSigma_{\dim}\|)\leq \tilde{F}_p(t)\leq F_{\dim}(t+\|\tilde{\mSigma}_{\dim}-\mSigma_{\dim}\|).
\end{align*}
Since we do not know if $F$ is continuous, 
we cannot directly apply Polya's theorem \citep[Theorem 11.2.9]{lehmann2005testing} to get uniform convergence of $F_{\dim}$ to $F.$ Instead, we will leverage the observation that the set of continuity points of $F,$ which we denote as $C(F),$ is dense in $\R,$ which gives the existence of a sequence of continuity points $t_m\in C(F)\cap (t-\frac{1}{m},t)$ and $t_m'\in C(F)\cap (t,t+\frac{1}{m})$ for each $m\in\N.$ Fix any $m\in\N,$ and note that $t+\|\tilde{\mSigma}_{\dim}-\mSigma_{\dim}\|\in(t_m,t_m') $ for large enough $\dim.$ Then, for large enough $\dim,$ we have: 
\begin{align*}
    F_{\dim}(t_m)\leq F_{\dim}(t-\|\tilde{\mSigma}_{\dim}-\mSigma_{\dim}\|)\leq \tilde{F}_p(t)\leq F_{\dim}(t+\|\tilde{\mSigma}_{\dim}-\mSigma_{\dim}\|)\leq F_{\dim}(t_m'),
\end{align*}
which implies that 
\begin{align*}
    F(t_m)\explain{(\#)}{=}\lm F_{\dim}(t_m)\leq \lm \tilde{F}_p(t)\leq \lm F_{\dim}(t_m')\explain{(\#)}{=}F(t_m'),
\end{align*}
where the equalities marked (\#) hold since $t_m\in C(F).$ Finally, since $t\in C(F),$ we have $\lim_{m\rightarrow\infty}F(t_m)=\lim_{m\rightarrow\infty}F(t_m')=F(t),$ which proves the claim that $\lm \tilde{F}_p(t)=F(t).$
\end{proof}

\begin{lemma}\label{lem:misc_HK}
Suppose that $\mSigma_{\dim}\in\R^{\dim\times\dim}$ is a sequence of matrices satisfying Assumption \mref{assump:mat}:
    \begin{enumerate}
        \item As $\dim\rightarrow\infty,$ $\lambda_i(\mSigma_{\dim})\rightarrow \gamma_i$ for all $i\in[\rnk+1]$ for some $\rnk\in\N$ (independent of $\dim$) and $\gamma_{1:\rnk+1}\in\R,$ where $\gamma_\rnk>\gamma_{\rnk+1}.$
        \item As $\dim\rightarrow\infty,$ $\mu_{\dim}\bydef\frac{1}{\dim-\rnk}\sum_{i=\rnk+1}^{\dim}\delta_{\lambda_i(\mSigma_{\dim})}\rightarrow\mu$ for some measure $\mu.$
    \end{enumerate} 
Then, for any $j\in[\rnk]$ and any sequence $\xi_j$ converging to $\gamma_j$ and for any $m\in\N:$  
\begin{align*}
    \frac{1}{\dim-\rnk}\sum_{i=\rnk+1}^{\dim}\frac{1}{(\xi_j-\lambda_i(\mSigma_{\dim}))^m}\rightarrow\int_{\R}\frac{\diff\mu(\lambda)}{(\gamma_j-\lambda)^m}.
\end{align*}
Moreover, for any $j,\ell\in[\rnk]:$
\begin{align*}
    \frac{1}{\dim-\rnk}\sum_{i=\rnk+1}^{\dim}\frac{1}{(\lambda_j(\mSigma_{\dim})-\lambda_i(\mSigma_{\dim}))(\lambda_\ell(\mSigma_{\dim})-\lambda_i(\mSigma_{\dim}))}\rightarrow \int_{\R}\frac{\diff\mu(\lambda)}{(\gamma_j-\lambda)(\gamma_\ell-\lambda)}.
\end{align*}
\end{lemma}
\begin{proof}
For convenience, we
let $\lambda_i\bydef\lambda_i(\mSigma_{\dim})$ for each $i\in[\dim].$ Fix any $j,\ell\in[\rnk],$ any sequence $\xi_j$ converging to $\gamma_j,$ and any $m\in\N,$ and recall that our goal is to show:
\begin{align*}
    \frac{1}{\dim-\rnk}\sum_{i=\rnk+1}^{\dim}\frac{1}{(\xi_j-\lambda_i)^m}\rightarrow\int_{\R}\frac{\diff\mu(\lambda)}{(\gamma_j-\lambda)^m}
\end{align*}
and 
\begin{align*}
    \frac{1}{\dim-\rnk}\sum_{i=\rnk+1}^{\dim}\frac{1}{(\lambda_j-\lambda_i)(\lambda_\ell-\lambda_i)}\rightarrow \int_{\R}\frac{\diff\mu(\lambda)}{(\gamma_j-\lambda)(\gamma_\ell-\lambda)},
\end{align*}
or equivalently: 
\begin{align*}
    \E_{\bf{\rLambda}\sim\mu_{\dim}}\left[\frac{1}{(\xi_j-\rLambda)^m}\right]\rightarrow \E_{\rLambda\sim\mu}\left[\frac{1}{(\gamma_j-\rLambda)^m}\right]
\end{align*}
and 
\begin{align*}
    \E_{\rLambda\sim\mu_{\dim}}\left[\frac{1}{(\lambda_j-\rLambda)(\lambda_\ell-\rLambda)}\right]\rightarrow \E_{\rLambda\sim\mu}\left[\frac{1}{(\gamma_j-\rLambda)(\gamma_\ell-\rLambda)}\right].
\end{align*}
Observe that for any $\gamma,\gamma'>\gamma_\rnk-\frac{\Delta}{3},$ where $\Delta\bydef\gamma_\rnk-\gamma_{\rnk+1},$ the functions
\begin{align*}
    \lambda\mapsto\frac{1}{(\gamma-\lambda)^m}\quad\text{and}\quad \lambda\mapsto \frac{1}{(\gamma-\lambda)(\gamma'-\lambda)}
\end{align*}
defined on $\left(-\infty,\gamma_{\rnk+1}+\Delta/3\right)$ are continuous and bounded. Then, since $\mathrm{supp}(\mu_{\dim})\subset(-\infty,\lambda_{\rnk+1}]\subset(-\infty,\gamma_{\rnk+1}+\Delta/3)$ for large enough $\dim,$ the assumption that $\mu_{\dim}\rightarrow\mu$ implies that
\begin{subequations}\label{eq:HK_obs}
\begin{align}
    \E_{\rLambda\sim\mu_{\dim}}\left[\frac{1}{(\gamma_j-\rLambda)^m}\right]\rightarrow \E_{\rLambda\sim\mu}\left[\frac{1}{(\gamma_j-\rLambda)^m}\right]
\end{align}
and
\begin{align}
    \E_{\rLambda\sim\mu_{\dim}}\left[\frac{1}{(\gamma_j-\rLambda)(\gamma_\ell-\rLambda)}\right]\rightarrow \E_{\rLambda\sim\mu}\left[\frac{1}{(\gamma_j-\rLambda)(\gamma_\ell-\rLambda)}\right].
\end{align}
\end{subequations}
Fix any $j,\ell\in[\rnk].$ To derive the limits of 
\begin{align*}
    \E_{\rLambda\sim\mu_{\dim}}\left[\frac{1}{(\xi_j-\rLambda)^m}\right]\quad\text{and}\quad \E_{\rLambda\sim\mu_{\dim}}\left[\frac{1}{(\lambda_j-\rLambda)(\lambda_\ell-\rLambda)}\right],
\end{align*}
observe that for any $\epsilon\in(0,\Delta/3),$ $\lambda_j\in(\gamma_j-\epsilon,\gamma_j+\epsilon)$ and $\lambda_\ell\in(\gamma_\ell-\epsilon,\gamma_\ell+\epsilon)$ for large enough $\dim,$ which implies:
\begin{align*}
    \E_{\rLambda\sim\mu_{\dim}}\left[\frac{1}{(\gamma_j+\epsilon-\rLambda)^m}\right]\leq\E_{\rLambda\sim\mu_{\dim}}\left[\frac{1}{(\xi_j-\rLambda)^m}\right]\leq \E_{\rLambda\sim\mu_{\dim}}\left[\frac{1}{(\gamma_j-\epsilon-\rLambda)^m}\right]
\end{align*}
and 
\begin{align*}
    \E_{\rLambda\sim\mu_{\dim}}\left[\frac{1}{(\gamma_j+\epsilon-\rLambda)(\gamma_\ell+\epsilon-\rLambda)}\right]&\leq\E_{\rLambda\sim\mu_{\dim}}\left[\frac{1}{(\lambda_j-\rLambda)(\lambda_\ell-\rLambda)}\right]\\
    &\hspace{2cm}\leq \E_{\rLambda\sim\mu_{\dim}}\left[\frac{1}{(\gamma_j-\epsilon-\rLambda)(\gamma_\ell-\epsilon-\rLambda)}\right].
\end{align*}
Then, taking $\dim\rightarrow\infty,$ it follows from the observation in \eqref{eq:HK_obs} that 
\begin{align*}
    \E_{\rLambda\sim\mu}\left[\frac{1}{(\gamma_j+\epsilon-\rLambda)^m}\right]\leq\lm\E_{\rLambda\sim\mu_{\dim}}\left[\frac{1}{(\xi_j-\rLambda)^m}\right]\leq \E_{\rLambda\sim\mu}\left[\frac{1}{(\gamma_j-\epsilon-\rLambda)^m}\right]
\end{align*}
and 
\begin{align*}
    \E_{\rLambda\sim\mu}\left[\frac{1}{(\gamma_j+\epsilon-\rLambda)(\gamma_\ell+\epsilon-\rLambda)}\right]&\leq\lm \E_{\rLambda\sim\mu_{\dim}}\left[\frac{1}{(\lambda_j-\rLambda)(\lambda_\ell-\rLambda)}\right]\\
    &\hspace{2cm}\leq \E_{\rLambda\sim\mu}\left[\frac{1}{(\gamma_j-\epsilon-\rLambda)(\gamma_\ell-\epsilon-\rLambda)}\right].
\end{align*}
Finally, by the dominated convergence theorem, we have:
\begin{align*}
    \underset{\epsilon\rightarrow0}{\lim}\;\E_{\rLambda\sim\mu}\left[\frac{1}{(\gamma_j\pm\epsilon-\rLambda)^m}\right]=\E_{\rLambda\sim\mu}\left[\frac{1}{(\gamma_j-\rLambda)^m}\right]
\end{align*}
and
\begin{align*}
   \underset{\epsilon\rightarrow0}{\lim}\;\E_{\rLambda\sim\mu}\left[\frac{1}{(\gamma_j\pm\epsilon-\rLambda)(\gamma_\ell\pm\epsilon-\rLambda)}\right]=\E_{\rLambda\sim\mu}\left[\frac{1}{(\gamma_j-\rLambda)(\gamma_\ell-\rLambda)}\right],
\end{align*}
which immediately implies the claimed statement.
\end{proof}

\begin{lemma}\label{lem:misc-tv-var-ratio}
For any sequences $\mu_{\dim}\in\R$ and $v_{\dim},\tilde{v}_{\dim}>0$ satisfying $\tilde{v}_{\dim}/v_{\dim}\rightarrow1,$ 
\begin{align*}
    \tv(\gauss{\mu_{\dim}}{v_{\dim}},\gauss{\mu_{\dim}}{\tilde{v}_{\dim}})\rightarrow0.
\end{align*}
\end{lemma}
\begin{proof}
    Fix any sequences $\mu_{\dim}\in\R$ and $v_{\dim},\tilde{v}_{\dim}>0$ satisfying $\tilde{v}_{\dim}/v_{\dim}\rightarrow1.$
    We will show:
    \begin{align*}
        \frac{\diff\gauss{\mu_{\dim}}{\tilde{v}_{\dim}}}{\diff\gauss{\mu_{\dim}}{v_{\dim}}}(\rz)\pc 1\quad\text{when}\quad\rz\sim\gauss{\mu_{\dim}}{v_{\dim}},
    \end{align*}
    which, along with Scheffé's lemma, implies the claim. Let $\rz\sim\gauss{\mu_{\dim}}{v_{\dim}},$ and observe that:
    \begin{align*}
        \frac{\diff\gauss{\mu_{\dim}}{\tilde{v}_{\dim}}}{\diff\gauss{\mu_{\dim}}{v_{\dim}}}(\rz)&=\sqrt{\frac{v_{\dim}}{\tilde{v}_{\dim}}}\exp\left((\rz-\mu_{\dim})^2\left(\frac{1}{2v_{\dim}}-\frac{1}{2\tilde{v}_{\dim}}\right)\right)\\
        &=\sqrt{\frac{v_{\dim}}{\tilde{v}_{\dim}}}\exp\left(\left(\frac{\rz-\mu_{\dim}}{\sqrt{v_{\dim}}}\right)^2\left(\frac{1}{2}-\frac{v_{\dim}}{2\tilde{v}_{\dim}}\right)\right)\pc 1
    \end{align*}
    by Slutsky's theorem since 
    \begin{align*}
        \left(\frac{\rz-\mu_{\dim}}{\sqrt{v_{\dim}}}\right)^2\dc \chi_1^2\quad\text{and}\quad \frac{1}{2}-\frac{v_{\dim}}{2\tilde{v}_{\dim}}\rightarrow 0.
    \end{align*}
\end{proof}

\begin{lemma}\label{lem:Slutsky} 
Let $\mX_{\dim},\mX_{\dim}'$ be two sequences of random variables satisfying
\begin{align*}
    \mX_{\dim}\dc \gauss{\mu}{\sigma^2},\quad \mX_{\dim}'\dc \gauss{\mu'}{\sigma^2}\quad\text{for some }\mu,\mu'\;\in\;\R,\;\sigma^2>0.
\end{align*}
Suppose there exists a sequence of deterministic values $c_{\dim}$ such that 
\begin{align}\label{lem:Slutsky-assump}
    \mX_{\dim}-\mX_{\dim}'=c_{\dim}.   
\end{align}
Then, $c_{\dim}\rightarrow \mu-\mu'.$
\end{lemma}
\begin{proof}
We show that $$\ls c_{\dim}=\li c_{\dim}=\mu-\mu'.$$ Let us pass to a subsequence that achieves the limit superior, and observe that Slutsky's theorem implies:
\begin{align*}
    \mX_{\dim}'\explain{\eqref{lem:Slutsky-assump}}{=}\mX_{\dim}-c_{\dim}\dc \gauss{\mu-\ls c_{\dim}}{\sigma^2}.
\end{align*}
Since $\mX_{\dim}^\prime \dc \gauss{\mu^\prime}{\sigma^2},$ we can deduce:
\begin{align*}
    \mu'=\mu-\ls c_{\dim}\implies \ls c_{\dim}=\mu-\mu'.
\end{align*}
An analogous argument implies that $\li c_{\dim}=\mu-\mu',$ which completes the proof.  
\end{proof}

%% file: appendix_notes_on_figures.tex
\section{Notes on Figures}\label{sec:figures-info} In this section, we elaborate on various aspects of the figures, including the 1000 Genomes dataset and the implementation of the codes, which can be found on \verb|https://github.com/yjyun97/DPPCA|. All of the experiments for the figures were run in \verb|R|, and the figures themselves were generated using \verb|MATLAB|. The computations for Figures \mref{fig:main-utility} and \mref{fig:privacy-tf} were performed on a cluster at the Center for High Throughput Computing \citep{https://doi.org/10.21231/gnt1-hw21}.  

\paragraph{Description of the 1000 Genomes Project Dataset} The dataset we use as a running example is a subset of the cleaned version of the 1000  Genomes Project Dataset that we obtained by following the tutorial of \citet{biostars}. The tutorial produces a dataset consisting of 466,487 SNPs for 2504 individuals, from which we randomly sample 200 SNPs to create a subset of $\ssize=2504$ individuals each with $\dim=200$ SNPs, which we refer to throughout this paper as the 1000 Genomes dataset. The dataset comes with information on the sub-populations that the individuals belong to (e.g., Finnish, Peruvian), which we group into a coarser category of five super-populations (African, European, East Asian, Hispanic, and South Asian)  as is done in \citep[Supplementary Information Table 1]{10002015global} and \citep{zhong2022empirical}. In all of our figures involving the 1000 Genomes dataset, except for the left panel of \Cref{fig:projections-preprocessing}, we work with the standardized, rank-transformed dataset (see \eqref{eq:rank-cov}) to ensure that the norm constraint (Assumption \mref{assump:data}(3)) is satisfied. Hence, unless specified otherwise, the 1000 Genomes dataset in this paper refers to its standardized, rank-transformed version (the conversion to raw ranks is done through the \verb|R| function \verb|rank| with the parameter \verb|ties.method| set to \verb|"average"|).  The left panel of \Cref{fig:projections-preprocessing} does not use the rank transformed dataset, but instead preprocesses the 1000 Genomes dataset using the standard normalization, where each feature is centered by its mean and rescaled to have unit variance. 

\paragraph{Sampling from the Exponential Mechanism} To sample from the exponential mechanism, we use the Gibbs sampler of \citet{hoff2009simulation}, whose stationary distribution is the Gibbs distribution introduced in Definition \mref{def:Gibbs}. While the number of iterations required for convergence is unknown, we observed that the sample means of various one-dimensional statistics tend to stabilize after around 50 iterations. Hence, to draw one sample, we run the sampler for 200 iterations using the function \verb|rbing.matrix.gibbs| from the \verb|R| package \verb|rstiefel| written by \citet{hoff2009simulation}. 

\paragraph{Visualizing Projections (Figures \mref{fig:proj} and \mref{fig:privacy-projections})} When visualizing the projections of the dataset onto  privatized PCs $\rV,$ we rotate $\rV$ so that the orientation of the projections is comparable to that associated with the true PCs $\mU_\star$, using the function \verb|PROCRUSTES| from the \verb|R| package \verb|EFA.dimensions|. This process does not change the subspace spanned by the privatized principal components, but only finds a basis of this subspace which is most similar to the true PCs. 

\paragraph{Estimating Estimation Error (Figure \mref{fig:main-utility})} To estimate the finite-$\dim$ estimation errors in Figure \mref{fig:main-utility}, we approximate the expectation with the sample mean computed using 100,000 samples from the Gibbs distribution, as follows: 
\begin{align*} 
   \E \|\mU_\star \mU_\star^\top - \rV \rV^\top\|^2 &\approx \frac{1}{100,000} \sum_{i=1}^{100,000} \|\mU_\star \mU_\star^\top - \rV_{(i)} \rV_{(i)}^\top\|^2\\
   \E \|\mU_\star \mU_\star^\top - \rV \rV^\top\|_{\fr}^2  &\approx \frac{1}{100,000} \sum_{i=1}^{100,000} \|\mU_\star \mU_\star^\top - \rV_{(i)}  \rV_{(i)} ^\top\|_{\fr}^2,
\end{align*}
where $\rV_{(i)}$ denotes the $i$th sample drawn from the Gibbs distribution.

\paragraph{Estimating Trade-off Functions (Figure \mref{fig:privacy-tf})} Figure \mref{fig:privacy-tf} presents an estimation of the finite-$\dim$ trade-off function 
\begin{align}\label{eq:tf-finite-p}
    \tf{\nu(\cdot\mid \mSigma(\mX),\beta,\rnk)}{\nu(\cdot \mid \mSigma(\mX\cup\{\vx_\star\}), \beta, \rnk)},
\end{align}
where $\vx_\star$ is the worst-case data point constructed in Proposition \mref{prop:var-lim-v2} (see Section \mref{sec:privacy}). 
% Recall from \Cref{prop:var-lim-v2} that the worst-case dataset $\tilde{\mX}_\star=\mX\cup \{\vx_\star\}$
% constructed as:
% \begin{align*}
%     \vx_\star & = \sqrt{\dim} \left(  \sqrt{t_\star}  \vu_{\rnk} + \sqrt{1-t_\star}  \vu_{\rnk+1} \right)\quad\text{with }t_\star\bydef\min\left(\frac{{\beta-H_{\mu}(\gamma_\rnk)}}{2(\beta-H_{\mu}(\gamma_\rnk)) + \Delta H'_{\mu}(\gamma_\rnk)}, 1\right)
% \end{align*}
% achieves the limiting worst-case privacy guarantee: 
% \begin{multline*}
%     \lm \inf_{\tilde{\mX}\in\calN(\mX)}\tf{\nu(\cdot\mid \mSigma(\mX),\beta,\rnk)}{\nu(\cdot \mid \mSigma(\tilde{\mX}), \beta, \rnk)}(\alpha)=\\
%     \lm \tf{\nu(\cdot\mid \mSigma(\mX),\beta,\rnk)}{\nu(\cdot \mid \mSigma(\mX\cup\{\vx_\star\}), \beta, \rnk)}(\alpha)\quad\forall\;\alpha\;\in\;[0,1].
% \end{multline*}
We estimate the trade-off function \eqref{eq:tf-finite-p} using the following procedure:
\begin{enumerate}
    \item Generate 100,000 samples of the test statistic $\|\rV^\top \vx_\star\|^2-\Tr[\rV^\top\mSigma(\mX)\rV]$ 
    from the null distribution $\nu(\cdot\mid \mSigma(\mX),\beta,\rnk):$
    \begin{align*}
        \|\rV_{(i)}^\top \vx_\star\|^2-\Tr[\rV_{(i)}^\top\mSigma(\mX)\rV_{(i)}],\quad \rV_{(i)}\explain{i.i.d.}{\sim}\nu(\cdot\mid \mSigma(\mX),\beta,\rnk),\;i\in[100,000] .
    \end{align*}
    \item Compute the $1-\alpha$-th sample quantile $t_\alpha$ of the 100,000 samples $\|\rV_{(i)}^\top \vx_\star\|^2-\Tr[\rV_{(i)}^\top\mSigma(\mX)\rV_{(i)}].$ 
    \item Generate 100,000 samples of the test statistic $\|\rV^\top \vx_\star\|^2-\Tr[\rV^\top\mSigma(\mX)\rV]$ 
    from the alternative distribution $\nu(\cdot \mid \mSigma({\mX} \cup \{\vx_\star\}), \beta, \rnk):$
    \begin{align*}
        \|\tilde{\rV}_{(i)}^\top \vx_\star\|^2-\Tr[\tilde{\rV}_{(i)}^\top\mSigma(\mX)\tilde{\rV}_{(i)}],\quad \tilde{\rV}_{(i)}\explain{i.i.d.}{\sim}\nu(\cdot\mid \mSigma({\mX} \cup \{\vx_\star\}),\beta,\rnk),\; i\in[100,000].
    \end{align*}
    \item Estimate the trade-off function \eqref{eq:tf-finite-p} by computing the proportion of the samples $\|\tilde{\rV}_{(i)}^\top \vx_\star\|^2-\Tr[\tilde{\rV}_{(i)}^\top\mSigma(\mX)\tilde{\rV}_{(i)}]$  that are less than $t_\alpha:$
    \begin{align*}
        &\tf{\nu(\cdot\mid \mSigma(\mX),\beta,\rnk)}{\nu(\cdot \mid \mSigma(\mX\cup\{\vx_\star\}), \beta, \rnk)}(\alpha)\\
        \approx&\frac{1}{100,000}\sum_{i=1}^{100,000} \ind_{\{\|\tilde{\rV}_{(i)}^\top \vx_\star\|^2-\Tr[\tilde{\rV}_{(i)}^\top\mSigma(\mX)\tilde{\rV}_{(i)}]< t_\alpha\}}.
    \end{align*}
\end{enumerate}

%% file: appendix_survey.tex
\section{Additional Dataset: ANES Survey Dataset}\label{app:survey} 
This section includes empirical results on an additional dataset consisting of the survey responses from the American National Election Studies (ANES) 2024 Pilot Study \citep{ANES2024Pilot}. We provide figures analogous to Figures \mref{fig:main-utility}, \mref{fig:privacy-parameter}, \mref{fig:privacy-projections}, and \mref{fig:privacy-tf}, complementing our illustrations of the 1000 Genomes dataset in the main text. The figures are generated in the same way as detailed in \Cref{sec:figures-info}.

\paragraph{Description of the ANES Survey Dataset} The ANES website provides {two files, one containing the questionnaire, and the other containing the responses}. We combine and organize information from the two sources to create a dataset, which is the dataset we use in our experiments and is available on \verb|https://github.com/yjyun97/DPPCA|. The cleaned dataset, which we refer to as the ANES survey dataset throughout this section, consists of $n=1548$ rows (each corresponding to a survey respondent) and $p=107$ columns (each corresponding to a survey question). As with the 1000 Genomes dataset, we standardize and rank-transform the ANES survey dataset in all experiments associated with the figures below. Before presenting the figures, we briefly summarize the steps we took to clean the dataset:
\begin{itemize}
    \item We removed individuals with incomplete responses or responses containing ``inapplicable, legitimate skip" or ``no answer."  
    \item We removed short-answer questions and questions that depend on previous questions. 
    \item We converted ordinal variables to numeric values and categorical variables to Boolean (or one-hot) encoding. 
    \item We only kept the questions that ask the respondent about their opinion on various political issues (e.g. ``How important is the issue of jobs and employment in the country today?"). For example, we removed questions that ask about one's habits (e.g. ``Generally speaking, how would you rate the quality of the sleep you usually get?") and one's demographic identity (e.g. ``Are you of Middle Eastern or North African descent?"). \item We removed features that either directly or indirectly ask about the respondent's political party preference (e.g. ``Please tell us which political party—the Democrats or the Republicans—would do a better job handling the issue of the cost of living and rising prices, or is there no difference?"). This is because we wanted to investigate if their political party preferences can be recovered from their opinions on various political issues in an unsupervised manner.  
\end{itemize}
We used the question ``Generally speaking, do you usually think of yourself as a Democrat, a Republican, an independent, or what?" to assign a label to each individual from the options Democrat, Republican, Independent, and Other. We excluded this question from the features. The aforementioned GitHub website contains additional details on how we generated the dataset. 

\paragraph{Figures} \Cref{fig:privacy-projections-survey} visualizes the projections of the ANES survey dataset onto the non-private PCs and the privatized PCs generated by the exponential mechanism for various values of $\beta$, along with the estimate of the asymptotic privacy loss derived from our privacy result (Theorem \mref{thm:privacy}).  We see that the first principal component appears to separate Democratic respondents from Republican respondents, suggesting that the principal components can recover information about participants' political party preferences from their political opinions.  
\begin{figure}
    \centering
    \includegraphics[width=0.8\linewidth]{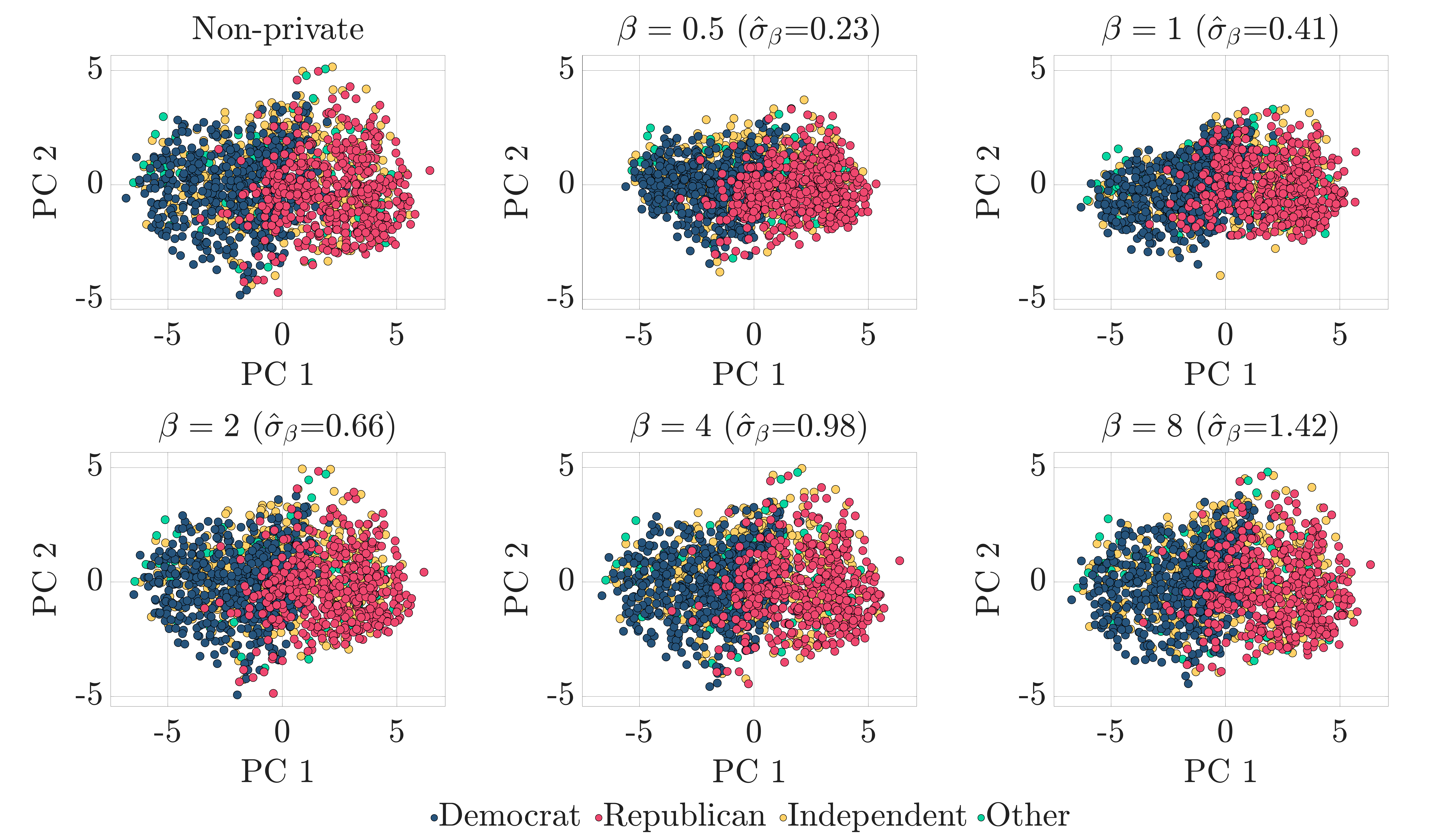}
    \caption{Projections of the ANES survey dataset onto the first $\rnk=2$ non-private PCs and privatized PCs for $\beta=0.5,$ 
    $\beta=1,$ $\beta=2,$ $\beta=4,$ and $\beta=8.$  Theorem \mref{thm:privacy} predicts that the privatized PCs with these $\beta$ values has reasonable estimated AGDP losses of $\hat{\sigma}_\beta=0.23,$ $\hat{\sigma}_\beta=0.41,$ $\hat{\sigma}_\beta=0.66,$ $\hat{\sigma}_\beta=0.98,$ and $\hat{\sigma}_\beta=1.42,$ respectively.} \label{fig:privacy-projections-survey}
\end{figure}

\Cref{fig:privacy-parameter-survey} plots the estimated AGDP loss of the exponential mechanism on the ANES survey dataset (obtained from Theorem \mref{thm:privacy}) as a function of the noise parameter $\beta$. Unexpectedly, for most values of $\beta$, publishing $k=2$ PCs has a higher estimated privacy cost than publishing $k=3$ PCs. This is because the eigengap between the 2nd and 3rd eigenvalues is smaller than the gap between the 3rd and 4th eigenvalues ($\lambda_2 - \lambda_3=0.485$ and $\lambda_3 - \lambda_4=1.062$) for the ANES dataset. 

\Cref{fig:main-utility-survey} provides the analog of Figure \mref{fig:main-utility} (in the main paper) for the ANES dataset. We find that the empirically estimated finite-$\dim$ estimation error of the exponential mechanism on the ANES dataset aligns closely with the theoretical predictions of Theorem \mref{thm:utility}.

Finally, \Cref{fig:privacy-tf-survey} provides the analog of Figure \mref{fig:privacy-tf} (in the main paper) for the ANES dataset. We find that for most values of $(k,\beta)$, the empirically estimated trade-off functions measuring the finite-$\dim$ privacy loss of the exponential mechanism closely track the theoretical predictions derived from Theorem \mref{thm:privacy}. An exception is the $(k = 3, \beta = 0.5)$ case, where the theoretical predictions do not approximate the finite-$\dim$ trade-off function well. We believe that this discrepancy is due to $\beta=0.5$ being very close to the (estimated) critical $\beta$ value of $H_\mu(\gamma_3)=0.494,$ below which our privacy result (Theorem \mref{thm:privacy}) does not apply.

\begin{figure}
  \centering
  \includegraphics[width=0.4\linewidth]{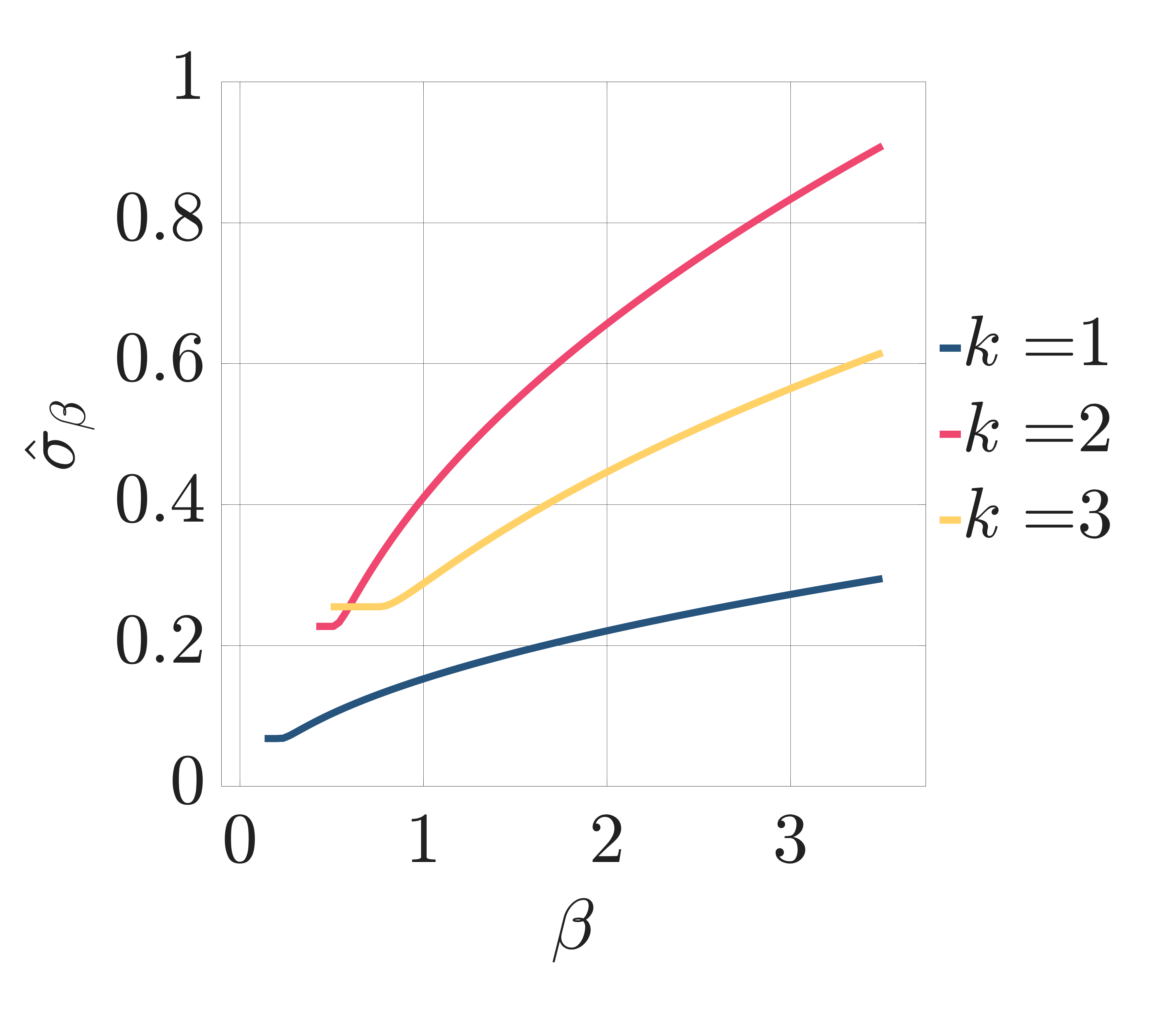}
  \caption{Noise parameter $\beta$ v.s. privacy parameter $\sigma_{\beta}$ for the ANES survey dataset for $k = 1,2,3$.}
  \label{fig:privacy-parameter-survey}
\end{figure}

\begin{figure}
    \centering
    \includegraphics[width=0.70\linewidth]{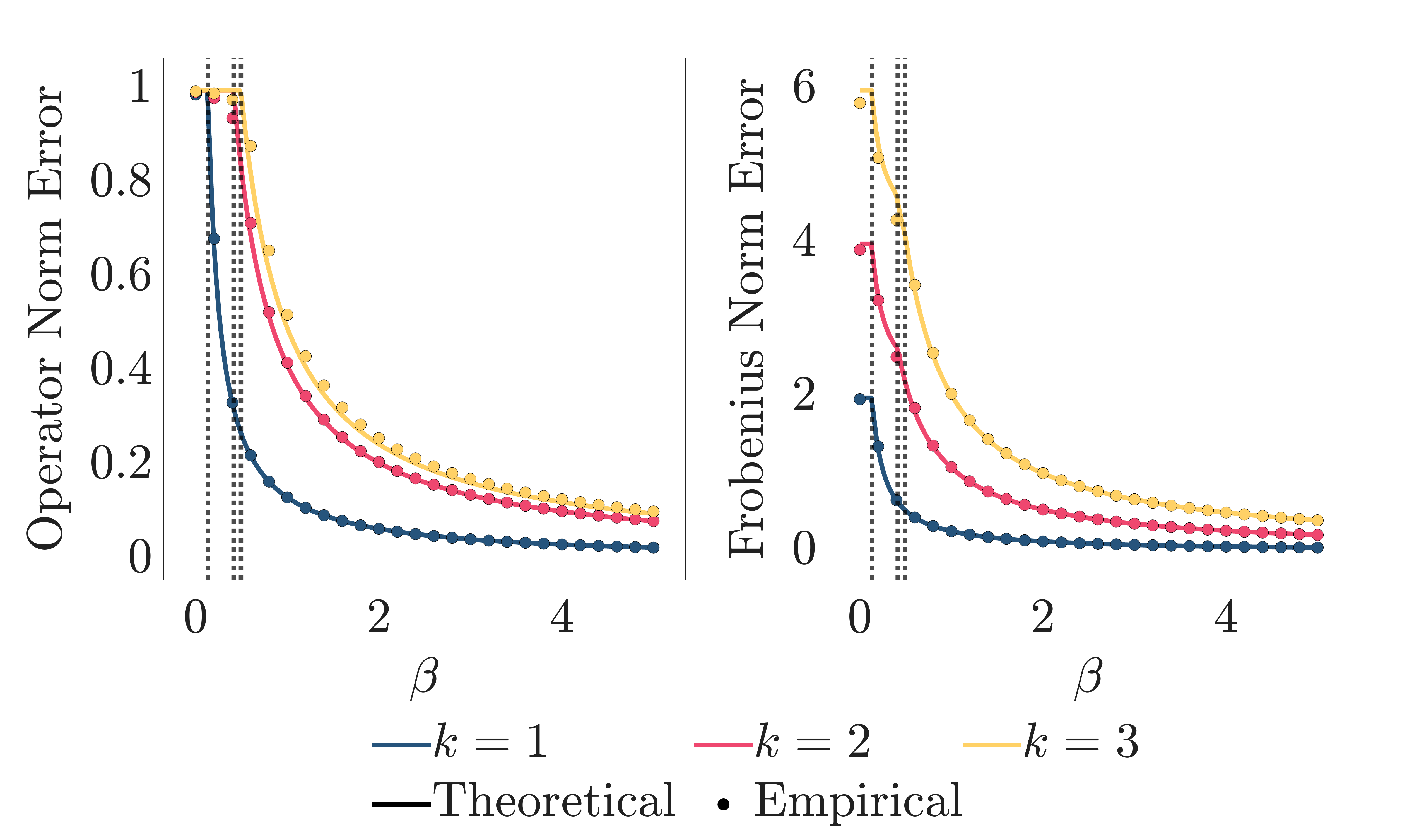}
    \caption{Estimation error of the exponential mechanism in operator norm (left panel) and Frobenius norm (right panel) as a function of $\beta$ for $\rnk = 1,2,3$ on the ANES survey dataset. Solid curves represent theoretical predictions derived from Theorem \mref{thm:utility} and circular markers represent the empirical error, averaged over 100,000 Monte Carlo simulations. The dotted vertical lines represent the threshold values $\beta = H_\mu(\gamma_{1:3})$, at which phase transitions occur.}
    \label{fig:main-utility-survey}
\end{figure}

\begin{figure}
    \centering
    \includegraphics[width=0.9\linewidth]{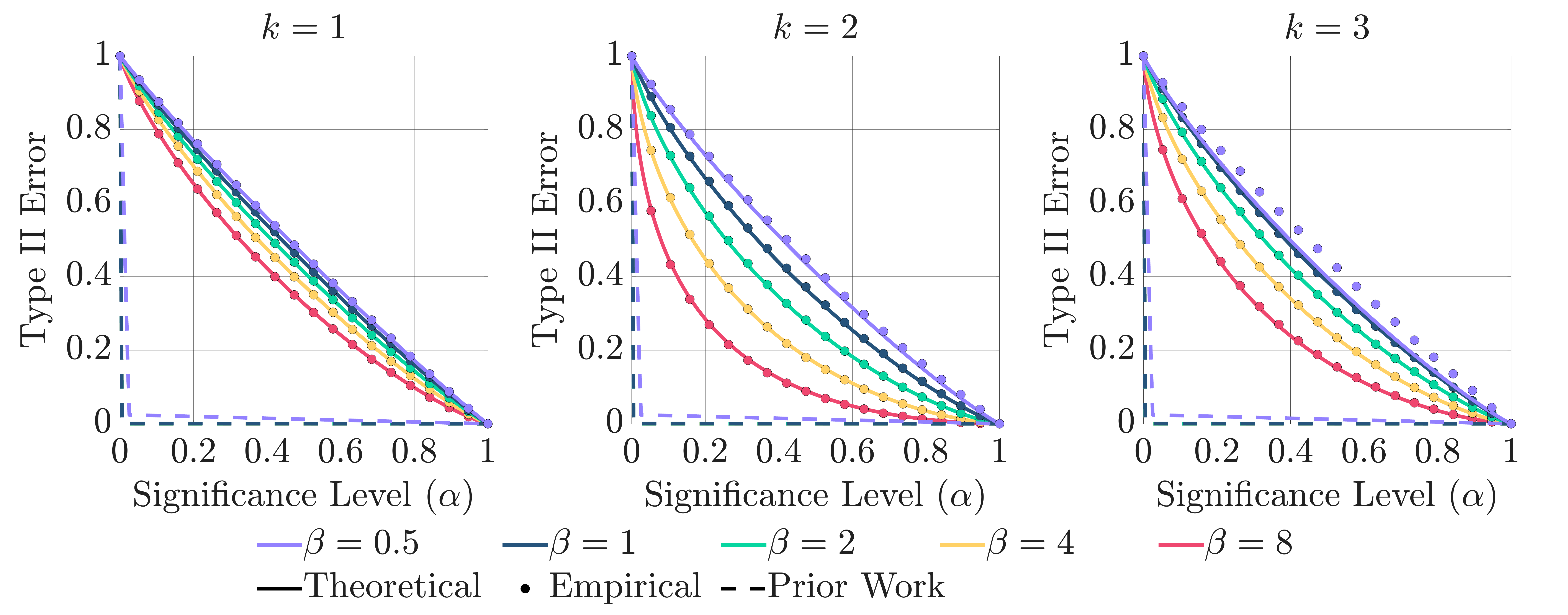}
    \caption{Trade-off functions for the exponential mechanism on the ANES survey dataset for rank $\rnk \in \{1, 2, 3\}$ and $\beta\in\{0.5,1,2,4,8\}.$  Solid curves: theoretical predictions from Theorem \mref{thm:privacy}.  Circular markers: empirically estimated trade-off functions (using 100,000 Monte Carlo samples). Dashed curves:  non-asymptotic privacy bounds from prior work \citep{chaudhuri2013near}, which are essentially vacuous for these values of $\beta$ and coincide with the axes. }\label{fig:privacy-tf-survey}
\end{figure}

%% file: refs.bib
@article{chaudhuri2013near,
  title={A near-optimal algorithm for differentially-private principal components},
  author={Chaudhuri, Kamalika and Sarwate, Anand D and Sinha, Kaushik},
  journal={The Journal of Machine Learning Research},
  volume={14},
  number={1},
  pages={2905--2943},
  year={2013},
  publisher={JMLR. org}
}

@misc{donoho2000highdimensional,
  author       = {Donoho, David L.},
  title        = {High-Dimensional Data Analysis: The Curses and Blessings of Dimensionality},
  year         = {2000},
  month        = aug,
  howpublished = {Aide-m{\'e}moire of a lecture at the American Mathematical Society conference
                  ``Mathematical Challenges of the 21st Century''},
  note         = {Los Angeles, California, August 8, 2000},
  url          = {https://web.archive.org/web/20030405145144id_/http://www-stat.stanford.edu:80/~donoho/Lectures/AMS2000/Curses.pdf}
}

@article{guionnet2021asymptotics,
  author  = {Guionnet, Alice and Husson, Jonathan},
  title   = {Asymptotics of $k$ dimensional spherical integrals and applications},
  journal = {ALEA: Latin American Journal of Probability and Mathematical Statistics},
  volume  = {19},
  pages   = {769--797},
  year    = {2022},
  doi     = {10.30757/ALEA.v19-30}
}

@article{liao2025testing,
  title={Testing for latent structure via the {W}ilcoxon--{W}igner random matrix of normalized rank statistics},
  author={Liao, Jonquil Z and Cape, Joshua},
  journal={arXiv preprint arXiv:2512.18924},
  year={2025}
}

@article{chatterjee2009fluctuations,
  title={Fluctuations of eigenvalues and second order {P}oincar{\'e} inequalities},
  author={Chatterjee, Sourav},
  journal={Probability Theory and Related Fields},
  volume={143},
  number={1},
  pages={1--40},
  year={2009},
  publisher={Springer}
}

@article{mukherjee2013statistics,
  author  = {Mukherjee, Sumit and Xu, Yuanzhe},
  title   = {Statistics of the two star {ERGM}},
  journal = {Bernoulli},
  volume  = {29},
  number  = {1},
  pages   = {24--51},
  month   = feb,
  year    = {2023},
  doi     = {10.3150/21-BEJ1448}
}

@article{redberg2021privately,
  title={Privately publishable per-instance privacy},
  author={Redberg, Rachel and Wang, Yu-Xiang},
  journal={Advances in Neural Information Processing Systems},
  volume={34},
  pages={17335--17346},
  year={2021}
}

@book{vershynin2018high,
  title={High-dimensional probability: An introduction with applications in data science},
  author={Vershynin, Roman},
  volume={47},
  year={2018},
  publisher={Cambridge university press}
}

@book{meckes2019random,
  title={The random matrix theory of the classical compact groups},
  author={Meckes, Elizabeth S},
  volume={218},
  year={2019},
  publisher={Cambridge University Press}
}

@article{dawid1977spherical,
  title={Spherical matrix distributions and a multivariate model},
  author={Dawid, AP},
  journal={Journal of the Royal Statistical Society Series B: Statistical Methodology},
  volume={39},
  number={2},
  pages={254--261},
  year={1977},
  publisher={Oxford University Press}
}

@book{searle1997linear,
  title={Linear models},
  author={Searle, Shayle R},
  year={1997},
  publisher={John Wiley \& Sons}
}

@book{van2000asymptotic,
  title={Asymptotic statistics},
  author={Van der Vaart, Aad W},
  volume={3},
  year={2000},
  publisher={Cambridge university press}
}

@article{chen2021spectral,
  title={Spectral methods for data science: A statistical perspective},
  author={Chen, Yuxin and Chi, Yuejie and Fan, Jianqing and Ma, Cong and others},
  journal={Foundations and Trends{\textregistered} in Machine Learning},
  volume={14},
  number={5},
  pages={566--806},
  year={2021},
  publisher={Now Publishers, Inc.}
}

@article{wedin1973perturbation,
  title={Perturbation theory for pseudo-inverses},
  author={Wedin, Per-{\AA}ke},
  journal={BIT Numerical Mathematics},
  volume={13},
  pages={217--232},
  year={1973},
  publisher={Springer}
}

@article{lewis1996derivatives,
  title={Derivatives of spectral functions},
  author={Lewis, Adrian S},
  journal={Mathematics of Operations Research},
  volume={21},
  number={3},
  pages={576--588},
  year={1996},
  publisher={INFORMS}
}

@article{guionnet2005fourier,
title = {A {F}ourier view on the {R}-transform and related asymptotics of spherical integrals},
journal = {Journal of Functional Analysis},
volume = {222},
number = {2},
pages = {435-490},
year = {2005},
author = {A. Guionnet and M. Maida}
}

@inproceedings{mironov2017renyi,
  title={R{\'e}nyi differential privacy},
  author={Mironov, Ilya},
  booktitle={2017 IEEE 30th computer security foundations symposium (CSF)},
  pages={263--275},
  year={2017},
  organization={IEEE}
}

@article{dong2022gaussian,
  title={Gaussian differential privacy},
  author={Dong, Jinshuo and Roth, Aaron and Su, Weijie J},
  journal={Journal of the Royal Statistical Society: Series B (Statistical Methodology)},
  volume={84},
  number={1},
  pages={3--37},
  year={2022},
  publisher={Wiley Online Library}
}

@inproceedings{dwork2006calibrating,
  title={Calibrating noise to sensitivity in private data analysis},
  author={Dwork, Cynthia and McSherry, Frank and Nissim, Kobbi and Smith, Adam},
  booktitle={Theory of cryptography conference},
  pages={265--284},
  year={2006},
  organization={Springer}
}

@book{lehmann2005testing,
  title={Testing statistical hypotheses},
  author={Lehmann, Erich Leo and Romano, Joseph P},
  year={2005},
  publisher={Springer}
}

@article{kairouz2015composition,
author = {Kairouz, Peter and Oh, Sewoong and Viswanath, Pramod},
title = {The Composition Theorem for Differential Privacy},
year = {2017},
issue_date = {June 2017},
publisher = {IEEE Press},
volume = {63},
number = {6},
issn = {0018-9448},
url = {https://doi.org/10.1109/TIT.2017.2685505},
doi = {10.1109/TIT.2017.2685505},
journal = {IEEE Trans. Inf. Theor.},
month = jun,
pages = {4037–4049},
numpages = {13}
}

@article{wasserman2010statistical,
  title={A statistical framework for differential privacy},
  author={Wasserman, Larry and Zhou, Shuheng},
  journal={Journal of the American Statistical Association},
  volume={105},
  number={489},
  pages={375--389},
  year={2010},
  publisher={Taylor \& Francis}
}

@book{james2013introduction,
  title={An introduction to statistical learning: with applications in R},
  author={James, Gareth and Witten, Daniela and Hastie, Trevor and Tibshirani, Robert},
  volume={103},
  year={2013},
  publisher={Springer}
}

@inproceedings{dwork2014analyze,
  title={Analyze {G}auss: optimal bounds for privacy-preserving principal component analysis},
  author={Dwork, Cynthia and Talwar, Kunal and Thakurta, Abhradeep and Zhang, Li},
  booktitle={Proceedings of the forty-sixth annual ACM symposium on Theory of computing},
  pages={11--20},
  year={2014}
}

@article{amin2019differentially,
  title={Differentially private covariance estimation},
  author={Amin, Kareem and Dick, Travis and Kulesza, Alex and Munoz, Andres and Vassilvitskii, Sergei},
  journal={Advances in Neural Information Processing Systems},
  volume={32},
  year={2019}
}

@inproceedings{kapralov2013differentially,
  title={On differentially private low rank approximation},
  author={Kapralov, Michael and Talwar, Kunal},
  booktitle={Proceedings of the twenty-fourth annual ACM-SIAM symposium on Discrete algorithms},
  pages={1395--1414},
  year={2013},
  organization={SIAM}
}

@inproceedings{wei2016analysis,
  title={Analysis of a privacy-preserving {PCA} algorithm using random matrix theory},
  author={Wei, Lu and Sarwate, Anand D and Corander, Jukka and Hero, Alfred and Tarokh, Vahid},
  booktitle={2016 IEEE Global Conference on Signal and Information Processing (GlobalSIP)},
  pages={1335--1339},
  year={2016},
  organization={IEEE}
}

@book{kato2013perturbation,
  title={Perturbation theory for linear operators},
  author={Kato, Tosio},
  volume={132},
  year={2013},
  publisher={Springer Science \& Business Media}
}

@article{xia2021normal,
  title={Normal approximation and confidence region of singular subspaces},
  author={Xia, Dong},
  journal={Electronic Journal of Statistics},
  volume={15},
  number={2},
  pages={3798--3851},
  year={2021},
  publisher={The Institute of Mathematical Statistics and the Bernoulli Society}
}

@inproceedings{koltchinskii2016asymptotics,
  title={Asymptotics and concentration bounds for bilinear forms of spectral projectors of sample covariance},
  author={Koltchinskii, Vladimir and Lounici, Karim},
  booktitle={Annales de l’Institut Henri Poincar{\'e}-Probabilit{\'e}s et Statistiques},
  volume={52},
  pages={1976--2013},
  year={2016}
}

@article{he2025differentially,
  title={Differentially private low-dimensional synthetic data from high-dimensional datasets},
  author={He, Yiyun and Strohmer, Thomas and Vershynin, Roman and Zhu, Yizhe},
  journal={Information and Inference: A Journal of the IMA},
  volume={14},
  number={1},
  pages={iaae034},
  year={2025},
  publisher={Oxford University Press}
}

@article{liu2022dp,
  title={{DP-PCA}: Statistically optimal and differentially private {PCA}},
  author={Liu, Xiyang and Kong, Weihao and Jain, Prateek and Oh, Sewoong},
  journal={Advances in neural information processing systems},
  volume={35},
  pages={29929--29943},
  year={2022}
}

@inproceedings{leake2021sampling,
  title={Sampling matrices from {H}arish-{C}handra--{I}tzykson--{Z}uber densities with applications to quantum inference and differential privacy},
  author={Leake, Jonathan and McSwiggen, Colin and Vishnoi, Nisheeth K},
  booktitle={Proceedings of the 53rd Annual ACM SIGACT Symposium on Theory of Computing},
  pages={1384--1397},
  year={2021}
}

@inproceedings{gonem2018smooth,
  title={Smooth sensitivity based approach for differentially private {PCA}},
  author={Gonen, Alon and Gilad-Bachrach, Ram},
  booktitle={Algorithmic Learning Theory},
  pages={438--450},
  year={2018},
  organization={PMLR}
}

@incollection{dwork2024differentially,
  author    = {Dwork, Cynthia and Tankala, Pranay and Zhang, Linjun},
  title     = {Differentially Private Learning Beyond the Classical Dimensionality Regime},
  booktitle = {Theory of Cryptography Conference (TCC 2025)},
  pages     = {321--355},
  year      = {2025},
  publisher = {Springer},
  doi       = {10.1007/978-3-032-12290-2_11}
}

@inproceedings{dwork2006our,
  title={Our data, ourselves: Privacy via distributed noise generation},
  author={Dwork, Cynthia and Kenthapadi, Krishnaram and McSherry, Frank and Mironov, Ilya and Naor, Moni},
  booktitle={Annual international conference on the theory and applications of cryptographic techniques},
  pages={486--503},
  year={2006},
  organization={Springer}
}

@inproceedings{mcsherry2007mechanism,
  title={Mechanism design via differential privacy},
  author={McSherry, Frank and Talwar, Kunal},
  booktitle={48th Annual IEEE Symposium on Foundations of Computer Science (FOCS'07)},
  pages={94--103},
  year={2007},
  organization={IEEE}
}

@article{mangoubi2022re,
  title={Re-analyze {G}auss: Bounds for private matrix approximation via {D}yson {B}rownian motion},
  author={Mangoubi, Oren and Vishnoi, Nisheeth},
  journal={Advances in Neural Information Processing Systems},
  volume={35},
  pages={38585--38599},
  year={2022}
}

@inproceedings{hardt2013beyond,
  title={Beyond worst-case analysis in private singular vector computation},
  author={Hardt, Moritz and Roth, Aaron},
  booktitle={Proceedings of the forty-fifth annual ACM symposium on Theory of computing},
  pages={331--340},
  year={2013}
}

@inproceedings{hardt2012beating,
  title={Beating randomized response on incoherent matrices},
  author={Hardt, Moritz and Roth, Aaron},
  booktitle={Proceedings of the forty-fourth annual ACM symposium on Theory of computing},
  pages={1255--1268},
  year={2012}
}

@article{hardt2014noisy,
  title={The noisy power method: A meta algorithm with applications},
  author={Hardt, Moritz and Price, Eric},
  journal={Advances in neural information processing systems},
  volume={27},
  year={2014}
}

@article{baik2016fluctuations,
  title={Fluctuations of the free energy of the spherical {S}herrington--{K}irkpatrick model},
  author={Baik, Jinho and Lee, Ji Oon},
  journal={Journal of Statistical Physics},
  volume={165},
  number={2},
  pages={185--224},
  year={2016},
  publisher={Springer}
}

@inproceedings{baik2017fluctuations,
  title={Fluctuations of the free energy of the spherical {S}herrington--{K}irkpatrick model with ferromagnetic interaction},
  author={Baik, Jinho and Lee, Ji Oon},
  booktitle={Annales Henri Poincar{\'e}},
  volume={18},
  pages={1867--1917},
  year={2017},
  organization={Springer}
}

@article{baik2021spherical,
  title={Spherical spin glass model with external field},
  author={Baik, Jinho and Collins-Woodfin, Elizabeth and Le Doussal, Pierre and Wu, Hao},
  journal={Journal of Statistical Physics},
  volume={183},
  number={2},
  pages={31},
  year={2021},
  publisher={Springer}
}

@article{kosterlitz1976spherical,
  title={Spherical model of a spin-glass},
  author={Kosterlitz, John M and Thouless, David J and Jones, Raymund C},
  journal={Physical Review Letters},
  volume={36},
  number={20},
  pages={1217},
  year={1976},
  publisher={APS}
}

@article{husson2025spherical,
  title={Spherical integrals of sublinear rank},
  author={Husson, Jonathan and Ko, Justin},
  journal={Probability Theory and Related Fields},
  pages={1--88},
  year={2025},
  publisher={Springer}
}

@article{landon2020fluctuations,
  title={Fluctuations of the 2-spin {SSK} model with magnetic field},
  author={Landon, Benjamin and Sosoe, Philippe},
  journal={arXiv preprint arXiv:2009.12514},
  year={2020}
}

@article{baik2018ferromagnetic,
  title={Ferromagnetic to paramagnetic transition in spherical spin glass},
  author={Baik, Jinho and Lee, Ji Oon and Wu, Hao},
  journal={Journal of Statistical Physics},
  volume={173},
  number={5},
  pages={1484--1522},
  year={2018},
  publisher={Springer}
}

@article{dungler2025iterative,
  title={An iterative algorithm for differentially private $ k $-{PCA} with adaptive noise},
  author={D{\"u}ngler, Johanna and Sanyal, Amartya},
  journal={Advances in Neural Information Processing Systems},
  volume={38},
  pages={132205--132250},
  year={2026}
}

@article{cai2024optimal,
  title={Optimal Differentially Private {PCA} and Estimation for Spiked Covariance Matrices},
  author={Cai, T Tony and Xia, Dong and Zha, Mengyue},
  journal={arXiv preprint arXiv:2401.03820},
  year={2024}
}

@article{singhal2021privately,
  title={Privately learning subspaces},
  author={Singhal, Vikrant and Steinke, Thomas},
  journal={Advances in Neural Information Processing Systems},
  volume={34},
  pages={1312--1324},
  year={2021}
}

@inproceedings{mcsherry2009differentially,
  title={Differentially private recommender systems: Building privacy into the {N}etflix prize contenders},
  author={McSherry, Frank and Mironov, Ilya},
  booktitle={Proceedings of the 15th ACM SIGKDD international conference on Knowledge discovery and data mining},
  pages={627--636},
  year={2009}
}

@inproceedings{hu2024provable,
  author    = {Hu, Yuchen and Sanyal, Amartya and Sch{\"o}lkopf, Bernhard},
  title     = {Provable Privacy with Non-Private Pre-Processing},
  booktitle = {Proceedings of the 41st International Conference on Machine Learning (ICML)},
  year      = {2024},
  series    = {Proceedings of Machine Learning Research},
  volume    = {235},
  pages     = {19402--19437},
  publisher = {JMLR.org},
  doi       = {10.5555/3692070.3692852}
}

@article{hoff2009simulation,
  title={Simulation of the matrix {B}ingham--von {M}ises--{F}isher distribution, with applications to multivariate and relational data},
  author={Hoff, Peter D},
  journal={Journal of Computational and Graphical Statistics},
  volume={18},
  number={2},
  pages={438--456},
  year={2009},
  publisher={Taylor \& Francis}
}

@article{kume2006sampling,
  title={Sampling from compositional and directional distributions},
  author={Kume, Alfred and Walker, Stephen G},
  journal={Statistics and Computing},
  volume={16},
  number={3},
  pages={261--265},
  year={2006},
  publisher={Springer}
}

@article{kent2004simulation,
  title={Simulation for the complex {B}ingham distribution},
  author={Kent, John T and Constable, Patrick DL and Er, Fikret},
  journal={Statistics and Computing},
  volume={14},
  number={1},
  pages={53--57},
  year={2004},
  publisher={Springer}
}

@article{kent2018new,
  title={A new unified approach for the simulation of a wide class of directional distributions},
  author={Kent, John T and Ganeiber, Asaad M and Mardia, Kanti V},
  journal={Journal of Computational and Graphical Statistics},
  volume={27},
  number={2},
  pages={291--301},
  year={2018},
  publisher={Taylor \& Francis}
}

@article{bombari2025better,
  title={Better Rates for Private Linear Regression in the Proportional Regime via Aggressive Clipping},
  author={Bombari, Simone and Seroussi, Inbar and Mondelli, Marco},
  journal={arXiv preprint arXiv:2505.16329},
  year={2025}
}

@inproceedings{kuleszamean,
  title={Mean Estimation in the Add-Remove Model of Differential Privacy},
  author={Kulesza, Alex and Suresh, Ananda Theertha and Wang, Yuyan},
  booktitle={Forty-first International Conference on Machine Learning},
year = {2024}
}

@article{dwork2014algorithmic,
  title={The algorithmic foundations of differential privacy},
  author={Dwork, Cynthia and Roth, Aaron},
  journal={Foundations and trends{\textregistered} in theoretical computer science},
  volume={9},
  number={3--4},
  pages={211--407},
  year={2014},
  publisher={Now Publishers, Inc.}
}

@article{wang2019per,
  title={Per-instance differential privacy},
  author={Wang, Yu-Xiang},
  journal={Journal of Privacy and Confidentiality},
  volume={9},
  number={1},
  year={2019}
}

@misc{biostars,
  author       = {Kevin Blighe},
  title        = {Tutorial: Produce {PCA} bi-plot for 1000 {G}enomes Phase III - Version 2},
  year         = {2018},
  howpublished = {\url{https://www.biostars.org/p/335605/}},
  note         = {Accessed: 2025-09-23}
}

@article{10002015global,
  title={A global reference for human genetic variation},
  author={{1000 Genomes Project Consortium}},
  journal={Nature},
  volume={526},
  number={7571},
  pages={68},
  year={2015}
}

@article{zhong2022empirical,
  title={Empirical Bayes {PCA} in high dimensions},
  author={Zhong, Xinyi and Su, Chang and Fan, Zhou},
  journal={Journal of the Royal Statistical Society Series B: Statistical Methodology},
  volume={84},
  number={3},
  pages={853--878},
  year={2022},
  publisher={Oxford University Press}
}

@article{homer2008resolving,
  title={Resolving individuals contributing trace amounts of {DNA} to highly complex mixtures using high-density {SNP} genotyping microarrays},
  author={Homer, Nils and Szelinger, Szabolcs and Redman, Margot and Duggan, David and Tembe, Waibhav and Muehling, Jill and Pearson, John V and Stephan, Dietrich A and Nelson, Stanley F and Craig, David W},
  journal={PLoS genetics},
  volume={4},
  number={8},
  pages={e1000167},
  year={2008},
  publisher={Public Library of Science San Francisco, USA}
}

@book{le2012asymptotic,
  title={Asymptotic methods in statistical decision theory},
  author={Le Cam, Lucien},
  year={2012},
  publisher={Springer Science \& Business Media}
}

@inproceedings{ge2021efficient,
  title={Efficient sampling from the Bingham distribution},
  author={Ge, Rong and Lee, Holden and Lu, Jianfeng and Risteski, Andrej},
  booktitle={Algorithmic Learning Theory},
  pages={673--685},
  year={2021},
  organization={PMLR}
}

@book{tao2012topics,
  title={Topics in random matrix theory},
  author={Tao, Terence},
  volume={132},
  year={2012},
  publisher={American Mathematical Soc.}
}

@inproceedings{montanari2018mean,
  title={Mean field asymptotics in high-dimensional statistics: From exact results to efficient algorithms},
  author={Montanari, Andrea},
  booktitle={Proceedings of the International Congress of Mathematicians: Rio de Janeiro 2018},
  pages={2973--2994},
  year={2018},
  organization={World Scientific}
}

@inproceedings{papernot2018scalable,
  title={Scalable Private Learning with {PATE}},
  author={Papernot, Nicolas and Song, Shuang and Mironov, Ilya and Raghunathan, Ananth and Talwar, Kunal and Erlingsson, Ulfar},
  booktitle={International Conference on Learning Representations},
  year={2018}
}

@inproceedings{tran2025spectral,
  author    = {Tran, Phuc and Vishnoi, Nisheeth K. and Vu, Van H.},
  title     = {Spectral Perturbation Bounds for Low-Rank Approximation with Applications to Privacy},
  booktitle = {Advances in Neural Information Processing Systems (NeurIPS)},
  year      = {2025},
  url       = {https://openreview.net/forum?id=F0JzotXYgC}
}

@misc{https://doi.org/10.21231/gnt1-hw21,
  doi = {10.21231/GNT1-HW21},
  url = {https://chtc.cs.wisc.edu/},
  author = {{Center for High Throughput Computing}},
  title = {Center for High Throughput Computing},
  publisher = {Center for High Throughput Computing},
  year = {2006}
}

@article{mitra2014multivariate,
  title={Multivariate analysis of nonparametric estimates of large correlation matrices},
  author={Mitra, Ritwik and Zhang, Cun-Hui},
  journal={arXiv preprint arXiv:1403.6195},
  year={2014}
}

@article{liu2012high,
  title={High-dimensional semiparametric {G}aussian copula graphical models},
  author={Liu, Han and Han, Fang and Yuan, Ming and Lafferty, John and Wasserman, Larry},
  journal={The Annals of Statistics},
  volume={40},
  number={4},
  pages={2293},
  year={2012},
  publisher={Institute of Mathematical Statistics}
}

@article{xue2012regularized,
  title={Regularized rank-based estimation of high-dimensional nonparanormal graphical models},
  author={Xue, Lingzhou and Zou, Hui},
  journal={Annals of Statistics},
  volume={40},
  number={5},
  pages={2541--2571},
  year={2012},
  publisher={Institute of Mathematical Statistics}
}

@article{cape2024robust,
  title={Robust spectral clustering with rank statistics},
  author={Cape, Joshua and Yu, Xianshi and Liao, Jonquil Z},
  journal={Journal of Machine Learning Research},
  volume={25},
  number={398},
  pages={1--81},
  year={2024}
}

@inproceedings{d2025tight,
  author    = {d'Orsi, Tommaso and Novikov, Gleb},
  title     = {Tight Differentially Private {PCA} via Matrix Coherence},
  booktitle = {Proceedings of the 2026 Annual {ACM}-{SIAM} Symposium on Discrete Algorithms (SODA)},
  pages     = {10--51},
  year      = {2026},
  doi       = {10.1137/1.9781611978971.2}
}

@inproceedings{papernot2016semi,
  author    = {Papernot, Nicolas and Abadi, Mart{\'\i}n and Erlingsson, {\'U}lfar and Goodfellow, Ian and Talwar, Kunal},
  title     = {Semi-supervised Knowledge Transfer for Deep Learning from Private Training Data},
  booktitle = {International Conference on Learning Representations (ICLR)},
  year      = {2017},
  url       = {https://openreview.net/forum?id=HkwoSDPgg}
}

@article{pandey2025infinitely,
  title={Infinitely divisible privacy and beyond I: resolution of the $ s^2= 2k $ conjecture},
  author={Pandey, Aaradhya and Maleki, Arian and Kulkarni, Sanjeev},
  journal={arXiv preprint arXiv:2512.00734},
  year={2025}
}

@misc{ANES2024Pilot,
  author       = {{American National Election Studies}},
  title        = {{ANES} 2024 Pilot Study},
  year         = {2024},
  howpublished = {Dataset and documentation},
  note         = {March 19, 2024 version},
  institution  = {American National Election Studies},
  url          = {https://electionstudies.org/data-center/2024-pilot-study/}
}

@article{jagielski2020auditing,
  title={Auditing differentially private machine learning: How private is private sgd?},
  author={Jagielski, Matthew and Ullman, Jonathan and Oprea, Alina},
  journal={Advances in Neural Information Processing Systems},
  volume={33},
  pages={22205--22216},
  year={2020}
}

@inproceedings{nasr2021adversary,
  title={Adversary instantiation: Lower bounds for differentially private machine learning},
  author={Nasr, Milad and Song, Shuang and Thakurta, Abhradeep and Papernot, Nicolas and Carlini, Nicholas},
  booktitle={2021 IEEE Symposium on security and privacy (SP)},
  pages={866--882},
  year={2021},
  organization={IEEE}
}

@article{lu2022general,
  title={A general framework for auditing differentially private machine learning},
  author={Lu, Fred and Munoz, Joseph and Fuchs, Maya and LeBlond, Tyler and Zaresky-Williams, Elliott and Raff, Edward and Ferraro, Francis and Testa, Brian},
  journal={Advances in Neural Information Processing Systems},
  volume={35},
  pages={4165--4176},
  year={2022}
}

@article{steinke2023privacy,
  title={Privacy auditing with one (1) training run},
  author={Steinke, Thomas and Nasr, Milad and Jagielski, Matthew},
  journal={Advances in Neural Information Processing Systems},
  volume={36},
  pages={49268--49280},
  year={2023}
}

@article{dwork2017exposed,
  title={Exposed! a survey of attacks on private data},
  author={Dwork, Cynthia and Smith, Adam and Steinke, Thomas and Ullman, Jonathan},
  journal={Annual Review of Statistics and Its Application},
  volume={4},
  pages={61--84},
  year={2017},
  publisher={Annual Reviews}
}

@article{lin2025high,
  title={High-Dimensional Privacy-Utility Dynamics of Noisy Stochastic Gradient Descent on Least Squares},
  author={Lin, Shurong and Kolaczyk, Eric D and Smith, Adam and Paquette, Elliot},
  journal={arXiv preprint arXiv:2510.16687},
  year={2025}
}

@article{dong2021central,
  title={A central limit theorem for differentially private query answering},
  author={Dong, Jinshuo and Su, Weijie and Zhang, Linjun},
  journal={Advances in neural information processing systems},
  volume={34},
  pages={14759--14770},
  year={2021}
}

@article{bai2012sample,
  title={On sample eigenvalues in a generalized spiked population model},
  author={Bai, Zhidong and Yao, Jianfeng},
  journal={Journal of Multivariate Analysis},
  volume={106},
  pages={167--177},
  year={2012},
  publisher={Elsevier}
}

@article{dey2019asymptotic,
  title={Asymptotic properties of principal component analysis and shrinkage-bias adjustment under the generalized spiked population model},
  author={Dey, Rounak and Lee, Seunggeun},
  journal={Journal of Multivariate Analysis},
  volume={173},
  pages={145--164},
  year={2019},
  publisher={Elsevier}
}

@article{benaych2011eigenvalues,
  title={The eigenvalues and eigenvectors of finite, low rank perturbations of large random matrices},
  author={Benaych-Georges, Florent and Nadakuditi, Raj Rao},
  journal={Advances in Mathematics},
  volume={227},
  number={1},
  pages={494--521},
  year={2011},
  publisher={Elsevier}
}

@article{cheng2024dimension,
  title={Dimension free ridge regression},
  author={Cheng, Chen and Montanari, Andrea},
  journal={The Annals of Statistics},
  volume={52},
  number={6},
  pages={2879--2912},
  year={2024},
  publisher={Institute of Mathematical Statistics}
}

@article{wei2025asymptotic,
  title={Asymptotic Matrix-Variate Bingham Distributions and Privacy-Preserving Principal Component Analysis},
  author={Wei, Lu and Ji, Tianxi},
  journal={IEEE Transactions on Signal Processing},
  volume={73},
  pages={1966--1978},
  year={2025},
  publisher={IEEE}
}

@inproceedings{sudakov1978typical,
  title={Typical distributions of linear functionals in finite-dimensional spaces of higher dimension},
  author={Sudakov, Vladimir Nikolaevich},
  booktitle={Doklady Akademii Nauk},
  volume={243},
  number={6},
  pages={1402--1405},
  year={1978},
  organization={Russian Academy of Sciences}
}

@article{power2026complexity,
  title={A Complexity Bound for the Kent-Ganeiber-Mardia Sampler for the Bingham Distribution},
  author={Power, Sam},
  journal={arXiv preprint arXiv:2608.20475},
  year={2026}
}
